\documentclass[11pt]{amsart}
\usepackage[T1]{fontenc}
\usepackage[utf8]{inputenc}
\usepackage{lmodern}
\usepackage{microtype}
\usepackage{amsmath,amssymb,mathtools,mathrsfs,bm}
\usepackage{booktabs,longtable,tabularx,array}
\usepackage{enumitem}
\usepackage[margin=1in]{geometry}
\usepackage{tikz-cd}
\usepackage[colorlinks=true,linkcolor=blue!60!black,citecolor=blue!60!black,urlcolor=blue!60!black]{hyperref}
\usepackage[nameinlink,capitalise,noabbrev]{cleveref}
\allowdisplaybreaks

\theoremstyle{plain}
\newtheorem{theorem}{Theorem}[section]
\newtheorem{proposition}[theorem]{Proposition}
\newtheorem{lemma}[theorem]{Lemma}
\newtheorem{corollary}[theorem]{Corollary}
\newtheorem{conjecture}[theorem]{Conjecture}

\newtheorem{inputtheorem}[theorem]{Geometric input}
\newtheorem*{thmA}{Theorem A}
\newtheorem*{thmB}{Theorem B}
\newtheorem*{thmC}{Theorem C}
\newtheorem*{thmD}{Theorem D}
\newtheorem*{thmE}{Theorem E}
\theoremstyle{definition}
\newtheorem{definition}[theorem]{Definition}
\newtheorem{construction}[theorem]{Construction}
\newtheorem{convention}[theorem]{Convention}
\newtheorem{example}[theorem]{Example}
\newtheorem{question}[theorem]{Question}
\theoremstyle{remark}
\newtheorem{remark}[theorem]{Remark}
\newtheorem{warning}[theorem]{Warning}

\crefname{theorem}{Theorem}{Theorems}
\crefname{proposition}{Proposition}{Propositions}
\crefname{lemma}{Lemma}{Lemmas}
\crefname{corollary}{Corollary}{Corollaries}
\crefname{hypothesis}{Hypothesis}{Hypotheses}
\crefname{inputtheorem}{Geometric input}{Geometric inputs}
\crefname{definition}{Definition}{Definitions}
\crefname{construction}{Construction}{Constructions}
\crefname{convention}{Convention}{Conventions}
\crefname{example}{Example}{Examples}
\crefname{question}{Question}{Questions}
\crefname{remark}{Remark}{Remarks}
\crefname{warning}{Warning}{Warnings}
\crefname{equation}{}{}
\crefname{section}{Section}{Sections}
\crefname{appendix}{Appendix}{Appendices}

\newcommand{\C}{\mathbb C}
\newcommand{\Q}{\mathbb Q}
\newcommand{\A}{\mathbb A}
\newcommand{\N}{\mathbb N}
\newcommand{\Z}{\mathbb Z}
\newcommand{\PP}{\mathbb P}
\newcommand{\cO}{\mathcal O}
\newcommand{\cP}{\mathcal P}
\newcommand{\cR}{\mathcal R}
\newcommand{\cE}{\mathcal E}
\newcommand{\cL}{\mathcal L}
\newcommand{\cK}{\mathcal K}
\newcommand{\cM}{\mathcal M}
\newcommand{\cH}{\mathcal H}
\newcommand{\cQ}{\mathcal Q}
\newcommand{\cS}{\mathcal S}

\newcommand{\cZ}{\mathcal Z}
\newcommand{\cI}{\mathcal I}
\newcommand{\cU}{\mathcal U}

\newcommand{\cT}{\mathcal T}
\newcommand{\cV}{\mathcal V}
\newcommand{\cW}{\mathcal W}
\newcommand{\cF}{\mathcal F}
\newcommand{\cG}{\mathcal G}
\newcommand{\cB}{\mathcal B}
\newcommand{\cN}{\mathcal N}
\newcommand{\cX}{\mathcal X}
\newcommand{\sP}{\mathscr P}
\newcommand{\sA}{\mathscr A}
\newcommand{\Hecke}{\mathscr H}
\newcommand{\scrD}{\mathscr D}
\newcommand{\mfS}{\mathfrak S}
\newcommand{\mfM}{\mathfrak M}
\newcommand{\mfN}{\mathfrak N}

\newcommand{\mfm}{\mathfrak m}
\newcommand{\mfa}{\mathfrak a}
\newcommand{\mfb}{\mathfrak b}
\newcommand{\frh}{\mathfrak h}
\newcommand{\Hilb}{\operatorname{Hilb}}
\newcommand{\Quot}{\operatorname{Quot}}
\newcommand{\Gr}{\operatorname{Gr}}
\newcommand{\Fl}{\operatorname{Fl}}
\newcommand{\Tot}{\operatorname{Tot}}
\newcommand{\SpecOp}{\operatorname{Spec}}

\newcommand{\ProjOp}{\operatorname{Proj}}
\newcommand{\SymOp}{\operatorname{Sym}}
\newcommand{\Tr}{\operatorname{Tr}}
\newcommand{\End}{\operatorname{End}}
\newcommand{\Hm}{\operatorname{Hom}}
\newcommand{\Hom}{\operatorname{Hom}}
\newcommand{\cHom}{\mathcal{H}\!om}
\newcommand{\cEnd}{\mathcal{E}\!nd}
\newcommand{\Ex}{\operatorname{Ext}}
\newcommand{\Ext}{\operatorname{Ext}}
\newcommand{\cExt}{\mathcal{E}\!xt}
\newcommand{\Tor}{\operatorname{Tor}}

\newcommand{\RcHom}{R\mathcal{H}\!om}
\newcommand{\RGamma}{R\Gamma}
\newcommand{\Ltensor}{\overset{\mathbf L}{\otimes}}

\newcommand{\soc}{\operatorname{soc}}
\newcommand{\rk}{\operatorname{rk}}
\newcommand{\wt}{\operatorname{wt}}
\newcommand{\Supp}{\operatorname{Supp}}
\newcommand{\codim}{\operatorname{codim}}
\newcommand{\length}{\operatorname{length}}
\newcommand{\Span}{\operatorname{Span}}
\newcommand{\Perf}{\operatorname{Perf}}
\newcommand{\Coh}{\operatorname{Coh}}
\newcommand{\id}{\mathrm{id}}
\newcommand{\eps}{\varepsilon}
\newcommand{\ev}{\varepsilon}
\newcommand{\grf}{\operatorname{gr}}
\newcommand{\gr}{\operatorname{gr}}
\newcommand{\ogr}{\operatorname{ogr}}
\newcommand{\chf}{\operatorname{ch}}
\newcommand{\ch}{\operatorname{ch}}
\newcommand{\triv}{\mathrm{triv}}
\newcommand{\Bal}{\operatorname{Bal}}
\newcommand{\Cat}{\operatorname{Cat}}
\newcommand{\qbinom}[2]{\genfrac{[}{]}{0pt}{}{#1}{#2}}
\newcommand{\Cone}{\operatorname{Cone}}
\newcommand{\Pic}{\operatorname{Pic}}

\newcommand{\im}{\operatorname{im}}
\newcommand{\coker}{\operatorname{coker}}
\newcommand{\Ann}{\operatorname{Ann}}
\newcommand{\grade}{\operatorname{grade}}
\newcommand{\htOp}{\operatorname{ht}}
\newcommand{\pd}{\operatorname{pd}}
\newcommand{\depth}{\operatorname{depth}}
\newcommand{\Refl}{\operatorname{Ref}}
\newcommand{\Hilbs}{\operatorname{Hilb}_{q,t}}
\newcommand{\BN}{\operatorname{BN}}
\newcommand{\ord}{\operatorname{ord}}

\newcommand{\schur}{\mathbb S}

\newcommand{\degh}{\deg_{\mathbf h}}
\newcommand{\degpol}{\deg_{\mathrm{pol}}}
\newcommand{\degt}{\deg_{\mathrm{tot}}}
\newcommand{\dinv}{\operatorname{dinv}}
\newcommand{\inv}{\operatorname{inv}}
\newcommand{\wtH}{\widetilde H}
\newcommand{\sK}{\mathscr K}

\title[Asymptotic $q,t$-Fuss--Catalan numbers for type $B$]{Asymptotic $q,t$-Fuss--Catalan numbers
  for type $B$}

\author{Alexei Oblomkov}
\address{
A.~Oblomkov\\
Department of Mathematics and Statistics\\
University of Massachusetts at Amherst\\
Lederle Graduate Research Tower\\
710 N. Pleasant Street\\
Amherst, MA 01003 USA
}
\email{oblomkov@math.umass.edu}

\begin{document}
\begin{abstract}
Let $W=W(B_n)$ act diagonally on $\frh\oplus\frh^*$, let
$S=\C[\frh\oplus\frh^*]$, let $J\subseteq S$ be the ideal
generated by the $W$-alternating polynomials and \(\mathfrak{m}_S\) is the maximal ideal of the origin.
For sufficiently large \(m\) we compute \(q,t\)-Fuss-Catalan polynomial $\mathrm{Cat}^{(m)}(B_n;q,t):=\Hilbs\left(\frac{J^m}{\mathfrak{m}_S J^m}\right)_{\!\det\text{-part}}\) and imply \(\mathrm{Cat}^{(m)}(B_n;1,1)=\binom{n(m+1)}{n}\).
For proofs,  we work with the $\Gamma$-equivariant Hilbert scheme
$Y_n=n\Gamma$-$\Hilb(\A^2)$, $\Gamma=\mu_2$  and  Haiman-type Koszul complex that defines
the punctual locus of \(Y_n\). Our formula for \(\mathrm{Cat}^{(m)}(B_n;q,t)\) is derived from
a localization computaion for the  Haiman-type Koszul complex.
\end{abstract}
\maketitle

\section{Introduction}\label{sec:intro}

\subsection{The type \(A\) vs \(B\)}

Let $W$ be a finite real reflection group acting on its reflection representation
$\frh$ and diagonally on $\frh\oplus\frh^*$, and let
\[
 S=\C[\frh\oplus\frh^*],\qquad A=S^W,\qquad \mfm=A_+ .
\]
Let $S^\eps=\{f\in S:\ w(f)=\det(w)f\}$ be the space of alternating polynomials and
$J=S\cdot S^\eps$ the ideal it generates.  Every alternating polynomial vanishes on
every reflection diagonal $\{\alpha_s=\alpha_s^\vee=0\}$, so
\begin{equation}\label{eq:JI-inclusion}
 J\ \subseteq\ I:=\bigcap_{s\in\Refl(W)}(\alpha_s,\alpha_s^\vee),
 \qquad\text{and}\qquad
 J^m\subseteq I^m\subseteq I^{(m)}:=\bigcap_{s\in\Refl(W)}(\alpha_s,\alpha_s^\vee)^m .
\end{equation}
For $W=S_n$ Haiman's polygraph theorem \cite{HaimanNfact,Haiman} implies  the remarkable
equalities $J^m=I^m=I^{(m)}$ for all $m$ as well as 
\[
 \mathrm{Cat}^{(m)}(S_n;q,t)
 =\chi_{\,\C^*_q\times\C^*_t}\bigl(\cO_{Z_n}(m)\bigr),
\]
where $Z_n\subset\Hilb^n(\C^2)$ is the punctual Hilbert scheme
\cite{HaimanCatalan,Haiman}.  Stump \cite{Stump} defined, for every finite reflection
group,
\begin{equation}\label{eq:stump-def}
 \mathrm{Cat}^{(m)}(W;q,t)
 :=\Hilbs\Bigl(\,e_\eps\, DR^{(m)}(W)\Bigr)
 =\Hilbs\Bigl(\frac{J^m}{\mathfrak{m}_S J^m}\Bigr)_{\!\det\text{-part}},
\end{equation}
the bigraded Hilbert series of the determinantal component of the generalized diagonal
coinvariants, equivalently
of the minimal generating space of $J^m$; and he
conjectured that its dimension is the Fuss--Catalan number of $W$ and that it matches
the graded character of the corresponding finite-dimensional spherical Cherednik module (see details below).

The Fuss--Catalan numbers in Stump's conjecture are the numbers
$\mathrm{Cat}^{(m)}(W)=\prod_{i=1}^{\ell}\frac{d_i+mh}{d_i}$, where $\ell$ is the rank of
$W$, $d_1,\dots,d_\ell$ are its degrees and $h$ is its Coxeter number, and they are the
subject of an extensive combinatorial theory, the Coxeter--Catalan combinatorics.  For
$m=1$ they count the noncrossing partitions of Reiner \cite{Reiner} and Bessis
\cite{Bessis}, the clusters of Fomin and Zelevinsky \cite{FominZelevinsky}, the
Coxeter-sortable elements of Reading \cite{Reading}, and the nonnesting partitions
and the dominant regions of the Catalan arrangement \cite{Athanasiadis04}.  For
general $m$ they count the $m$-divisible noncrossing partitions of Armstrong
\cite{Armstrong}, the facets of the generalized cluster complexes of Fomin and Reading
\cite{FominReading}, and the dominant regions of the extended Catalan arrangements
and the $m$-nonnesting partitions of Athanasiadis \cite{Athanasiadis04,Athanasiadis05}.
The recent memoir of Stump, Thomas and Williams \cite{Cataland} gives a uniform
framework for the noncrossing side of this theory, and the rational Catalan
combinatorics, in which the parameter $m+\frac1h$ is replaced by an arbitrary rational
number, is developed in \cite{ARW,GMV}.  The $q,t$-Catalan numbers of type $A$ were
introduced by Haiman \cite{HaimanConj} and by Garsia and Haiman \cite{GarsiaHaiman},
and the book of Haglund \cite{HaglundBook} is a reference for their combinatorics and
for the space of diagonal harmonics.  For $W=B_n$ one has
$\mathrm{Cat}^{(m)}(B_n)=\binom{n(m+1)}{n}$, the number of lattice paths with unit north
and east steps from $(0,0)$ to $(mn,n)$, and we discuss the type $B$ lattice path
models in \cref{sec:open-problems,sec:shuffle,sec:labelled}.

For groups other than $S_n$ the inclusions \eqref{eq:JI-inclusion} are strict in
general: Kivinen \cite[Ex.~B.2]{Kivinen} computes for $W(B_3)$ that $J\subsetneq I$
while $eJ=eI$ for the symmetrizing idempotent $e$.  This suggests that the correct
general statements are \emph{spherical}.  Kivinen \cite{Kivinen} attaches to
these data three ``Hilbert-scheme'' models,
\begin{equation}\label{eq:three-models}
 X_{\mathfrak g,\mathrm{sgn}}=\ProjOp\bigoplus_{m\ge0}(\Delta S^\eps)^m,\quad
 X_{\mathfrak g,\mathrm{diag}}=\ProjOp\bigoplus_{m\ge0}\bigl(e(\Delta I)\bigr)^m,\quad
 X_{\mathfrak g,\mathrm{symb}}=\ProjOp\bigoplus_{m\ge0}e\bigl(\Delta^mI^{(m)}\bigr),
\end{equation}
where $\Delta=\prod_{\alpha\in\Phi^+}\alpha$
(\cite[Defs.~3.1, 3.4, 3.6]{Kivinen}); the symbolic model is normal
\cite[Thm.~3.8]{Kivinen} and is the commutative degeneration of a $\Z$-algebra built
from shift bimodules of rational Cherednik algebras
\cite[Thm.~4.3, Cor.~4.5]{Kivinen}.  The three models are connected by canonical
birational morphisms
$X_{\mathrm{symb}}\to X_{\mathrm{diag}}\to X_{\mathrm{sgn}}$, and it is not clear in
general whether these are isomorphisms.

This paper solves the type-$B$ problem asymptotically in the Fuss parameter \(m\), with
explicit $q,t$-characters.

\subsection{Main results}\label{sec:main-results}

Fix $W=W(B_n)=\mu_2\wr S_n$, with $\Refl(W)$ consisting of the $n$ coordinate
reflections and the $2\binom n2$ signed transpositions; $h=2n$ is the Coxeter number,
$\deg\Delta=n^2$.  Write
$e_m=\frac1{|W|}\sum_w\det(w)^mw$, so $e_0=e$ and $e_1=e_-$.  The phrase ``$m\gg0$''
always means: for each fixed $n$ there is a (non-effective) $m_0(n)$ such that the
statement holds for all $m\ge m_0(n)$.

Our first result settles the comparison of the three models.

\begin{thmA}[Asymptotic spherical collapse; Theorem~\ref{thm:collapse}]
For every fixed $n$ and all $m\gg0$,
\[
 e_mJ^m\;=\;e_mI^m\;=\;e_mI^{(m)} .
\]
Consequently the three spherical models \eqref{eq:three-models} coincide for
$\mathfrak g$ of type $B_n/C_n$:
$X_{\mathfrak g,\mathrm{symb}}\cong X_{\mathfrak g,\mathrm{diag}}\cong
X_{\mathfrak g,\mathrm{sgn}}$.
\end{thmA}

The $q,t$-characters are expressed by an explicit localization polynomial.  Let
$\Bal(2n)$ denote the partitions $\lambda\vdash2n$ with empty $2$-core (equivalently,
with $n$ boxes of each checkerboard parity), and for $\lambda\in\Bal(2n)$ define, with
boxes $(a,b)\in\N^2$ in English notation and arm and leg lengths
$a_\lambda(s),l_\lambda(s)$,
\begin{equation}\label{eq:statistics}
\begin{gathered}
 B_\lambda=\sum_{(a,b)\in\lambda}q^at^b,\qquad
 \ell_\lambda=\prod_{\substack{(a,b)\in\lambda\\ a+b\ \mathrm{odd}}}q^at^b,\qquad
 \Pi^{(0)}_\lambda=\prod_{\substack{(a,b)\in\lambda\setminus\{(0,0)\}\\
 a+b\ \mathrm{even}}}\bigl(1-q^at^b\bigr),\\
 D_\lambda=\prod_{\substack{s\in\lambda\\ a_\lambda(s)+l_\lambda(s)\ \mathrm{odd}}}
 \bigl(1-q^{a_\lambda(s)+1}t^{-l_\lambda(s)}\bigr)
 \bigl(1-q^{-a_\lambda(s)}t^{l_\lambda(s)+1}\bigr),
\end{gathered}
\end{equation}
and set
\begin{equation}\label{eq:F-def}
\frac{ F_{n,m}(q,t)}{(1-q)(1-t)}=\sum_{\lambda\in\Bal(2n)}
 \frac{\ell_\lambda^{\,m}\,B_\lambda\Pi^{(0)}_\lambda}{D_\lambda},
 \quad
 P_{n,m}(q,t)=\tfrac12\Bigl(F_{n,m}(q,t)+(-1)^{nm}F_{n,m}(-q,-t)\Bigr).
\end{equation}

\begin{thmB}[$q,t$-Fuss--Catalan character; Theorem~\ref{thm:qtCat}]
For every fixed $n$ and all $m\gg0$,
\[
 \mathrm{Cat}^{(m)}(B_n;q,t)\;=\;P_{n,m}(q,t).
\]
\end{thmB}

\begin{thmC}[Principal specialization; Theorem~\ref{thm:principal}]
For all $n\ge1$ and $m\ge0$, with $M=n(m+1)$ and $Q=q^2$,
\[
 F_{n,m}(q,q^{-1})
 =q^{-mn^2}\qbinom{M}{n}_{Q}
 -q^{1-mn^2}\,[m]_{Q^n}\,\qbinom{M}{n-1}_{Q},
 \qquad [m]_{Q^n}=\tfrac{1-Q^{nm}}{1-Q^n},
\]
the two summands having opposite exponent parities.  Hence
\[
 q^{mn^2}\,P_{n,m}(q,q^{-1})=\qbinom{n(m+1)}{n}_{q^2},
 \qquad
 P_{n,m}(1,1)=\binom{n(m+1)}{n},
\]
and the opposite-parity part $N_{n,m}:=P_{n,m}-F_{n,m}$ satisfies
$N_{n,m}(1,1)=m\binom{n(m+1)}{n-1}$ and
$F_{n,m}(1,1)=\frac1{nm+1}\binom{n(m+1)}{n}$.
In particular $\mathrm{Cat}^{(m)}(B_n;q,q^{-1})=q^{-mn^2}\qbinom{n(m+1)}{n}_{q^2}$ and
$\mathrm{Cat}^{(m)}(B_n;1,1)=\binom{n(m+1)}n$ for $m\gg0$.
\end{thmC}

Let $H_c$ be the equal-parameter rational Cherednik algebra of $W(B_n)$,
$U_c=eH_ce$, and let $L_c(\triv)$ be the finite-dimensional simple at
$c=\frac1{2n}+m$.  Set
$\mfS^{\mathrm{symb}}_{n,m}=e_mI^{(m)}/\mfm\,e_mI^{(m)}$.

\begin{thmD}[Cherednik identification; Theorem~\ref{thm:cherednik}]
For every fixed $n$ and all
$m\gg0$ the canonical filtered shift-bimodule surjection is an isomorphism of
Gordon-bigraded vector spaces
\[
 \mfS^{\mathrm{symb}}_{n,m}\ \xrightarrow{\ \sim\ }\
 \grf_F\, eL_{\frac1{2n}+m}(\triv),
 \qquad\text{and}\qquad
 P_{n,m}(q,t)=\chf^{G}_{q,t}\,\grf_F\, eL_{\frac1{2n}+m}(\triv),
\]
where $F$ is the iterated tensor-product filtration of \cref{def:bigrading}.
\end{thmD}

All of the above rests on the following geometric theorem, which occupies the bulk of
the paper.  Let $\Gamma=\mu_2\subset SL_2(\C)$ act on $\A^2$ by
$(x,y)\mapsto(-x,-y)$, let $Y_n=n\Gamma\text{-}\Hilb(\A^2)$ be the corresponding smooth
$2n$-dimensional Nakajima quiver variety, $\eta:\cU_n\to Y_n$ the universal
$\Gamma$-cluster, $\cP_n=\eta_*\cO_{\cU_n}=\cR_0\oplus\cR_1$ the rank-$2n$ cluster
algebra bundle, $\cT_n=\ker(\Tr_{\cR_0})$ the trace-zero even part, and
\[
 \cE_n=\cT_n\oplus\cO_q\oplus\cO_t,\qquad
 \cK_n^{-a}=\cP_n\otimes{\textstyle\bigwedge^a}\cE_n,\qquad
 \cI_n=\cP_n/J_n,
\]
the type-$B$ analogue of the Haiman--Garsia Koszul complex
(\cref{sec:haiman}); here $J_n$ is the ideal
generated by $\cT_n$ and the universal sections $x,y$.  Let $\rho:Y_n\to Q_n$ be the
projective symplectic resolution of the wall quiver variety
$Q_n\cong\Quot^n_X(R_1)_{\mathrm{red}}\cong X_{\mathfrak g,\mathrm{symb}}$
(\cref{sec:models}), $\cL_n=\det\cR_1=\rho^*\cO_{Q_n}(1)$, and
$\cZ_n^{\mathrm{symb}}=Q_n\times_{S_n}\{0\}$ the scheme-theoretic symbolic punctual
fiber, $S_n=(\frh\oplus\frh^*)/W$.

\begin{thmE}[Geometric theorem; Theorems~\ref{thm:two-term}, \ref{thm:dual-socle}, \ref{thm:vanishing}, \ref{thm:SD}]
\leavevmode
\begin{enumerate}[label=\textup{(\arabic*)}]
\item The complex $\cK_n^\bullet$ has exactly two nonzero homology sheaves,
$\cH^0(\cK_n^\bullet)=\cI_n$ and $\cM_n:=\cH^{-1}(\cK_n^\bullet)$; and
$\cI_n\cong\cO_{F_n}$, the structure sheaf of the scheme-theoretic punctual fiber.
\item There is a canonical $T$-equivariant isomorphism
$\cM_n\cong\cExt^n_{\cO_{Y_n}}(\cS_n,\cO_{Y_n})\otimes\det\cE_n$, where
$\cS_n=\cP_n^\vee/J_n\cP_n^\vee$ is the dual-socle sheaf.
\item $R\rho_*\cM_n^+=0$, where $\cM_n^\pm$ are the central-parity summands.
\item \textup{(}Symbolic descent\textup{)} For every $d\in\Z$,
\[
 R\rho_*\bigl(\cL_n^{\otimes d}\otimes\cI_n\bigr)\;\simeq\;
 \cO_{\cZ_n^{\mathrm{symb}}}(d).
\]
\end{enumerate}
\end{thmE}

A feature with no type-$A$ counterpart is that $\cK_n^\bullet$ is \emph{not} a
resolution: it has $n+1$ generators against expected codimension $n$, and the excess
sheaf $\cM_n$, in particular its even part $\cM_n^+$, is nonzero in every rank
$n\ge2$ (\cref{cor:trace-CM}).  Part (3) says that this excess is invisible to
the wall.  We prove it in \cref{sec:hecke} through a
rank-one Hecke correspondence $Y_n\xleftarrow{p}\Hecke\xrightarrow{q}
\mfM_\theta((n-1,n),(1,0))$; the Grassmannian wall fibers $F_J\cong\Gr(k,2k-1)$
and their Borel--Weil--Bott vanishing (\cref{sec:wall}) reappear as the
fixed-point shadow of this proof, but are not needed for it.  Part (4) is proved
in \cref{sec:SD} on the formal neighbourhood of each wall fiber, where, by the
Bellamy--Craw normal form, the even cluster bundle is a trivial summand plus the
tautological quotient bundle of a cotangent Grassmannian: Borel--Weil--Bott makes
the positive Haiman complex termwise acyclic for the pushforward and its first
differential computable, an elementary normalized-trace lemma identifies the
cokernel with the punctual quotient, and part (3) forces concentration in degree
zero.  



The paper is organized as follows.  \Cref{sec:models} fixes notation and identifies
the wall variety and its polarization.  \Cref{sec:haiman} introduces the Koszul
complex and proves parts (1) and (2) of Theorem~E.  \Cref{sec:wall} describes the
wall fibers, the local model at a wall point and the Brill--Noether thresholds.
\Cref{sec:hecke} proves part (3) and \cref{sec:SD} part (4).
\Cref{sec:localization,sec:cherednik,sec:collapse} deduce
Theorems C, B, D and A, in that order.  \Cref{sec:shuffle} formulates a lattice
path model for the polynomials $P_{n,m}$, a type-$B$ analogue of the shuffle
conjecture.  \Cref{sec:labelled} lifts this model to the level of symmetric
functions.  The paths of the rectangle acquire labels, each labelled path
contributes a fundamental quasisymmetric function, and we prove that the resulting
series $\mathcal H^+_{n,m}$ is symmetric and Schur positive, by identifying its
path-by-path pieces with LLT polynomials, and that its $s_n$-coefficient is the
polynomial $\widehat P_{n,m}$ of \cref{sec:shuffle}.  We conjecture that
$\mathcal H^+_{n,m}$ is the bigraded Frobenius character of the $G_n$-invariant
part of Gordon's module $L_{\frac1{2n}+m}(\triv)$, whose spherical part is the
subject of Theorem~D.  At $m=1$ we prove a schedule formula in the sense of Haglund
and Loehr, a sum over the permutations of $\{0,1,\dots,n\}$ of a monomial in $q,t$
times a product of $qt$-integers, whose two extreme sectors are the type-$A$
formula, and we conjecture that the corresponding $(n+1)^n$ monomials form a basis
of the module.

\subsection{Techniques}\label{sec:techniques}

Let us describe the techniques which enter the proofs of the main results.  The
geometric part of the paper rests on the theory of Nakajima quiver varieties for the
framed affine $A_1$ McKay quiver \cite{Nakajima,Nakajima98}.  The space $Y_n$ is the
quiver variety at a generic stability, the wall variety $Q_n$ is the quiver variety at
the wall stability, and $\rho$ is the variation of GIT morphism.  The identification of
$Q_n$ with the reduced Quot scheme of the odd McKay module is due to Craw, Gammelgaard,
Gyenge and Szendr\H{o}i \cite{CGGS,CGGSpunctual}, and its identification with the
symbolic model is due to Kivinen \cite{Kivinen}.  We add to this the identification of
the polarization of $Q_n$ with the sign component of the Procesi bundle of Losev
\cite{LosevProcesi} and of Boixeda Alvarez and Losev \cite{BoixedaLosev}
(\cref{thm:polarization}), which is proved by intersecting with the two extremal
curves of $Y_n$.

The proof of Theorem~E combines several tools.  Parts (1) and (2) follow from a local
cohomology argument, which shows that a perfect complex of Tor-amplitude $[-n-1,0]$
whose cohomology is supported in codimension $n$ has only two homology sheaves, and
from Koszul duality, which converts the second homology sheaf into an Ext sheaf of the
dual-socle sheaf.  The codimension bound is the semismallness of symplectic resolutions
\cite{Kaledin,KaledinPoisson}.  Part (3), the even-excess vanishing, is proved with the
rank-one Hecke correspondence of Gonz\'alez, Gorsky and Simental \cite{GGS}, which
parametrizes a cluster together with an even socle line.  Its two projections are cut
out by regular sections in projective bundles, and the proof uses Bott's formula for
twisted differential forms on projective space, the Hilbert--Burch and Buchsbaum--Rim
resolutions of the even socle locus, and Grothendieck duality for finite morphisms.
Part (4), symbolic descent, uses the \'etale-local normal form of Bellamy and Craw
\cite{BellamyCraw} at a wall point, which exhibits the formal neighbourhood of a wall
fiber as a cotangent Grassmannian times a smooth core, and the Borel--Weil--Bott
theorem on the Grassmannian, applied to positive tensor constructions of the
tautological quotient bundle.  An elementary normalized-trace lemma then computes the
pushforward of the punctual quotient, and the theorem on formal functions passes from
the formal model to $Q_n$.

The characters of Theorems~B and~C are computed with equivariant localization on
$Y_n$, the principal specialization reduces to the finite $q$-binomial theorem, and
Serre vanishing with cohomology and base change turns the Euler characteristic into a
Hilbert series for $m\gg0$.  On the representation-theoretic side we use the shift
bimodules of the rational Cherednik algebra of type $B_n$.  Kivinen's
associated-graded theorem \cite{Kivinen} identifies the associated graded of the
iterated shift bimodule with $e(\Delta^mI^{(m)})$, the Procesi bundles of Losev and of
Boixeda Alvarez and Losev give the surjection from the symbolic model onto the
finite-dimensional module $eL_{\frac1{2n}+m}(\triv)$, and the theorem of Berest,
Etingof and Ginzburg \cite{BEG} gives its dimension.  Theorem~D is then a dimension
count, and Theorem~A follows from Theorem~D, from Stump's generalized Gordon surjection
\cite{GordonDiag,Stump} and from graded Nakayama.


{\bf Acknowledgments:} The author would like to thank E. Gorsky, I. Losev, J. Haglund and O. Kivinen for useful discussions.
Claude-code was used to improving exposition, for reference search and writing python code for testing  conjectures in Sections~\ref{sec:shuffle} and~\ref{sec:labelled}.

\section{Notation and the spaces}\label{sec:models}

In this section we fix the notation which is used throughout the paper, and we collect
the facts about the spaces $Y_n$ and $Q_n$ which we need.  The two spaces are the
Nakajima quiver varieties of the framed affine $A_1$ McKay quiver at a generic
stability and at the wall stability, and the morphism $\rho:Y_n\to Q_n$ between them
is the variation of GIT morphism.  The results of Craw, Gammelgaard, Gyenge and
Szendr\H{o}i \cite{CGGS,CGGSpunctual} and of Kivinen \cite{Kivinen} identify $Q_n$ with
the reduced Quot scheme of the odd McKay module and with the symbolic model
$X_{\mathfrak g,\mathrm{symb}}$, and this is what connects the geometry of the paper
with the ideals $I^{(m)}$ of the introduction.

The main new statement of the section is the identification of the polarization,
\cref{thm:polarization}.  Boixeda Alvarez and Losev construct a partial resolution of
$(\frh\oplus\frh^*)/W$ on which a power of the sign component of the Procesi sheaf is
ample, and they remark that in type $B_n$ it is the wall quiver variety.  We prove
that the sign component is the determinant $\det\cR_1$ of the odd tautological bundle,
by pairing both line bundles with the two extremal curves of $Y_n$.  This is what
allows us to transport the Procesi package of Boixeda Alvarez and Losev to $Q_n$ in
\cref{sec:cherednik}.  In \cref{sec:two-quotients} we also compare the two
finite-dimensional quotients of $e_mJ^m$ which occur in the literature, Stump's
quotient and the quotient by the invariants, since Theorem~B is stated for the former
while the geometry of the paper computes the latter.

\subsection{McKay data}\label{sec:mckay}

Throughout, $\Gamma=\mu_2=\{\pm1\}\subset SL_2(\C)$ acts on $\A^2$ by
$(-1)\cdot(x,y)=(-x,-y)$, and
\[
 R=\C[x,y]=R_0\oplus R_1,\qquad
 R_0=\C[x,y]^\Gamma=\C[u,v,w]/(uv-w^2),\qquad
 R_1=\C[x,y]^{\mathrm{odd}}=R_0x+R_0y,
\]
with $u=x^2$, $v=y^2$, $w=xy$.  The quotient $X=\SpecOp R_0=\A^2/\Gamma$ is the $A_1$
singularity and $\mfm_X=(u,v,w)=R_1R_1$ is its maximal ideal at the origin.

We identify
$\C[x,y]=\SymOp(\frh\oplus\frh^*)$ for the rank-one situation and, for rank $n$,
$S=\C[\frh\oplus\frh^*]=\C[x_1,\dots,x_n,y_1,\dots,y_n]$ with
$W=W(B_n)=\mu_2\wr S_n$ acting by signed permutations, diagonally in $x$'s and $y$'s.
The torus $T=\C^*_q\times\C^*_t$ scales $(x,y)$ with weights $(q,t)$.  The central
element $\gamma=(-1,-1)\in T$ generates $\Gamma$, and for any $T$-equivariant sheaf
$\cF$ on a space with trivial $\gamma$-action we write $\cF=\cF^+\oplus\cF^-$ for the
$\gamma$-parity decomposition.  On local classes, parity is the parity of the total
$(q,t)$-degree.  We write $\Lambda=\C[x,y]\rtimes\Gamma$ for the skew group algebra,
and $S_0$, $S_1$ for its even and odd one-dimensional simple modules at the origin,
hence $x$ and $y$ act by zero on $S_0$ and $S_1$.

For type $B_n$ we normalize the reflection data such that, up to nonzero scalars,
\begin{equation}\label{eq:Bn-ideals}
 I^{(m)}=\bigcap_{i=1}^n(x_i,y_i)^m\cap
 \bigcap_{i<j}(x_i-x_j,\,y_i-y_j)^m\cap
 \bigcap_{i<j}(x_i+x_j,\,y_i+y_j)^m,
 \qquad
 \Delta=\prod_{i=1}^nx_i\prod_{i<j}(x_i^2-x_j^2),
\end{equation}
hence $\deg_q\Delta=n^2$.  We write $A=S^W$, $\mfm=A_+$,
$e_m=\frac{1}{|W|}\sum_w\det(w)^mw$, $e=e_0$, $e_-=e_1$.  Since $W$ is real,
$e(\Delta^mf)=\Delta^m\,e_m(f)$, hence the  multiplication by $\Delta^m$ identifies
\begin{equation}\label{eq:two-normalizations}
 \mfS^{\mathrm{symb}}_{n,m}:=\frac{e_mI^{(m)}}{\mfm\,e_mI^{(m)}}
 \ \xrightarrow[\ \cdot\Delta^m\ ]{\ \sim\ }\
 \widetilde\mfS^{\mathrm{symb}}_{n,m}:=\frac{e(\Delta^mI^{(m)})}{\mfm\,e(\Delta^mI^{(m)})},
\end{equation}
a bigrading shift by $q^{mn^2}$.

\subsection{The generic and wall quiver varieties}\label{sec:quiver}

Let $Y_n=n\Gamma\text{-}\Hilb(\A^2)$ be the $\Gamma$-equivariant Hilbert scheme of
clusters $Z\subset\A^2$ with $\C[x,y]/I_Z\cong\C[\Gamma]^{\oplus n}$ as
$\Gamma$-modules.  Equivalently, $Y_n$ is the Nakajima quiver variety
$\mfM_\theta((n,n),(1,0))$ for the framed affine $A_1$ McKay quiver at the
cyclic-chamber stability $\theta$ \cite{Kuznetsov,CGGSpunctual}.  It is smooth,
symplectic, of dimension $2n$ \cite{Nakajima98}, with universal
cluster $\eta:\cU_n\to Y_n$ finite flat of degree $2n$ and cluster algebra bundle
\[
 \cP_n=\eta_*\cO_{\cU_n}=\cR_0\oplus\cR_1,\qquad
 \rk\cR_0=\rk\cR_1=n,\qquad
 \cR_0=\cO_{Y_n}\!\cdot\!1\oplus\cT_n,\quad\cT_n=\ker(\Tr_{\cR_0}).
\]
The splitting uses $\Tr_{\cR_0}(1)=n$, and $\cT_n$ is a vector bundle of rank $n-1$.

Let $\theta_0=(0,1)$ be the wall character and $Q_n:=\mfM_{\theta_0}((n,n),(1,0))$
(with its reduced structure), and let
\[
 \rho:Y_n\longrightarrow Q_n
\]
be the VGIT morphism.  Both spaces map to $\SymOp^n(X)=(\A^2)^n/(\mu_2\wr S_n)$.  We
write $\pi_Q:Q_n\to S_n:=\SpecOp A$ for the support-cycle morphism, $\mfm_0$ for the
ideal of the origin, and
\[
 \cZ_n^{\mathrm{symb}}:=Q_n\times_{S_n}\SpecOp(A/\mfm_0),
 \qquad
 F_n:=Y_n\times_{S_n}\SpecOp(A/\mfm_0)
\]
for the scheme-theoretic central fibers downstairs and upstairs.  
Below we collect known facts about the spaces from above.

\begin{proposition}\label{prop:models-identified}
\leavevmode
\begin{enumerate}[label=\textup{(\alph*)}]
\item \textup{\cite[Thm.~1.1]{CGGS}} The reduced scheme $Q_n$ is isomorphic to the
reduced Quot scheme $\Quot^n_X(R_1)_{\mathrm{red}}$ of colength-$n$ quotients of the
odd McKay module.  It is irreducible and normal with symplectic, hence rational
Gorenstein, singularities, and it admits a projective symplectic resolution. 

\item \textup{\cite[Prop.~5.3]{CGGS}} The VGIT morphism $\rho$ is such a resolution,
that is, it is projective, symplectic and birational.
\item \textup{\cite[Prop.~A.1, Def.~3.6, Thm.~3.8]{Kivinen}} There is a torus-equivariant isomorphism
$Q_n\cong X_{\mathfrak g,\mathrm{symb}}$ for $\mathfrak g$ of type $B_n/C_n$,
identifying $\cO_{Q_n}(m)$ with the sheaf whose module of global sections is
$e(\Delta^mI^{(m)})$, $Q_n$ is normal and
\begin{equation}\label{eq:sections}
 \Gamma(Q_n,\cO_{Q_n}(m))\;=\;e\bigl(\Delta^mI^{(m)}\bigr),
 \qquad
 \pi_{Q*}\cO_{Q_n}(m)\ \text{is the sheaf associated to }e(\Delta^mI^{(m)}).
\end{equation}
\item The polarization pulls back to the determinant of the odd tautological bundle,
namely
\[\cL_n:=\rho^*\cO_{Q_n}(1)=\det\cR_1.\]
\end{enumerate}
\end{proposition}

\begin{proof}
(a)--(c) are the cited results.  Let us give more details for (c).  The two-vertex
cyclic quiver of \cite[Prop.~A.1]{Kivinen} with dimension vector $(n,n)$ at the wall
stability is the framed affine $A_1$ McKay quiver at $\theta_0$.  Kivinen's parameter
computation gives Gordon's stability $(0,-2c)$ in the equal-parameter case, and quiver
varieties are isomorphic under $\theta\mapsto-\theta$.  Hence we can use the
identification of the section rings in \cite[Thm.~4.3(2)]{Kivinen} and the normality
statement \cite[Thm.~3.8]{Kivinen}.  Let us also notice that the symbolic Rees algebra
need not be known to be finitely generated for \eqref{eq:sections}, since
$e(\Delta^mI^{(m)})$ is the space of sections of a unique line bundle on the normal
$Q_n$, as observed in \cite{Kivinen}.

For (d), note that on the Quot side $\cO(1)$ is the determinant line bundle of the
tautological quotient, whose pullback is $\det(\cR_1)$.  Equivalently, on the GIT side,
$\theta_0=(0,1)$ is the determinant character of the odd vertex.
\end{proof}


\subsection{The two punctual quotients}\label{sec:two-quotients}
Two finite-dimensional quotients of $e_mJ^m$ occur in the literature and we compare them below.
Recall the notation $\mathfrak m_S=S_+\subset S$ for the maximal ideal of the origin. For a $W$-representation $\chi$ let
$e_\chi$ be the corresponding idempotent, for example $e_m=e_{\det^m}$.

\begin{definition}\label{def:two-quotients}
For a $W$-stable ideal $\mathcal J\subseteq S$ set
\[
 \mathsf Q^{\mathrm{inv}}_m(\mathcal J)
 :=\frac{e_m\mathcal J}{\mathfrak m\cdot e_m\mathcal J},
 \qquad
 \mathsf Q^{\mathrm{full}}_m(\mathcal J)
 :=e_m\!\left(\frac{\mathcal J}{\mathfrak m_S\mathcal J}\right)
 =\frac{e_m\mathcal J}{e_m(\mathfrak m_S\mathcal J)}.
\]
\end{definition}

Stump's $\mathrm{Cat}^{(m)}$ is $\dim_{q,t}\mathsf Q^{\mathrm{full}}_m(J^m)$
\cite[Def.\ 3.8, Thm.\ 3.5]{Stump}, while the geometric arguments of the paper compute
$\mathsf Q^{\mathrm{inv}}_m$.  Since $\mathfrak m\cdot e_m\mathcal J\subseteq
e_m(\mathfrak m_S\mathcal J)$, there is a canonical bigraded surjection
\begin{equation}\label{eq:two-quotient-surj}
 \pi_m:\ \mathsf Q^{\mathrm{inv}}_m(\mathcal J)\twoheadrightarrow
 \mathsf Q^{\mathrm{full}}_m(\mathcal J).
\end{equation}

\begin{lemma}[Sandwich]\label{lem:sandwich}
Fix $n$, and let $m\gg0$ be in the range of \cref{thm:collapse} and
\cref{thm:character}.  Then $\pi_m$ is an isomorphism for
$\mathcal J=J^m$, and
\[
 \mathrm{Cat}^{(m)}(B_n;q,t)
 =\dim_{q,t}\mathsf Q^{\mathrm{full}}_m(J^m)
 =\dim_{q,t}\mathsf Q^{\mathrm{inv}}_m(J^m)
 =P_{n,m}(q,t).
\]
\end{lemma}

\begin{proof}
By \cref{thm:collapse}, $e_mJ^m=e_mI^{(m)}$, and by
\cref{thm:character}, we have
$\dim\mathsf Q^{\mathrm{inv}}_m(I^{(m)})=P_{n,m}(1,1)=\binom{n(m+1)}n$.  By
\eqref{eq:two-quotient-surj}, we have
$\dim\mathsf Q^{\mathrm{full}}_m(J^m)\le\binom{n(m+1)}n$.  On the other hand Stump's
surjection \cite[Thm.\ 4.5 and \S4.2.1]{Stump} (reproduced as
\cref{thm:stump-surjection} below) maps $\mathsf Q^{\mathrm{full}}_m(J^m)$, up to
the determinant twist, \emph{onto} $\gr_F eL_{\frac1{2n}+m}(\triv)$, whose dimension is
$\binom{n(m+1)}n$ by \cref{thm:typeB-dim}.  Hence
$\dim\mathsf Q^{\mathrm{full}}_m(J^m)\ge\binom{n(m+1)}n$, and therefore both dimensions
equal $\binom{n(m+1)}n$ and $\pi_m$ is an isomorphism of bigraded spaces.  The common
bigraded character is $P_{n,m}$ by \cref{thm:character}.
\end{proof}

\subsection{The polarization and the sign component of the Procesi bundle}
\label{sec:polarization}

 Boixeda Alvarez--Losev \cite[\S2]{BoixedaLosev} construct, for any finite Coxeter
group $W$, a partial resolution $X$ of $(\frh\oplus\frh^*)/W$, which is dominated by a
$\Q$-factorial terminalization $\widetilde X$, together with Procesi sheaves on both
spaces.  From their construction we see that on $X$ a power of the line bundle
$\cO^{\mathrm{reg}}(1):=\widetilde{\sP}^{\mathrm{reg}}\eps_-$ becomes ample
\cite[\S3.3]{BoixedaLosev}, where $\widetilde{\sP}^{\mathrm{reg}}\eps_-$ is the
sign-isotypic component of the Procesi sheaf $\widetilde\sP$ on $\widetilde X$,
normalized to the unit of the Namikawa--Weyl group.  In type $B_n$ they remark that
$X$ is the wall quiver variety at $\theta_0=(0,1)$ \cite[Rem.~2.2]{BoixedaLosev}.

We expand on their remark below.
Here $\widetilde X=Y_n$ and $\nu:=c_1(\widetilde{\sP}_n\eps_-)$, and the statement
to be proved is
\begin{equation}\label{eq:sharp-remark22}
 \nu=c_1(\det\cR_1)\qquad\text{in }N^1(Y_n/S_n).
\end{equation}

\begin{definition}[The two extremal curves]\label{def:two-curves}
Let $s$ be a reflection of $W(B_n)$ and $y\in S_n=\frh\oplus\frh^*/W$  a point of the codimension-two
symplectic leaf attached to its conjugacy class.  All formal slices to these
leaves are of type $A_1$ \cite[\S2.1]{BoixedaLosev}, hence the fibre of $Y_n\to S_n$ over $y$
is the exceptional curve of the minimal resolution of an $A_1$ surface germ,
a $\PP^1$.  Write
\begin{itemize}[leftmargin=1.6em]
\item $C_{\mathrm{sh}}$ for the fibre at the \emph{short} class $s_i:e_i\mapsto-e_i$,
along which $n-1$ distinct free $\Gamma$-orbits are frozen and one $\Gamma$-cluster
supported at the origin moves.
\item $C_{\mathrm{lg}}$ for the fibre at the \emph{long} class $(ij)$
\textup(needs $n\ge2$\textup), along which $n-2$ free orbits are frozen and two free
orbits collide.
\end{itemize}
\end{definition}

\begin{lemma}[The tautological degrees]\label{lem:taut-degrees}
\[
  \det\cR_0\cdot C_{\mathrm{sh}}=0,\quad
  \det\cR_1\cdot C_{\mathrm{sh}}=1,\quad
  \det\cR_0\cdot C_{\mathrm{lg}}=
  \det\cR_1\cdot C_{\mathrm{lg}}=1 .
\]
\end{lemma}

\begin{proof}
The fibre of $Y_n\to S_n$ over a point of either leaf is the exceptional curve of the
minimal resolution of an $A_1$ surface germ, hence a \emph{reduced} $\PP^1$, since the
minimal resolution of a surface rational double point has reduced exceptional
divisor, and for $A_1$ it is a single $(-2)$-curve.  Along either curve the supports
stay disjoint throughout. Indeed, on $C_{\mathrm{sh}}$ the $n-1$ frozen orbits sit at
fixed distinct points of $\C^2/\Gamma$ while the moving cluster stays over the
origin. Respectively,  on $C_{\mathrm{lg}}$ the $n-2$ frozen orbits stay away from the
colliding pair.

Hence along either curve the $\Gamma$-cluster $Z$ is a disjoint union of a frozen
part and a moving part, the sheaf $\cO_Z$ splits accordingly, and the frozen summands
are constant, of degree $0$.

On $C_{\mathrm{sh}}$ the moving part is a $\Gamma$-cluster $Z_0$ at the origin,
hence $\cO_{Z_0}\cong\C[\Gamma]=\rho_0\oplus\rho_1$ with each isotypic piece
one-dimensional.  Its even part is $\C\cdot1$, constant, of degree $0$, and its odd
part is $\langle x,y\rangle/(I\cap\langle x,y\rangle)$.  Writing $V=\langle
x,y\rangle$, the curve is $\PP(V)$ with $\cS:=I\cap V$ the tautological sub line bundle, hence the
odd part is the tautological quotient $V/\cS$ and
$\det(V/\cS)=\cO_{\PP^1}(1)$, of degree $+1$.

On $C_{\mathrm{lg}}$ the two colliding orbits are free, hence on a small ball
$U\subset\C^2/\Gamma$ around their common image the covering is trivial,
$\pi^{-1}(U)\cong U\times\Gamma$, and for the length-two subscheme $W_0\subset U$
determined by the cluster we get $\cO_{\widetilde W_0}\cong\cO_{W_0}\otimes\C[\Gamma]$.
Hence $\cR_0$ and $\cR_1$ both restrict to the tautological bundle
$\cO^{[2]}$ of $\Hilb^2(U)$.  Its determinant has degree $+1$ on the fibre.
Indeed, over a point of the diagonal the fibre is $\PP(V)$, $V=\mathfrak m/\mathfrak m^2$,
and $\cO^{[2]}|_{[I]}=\C[x,y]/I=\cO\oplus(\mathfrak m/I)=\cO\oplus(V/\cS)$, hence again
$\det=\cO_{\PP^1}(1)$.  Different trivializations of $\pi^{-1}(U)$ swap
$\rho_0\leftrightarrow\rho_1$, but both give $+1$, hence nothing depends on the
choice.
\end{proof}

\begin{lemma}[{The sign degrees, from \cite{BoixedaLosev}}]\label{lem:nu-degrees}
$\nu\cdot C_{\mathrm{sh}}=\nu\cdot C_{\mathrm{lg}}=1$.
\end{lemma}

\begin{proof}
This is contained in the proof of \cite[Lemma~2.5]{BoixedaLosev}, run there for an
\emph{arbitrary} reflection $s$ and hence covering both classes at once.  At a
point $y$ of the leaf of $s$ they show, by the direct analogue of
\cite[Prop.~4.1]{LosevProcesi}, that
$\widetilde{\sP}^{\wedge y}=\Hom_{\C\{1,s\}}(\C W,\sP^{\wedge0})$ with $\sP$ the
Procesi bundle over $(\frh\oplus\frh^*)^s\times T^*\PP^1$. Then they argue that
the two possibilities for $\sP$ are $\cO\oplus\cO(1)$ and $\cO\oplus\cO(-1)$.
Using results from \cite[Section 4.2]{LosevProcesi} they conclude that the
restriction of the line bundle $\sP\eps_-$ to $X^{\wedge y}$ is $\cO(1)$.  That
is degree $+1$ on the slice $\PP^1$, which is $C_{\mathrm{sh}}$ or
$C_{\mathrm{lg}}$ according to the class of $s$.
\end{proof}

\begin{theorem}[The polarization]\label{thm:polarization}
For every $n\geq2$,
\[
  c_1\bigl(\widetilde{\sP}_n\eps_-\bigr)=c_1(\det\cR_1)
  \qquad\text{in }N^1(Y_n/S_n),
\]
an equality of classes, not merely of rays, hence
$\widetilde{\sP}_n\eps_-\cong\det\cR_1=\rho^*\cO_{Q_n}(1)$.  Consequently the
partial resolution $X$ of \cite{BoixedaLosev} is $Q_n=\mfM_{\theta_0}$,
$\bar\rho=\rho$, and $\cO_{Q_n}(1)$ is its VGIT ample generator.
\end{theorem}

\begin{proof}
By \cref{lem:taut-degrees} the matrix of $\{\det\cR_0,\det\cR_1\}$
against $\{C_{\mathrm{sh}},C_{\mathrm{lg}}\}$ is
$\left(\begin{smallmatrix}0&1\\1&1\end{smallmatrix}\right)$, of determinant
$-1$.  Hence the two curve classes are independent, and they span
$N_1(Y_n/S_n)_\Q$, which is two-dimensional because $b_2(Y_n)=2$.  Respectively,
$\det\cR_0,\det\cR_1$ form a basis of $N^1$, 
hence $\nu=a\,c_1(\det\cR_0)+b\,c_1(\det\cR_1)$ and pairing with the two
curves, \cref{lem:nu-degrees} gives $b=1$ and $a+b=1$, i.e.\ $(a,b)=(0,1)$.
Since numerical and linear equivalence agree on $Y_n$ \textup($H^1(\cO)=0$ and
$\Pic$ is torsion-free\textup), $\widetilde{\sP}_n\eps_-\cong\det\cR_1$.

The last sentence follows because $X$ is by construction the partial resolution on which
a power of $\widetilde\sP_n\eps_-$ is ample, and $\det\cR_1=\rho^*\cO_{Q_n}(1)$
with $\cO_{Q_n}(1)$ relatively ample on $Q_n$
(\cref{prop:models-identified}(d)).  To complete the argument, we use that partial resolutions between $Y_n$ and $S_n$
are classified by the faces of the ample cone of $Y_n$
\cite[proof of Prop.~2.1]{BoixedaLosev}, and the face is determined by the class.
\end{proof}

\begin{remark}\label{rem:polarization-controls}
Let us mention some special cases and discuss the situation in type \(A\).
  At $n=1$ only $C_{\mathrm{sh}}$ exists, $\det\cR_0=\cO$, and
$Y_1=T^*\PP^1$ with Procesi bundle the McKay tautological bundle
$\cR_0\oplus\cR_1$.  Then $\widetilde\sP_1\eps_-=\cR_1$ and
$\nu=c_1(\det\cR_1)$ directly.
For $\Gamma$ trivial we have $W=S_n$ and
$Y_n=\Hilb^n(\C^2)$, only $C_{\mathrm{lg}}$ exists, and the same
argument gives $\nu=c_1(\det\cR)$. Thus we recover Haiman's
$\sP\eps_-=\bigwedge^{\,n}\cB$.

\end{remark}

\section{The Haiman complex and the dual-socle identity}\label{sec:haiman}

In this section we introduce the type $B$ analogue $\cK_n$ of the Koszul complex
which Haiman used to resolve the punctual Hilbert scheme, and we prove parts (1) and
(2) of Theorem~E.  The complex is built from the trace-zero part $\cT_n$ of the even
cluster bundle together with the two universal coordinates $x,y$.  Its zeroth
homology is the structure sheaf of the scheme-theoretic punctual fiber $F_n$
(\cref{prop:H0}).  The new feature, compared with type $A$, is that $\cK_n$ has $n+1$
generators against the expected codimension $n$, hence it is not a resolution.

The main idea of the section is that this failure is very controlled.  The homology
of $\cK_n$ is supported on $F_n$, which has codimension at least $n$ by semismallness,
and a general local-cohomology argument (\cref{lem:codim-amplitude}) shows that only
two homology sheaves survive, namely $\cI_n=\cO_{F_n}$ and the excess sheaf $\cM_n$ in
degree $-1$.  Koszul duality then identifies $\cM_n$ with an Ext sheaf of the
dual-socle sheaf $\cS_n$, whose fibers are the duals of the socles of the cluster
algebras.  This dual-socle formula is what makes the excess computable in
\cref{sec:hecke}.

\subsection{Conventions}\label{sec:conventions}

We use cohomological indexing, $(A[r])^i=A^{i+r}$ and $\cH^i(A[r])=\cH^{i+r}(A)$,
hence a sheaf $F[1]$ sits in degree $-1$.  For a vector bundle $V$ on a scheme $X$ and
a map $s:V\to\cO_X$ the Koszul complex is
\[
 K_X(V,s)^{-j}=\Lambda^jV,\qquad d=\iota_s ,
\]
in degrees $[-\rk V,0]$. If the operators act on a locally free module $P$ rather
than on $\cO_X$ we write $K_P(V\to P)$ for the complex with terms $P\otimes\Lambda^jV$.
$\RcHom_X$ denotes \emph{internal} derived Hom.  For a perfect complex, Tor-amplitude
$[u,v]$ means that its derived tensor product with every quasicoherent module has
cohomology only in $[u,v]$. Equivalently, locally it admits a finite locally free
representative in those degrees.  This is stronger than having its own cohomology in
$[u,v]$.

\subsection{The complex}\label{sec:complex-def}

Recall from \cref{sec:quiver} the splitting $\cR_0=\cO_{Y_n}\cdot1\oplus\cT_n$,
$\cT_n=\ker(\Tr_{\cR_0})$, $\rk\cT_n=n-1$.  Set
\[
 \cE_n=\cT_n\oplus\cO_q\oplus\cO_t,\qquad\rk\cE_n=n+1,
\]
where $\cO_q$, $\cO_t$ are the trivial line bundles with the $T$-characters of the
coordinates $x,y$ (as ordinary bundles, $\cE_n=\cT_n\oplus\C x\oplus\C y$).  Let
$\nu:\cE_n\to\cP_n$ be the inclusion on $\cT_n$ together with the universal sections
$x,y\in\cP_n$, let $J_n\subseteq\cP_n$ be the ideal sheaf generated by $\im\nu$, and let
\[
 \cK_n=K_{\cP_n}(\cE_n\to\cP_n):\qquad
 \cK_n^{-a}=\cP_n\otimes{\textstyle\bigwedge^a}\cE_n,\quad 0\le a\le n+1,
\]
\[
 d(p\otimes e_1\wedge\cdots\wedge e_a)
 =\sum_i(-1)^{i-1}p\,\nu(e_i)\otimes e_1\wedge\cdots\widehat{e_i}\cdots\wedge e_a .
\]
This is the direct $A_1$-McKay analogue of the complex Haiman used to resolve the
punctual fiber in type $A$ \cite[Prop.~2.4, eq.~(34)]{Haiman}. The main difference
is that here $\rk\cE_n=n+1$ exceeds the expected codimension $n$ by one, and exactness
fails.  Define
\[
 \cI_n=\cH^0(\cK_n)=\cP_n/J_n,\qquad
 \cM_n=\cH^{-1}(\cK_n),\qquad
 \cM_n=\cM_n^+\oplus\cM_n^-,\qquad \cM_n^+=\cM_n^\Gamma .
\]
We shall also use the \emph{transposed} complex.  The dual $\cP_n^\vee$ is a
$\cP_n$-module through $(a\phi)(b)=\phi(ab)$, and
\[
 \cK_n^T=K_{\cP_n^\vee}(\cE_n\to\cP_n^\vee),\qquad
 \cS_n=\cH^0(\cK_n^T)=\cP_n^\vee/J_n\cP_n^\vee,\qquad
 \cN_n=\cH^{-1}(\cK_n^T).
\]
The sheaf $\cS_n$ is the \emph{dual-socle sheaf}.  When $n$ is fixed we drop the
subscript $n$ from all of these symbols.

\begin{remark}[Equivariance]\label{rem:equivariance}
The parity summands are selected in the $T$-equivariant Koszul complex. In particular, the generator $\gamma$ of \(\Gamma\)
acts trivially on $Y_n$, hence the parity decomposition is a decomposition of
complexes of $\cO_{Y_n}$-modules.  Where it is convenient we suppress the torus
linearizations and compute with ordinary line bundles.  In particular,
$\det\cE_n=\det\cT_n$ as ordinary line bundles. The two coordinate lines
contribute the constant character $qt$, whose central $\Gamma$-character is
trivial.  The vanishing $R\rho_*\cM_n^+=0$ for ordinary coherent sheaves implies
the equivariant statement, since forgetting equivariance is conservative.
\end{remark}

\subsection{Punctual support and the zeroth homology}\label{sec:H0}
In this subsection we collect the most basic facts about the Haiman complex \(\cK_\bullet\).
\begin{lemma}\label{lem:punctual-support}
Let $z\in Y$ be a closed point and $E=\cP_z$ the corresponding cluster algebra.
Then $E/J_zE\ne0$ if and only if $E$ is supported at the origin.  Consequently all
cohomology sheaves of $\cK$ and of $\cK^T$ are supported on the set $|F_n|$.
\end{lemma}

\begin{proof}
A point in the support of $E/J_zE$ is killed by $x,y$, hence it is the origin.  For
every even polynomial $f$ vanishing at the origin, the element
$f-\frac1n\Tr_{E_0}(f)\,1$ belongs to the trace-zero space $\cT_z$, and its value in
the residue field at the origin is $-\frac1n\Tr_{E_0}(f)$.  Thus a nonzero quotient
forces all such traces to vanish.  If $E$ has a nonzero support orbit, choose an
invariant polynomial $f$ which is zero at the origin and at the other support
orbits and equals one at the chosen orbit.  Multiplication by $f$ on each local
factor has trace equal to its residue value times the dimension, the nilpotent
part contributing zero.  Hence the chosen free orbit contributes a strictly positive
integer to $\Tr_{E_0}(f)$, a contradiction.  Conversely, for a punctual algebra the
trace of multiplication by $a\in E_0$ is $n$ times its residue, hence $\cT_z$, $x$
and $y$ lie in the maximal ideal and the quotient is nonzero.

Away from this locus, $J=\cP$ locally, and an expression $1=\sum a_if_i$ gives the
standard contracting homotopy of the Koszul complex.  Transposing the commuting
multiplication operators gives the same contraction for $\cK^T$.
\end{proof}

\begin{lemma}[Power sums generate the symmetric tensors]\label{lem:powersums}
Let $C$ be a commutative $\C$-algebra with an augmentation $a:C\to\C$, and for $h\in C$
let $p(h)=\sum_{j=1}^nh^{(j)}\in(C^{\otimes n})^{S_n}$, where $h^{(j)}$ is $h$ in the
$j$-th factor and $1$ elsewhere.  Then $(C^{\otimes n})^{S_n}$ is generated as an
algebra by the $p(h)$, and the kernel of the augmentation $a^{\otimes n}$ is generated
as an ideal by the $p(h)$ with $a(h)=0$.
\end{lemma}
\begin{proof}
For $1\le s\le n$ put $D_s(h_1,\dots,h_s)=\sum h_1^{(j_1)}\cdots h_s^{(j_s)}$, the sum
over pairwise distinct $j_1,\dots,j_s$.  In the product $p(h_1)D_{s-1}(h_2,\dots,h_s)$
the position assigned to $h_1$ is either distinct from all others or agrees with
exactly one of them, hence
\[
 p(h_1)D_{s-1}(h_2,\dots,h_s)=D_s(h_1,\dots,h_s)+\sum_{b=2}^sD_{s-1}(h_2,\dots,h_1h_b,\dots,h_s),
\]
and induction on $s$ expresses every $D_s$ in power sums.  Every invariant tensor is a
linear combination of symmetrizations of pure tensors (average over $S_n$), and the
symmetrization of $h_1\otimes\cdots\otimes h_n$ is $D_n(h_1,\dots,h_n)/n!$.  This proves
generation.  Writing $h=a(h)1+h_+$ with $a(h_+)=0$ gives $p(h)=na(h)+p(h_+)$.  Hence modulo
the ideal generated by the $p(h_+)$ every generator is a scalar, the quotient is a
quotient of $\C$, and the augmentation shows it is $\C$.
\end{proof}

With $C=R_0$ and $a(f)=f(0)$ this describes $\cO(\SymOp^nX)=(R_0^{\otimes n})^{S_n}$ and
its maximal ideal $\mfm_0=(p(f):\ f\in R_0,\ f(0)=0)$ at $n[0]$.

\begin{proposition}\label{prop:H0}
There is a canonical $T$-equivariant isomorphism of $\cO_{Y_n}$-algebras
$\cI_n\cong\cO_{F_n}$, where $F_n=Y_n\times_{S_n}\{0\}$ is the scheme-theoretic
punctual fiber.  In particular $\mfm_0\,\cI_n=0$.
\end{proposition}

\begin{proof}
\emph{Step 0.}  Since $\cU_n\subseteq Y_n\times\A^2$, we have
$\cP_n/(x,y)\cP_n=\eta_*\cO_{\cU^0_n}$ with
$\cU^0_n:=\cU_n\cap(Y_n\times\{0\})$.  This sheaf is $\cO_{Y_n}$-cyclic, generated by
the image of $1$, hence equals $\cO_{Y_n}/\mfb$ for the ideal
$\mfb=\Ann(\eta_*\cO_{\cU^0_n})$.  Note that $\mfb\ne0$, since fiberwise
$E/(x,y)E\ne0$ iff $0\in\Supp E$, which is a codimension-two condition.

\emph{Step 1: $\mfb\subseteq\mfm_0\cO_{Y_n}$.}  Equivalently
$\mfb\cdot\cO_{F_n}=0$, i.e.\ the natural surjection
$\cO_{F_n}\twoheadrightarrow(\cP_n/(x,y)\cP_n)|_{F_n}$ is injective.  For
$y\in F_n$ the fiber $E$ is the local Artinian algebra of a \emph{punctual}
cluster, with maximal ideal $(x,y)E$, hence
$(\cP_n|_{F_n}/(x,y))\otimes k(y)=E/(x,y)E=\C$.  Thus the displayed
surjection of coherent sheaves has kernel with vanishing fibers, and by Nakayama the
kernel is zero.

\emph{Step 2: the trace ideal.}  Let
$\sigma:\cT_n\to\cP_n/(x,y)\cP_n$ be the composite of the inclusion with the
projection, and therefore $\cI_n=\cO_{Y_n}/(\mfb+\widetilde{\im}\,\sigma)$ for any
lift $\widetilde{\im}\,\sigma\subseteq\cO_{Y_n}$ of its image.  For a
positive-degree invariant $f\in(R_0)_+$ let $\bar f\in\cR_0$ be its universal image
and $\tau_f:=\Tr_{\cR_0}(m_{\bar f})\in\cO_{Y_n}$.  Then
$\bar f^{\,0}:=\bar f-\tfrac1n\tau_f\cdot1\in\cT_n\subseteq J_n$, while
$\bar f\in(x,y)\cP_n\subseteq J_n$ (positive degree).  Subtracting,
$\sigma$ takes the value $-\tfrac1n\tau_f$ on $\bar f^{\,0}$, and every element of
$\cT_n$ is an $\cO_{Y_n}$-combination of such.  Hence
$\mfb+\widetilde{\im}\,\sigma=\mfb+(\tau_f:\ f\in(R_0)_+)$.  Write
$\mathfrak a$ for the latter ideal.

By the equivariant form of Haiman's trace identity
\cite[eq.~(28)]{Haiman} (take $\Gamma$-invariants along the fixed locus
$Y_n\subset\Hilb^{2n}(\A^2)$), the function $\tau_f$ is the pullback, under the
support-cycle morphism $Y_n\to\SymOp^n(X)$, of the polarized power sum attached to
$f$.  (The identity can also be checked directly.  On the dense open locus of $n$ disjoint
free $\Gamma$-orbits, $\cR_0$ is a sum of $n$ one-dimensional algebras and
$\Tr(m_{\bar f})$ is the sum of the values of $f$ at the orbits, and both sides are
regular functions on the reduced irreducible $Y_n$.)  The positive-degree polarized
power sums generate the maximal ideal $\mfm_0$ of $n[0]\in\SymOp^n(X)$
(\cref{lem:powersums}).  Hence
$(\tau_f)=\mfm_0\cO_{Y_n}$, and with Step 1,
$\mathfrak a=\mfb+(\tau_f)=\mfm_0\cO_{Y_n}$, hence $\cI_n=\cO_{F_n}$.
\end{proof}

\begin{remark}[Generator count]\label{rem:generator-count}
Step 1 is not a formality.  The map $\sigma$ has locally free source of rank
$n-1$, hence $\im\sigma$ is locally $(n-1)$-generated, whereas $\codim F_n=n$.
The ideal of $F_n$ can never be generated by $n-1$ elements, and the missing
generators come from $\mfb$, the ideal of the divisor of clusters through the
origin.  In the Koszul complex $\cK_n$ the ideal $\mfb$ is accounted for by the two
sections $x,y$ of $\cP_n$, hence $J_n$ has $n+1$ generators for a height-$n$ ideal.
This deviation by one is the source of $\cM_n\ne0$.  In type $A$ the
same count gives $n+1$ generators of height $n+1$, a complete intersection.
\end{remark}

The following fiberwise form of \cref{prop:H0} will be used in the unit
normalization of \cref{sec:SD}.

\begin{lemma}[The incidence quotient is the augmentation line]\label{lem:incidence-line}
Let $\xi\in F_n$ be a closed point, $A_\xi=\cP_n\otimes k(\xi)$ the corresponding
punctual cluster algebra, $\eps_\xi:A_\xi\twoheadrightarrow\C$ its residue-field
augmentation, and $J_{n,\xi}\subset A_\xi$ the specialization of $J_n$.  Then
\[
  J_{n,\xi}=\ker\eps_\xi=(x,y)A_\xi,
 \qquad\text{hence}\qquad
 \cI_n\otimes k(\xi)=A_\xi/J_{n,\xi}\cong\C\cdot[1],
\]
and the unit map $\C\to A_\xi/J_{n,\xi}$, $1\mapsto[1]$, is the identity under this
identification.
\end{lemma}

\begin{proof}
The inclusion $(x,y)A_\xi\subseteq J_{n,\xi}$ holds because $x,y$ are among the
generators of $J_n$.  Conversely, let $g$ be a specialized trace-zero even
generator and write $g=\eps_\xi(g)1+g_{\mathrm{nil}}$ with
$g_{\mathrm{nil}}\in\ker\eps_\xi$, which is nilpotent.  Being even, $g_{\mathrm{nil}}$
preserves the even summand $(A_\xi)_0$ and is nilpotent there, hence has trace
zero, while the scalar term acts on the $n$-dimensional even summand as
$\eps_\xi(g)$ times the identity.  Therefore $\Tr_{(A_\xi)_0}(m_g)=n\,\eps_\xi(g)$.
The left side is zero because $g$ is trace-free, hence $\eps_\xi(g)=0$.  Thus every
trace-zero even generator lies in $\ker\eps_\xi=(x,y)A_\xi$.
\end{proof}

\subsection{The dimension of the punctual fiber}\label{sec:dimF}

\begin{lemma}\label{lem:dimF}
$\dim F_n\le n$.
\end{lemma}

\begin{proof}
The affinization map $Y_n\to\SymOp^n(\A^2/\Gamma)$ is a projective symplectic
resolution, since $Y_n$ is smooth symplectic and the map is birational onto the symmetric
power \cite{Kuznetsov,CGGSpunctual}.  Symplectic resolutions are semismall
\cite[Lemma~2.11]{KaledinPoisson}, \cite[Thm.~1.9]{Kaledin}, and a proper map
$f:Y\to X$ is semismall exactly when $\dim(Y\times_XY)\le\dim Y$.  Since
$F_n\times F_n\subset Y_n\times_{\SymOp^n(\A^2/\Gamma)}Y_n$, we get
$2\dim F_n\le\dim Y_n=2n$.  (Alternatively, the central fiber of a quiver variety
over its affinization is Lagrangian by Nakajima \cite[Thm.~5.8]{Nakajima}, see
also \cite[\S\S3--4]{Gordon}.)  Only this inequality is needed, not reducedness or
Cohen--Macaulayness of the scheme $F_n$.
\end{proof}

\subsection{Homological lemmas}\label{sec:homological}

\begin{lemma}[Support versus Tor-amplitude]\label{lem:codim-amplitude}
Let $X$ be a regular Noetherian scheme, let $Z\subset X$ be closed of
codimension at least $c$, and let $A\in\Perf(X)$ have Tor-amplitude
$[-m,0]$.  If every $\cH^i(A)$ is supported on $Z$, then
$A\in D^{[c-m,0]}(X)$.  In particular, if $m=c$, the complex is a sheaf in degree zero.
\end{lemma}
\begin{proof}
The assertion is local on $X$.  Write $X=\SpecOp R$, choose an ideal $I$
defining $Z$, and represent $A$ by finite free modules in degrees
$-m,\ldots,0$.  Since $R$ is regular, it is Cohen--Macaulay, and the codimension
assumption gives $H_I^r(R)=0$ for $r<c$, and the same for each finite free term.
This is Grothendieck's local-cohomology characterization of depth,
$\depth_I(R)=\min\{r:H^r_I(R)\ne0\}$, together with $\depth_I(R)=\grade(I)=
\htOp(I)$ in a Cohen--Macaulay ring
\cite[Thm.~3.5.7, Cor.~2.1.4]{BrunsHerzog}, \cite[Tag 0AVZ]{Stacks}.
The hypercohomology spectral sequence of $R\Gamma_I$ applied to the bounded
complex of free modules $A$ has $E_1^{i,r}=H_I^r(A^i)$, and every nonzero entry
has $i\ge-m$ and $r\ge c$, hence the total degree is at least $c-m$.  As the
cohomology of $A$ is supported on $V(I)$, the natural map $R\Gamma_I(A)\to A$ is an
isomorphism ($A$ restricts to zero on $X\setminus Z$).  Thus $A\in D^{\ge c-m}$,
and the locally free representative gives $A\in D^{\le0}$.
\end{proof}

\begin{remark}
The hypothesis is Tor-amplitude, not ordinary cohomological amplitude.  A sheaf
supported in positive codimension has amplitude $[0,0]$ and arbitrary projective
dimension.  In our applications the amplitude comes from pushing down a
regular-section Koszul resolution.
\end{remark}

\begin{lemma}[Low Ext vanishing]\label{lem:low-ext}
If $X$ is regular and a coherent sheaf $N$ is supported in codimension at
least $c$, then $\cExt_X^i(N,\cO_X)=0$ for $i<c$.
\end{lemma}
\begin{proof}
Locally $N$ is annihilated by an ideal of grade at least $c$.  The grade $\grade(\Ann N,R)$
is the first nonzero index of $\Ext_R^i(N,R)$ \cite[\S1.2]{BrunsHerzog}, and over
a Cohen--Macaulay ring grade equals height \cite[Cor.~2.1.4]{BrunsHerzog}, i.e.\
the codimension of the support.
\end{proof}

\begin{lemma}[Koszul duality, including shifts]\label{lem:koszul-duality}
Let $V$ have rank $r$, and let $K_X(V,s)$ be in degrees $[-r,0]$.  Then
\begin{equation}\label{eq:koszul-self-duality}
 \RcHom_X(K_X(V,s),\cO_X)\simeq K_X(V,s)[-r]\otimes(\det V)^{-1}.
\end{equation}
If the Koszul operators act on a locally free module $P$ rather than on
$\cO_X$, the right-hand complex has coefficient module $P^\vee$ and the
transposed operators.
\end{lemma}
\begin{proof}
In dual cohomological degree $j$ the term is $(\Lambda^jV)^\vee$, which the perfect
exterior pairing identifies with $\Lambda^{r-j}V\otimes(\det V)^{-1}$, the
degree-$j$ term of the right-hand side.  Under this pairing the transpose of
contraction becomes contraction on the complementary exterior power up to sign.
Rescaling the degree-$j$ identification by $(-1)^{\binom j2}$, together with the
shift sign, makes the differentials agree.  With coefficients, the same
calculation transposes the multiplication operators.
\end{proof}

\begin{lemma}[Canonical modules are semidualizing]\label{lem:semidualizing}
Let $B$ be Cohen--Macaulay with a dualizing module $\omega_B$, and let $L$
be invertible.  For $\Sigma=\omega_B\otimes L$, $\RcHom_B(\Sigma,\Sigma)\simeq\cO_B$.
\end{lemma}
\begin{proof}
A defining property of a dualizing complex is
$\cO_B\simeq\RcHom_B(\omega_B^\bullet,\omega_B^\bullet)$, and on a Cohen--Macaulay
scheme the complex is locally a single canonical module with a shift, which
cancels in its self-Hom.  Concretely, for a Cohen--Macaulay local ring with
canonical module $\omega$ one has $\Hom(\omega,\omega)=R$ and
$\Ext^{i}(\omega,\omega)=0$ for $i>0$ \cite[Thm.~3.3.10]{BrunsHerzog}, and tensoring both
entries with $L$ does not change internal derived Hom.  See also
\cite[Tags 0A7A, 0DW6]{Stacks}.
\end{proof}

\subsection{Only two homology sheaves}\label{sec:two-term}

\begin{theorem}\label{thm:two-term}
$\cK_n,\cK_n^T\in D^{[-1,0]}(Y_n)$.  In other words, the only homology sheaves of $\cK_n$ are
$\cH^0=\cI_n$ and $\cH^{-1}=\cM_n$, and those of $\cK_n^T$ are $\cS_n$ and $\cN_n$.
\end{theorem}

\begin{proof}
Both complexes have locally free terms in $[-n-1,0]$ on the smooth $2n$-dimensional
$Y_n$, and by \cref{lem:punctual-support} their cohomology is supported on
$F_n$, of codimension at least $n$ by \cref{lem:dimF}.  Apply
\cref{lem:codim-amplitude} with $(m,c)=(n+1,n)$.
\end{proof}

\subsection{The dual-socle formula}\label{sec:dual-socle}

\begin{theorem}[Dual-socle formula]\label{thm:dual-socle}
There is a canonical $T$-equivariant isomorphism
\begin{equation}\label{eq:dual-socle}
 \cM_n\;\simeq\;\cExt_{Y_n}^n(\cS_n,\cO_{Y_n})\otimes\det\cE_n .
\end{equation}
Its even part is
\begin{equation}\label{eq:even-dual-socle}
 \cM_n^+\;\simeq\;\cExt_{Y_n}^n(\cS_n^+,\cO_{Y_n})\otimes\det\cT_n ,
\end{equation}
where the last equality suppresses the constant coordinate character $qt$.
\end{theorem}
\begin{proof}
Dualize the truncation triangle $\cN_n[1]\to\cK_n^T\to\cS_n\to\cN_n[2]$ of
\cref{thm:two-term}.  By \cref{lem:low-ext} (supports have codimension $\ge n$),
the two outer groups in
\[
 \cExt_Y^{n-2}(\cN_n,\cO_Y)\longrightarrow
 \cExt_Y^n(\cS_n,\cO_Y)\longrightarrow
 \cH^n\RcHom_Y(\cK_n^T,\cO_Y)\longrightarrow
 \cExt_Y^{n-1}(\cN_n,\cO_Y)
\]
vanish, hence the middle arrow is an isomorphism.  By \cref{lem:koszul-duality},
$\RcHom_Y(\cK_n^T,\cO_Y)\simeq\cK_n[-n-1]\otimes(\det\cE_n)^{-1}$, whose
degree-$n$ cohomology is $\cM_n\otimes(\det\cE_n)^{-1}$.  This proves
\eqref{eq:dual-socle}.  The determinant of the two coordinate lines is
$\Gamma$-even, as is $\det\cT_n$, and taking the even summand proves
\eqref{eq:even-dual-socle}.
\end{proof}

\begin{remark}[The fiber interpretation]\label{rem:fiber-interpretation}
At a punctual closed point $z$, the dual of $\cS_z\otimes k(z)$ is the socle of
$E=\cP_z$, since $(E^\vee/\mfm_EE^\vee)^\vee=\{a\in E:\mfm_Ea=0\}$.  This is a
statement about the cokernel of transposed multiplication, and the sheaf kernel of
multiplication on the ambient $Y$ is not a substitute for it.
\end{remark}

Since $\cI_n$ is $\gamma$-even, the parity split of the truncation triangle
$\cM_n[1]\to\cK_n\to\cI_n\to$ gives the canonical triangle
\begin{equation}\label{eq:positive-triangle}
 \cM_n^+[1]\longrightarrow\cK_n^+\longrightarrow\cI_n\xrightarrow{\ +1\ } ,
\end{equation}
and $\cK_n^-\simeq\cM_n^-[1]$.

\begin{remark}\label{rem:M-nonzero}
The bound of \cref{thm:two-term} is sharp, that is, $\cM_n^+\ne0$ for every $n\ge2$.  The
balanced complete intersection $S/(xy,\,x^2+y^{2n-2})$ has basis
$1,x,y,y^2,\dots,y^{2n-2}$, with $n$ even and $n$ odd basis elements, and its socle
is the even line spanned by $y^{2n-2}$.  \Cref{cor:trace-CM} shows that $\cM_n^+$
is nonzero at this point and is Cohen--Macaulay of pure dimension $n$.  Thus,
unlike type $A$, the incidence is never a complete intersection compatible with
the natural generators, and \cref{sec:wall,sec:hecke} are the management of the
excess $\cM_n^+$.  For $n=1$, $\cM_1^+=0$ (\cref{thm:vanishing}).
\end{remark}

Parts (1) and (2) of Theorem~E are \cref{prop:H0,thm:two-term,thm:dual-socle}.

\section{Wall fibers, length duality, and the local model}\label{sec:wall}

In this section we describe the geometry of the contraction $\rho:Y_n\to Q_n$ over
the torus-fixed points of the wall.  The fiber over a fixed point $z_J$ is a
Grassmannian $\Gr(k,V_J)$ (\cref{thm:fiber-grass}), and the length duality
$\dim V_J=2k-1$ is the combinatorial fact which makes these Grassmannians special.
We then compute the cluster bundle on such a fiber, the polystable representative
and the Ext-quiver of a wall point, and we deduce from the theorem of Bellamy and
Craw a local model of $\rho$ at a wall point, namely the Springer resolution
$T^*\Gr(r+1,2r+1)\to\mfN_{r,2r+1}$ times a smooth core.  The last subsection recalls
the Brill--Noether theory of Gonz\'alez, Gorsky and Simental for the socle strata
and derives the threshold $n\ge r(r+1)$ for the even socle rank $r$.

The main idea is that everything is governed by the even socle of the cluster
algebras.  Its rank jumps along the wall fibers, the tautological quotient bundle of
the Grassmannian is the varying socle, and the local model is determined by the
dimension $2r+1$ of the extension space.  The results of this section are used in
\cref{sec:SD}, where the positivity of the tautological quotient bundle drives the
symbolic descent argument.  The Grassmannian acyclicity of \cref{thm:BWB} is the
fixed-point shadow of the vanishing theorem of \cref{sec:hecke}, and it explains at
a single wall fiber why the excess sheaf is invisible to the wall.

\subsection{Hearts and Grassmannian fibers}\label{sec:hearts}

A $T$-fixed point $z_J\in Q_n^T$ corresponds to a monomial submodule
$J\subset R_1$ with $\dim_\C R_1/J=n$.  Its \emph{heart} is the order ideal
$H=H(J)=\{x^ay^b\notin J\}$ of $n$ odd lattice points.  Define
\[
 J'=R_1J\subseteq R_0,\qquad
 J''=(J:R_1)=\{f\in R_0:\ xf,yf\in J\}\subseteq R_0,\qquad
 V_J=J''/J' .
\]

\begin{theorem}[Grassmannian fibers and length duality]\label{thm:fiber-grass}
\leavevmode
\begin{enumerate}[label=\textup{(\alph*)}]
\item $\mfm_XV_J=0$, and consequently every $\C$-subspace of $V_J$ is an
$R_0$-submodule.
\item With $k=k_J:=\length(R_0/J')-n$, the reduced fiber is
\[
  F_J:=\rho^{-1}(z_J)\;\cong\;\Gr(k,V_J),
\]
where a point of the fiber is a balanced ideal $I=I^+\oplus J$ with even part
$J'\subseteq I^+\subseteq J''$, and $W=I^+/J'\subset V_J$ is an arbitrary
$k$-plane.
\item \textup{(}Length duality\textup{)}
$\length(R_0/J')+\length(R_0/J'')=2n+1$, hence
\[
 N:=\dim V_J=2k-1,
\]
and every positive-dimensional fixed fiber has $k\ge2$.
\end{enumerate}
\end{theorem}

\begin{proof}
(a) $R_1J''\subseteq J$, hence $\mfm_XJ''=R_1(R_1J'')\subseteq R_1J=J'$.

(b) A cluster in the fiber has odd part exactly $J$, because the wall parameter
$\theta_0=(0,1)$ fixes the odd quotient $R_1\twoheadrightarrow R_1/J$, and $\rho$
remembers exactly this datum (\cref{prop:models-identified}(a)).  The even
part $I^+$ must satisfy $R_1I^+\subseteq J$ (ideal condition), i.e.\
$I^+\subseteq J''$, and $I^+\supseteq R_1J=J'$.  Conversely, by (a) any intermediate
subspace is an ideal.  The colength condition
$\length(R_0/I^+)=n$ becomes $\dim(I^+/J')=\length(R_0/J')-n=k$.

(c) Let $C$ and $D$ be the monomial bases of $R_0/J'$ and $R_0/J''$.  Directly from
the definitions,
\[
 C=\Bigl\{e\ \text{even}:\
 \bigl(a{=}0\ \text{or}\ e{-}(1,0)\in H\bigr)\ \text{and}\
 \bigl(b{=}0\ \text{or}\ e{-}(0,1)\in H\bigr)\Bigr\},
\]
\[
 D=\Bigl\{f\ \text{even}:\ f{+}(1,0)\in H\ \text{or}\ f{+}(0,1)\in H\Bigr\}.
\]
Let $h_x=|H\cap\{b=0\}|$, $h_y=|H\cap\{a=0\}|$, let $P$ be the number of even points
$(a,b)$, $a,b\ge1$, whose two odd predecessors $(a{-}1,b)$, $(a,b{-}1)$ both lie in
$H$, and let $Q$ be the number of even points whose two odd successors both lie in
$H$.  Every element of $C$ is the origin, or a boundary point $(a,0)$ with
$(a{-}1,0)\in H$ (there are $h_x$ of these), or $(0,b)$ with $(0,b{-}1)\in H$ ($h_y$
of them), or an interior point counted by $P$, hence
$|C|=1+h_x+h_y+P$.  For $D$, count the even points with odd successor in $H$ in the
$x$-direction.  These biject with $H\setminus\{a=0\}$, giving $n-h_y$, and similarly
there are $n-h_x$ in the $y$-direction.  The double-counted points are exactly those
counted by $Q$, hence $|D|=2n-h_x-h_y-Q$.  Translation by $(-1,-1)$ is a bijection
between $P$-configurations and $Q$-configurations, hence $P=Q$ and $|C|+|D|=2n+1$.  Then
$|D|=n-k+1$ and $N=|C|-|D|=2k-1$.  If $\dim F_J>0$ then $0<k<N$, i.e.\ $k\ge2$.
\end{proof}

\begin{remark}\label{rem:census}
Exact enumeration through rank $16$ shows that the positive-dimensional fixed
fibers occurring in $Q_n$ are precisely of the form $\Gr(m,2m-1)$, $m\ge2$, with
$\Gr(m,2m-1)$ appearing first at $n=m(m-1)$, where $\dim\Gr(m,2m-1)=m(m-1)=n$.  Thus
the first occurrence is an entire Lagrangian irreducible component of $F_n$ contracted
to a point.  See \cref{app:census} for the census through $n=12$.  The
thresholds $n\ge m(m-1)=r(r+1)$, $r=m-1$, are those of the Brill--Noether theory of
\cref{sec:BN}.
\end{remark}

\subsection{The cluster bundle, the socle, and the universal extension}
\label{sec:cluster-socle}

Fix a heart $J$ and write $F_J=\Gr(k,V)$, $V=V_J$, with tautological sequence
$0\to\cS\to V\otimes\cO\to\cQ_J\to0$ and $\cO_{\mathrm{Pl}}(1)=\det\cQ_J$ (quotient
Pl\"ucker convention).  Let
\[
 \mathsf M_J:=\C[x,y]/(J''\oplus J),
 \qquad
 A_W:=\C[x,y]/(I^+_W\oplus J)\quad (W\subset V,\ I^+_W=\text{preimage of }W),
\]
regarded as modules over $\Lambda=\C[x,y]\rtimes\Gamma$.

\begin{proposition}\label{prop:cluster-socle}
\leavevmode
\begin{enumerate}[label=\textup{(\alph*)}]
\item $\cP_n|_{F_J}\cong\bigl((R_0/J'')\oplus(R_1/J)\bigr)\otimes\cO_{F_J}\oplus\cQ_J$.
In particular $\cL_n|_{F_J}$ is the constant character
$\ell_J=\prod_{s\in H}w(s)$, and
$\det\cE_n|_{F_J}\cong\det\cT_n|_{F_J}\otimes qt\cong\cO_{\mathrm{Pl}}(1)\otimes\chi_J$
for a fixed character $\chi_J$.
\item $\soc_0(A_W)=J''/I^+_W=V/W$, and globally
$\soc_0\bigl(\cP_n|_{F_J}\bigr)\cong\cQ_J$.  \textup(This holds with no Gorenstein
hypothesis.  On the non-Gorenstein strata the even socle simply has rank $>1$.\textup)
\item There is a universal extension of families of $\Lambda$-modules
\begin{equation}\label{eq:universal-extension}
 0\longrightarrow S_0\otimes\cQ_J\longrightarrow \cP_n|_{F_J}
 \longrightarrow \mathsf M_J\otimes\cO_{F_J}\longrightarrow0 .
\end{equation}
\end{enumerate}
\end{proposition}

\begin{proof}
At $W$, the even part splits $T$-equivariantly as
$R_0/I^+_W\cong R_0/J''\oplus J''/I^+_W$ (choose the monomial complement, the
splitting is canonical as a $T$-module and the extension is by the
$\mfm_X$-torsion piece), and $J''/I^+_W=V/W$, while the odd part is the constant $R_1/J$.
This gives (a), and the determinant statements follow since
$\det\cE_n=qt\cdot\det\cT_n$ and the only varying summand of $\cT_n|_{F_J}$ is
$\cQ_J$.  For (b), note that $f\in R_0/I^+_W$ is killed by $x$ and $y$ iff $f\in J''/I^+_W$.
For (c), the surjection $A_W\twoheadrightarrow\mathsf M_J$ has kernel
$J''/I^+_W$ concentrated in even degree and killed by $x,y$ (as $xJ'',yJ''\subseteq
J$), i.e.\ $S_0\otimes(V/W)$, and this globalizes.
\end{proof}

\subsection{The wall polystable and the Ext-quiver of a heart}
\label{sec:ext-quiver}

\begin{proposition}\label{prop:ext-quiver}
\leavevmode
\begin{enumerate}[label=\textup{(\alph*)}]
\item $\mathsf M_J$ is generated by $1$ and has no $\Lambda$-submodule isomorphic to
$S_0$.  The $\theta_0$-polystable representative of every point of $F_J$ is
\[
  \mathsf M_J\ \oplus\ S_0^{\oplus m},\qquad m=k-1 ,
\]
and no $S_1$-summand ever splits off.
\item There are natural $T$-equivariant isomorphisms
\[
 \Ex^1_\Lambda(S_0,\mathsf M_J)\cong V_J,\qquad
 \Ex^1_\Lambda(S_0,S_0)=0,\qquad
 \Ex^1_\Lambda(\mathsf M_J,S_0)\cong V_J^\vee\otimes(qt) .
\]
\item Under \textup{(b)}, the class of the extension \eqref{eq:universal-extension}
at $W$ is the quotient map $V_J\to V_J/W$, and globally the classifying map of
\eqref{eq:universal-extension} is the tautological quotient
$V_J\otimes\cO_{F_J}\twoheadrightarrow\cQ_J$.
\end{enumerate}
\end{proposition}

\begin{proof}
(a) If $\bar f\in(R_0/J'')$ satisfies $x\bar f=y\bar f=0$ then $f\in(J:R_1)=J''$, hence
$\bar f=0$, and there is no $S_0$-submodule.  For the polystable representative, we use
King stability at $\theta_0=(0,1)$ (framing normalized such that the framed module has
slope $0$).  A subrepresentation not containing the framing vector has
$\theta_0$-pairing equal to its odd dimension, which is $\ge0$ with equality iff the
subrepresentation is a sum of copies of $S_0$.  Dually, $S_1$ has $\theta_0$-degree
$1\neq0$ and can neither destabilize nor split off at the wall.  The largest
$S_0$-isotypic submodule of $A_W$ is $\soc_0(A_W)=V/W$ of dimension $k-1$
(\cref{prop:cluster-socle}(b)), and the quotient is the stable framed
factor $\mathsf M_J$.

(b) Apply $\Hm_\Lambda(-,\mathsf M_J)$ to the $\Gamma$-equivariant Koszul resolution
of $S_0$,
\begin{equation}\label{eq:S0-resolution}
 0\to P_0\otimes(qt)\xrightarrow{\ \binom{-y}{x}\ }
 (P_1\otimes q)\oplus(P_1\otimes t)\xrightarrow{\ (x\ \ y)\ }P_0\to S_0\to0,
\end{equation}
$P_i=\Lambda e_i$.  The resulting complex is
$(\mathsf M_J)_0\to(\mathsf M_J)_1^{\oplus2}\to(\mathsf M_J)_0$ with maps
$m\mapsto(xm,ym)$ and $(b,c)\mapsto-yb+xc$, whose middle cohomology is the even part
of $(J''\oplus J)/(x,y)(J''\oplus J)$, namely $J''/(xJ+yJ)=J''/J'=V_J$.  The first map
is injective by (a), hence $\Ex^0$ contributes nothing in the middle.  Taking
$\mathsf M_J=S_0$ gives zero middle term.  The last isomorphism is the
$2$-Calabi--Yau duality of $\Lambda$ ($\Gamma\subset SL_2$), whose Serre twist is the
symplectic character $qt$.

(c) The extension of $\mathsf M_J$ by $S_0\otimes(V/W)$ determined by $A_W$
corresponds, under the Koszul identification in (b), to the cocycle
$J''\ni g\mapsto g\bmod I^+_W\in V/W$, i.e.\ to the quotient map with kernel $W$,
and these globalize to the tautological quotient.
\end{proof}

\subsection{Grassmannian acyclicity}\label{sec:BWB}

The following vanishing is the fixed-point shadow of the excess vanishing of
\cref{sec:hecke}.  It is not used in the proof of the vanishing, but it explains the
mechanism at the level of a single wall fiber. 

\begin{theorem}\label{thm:BWB}
Let $F=\Gr(k,N)$ with quotient bundle $\cQ$, $\cO_{\mathrm{Pl}}(1)=\det\cQ$, and
suppose $N=2k-1$.  Then
\[
 \RGamma\bigl(F,\ \cQ\otimes\cO_{\mathrm{Pl}}(1-N)\bigr)=0
 \quad\Longleftrightarrow\quad k\ge2 .
\]
\end{theorem}

\begin{proof}
In the quotient-block convention, $\cQ\otimes\cO_{\mathrm{Pl}}(1-N)$ corresponds to
the $GL_N$-weight
$\mu=(2-N,\,\underbrace{1-N,\dots,1-N}_{N-k-1},\,\underbrace{0,\dots,0}_{k})$.
Adding $\varrho=(N-1,\dots,1,0)$,
\[
 \mu+\varrho=\bigl(1,\ -1,-2,\dots,1-(N-k),\ k-1,k-2,\dots,0\bigr).
\]
Since $N-k=k-1$, the final block is $(k-1,\dots,1,0)$ and contains the entry $1$
exactly when $k\ge2$, repeating the first entry.  A $\varrho$-shifted weight with a
repetition is singular, and Bott's theorem \cite{Bott} gives vanishing in all
degrees.  For $k=1$ the entries are pairwise distinct and exactly one cohomology
group (of dimension one) survives.
\end{proof}

\subsection{The local model at a wall point}\label{sec:local-model}

\begin{theorem}[{Bellamy--Craw \cite[Thm.~1.4; Thms.~3.2, 3.7; Prop.~3.8,
\S3.4]{BellamyCraw}}]\label{thm:BC}
Let $f:\mfM_\theta(\mathbf v,\mathbf w)\to\mfM_{\theta_0}(\mathbf v,\mathbf w)$ be
the VGIT morphism of framed McKay quiver varieties for $\Gamma\subset SL_2(\C)$, from
a generic chamber to a wall.  Fix a closed point $z$ with polystable representative
$y$ and local Ext-quiver datum $(\mathbf m,\mathbf n)$.  Then the following hold.
\begin{enumerate}[label=\textup{(\alph*)}]
\item \'etale locally over $z$, $f$ is the product of the morphism
$\mfM_\varrho(\mathbf m,\mathbf n)\to\mfM_0(\mathbf m,\mathbf n)$ of local quiver
varieties with the identity of a smooth symplectic factor;
\item the equivalence is realized before GIT, on a common affine
$H$-variety $U$ with \'etale maps to both moment-map fibers and Cartesian squares
before and after quotient.  At representation-space level the map is
$c\mapsto y+c$, hence the pullback of the global universal representation is the family
$y+c$ of representations of the \emph{original} framed preprojective algebra;
\item $p_\theta^*\cR_i(\mathbf v,\mathbf w)\cong\bigoplus_j
q_\varrho^*\cR_j(\mathbf m,\mathbf n)^{\oplus\beta^{(j)}_i}$, where $\beta^{(j)}$ run
over the dimension vectors of the stable summands of $y$.  The framing summand
$\cR_\infty$ is trivialized, removing the common line-bundle ambiguity of the
universal family.
\end{enumerate}
\end{theorem}

Fix a closed point $z\in\cZ_n^{\mathrm{symb}}$ of exact even-socle rank $r$.  By
\cref{prop:ext-quiver}(a) (the odd socle never splits off at this wall) the
polystable representative is
\begin{equation}\label{eq:polystable}
 E_z=E_{\mathrm{fr}}\oplus S_0^{\oplus r},
\end{equation}
where $E_{\mathrm{fr}}$ is the stable summand containing the framing.  The framing
rigidification kills its scalar automorphisms, while the repeated simple factor has
stabilizer $H=GL(U)$, $U=\C^r$.  Put $V:=\Ex^1(S_0,E_{\mathrm{fr}})$.  At a heart,
$E_{\mathrm{fr}}=\mathsf M_J$, $V=V_J$ and $\dim V=2r+1$ by
\cref{thm:fiber-grass}(c) (the dimension is constant along the stratum by the
\'etale-local triviality of the stratification \cite{BellamySchedler}).  The
$2$-Calabi--Yau pairing identifies $\Ex^1(E_{\mathrm{fr}},S_0)\cong V^\vee$ up to
the symplectic character, and $\Ex^1(S_0,S_0)=0$, that is, there are no loop coordinates at
the multiplicity vertex.  Applying \cref{thm:BC} at $z$, the local model is
\begin{equation}\label{eq:local-model}
 \bigl\{(A,B,u):\ AB=0\bigr\}/\!\!/ GL(U),\quad
 A\in\Hm(V,U),\ B\in\Hm(U,V),\ u\in U_z:=\Ex^1(E_{\mathrm{fr}},E_{\mathrm{fr}}),
\end{equation}
with the invariant-theoretic wall map $(A,B,u)\mapsto(X,u)$, $X:=BA\in\End(V)$.  The
moment-map equation gives $X^2=B(AB)A=0$ and $\rk X\le r$, hence the transverse wall
singularity is the orbit closure
\begin{equation}\label{eq:Nr}
 \mfN_{r,2r+1}:=\{X\in\End(V):\ X^2=0,\ \rk X\le r\}.
\end{equation}

\emph{The $A$-surjective chamber, the source of $\rho$.}  In the chamber defining
$Y_n$, stability requires $A$ to be surjective.  Let $K:=\ker A$, of dimension
$(2r+1)-r=r+1$.  Then $A$ identifies $U\cong V/K$, and $AB=0$ says $\im B\subset K$, hence
$B$ becomes a map $V/K\to K$, i.e.\ a cotangent vector
$T^*_K\Gr(r+1,V)=\Hm(V/K,K)$.  Hence
\begin{equation}\label{eq:Yplus}
 Y^{\mathrm{loc}}_+\cong T^*\Gr(r+1,V)\times(\text{core}),
 \qquad (K,B)\longmapsto BA .
\end{equation}
The fiber over $X=0$ is the zero section, because if $BA=0$ and $A$ is surjective then
$B=0$.  Hence $\rho^{-1}(z)\cong\Gr(r+1,V)$, which is the wall fiber $F_J$ with
$k=r+1$, its tautological quotient $V\otimes\cO\twoheadrightarrow\cQ$ of rank $r$
being the varying even socle (\cref{prop:cluster-socle}).  On the standard big cell
$V=W_0\oplus Q_0$, $\dim W_0=k$, $\dim Q_0=r$,
\begin{equation}\label{eq:bigcell}
 A_C=(C\ \ I_{Q_0}),\qquad B_{C,D}=\binom{D}{-CD},\qquad
 B_{C,D}A_C=\begin{pmatrix}DC&D\\-CDC&-CD\end{pmatrix},
\end{equation}
hence $BA=0$ if and only if $D=0$.  Thus the zero fiber of the local wall contraction
is cut out by the cotangent coordinates, independently of the Grassmannian
coordinates $C$.  Moreover the center of $H$ acts by
$\zeta\cdot(A,B,u)=(\zeta A,\zeta^{-1}B,u)$.  Hence every $H$-equivariant formal
coordinate change preserves $\{B=0\}$, and on $\{B=0\}$ transforms
$A\mapsto A\Psi(u)$ with $\Psi(u)$ invertible.  Therefore the induced point of the
Grassmannian and its tautological quotient bundle are canonically unchanged.

\emph{The $B$-injective chamber.}  For the opposite sign of the stability
character, stability requires $B$ to be injective.  With $L:=\im B$, of dimension
$r$, $A$ factors through $V/L\to L$, and the opposite quotient is
$T^*\Gr(r,V)$ with zero fiber $\Gr(r,V)$.  We obtain the two-sided local diagram
\begin{equation}\label{eq:two-chambers}
 T^*\Gr(r,V)\ \longrightarrow\ \mfN_{r,2r+1}\ \longleftarrow\ T^*\Gr(r+1,V).
\end{equation}
Only the right-hand map is the local model of the given $\rho$.  For $X$ of
maximal rank $r$, $L=\im X\subset K=\ker X$ because $X^2=0$.  The two resolutions
remember the two halves of the same flag $L\subset K\subset V$, and are canonically
dual through the $2$-Calabi--Yau pairing.  The left-hand Grassmannian
$\Gr(r,2r+1)$ is the one of the Brill--Noether recursion below.

\subsection{Brill--Noether loci and the thresholds}\label{sec:BN}

For a stable representation $D\in Y_n$ let $\sigma_i(D)=\dim\Hom(S_i,D)$ be the
even and odd socle dimensions, and
\[
 \BN_n(a,b):=\{D\in Y_n:\ \sigma_0(D)\ge a,\ \sigma_1(D)\ge b\},\qquad
 \BN^+_{n,r}:=\BN_n(r,0).
\]
For the framed affine $A_1$ quiver with dimension vector $(n,n)$ and framing
$(1,0)$ the theory of Gonz\'alez--Gorsky--Simental \cite[Thms.~1.3, 1.11]{GGS}
gives, whenever the locus is nonempty,
\begin{equation}\label{eq:BN-dim}
 \dim\BN_n(a,b)=2n-a(a+1)-b^2,
\end{equation}
that $\BN_n(a,b)$ is irreducible and Cohen--Macaulay, that the exact-socle stratum
$\{\sigma_0=a,\sigma_1=b\}$ is smooth and dense in it, and that it is resolved by a
split-parabolic correspondence.  In particular
$\dim\BN^+_{n,r}=2n-r(r+1)$, and the boundary on which the even socle jumps from
$r$ to at least $r+1$ has codimension
\begin{equation}\label{eq:BN-boundary}
 \dim\BN^+_{n,r}-\dim\BN^+_{n,r+1}=2r+2
\end{equation}
in $\BN^+_{n,r}$.  We use \eqref{eq:BN-dim} with $(a,b)=(1,0)$ and $(2,0)$ in
\cref{sec:hecke}.

The split-parabolic resolution of \cite{GGS} is built from a recursion, which we
make explicit.  For a dimension vector $d=(d_0,d_1)$ define the defect vector
\begin{equation}\label{eq:defect}
 k^0[d]=\bigl(2d_0-2d_1-1,\ 2d_1-2d_0\bigr),
\end{equation}
start with $d^{(0)}=(n,n)$, $k^{(0)}=(r,0)$, and set recursively
\begin{equation}\label{eq:recursion}
 d^{(j+1)}=d^{(j)}-k^{(j)},\qquad
 k^{(j+1)}=\max\bigl(k^{(j)}+k^0[d^{(j+1)}],\,0\bigr)
\end{equation}
(componentwise maximum).  In words, one removes the prescribed socle, computes the forced
socle of the quotient, and repeats.

\begin{proposition}[Explicit recursive sequence]\label{prop:BN-recursion}
The sequence \eqref{eq:recursion} is
\[
  (r,0),\ (0,2r),\ (2r-1,0),\ (0,2r-2),\ \dots,\ (0,2),\ (1,0),\ (0,0),
\]
that is, $k^{(2a+1)}=(0,2r-2a)$ for $0\le a\le r-1$ and $k^{(2a)}=(2r-2a+1,0)$ for
$1\le a\le r$, with the exceptional initial value $k^{(0)}=(r,0)$.  The total amount
removed at each vertex is $r(r+1)$, hence the terminal dimension vector is
$d^{(2r+1)}=(n-r(r+1),\,n-r(r+1))$.
\end{proposition}

\begin{proof}
From $d^{(1)}=(n-r,n)$ we get $k^0[d^{(1)}]=(-2r-1,2r)$, hence
$k^{(1)}=\max((r,0)+(-2r-1,2r),0)=(0,2r)$.  Suppose $k^{(2a)}=(2r-2a+1,0)$ with
$1\le a\le r$.  One checks $d^{(2a+1)}_0-d^{(2a+1)}_1=a-r$, hence
$k^0[d^{(2a+1)}]=(-2(r-a)-1,2(r-a))$ and $k^{(2a+1)}=(0,2(r-a))$.  Subtracting reverses
the sign of the difference, $d^{(2a+2)}_0-d^{(2a+2)}_1=r-a$, hence
$k^0[d^{(2a+2)}]=(2(r-a)-1,-2(r-a))$ and $k^{(2a+2)}=(2(r-a)-1,0)$, the even-stage
formula with $a+1$.  The sums are $r+(2r-1)+\dots+1=r(r+1)$ and
$2r+(2r-2)+\dots+2=r(r+1)$.
\end{proof}

\begin{theorem}[Even-socle Brill--Noether recursion]\label{thm:BN}
$\BN^+_{n,r}\ne\emptyset$ if and only if $n\ge r(r+1)$.  When nonempty,
$\BN^+_{n,r}$ is birational to a $\Gr(r,2r+1)$-bundle over a dense open subset of
$Y_{n-r(r+1)}$.
\end{theorem}

\begin{proof}
The split-parabolic variety resolving $\BN^+_{n,r}$ \cite[Thm.~1.11]{GGS} is
nonempty precisely when the recursion reaches a nonempty terminal quiver variety,
which by \cref{prop:BN-recursion} is $Y_{n-r(r+1)}$.  The generic fiber of the
first split-parabolic projection is $\Gr(k^{(0)}_0,-k^0_0[d^{(1)}])=\Gr(r,2r+1)$.  At
every later step the recurrence forces $-k^0_v[d^{(j+1)}]=k^{(j)}_v$ at the active
vertex, hence the generic Grassmannian is a point.
\end{proof}

\begin{corollary}[Triangular threshold]\label{cor:threshold}
An exact even-socle rank $r$ occurs on $Y_n$ only when $r(r+1)\le n$.  The first
possible values are $r=1,2,3,4$ for $n\ge2,6,12,20$.  In particular every
positive-dimensional wall fiber for $n\le5$ is a $\PP^2=\Gr(2,3)$, and $\Gr(3,5)$
first appears at $n=6$.
\end{corollary}

\begin{lemma}[Trace preservation along the recursion]\label{lem:trace-preservation}
Let $0\to K\to D\to\widehat D\to0$ be one of the universal split-parabolic exact
sequences, where the successive subquotients of $K$ are vertex simples.  Then for
every positive-degree invariant polynomial $f$, $\Tr(f|D)=\Tr(f|\widehat D)$.
Consequently the recursion restricts to the punctual fibers, and reversing it over an
irreducible component of the punctual fiber of $Y_{n-r(r+1)}$ produces an
$n$-dimensional subvariety of $F_n\cap\BN^+_{n,r}$ birational to a
$\Gr(r,2r+1)$-bundle over that component.
\end{lemma}

\begin{proof}
Every positive-length path acts by zero on a vertex simple, hence relative to a
filtration adapted to $K$ the action of $f$ on $D$ is block upper triangular with
zero diagonal block on $K$.
\end{proof}

\begin{remark}[What the recursion does and does not prove]\label{rem:recursion-role}
The recursion identifies the complementary Grassmannian $\Gr(r,2r+1)$ of
\eqref{eq:two-chambers}, gives the threshold $r(r+1)\le n$, proves that the exact
rank stratum is smooth and dense with boundary of codimension $2r+2$, and shows that
the scalar trace functions are unchanged along it.  It is not an induction on $n$
proving symbolic descent, and it will not be used as one.  The descent argument of
\cref{sec:SD} is local at each wall point and uses only the Grassmannian structure
of the fiber.
\end{remark}

\section{The rank-one Hecke correspondence and the excess vanishing}\label{sec:hecke}

In this section we prove part (3) of Theorem~E, the vanishing $R\rho_*\cM_n^+=0$ of
the even part of the excess sheaf.  This is the central geometric result of the
paper.  The main idea is to realize $\cM_n^+$ as a pushforward from the rank-one
Hecke correspondence $\Hecke$, which parametrizes a cluster together with an even
socle line.  The correspondence has two projections, to $Y_n$ and to the smaller
quiver variety $Y'=\mfM_\theta((n-1,n),(1,0))$, and both projections are cut out by
regular sections inside projective bundles.  Bott's formula on the first projective
bundle bounds the Tor-amplitude of the pushforward of every term by $[-2,0]$, and
since the excess is supported in codimension $n$, the pushed-forward complex is
concentrated in a single degree.  On the other hand the inverse $\cL^{-1}$ of the universal socle line lies in the
negative acyclic window of the second projection, hence the pushforward to the wall
vanishes.

The identification of $\cM_n^+$ with the object on the Hecke correspondence is the
technical heart of the section.  It passes through a determinantal model of the even
socle locus, the canonical module of that locus, and the dual-socle formula of
\cref{sec:haiman}.  We conclude with the Cohen--Macaulay consequences and with the
two lowest ranks in which the argument can be followed by hand.

Throughout the section $n$ is fixed and we write $Y=Y_n$, $\cP=\cP_n=\cR_0\oplus\cR_1$,
$\cT=\cT_n$, $\cE=\cE_n$, $\cK=\cK_n$, $J=J_n$, $\cI=\cI_n$, $\cM=\cM_n$,
$\cS=\cS_n$.  The main result of the section is the following theorem.

\begin{theorem}[Even-excess vanishing]\label{thm:vanishing}
For every $n\ge1$,
\begin{equation}\label{eq:main-vanishing}
 R\rho_*\cM_n^+=0,
\end{equation}
i.e.\ $R^i\rho_*\cM_n^+=0$ for every $i\ge0$, and the same holds in the
$T$-equivariant derived category.
\end{theorem}

For $n\ge2$ the proof produces a more precise statement.  Put
\[
 Y'=\mfM_\theta((n-1,n),(1,0)),
 \qquad
 Y\xleftarrow{\ p\ }\Hecke\xrightarrow{\ q\ }Y',
\]
where $\Hecke$ classifies exact sequences
\begin{equation}\label{eq:universal-sequence}
 0\longrightarrow S_0\otimes\cL\longrightarrow p^*\cP\longrightarrow q^*\cP'\longrightarrow0
\end{equation}
of framed $\Lambda$-modules, and $\cL$ is the universal even socle line.  Define
\begin{equation}\label{eq:D-lines}
 \scrD=(\det\cR_1)^2(\det\cR_0)^{-2},\qquad
 \scrD'=(\det\cR'_1)^2(\det\cR'_0)^{-2}.
\end{equation}
There is a universal augmentation $\ev':\cP'\to\cO_{Y'}$ at the origin
(\cref{lem:smaller-augmentation}), and we use the modified trace and its kernel
\begin{equation}\label{eq:modified}
 \lambda'=\Tr_{\cR'_0}+\ev',\qquad
 \cT'=\ker\lambda',\qquad
 C'=K_{Y'}(\cT',\ev'|_{\cT'}).
\end{equation}
Note that $\lambda'(1)=(n-1)+1=n$ and $\rk\cT'=n-2$.

\begin{theorem}[Hecke realization]\label{thm:realization}
Let $n\ge2$.  After forgetting the auxiliary torus linearizations, there is an isomorphism
\begin{equation}\label{eq:main-realization}
 \cM_n^+\simeq Rp_*\bigl(\cL^{-1}\otimes q^*(\scrD'\otimes C')\bigr),
\end{equation}
and the object on the right is concentrated in degree zero.  Moreover
\begin{equation}\label{eq:q-window-intro}
 Rq_*\cL^{-1}=Rq_*\cL^{-2}=0.
\end{equation}
\end{theorem}

The order of the argument matters.  The statement that a complex is annihilated
by $R\rho_*$ does not imply that its zeroth cohomology is annihilated, and we prove
that the right-hand side of \eqref{eq:main-realization} is a sheaf \emph{before}
identifying it with $\cM_n^+$.  The argument proceeds in six steps.  We construct a
codimension-two determinantal model of the even socle locus, we identify the canonical
module of that locus by means of the dual-socle formula, we express the excess as the
first homology of a trace Koszul complex, we use the rank-one Hecke filtration to
obtain a uniform Tor-amplitude $[-2,0]$, we deduce concentration from the punctual
support in codimension $n$, and finally we use the negative projective window for the
second projection.

\subsection{Geometric inputs}\label{sec:geometry-input}

\begin{inputtheorem}\label{input:geometry}
In the cyclic chamber, the quiver varieties with dimension vectors $(n,n)$ and
$(n-1,n)$ are smooth schemes of cyclic commutative $\Gamma$-cluster algebras, of
dimensions $2n$ and $2n-4$.  The rank-one even Hecke scheme $\Hecke$ is smooth of
pure dimension $2n-2$.  Its first image is the even socle locus of codimension two,
which has a dense exact socle-rank-one stratum.  If nonempty, the locus of even socle
rank at least two has codimension six in $Y$.  Finally, $\dim F_n\le n$.
\end{inputtheorem}

Smoothness of $\mfM_\theta(v,w)$ for generic $\theta$ and the dimension formula
$\dim\mfM_\theta(v,w)=2w^Tv-v^TCv$ are Nakajima's \cite{Nakajima98}.  That the cyclic
chamber produces $n\Gamma$-Hilb, with the universal object a commutative cluster
algebra rather than merely a framed representation, is \cite{Kuznetsov,CGGSpunctual}
(\cref{app:ADHM} gives the algebraic content of this identification over an arbitrary
base).  For the correspondence and the Brill--Noether assertions we use
\cite[Thms.~1.3, 1.11, 2.23; Lemmas~3.7, 3.14]{GGS}, in the version arXiv:2601.22287v1
of 29 January 2026.  The stability convention there is generation by the framing,
which is ours, and there are no loops at the modified vertex, hence their
split-parabolic correspondence is the usual rank-one Nakajima correspondence.  The
last assertion is \cref{lem:dimF}.

Let us record the numerical checks.  With $C=\left(\begin{smallmatrix}2&-2\\-2&2\end{smallmatrix}\right)$
and $w=(1,0)$ we get $\dim Y=2n$ and $\dim Y'=2(n-1)-2=2n-4$.  For $v=(n,n)$ the vector
$k^0=Cv-w$ is $(-1,0)$, and for the modification $k=(1,0)$ the codimension is
$\sum_ik_i(k_i-k_i^0)=1\cdot(1+1)=2$, hence the correspondence has dimension $2n-2$.
For $v'=(n-1,n)$, $k^0=(-3,2)$, and the second projection has fiber
$\PP^{s_0(E')+2}$ where $s_0(E')=\dim\soc_0(E')$, which also follows from the bundle
presentation of \cref{sec:q-presentation}.

\begin{remark}[Nonemptiness]
For $n=2$ the smaller cluster is $S/(x,y)^2$, of dimension vector $(1,2)$.  For
$n>2$, adjoining $n-2$ distinct free $\Gamma$-orbits gives a point of $Y'$, and a
balanced length-four cluster with even socle together with free orbits gives a point
of $\Hecke$.  Thus the nonemptiness qualifications in the cited results hold.
\end{remark}

\begin{convention}[Projectivizations]\label{conv:proj}
For a vector bundle $V$, $\PP_{\mathrm{lines}}(V)$ carries $\cO(-1)\hookrightarrow V$ and
$\PP_{\mathrm{quot}}(V)$ carries $V\twoheadrightarrow\cO(1)$.  On $\Hecke$ the same line
satisfies
\begin{equation}\label{eq:L-conventions}
 \cL=\cO_{\PP_{\mathrm{lines}}(\cT)}(-1)|_{\Hecke}
  =\cO_{\PP_{\mathrm{quot}}(\cV')}(1)|_{\Hecke},
\end{equation}
and the two projective bundles have different bases and different moduli meanings.
\end{convention}

\subsection{The Hecke scheme and its two incidence presentations}\label{sec:hecke-scheme}

\subsubsection{The moduli functor}
For a test scheme $T_0$, a $T_0$-point of $\Hecke$ consists of a family of
balanced cyclic $S$-algebras $E$ (a family of points of $Y$) and an even line
subbundle $L\subset E_0$ annihilated by $x$ and $y$.  The quotient $E'=E/L$ is
locally free with ranks $(n-1,n)$ and cyclic with the induced framing, hence a
family of points of $Y'$.  Conversely, a family of exact sequences of this form gives
the same datum.  The two projections send the sequence to its larger and smaller
algebra, the universal sequence is \eqref{eq:universal-sequence}, and in particular
\begin{equation}\label{eq:vertex-exact-sequences}
 0\to\cL\to p^*\cR_0\to q^*\cR'_0\to0,
 \qquad p^*\cR_1\simeq q^*\cR'_1.
\end{equation}
The descriptions below are descriptions of this functor, not only of its closed
points.

\subsubsection{The smaller cluster has a universal augmentation}

\begin{lemma}\label{lem:smaller-augmentation}
Every closed point of $Y'$ contains the origin in its support.  Evaluation at the
origin descends to a morphism of universal algebras $\ev':\cP'\to\cO_{Y'}$,
compatible with arbitrary base change.
\end{lemma}
\begin{proof}
A $\Gamma$-stable finite subscheme supported away from the origin is a union of
free $\Gamma$-orbits with multiplicities, each of which has equal even and odd
dimensions.  The dimension vector $(n-1,n)$ is unbalanced, hence there is support at
the origin.  To pass from points to the universal family, work on an affine open
$\SpecOp A\subset Y'$ with universal ideal $I'\subset A[x,y]$.  For $f\in I'$ the
regular function $f(0,0)\in A$ vanishes at every closed point of $\SpecOp A$.  Since
$Y'$ is smooth, $A$ is reduced, and therefore $f(0,0)=0$.  Hence evaluation $A[x,y]\to A$ kills
$I'$ and factors through the universal quotient.  These local maps glue, and the
resulting universal map pulls back to any, including nonreduced, test scheme.
\end{proof}

\begin{warning}
Fiberwise containment of the origin is not by itself enough to produce such an
augmentation over an arbitrary nonreduced base.  The reducedness of the universal
moduli scheme is used first, and the universal map is then pulled back.
\end{warning}

\subsubsection{Trace additivity and the universal trace-bundle sequence}
Pull back $\ev'$ along $q$ and compose with the quotient in
\eqref{eq:universal-sequence}.  This gives an augmentation $\ev_\Hecke:p^*\cP\to\cO_\Hecke$.
Polynomial multiplication on $\cL$ is through evaluation at the origin, because
$x\cL=y\cL=0$.  Hence for a section $a$ of $p^*\cR_0$ with image $a'$ in $q^*\cR'_0$,
multiplication on the line $\cL$ is the scalar $\ev'(a')$.

\begin{proposition}\label{prop:trace-additivity}
There is an identity of linear functionals on $p^*\cR_0$,
\begin{equation}\label{eq:trace-additivity}
 \Tr_{p^*\cR_0}(a)=\Tr_{q^*\cR'_0}(a')+\ev'(a').
\end{equation}
In particular $\cL\subset p^*\cT$, and there is a canonical exact sequence of
vector bundles
\begin{equation}\label{eq:trace-bundle-sequence}
 0\longrightarrow\cL\longrightarrow p^*\cT\longrightarrow q^*\cT'\longrightarrow0,
 \qquad \cT'=\ker(\Tr_{\cR'_0}+\ev').
\end{equation}
\end{proposition}
\begin{proof}
Multiplication by $a$ preserves the subbundle $\cL$ in
\eqref{eq:vertex-exact-sequences}, and in a local splitting its matrix is block
triangular with diagonal blocks $\ev'(a')$ and multiplication by $a'$ on
$q^*\cR'_0$.  Trace is additive on these blocks over any base ring, proving
\eqref{eq:trace-additivity}.  For $a\in\cL$ the image $a'$ is zero, hence
$\Tr(a)=0$ and $\cL\subset p^*\cT$.  Since $\Tr_{\cR_0}(1)=n$ and
$(\Tr_{\cR'_0}+\ev')(1)=(n-1)+1=n$, both functionals are split surjections, and taking
kernels in \eqref{eq:vertex-exact-sequences} gives
\eqref{eq:trace-bundle-sequence}.
\end{proof}

\begin{remark}[Why the unmodified trace is wrong]
The lost even line contributes its scalar action to the trace, hence the trace
inherited by the quotient is $\Tr_{\cR'_0}+\ev'$, not $\Tr_{\cR'_0}$, and the ranks alone
do not repair the discrepancy.  For example, let the smaller algebra be the disjoint
union of $S/(x,y)^2$ and a free two-point orbit.  Its even algebra has two idempotent
factors, on which the trace is $(1,1)$ whereas the augmentation is $(1,0)$, and adding an
even socle line at the origin changes the trace to $(2,1)$, exactly the modified
trace.
\end{remark}

Taking determinants in \eqref{eq:vertex-exact-sequences} gives
\begin{equation}\label{eq:det-Hecke}
 p^*\scrD=\cL^{-2}\otimes q^*\scrD'.
\end{equation}

\subsubsection{The presentation over the larger cluster}\label{sec:p-presentation}
On $Y$ put
\[
 \cX=\cR_1\oplus\cR_1\oplus\cO_Y,
 \qquad M_2=(x,y,i):\cX\to\cR_0,
 \qquad M_1=(-y,x,0):\cR_0\to\cX,
\]
$i$ the framing.  The relation $M_2M_1=0$ is commutativity.  Cyclicity implies that
$M_2$ is surjective, because every nonconstant even monomial is $x$ or $y$ times an odd
monomial, and the constant monomial is the framing.  Hence $\cV=\ker M_2$ is locally
free of rank $n+1$.  The section $u=M_1(1)$ is nowhere zero, because if it vanished in a
fiber, both $x$ and $y$ would vanish in a cyclic algebra of dimension $2n>1$.  Thus
$\cO_Yu$ is a line subbundle of $\cV$, and we set
\begin{equation}\label{eq:unit-cancelled-map}
 \cG=\cV/\cO_Yu,\qquad
 \psi:\cT\longrightarrow\cG,
 \qquad \rk\cT=a:=n-1,\quad\rk\cG=a+1=n,
\end{equation}
$\psi$ being induced by $M_1|_\cT$.

\begin{proposition}[Unit cancellation on the incidence functor]\label{prop:unit-cancellation}
The Hecke scheme is the zero scheme
\begin{equation}\label{eq:p-incidence}
 \Hecke=Z\!\left(\cO_{\PP_{\mathrm{lines}}(\cT)}(-1)\longrightarrow\pi^*\cG\right)
 \subset\PP_{\mathrm{lines},Y}(\cT),
\end{equation}
with tautological line $\cL$.  This is an equality of moduli schemes, compatible
with arbitrary base change.
\end{proposition}
\begin{proof}
In the uncancelled presentation, choosing a line in $\ker M_1\subset\cR_0$ is
precisely choosing an even line killed by the arrows.  Such a line is an ideal in the cyclic
algebra, and its quotient is the smaller cluster, hence this kernel incidence is the
Hecke functor.  Use the splitting $\cR_0=\cO\cdot1\oplus\cT$.  Let $L_0$ be a line in
$\ker M_1$ over a test scheme.  Its projection to $\cT$ is a line subbundle (it
cannot meet the unit line in any geometric fiber because $M_1(1)$ is nowhere zero)
and is annihilated by $\psi$.  Conversely, if $\widetilde L\subset\cT$ is a line
subbundle with $\psi|_{\widetilde L}=0$, then $M_1|_{\widetilde L}$ factors uniquely
through $\cO u$, say by $c:\widetilde L\to\cO$, and the graph
$\ell\mapsto\ell-c(\ell)1$ is a line subbundle of $\cR_0$ annihilated by $M_1$.
These constructions are inverse over arbitrary rings.  Under this isomorphism the
graph line is the universal kernel $\cL$.  By \cref{prop:trace-additivity} it has
trace zero, and since $\Tr(\ell-c(\ell)1)=-nc(\ell)$ for $\ell\in\cT$, one has $c=0$
on the incidence.  Thus the actual socle line is the tautological line of
\eqref{eq:p-incidence}.
\end{proof}

The ambient projective bundle has dimension $2n+(n-2)=3n-2$, whereas $\Hecke$ is
smooth of dimension $2n-2$, and the defining section has rank $n$.  Therefore it is a
regular section (\cref{app:regular-section}), the zero scheme carries its ordinary
scheme structure, and its Koszul resolution is exact.  Let $\overline\cW$ be the
universal quotient on $\PP_{\mathrm{lines}}(\cT)$ and $\cW=\overline\cW|_\Hecke$, thus
\begin{equation}\label{eq:W-sequence}
 0\to\cL\to p^*\cT\to\cW\to0,\qquad\cW\simeq q^*\cT'
\end{equation}
by \cref{prop:trace-additivity}.

\subsubsection{The presentation over the smaller cluster}\label{sec:q-presentation}
On $Y'$ define $\cX'=\cR'_1\oplus\cR'_1\oplus\cO_{Y'}$, $M'_2=(x',y',i'):\cX'\to\cR'_0$
and $M'_1=(-y',x',0):\cR'_0\to\cX'$.  Again $M'_2$ is surjective, its kernel $\cV'$
has rank $n+2$, and $M'_1$ factors through a map $\phi':\cR'_0\to\cV'$.

\begin{proposition}[Quotient incidence, with reconstruction]\label{prop:q-incidence}
There is a functorial isomorphism
\begin{equation}\label{eq:q-incidence}
 \Hecke\simeq\PP_{\mathrm{quot},Y'}(\coker\phi')\subset\PP_{\mathrm{quot},Y'}(\cV'),
\end{equation}
with universal quotient line $\cL$, and the closed immersion is cut out by the section
$\cR'_0\to\cV'\to\cO(1)$.
\end{proposition}
\begin{proof}
A test-scheme point on the right is a line quotient of $\cV'$ killing $\im M'_1$.
If $U\subset\cV'$ is its kernel, then $\im M'_1\subset U\subset\ker M'_2\subset\cX'$
and $\cV'/U=L$.  Define $E_0=\cX'/U$ and $E_1=E'_1$.  The map $M'_2$ induces
$\pi_0:E_0\twoheadrightarrow E'_0$ with kernel $L$.  The incoming maps
$x_1,y_1:E'_1\to E_0$ are induced by the first and second inclusions into $\cX'$,
the framing by the third, and the outgoing maps are $x_0=x'_0\pi_0$, $y_0=y'_0\pi_0$.
The even commutativity relation is exactly $\im M'_1\subset U$, the odd one is
inherited from $E'$, and the line $L$ is killed by $x_0,y_0$.  As for cyclicity, the
submodule generated by the new framing surjects onto the cyclic module $E'$.  Because the
kernel is even, this submodule contains all of $E_1=E'_1$, and applying $x_1,y_1$
and adjoining the framing then generates all of $E_0=\cX'/U$.  This works over every
base, using surjectivity between the corresponding locally free modules.  For the
inverse, start with a Hecke sequence.  The map $M_2:E'_1\oplus E'_1\oplus\cO\to E_0$ is
surjective by cyclicity, its kernel $U$ lies in $\ker M'_2$ since $M'_2=\pi_0M_2$,
the outgoing maps kill the kernel line and therefore factor through $E'_0$, the even
moment-map equation gives $M_2M'_1=0$ hence $\im M'_1\subset U$, and
$(\ker M'_2)/U$ is canonically $L$.  The constructions are inverse and commute with
base change.
\end{proof}

The ambient projective bundle has dimension $(2n-4)+(n+1)=3n-3$, the defining
section has rank $n-1$, and its zero scheme is the smooth $(2n-2)$-dimensional
$\Hecke$, hence the section is regular.  The fiber over $E'$ is a projective space of
dimension $s_0(E')+2$, because $\ker\phi'(E')=\soc_0(E')$ and therefore
$\dim\coker\phi'(E')=(n+2)-(n-1-s_0(E'))=3+s_0(E')$.  Thus $q$ is not a projective
\emph{bundle}, it is an incidence inside one.

\subsection{The determinantal scheme and its canonical module}\label{sec:determinantal}

Let $B\subset Y$ have the scheme structure defined by the maximal minors of
$\psi:\cT\to\cG$, and $i:B\hookrightarrow Y$.  By \cref{prop:unit-cancellation} its
underlying set is the positive even-socle locus $\BN^+_{n,1}$, of codimension two by
\eqref{eq:BN-dim}.

\subsubsection{Determinants and Hilbert--Burch}
The two exact sequences defining $\cV$ and $\cG$ give
$\det\cG=\det\cV=(\det\cR_1)^2(\det\cR_0)^{-1}$ and $\det\cT=\det\cR_0$, hence
\begin{equation}\label{eq:D-as-determinant}
 \det\cG\otimes(\det\cT)^{-1}=\scrD.
\end{equation}
The maximal-minor map is naturally $\cG\otimes\scrD^{-1}\to\cO_Y$,
$g\otimes(\det\cT\otimes\det\cG^\vee)\mapsto\Lambda^a\psi\wedge g$, with image the
ideal of $B$.

\begin{proposition}\label{prop:HB}
The sequence
\begin{equation}\label{eq:HB}
 0\longrightarrow\cT\otimes\scrD^{-1}\xrightarrow{\ \psi\ }\cG\otimes\scrD^{-1}
 \longrightarrow\cO_Y\longrightarrow i_*\cO_B\longrightarrow0
\end{equation}
is exact.  In particular $B$ is Cohen--Macaulay, pure of codimension two, and reduced.
\end{proposition}
\begin{proof}
The grade of the maximal-minor ideal is two, since $Y$ is regular and its zero set
has codimension two, and Hilbert--Burch gives \eqref{eq:HB}
\cite[Thm.~20.15]{Eisenbud}, \cite[Thm.~1.4.17]{BrunsHerzog}.  An ideal of grade two
with a length-two free resolution is perfect, hence $B$ is Cohen--Macaulay of pure
codimension two \cite[Thm.~2.1.5, Cor.~2.1.4]{BrunsHerzog}.  On the dense
exact-corank-one locus the kernel-incidence map $\Hecke\to B$ is an isomorphism of
schemes.  Indeed, locally choose an invertible $(a-1)$-minor of $\psi$ and reduce the
matrix to an identity block and a remaining $2\times1$ column, then both the
determinantal ideal and the incidence are cut out by the two entries of that column.
This locus is smooth because $\Hecke$ is, hence $B$ is generically reduced.  A
Cohen--Macaulay scheme has no embedded primes \cite[\S2.1]{BrunsHerzog}, hence $B$ is
reduced.
\end{proof}

\subsubsection{Identifying the transpose cokernel}
The image of $M_1$ lies in $\cV$, and since $\cX^\vee\twoheadrightarrow\cV^\vee$ is
surjective, passing from $\cX$ to $\cV$ does not change the cokernel of the
transpose.  Cancelling the common unit line in domain and image, in local compatible
splittings the map has block form
$\left(\begin{smallmatrix}1&c\\0&\psi\end{smallmatrix}\right)$, and elementary
invertible row and column operations on its transpose identify its cokernel with
$\coker\psi^\vee$, and the zero framing component of $M_1$ does not affect that
cokernel either.  Therefore
\begin{equation}\label{eq:Sigma-cokernel}
 \coker\bigl((-y,x)^\vee:\cR_1^\vee\oplus\cR_1^\vee\to\cR_0^\vee\bigr)\simeq\coker\psi^\vee ,
\end{equation}
and the sheaf on the left is the even part of $\cP^\vee/(x,y)\cP^\vee$, the even
ambient dual-socle sheaf.  Dualizing \eqref{eq:HB},
\begin{equation}\label{eq:canonical-Sigma}
 \cExt_Y^2(i_*\cO_B,\cO_Y)\simeq\coker\psi^\vee\otimes\scrD ,
\end{equation}
and the right side is annihilated by the ideal of $B$.  We write
\begin{equation}\label{eq:Sigma-on-B}
 i_*\Sigma=\coker\psi^\vee ,
\end{equation}
hence $\omega_{B/Y}:=\cExt_Y^2(i_*\cO_B,\cO_Y)=i_*(\Sigma\otimes\scrD|_B)$.  Since
the absolute canonical module differs from this by the invertible $\omega_Y|_B$,
$\Sigma$ is a canonical module of $B$ up to an invertible twist.  In particular it is
maximal Cohen--Macaulay on $B$ and
\begin{equation}\label{eq:Sigma-self-Hom}
 \RcHom_B(\Sigma,\Sigma)\simeq\cO_B
\end{equation}
by \cref{lem:semidualizing}.

\begin{warning}[Not a regular-embedding argument]
For larger even socle rank the scheme $B$ is singular and need not be a local
complete intersection.  We use the canonical module supplied by \eqref{eq:HB} and
never write $i^!\cO_Y$ as a determinant of a normal bundle.  Instead,
$i^!\cO_Y\simeq(\Sigma\otimes\scrD|_B)[-2]$, which follows from the explicit resolution
and is valid on all of $B$.
\end{warning}

\begin{remark}[The uncancelled form]
Before cancelling the unit, $\coker\psi^\vee$ is the cokernel of
$\sigma^\vee:\cV^\vee\to\cR_0^\vee$ with $\rk\cV=n+1$, $\rk\cR_0=n$, and
\eqref{eq:HB} is the Buchsbaum--Rim complex of $\sigma^\vee$,
$0\to(\det\cV)^{-1}\otimes\det\cR_0\to\cV^\vee\to\cR_0^\vee\to\coker\sigma^\vee\to0$,
whose acyclicity criterion is again grade two of the maximal minors.  The two
presentations agree because $\det\cG=\det\cV$ and the unit column is split.
\end{remark}

\subsection{Trace complexes and the expression for the excess}\label{sec:traces}

\subsubsection{The augmentation on the determinantal scheme}
Every point of $B$ contains the origin, since it has a nonzero element annihilated by
$x$ and $y$.  Since $B$ is reduced by \cref{prop:HB}, the argument of
\cref{lem:smaller-augmentation} gives a universal augmentation
$\ev_B:\cP|_B\to\cO_B$.  Define
\begin{equation}\label{eq:CB-definition}
 s_B=\ev_B|_{\cT|_B},\qquad C_B=K_B(\cT|_B,s_B),\qquad Z=V(s_B)\subset B.
\end{equation}
For a positive-degree even polynomial $f$,
$\ev_B\bigl(f-\frac{\Tr(f)}n1\bigr)=-\frac{\Tr(f)}n$, and the projections of
even polynomials to $\cT$ generate it.  Hence the ideal of $Z$ is the ideal of all
positive-degree even traces restricted to $B$, and
\begin{equation}\label{eq:Z-support}
 |Z|=|B|\cap|F_n|,\qquad\dim Z\le n .
\end{equation}
This concerns support and the trace ideal on $B$, and it does not assert that the
punctual scheme $F_n$ is Cohen--Macaulay.

\subsubsection{The dual-socle quotient is a trace quotient}
Set $\boldsymbol\Sigma=\cP^\vee/(x,y)\cP^\vee$.  The $S$-action on this sheaf factors
through evaluation at the origin, and its even summand is $i_*\Sigma$ by
\eqref{eq:Sigma-cokernel}.  Hence multiplication by $t\in\cT$ on $i_*\Sigma$ is the
scalar $\ev_B(t)$, and
\begin{equation}\label{eq:S-trace-quotient}
 \cS^+=i_*(\Sigma\otimes\cO_Z)=i_*\cH^0(\Sigma\otimes C_B).
\end{equation}
This is an ordinary quotient statement, and the whole complex $\Sigma\otimes C_B$ is
retained when applying duality.  The scheme $B$ and the module $\Sigma$ are
Cohen--Macaulay of dimension $2n-2$.  For a maximal Cohen--Macaulay module $M$ over a
Cohen--Macaulay ring and an ideal $I$ we have $\grade(I,M)=\dim M-\dim(M/IM)$
\cite[Thm.~2.1.2]{BrunsHerzog}, hence by \eqref{eq:Z-support} the trace ideal has grade
at least $(2n-2)-n=n-2$ on each of $\cO_B$ and $\Sigma$.  Since $C_B$ has $n-1$
generators, Koszul grade-sensitivity \cite[Thm.~1.6.17]{BrunsHerzog} gives
\begin{equation}\label{eq:CB-two-homology}
 C_B,\ \Sigma\otimes C_B\in D^{[-1,0]}(B),
\end{equation}
with all cohomology sheaves supported on $Z$.

\begin{proposition}[The excess as first trace homology]\label{prop:M-trace}
There is an isomorphism of ordinary coherent sheaves
\begin{equation}\label{eq:M-trace}
 \cM_n^+\simeq i_*\bigl(\scrD|_B\otimes\cH^{-1}(C_B)\bigr).
\end{equation}
\end{proposition}
\begin{proof}
Let $N_B=\cH^{-1}(\Sigma\otimes C_B)$ and $S_B=\cH^0(\Sigma\otimes C_B)$.  The
triangle $i_*N_B[1]\to i_*(\Sigma\otimes C_B)\to i_*S_B\to i_*N_B[2]$ has outer
sheaves supported in codimension $\ge n$ in $Y$, hence \cref{lem:low-ext}, exactly
as in the proof of \cref{thm:dual-socle}, yields
\begin{equation}\label{eq:trace-edge-Ext}
 \cExt_Y^n(\cS^+,\cO_Y)\simeq\cH^n\RcHom_Y\bigl(i_*(\Sigma\otimes C_B),\cO_Y\bigr),
\end{equation}
the obstruction terms $\cExt_Y^{n-2}(i_*N_B,\cO_Y)$ and $\cExt_Y^{n-1}(i_*N_B,\cO_Y)$
both vanishing.  Duality for the finite morphism $i$, in the form
$\RcHom_Y(i_*F,\cO_Y)\simeq i_*\RcHom_B(F,i^!\cO_Y)$ with
$i^!\cO_Y=\RcHom_Y(i_*\cO_B,\cO_Y)$ regarded on $B$
\cite[Ch.~III, \S6]{RD}, \cite[Tag 0A7A]{Stacks}, together with
\eqref{eq:canonical-Sigma} and \eqref{eq:Sigma-self-Hom}, gives
\begin{equation}\label{eq:dual-trace-complex}
 \RcHom_Y\bigl(i_*(\Sigma\otimes C_B),\cO_Y\bigr)
 \simeq i_*\RcHom_B\bigl(\Sigma\otimes C_B,\Sigma\otimes\scrD|_B\bigr)[-2]
 \simeq i_*\bigl(\scrD|_B\otimes C_B^\vee\bigr)[-2],
\end{equation}
the second step because $C_B$ is a bounded complex of vector bundles, hence perfect,
and $\RcHom(C\otimes X,W)\simeq C^\vee\otimes\RcHom(X,W)$ for perfect $C$.  By
Koszul self-duality in rank $n-1$ (\cref{lem:koszul-duality}),
$C_B^\vee\simeq C_B[-n+1]\otimes(\det\cT|_B)^{-1}$, hence the degree-$n$ cohomology
of \eqref{eq:dual-trace-complex} is
$i_*\bigl(\scrD|_B\otimes\cH^{-1}(C_B)\otimes(\det\cT|_B)^{-1}\bigr)$.
Multiplication by $\det\cT$ in \eqref{eq:even-dual-socle} cancels the last factor.
\end{proof}

\subsubsection{The smaller trace complex and its support}
Recall $\cT'$ and $C'$ from \eqref{eq:modified}, and set $Z'=V(\ev'|_{\cT'})\subset Y'$.
If $f$ is positive-degree and even then $\ev'(f)=0$, hence
$f-\frac{\Tr_{\cR'_0}(f)}n1\in\cT'$ with $\ev'$-value $-\frac{\Tr_{\cR'_0}(f)}n$.
Thus $Z'$ is defined by the positive-degree even traces on $Y'$, and at a closed
point vanishing of the section is equivalent to $\Tr_{E'_0}=(n-1)\ev'$, which by the
separating-polynomial argument of \cref{lem:punctual-support} holds exactly when the
smaller cluster is punctual.  We use only the resulting support statement, namely
\begin{equation}\label{eq:qC-support}
 \Supp(q^*C')\subset q^{-1}(|Z'|),\qquad p\bigl(q^{-1}(|Z'|)\bigr)\subset|F_n| ,
\end{equation}
the second inclusion because an extension of a module supported at the origin by
$S_0$ is again supported at the origin.  Here $q^*C'$ denotes the termwise
pullback of the locally free complex $C'$, and no flatness and no assertion that
$q^*C'$ has only zeroth cohomology is used.

\subsubsection{The exact Koszul filtration}\label{sec:koszul-filtration}
The two augmentations on $\Hecke$ agree, both evaluating the original polynomial
algebra at the origin.  By \cref{prop:trace-additivity} the section of $p^*(\cT|_B)$
vanishes on $\cL$ and descends to the section of $q^*\cT'$.  Let
$\bar p:\Hecke\to B$ be the factorization of $p$.

\begin{proposition}\label{prop:Koszul-filtration}
There is a short exact sequence of complexes of vector bundles on $\Hecke$,
\begin{equation}\label{eq:Koszul-filtration}
 0\longrightarrow\cL\otimes q^*C'[1]\longrightarrow\bar p^*C_B\longrightarrow q^*C'\longrightarrow0.
\end{equation}
\end{proposition}
\begin{proof}
For an exact sequence $0\to L\to V\to W\to0$ with $L$ a line there are canonical
exact sequences $0\to L\otimes\Lambda^{j-1}W\to\Lambda^jV\to\Lambda^jW\to0$.  If the
Koszul section kills $L$ these are compatible with the Koszul differentials, and
locally $\iota_s(\ell\wedge w)=-\ell\wedge\iota_s(w)$, hence the differential on the
kernel has the minus sign of the shift $[1]$.  Apply this to
\eqref{eq:trace-bundle-sequence}, where local splittings are used only to check the sign.
\end{proof}

\subsection{The uniform projective-space calculation}\label{sec:Bott}

\subsubsection{Cohomology of twisted differential forms}
Let $V$ be an $a$-dimensional vector space, $\pi:\PP_{\mathrm{lines}}(V)\to\SpecOp\C$,
$m=a-1$, with universal sequence $0\to L=\cO(-1)\to V\otimes\cO\to W\to0$.  Dualizing,
$W^\vee\simeq\Omega^1(1)$.

\begin{lemma}[Projective Bott calculation]\label{lem:projective-Bott}
For $0\le b\le m$ the only possible nonzero cohomology of $\Omega^b(k)$ on $\PP^m$ is given by the following table.
\begin{equation}\label{eq:Omega-table}
 \begin{array}{c|c}
 \text{condition}&\text{cohomological degree}\\ \hline
 k>b&0\\
 k=0&b\\
 k<b-m&m
 \end{array}
\end{equation}
All other cases are acyclic, and in the middle case the cohomology is one-dimensional.
\end{lemma}
\begin{proof}
This is Bott's formula \cite{Bott}, \cite[Ch.~I, \S1.1]{OSS}, and here is the elementary
derivation.  The exterior Euler sequence is
\begin{equation}\label{eq:exterior-Euler}
 0\to\Omega^b(k)\to\Lambda^bV^\vee\otimes\cO(k-b)\to\Lambda^{b-1}V^\vee\otimes\cO(k-b+1)
 \to\cdots\to\cO(k)\to0
\end{equation}
\cite[Tags 0FUK, 01XS]{Stacks}.  If $k>0$, every twist is $>-m-1$, hence no term has
higher cohomology, and the complex of global sections is the indicated truncation of
the degree-$k$ strand of the polynomial Koszul complex on $V^\vee$, which is exact
for $k>0$ ($dh+hd=k\,\mathrm{id}$ on total degree $k$ for $d=\sum z_i\iota_i$ and
$h=\sum e_i\wedge\partial_{z_i}$).  If $0<k\le b$ the full strand occurs and
$\Omega^b(k)$ is acyclic, and if $k>b$ only the leftmost kernel, $H^0(\Omega^b(k))$,
remains.  If $k=0$, all negative twists in \eqref{eq:exterior-Euler} lie in $[-m,-1]$
and are acyclic.  Only $H^0(\cO)$ remains, and the long exact sequences place it in
$H^b(\Omega^b)$.  If $k<0$, Serre duality with $\omega=\Omega^m=\cO(-m-1)$ and the
perfect pairing $\Omega^b\otimes\Omega^{m-b}\to\Omega^m$ give
$H^r(\Omega^b(k))^\vee\simeq H^{m-r}(\Omega^{m-b}(-k))$, and the positive-twist case
gives the third row.  The case $m=0$ is included.
\end{proof}

The same degree statements hold relatively for a projective bundle over any base,
by checking on trivializing opens, and every nonzero relative cohomology sheaf occurring
below is locally free and commutes with arbitrary base change.

\subsubsection{The two rows of Hecke bundles}
Retain $a=n-1$ and the presentation $\Hecke\subset\PP_{\mathrm{lines},Y}(\cT)$ of
\cref{prop:unit-cancellation}, with ambient projection $\pi$.  Its exact Koszul
resolution has terms $\Lambda^\ell\pi^*\cG^\vee\otimes\cO(-\ell)$, $0\le\ell\le a+1$,
in degrees $-\ell$.

\begin{proposition}[Uniform Tor-amplitude]\label{prop:two-row-amplitude}
For every $n\ge2$, $e\in\{0,1\}$ and $0\le j\le n-2$,
\begin{equation}\label{eq:two-row-amplitude}
 Rp_*\bigl(\cL^e\otimes\Lambda^j\cW\bigr)\in\Perf^{[-2,0]}(Y).
\end{equation}
This is a Tor-amplitude statement, not only an ordinary cohomology bound.
\end{proposition}
\begin{proof}
The line and quotient bundle on $\Hecke$ extend to the ambient projective bundle,
and tensoring its regular-section Koszul resolution gives terms
\begin{equation}\label{eq:ambient-Koszul-term}
 \Lambda^\ell\pi^*\cG^\vee\otimes L^{\ell+e}\otimes\Lambda^j\overline\cW,
 \qquad\text{in degree }-\ell .
\end{equation}
Put $b=a-1-j$.  Exterior duality for the quotient bundle gives
\begin{equation}\label{eq:quotient-Omega}
 \Lambda^j\overline\cW\simeq\Lambda^{a-1-j}\overline\cW^\vee\otimes\det\overline\cW
 \simeq\Omega^b_\pi(b+1)\otimes\pi^*\det\cT,
\end{equation}
since $\overline\cW^\vee=\Omega^1_\pi(1)$ and $\det\overline\cW=\pi^*\det\cT\otimes L^{-1}$.
Thus the projective cohomology to be calculated is that of $\Omega^b_\pi(k)$ with
$k=b+1-\ell-e=a-j-\ell-e$, up to a vector bundle pulled back from $Y$.  Let us read
\eqref{eq:Omega-table} with $m=a-1$.  The first condition $k>b$ reads $\ell+e<1$,
forcing $(e,\ell)=(0,0)$ with relative degree $r=0$.  The second, $k=0$, reads
$\ell+e=a-j$, giving $\ell=a-j$ for $e=0$ and $\ell=a-j-1$ for $e=1$, with
$r=b=a-j-1$.  The third, $k<b-m$, reads $\ell+e\ge a+1$, giving $\ell=a+1$ for $e=0$
and $\ell\in\{a,a+1\}$ for $e=1$, with $r=a-1$.  The complete list is
\begin{equation}\label{eq:two-row-table}
 \begin{array}{c|c|c|c}
 e&\ell&\text{relative cohomology degree }r&r-\ell\\ \hline
 0&0&0&0\\
 0&a-j&a-j-1&-1\\
 0&a+1&a-1&-2\\ \hline
 1&a-j-1&a-j-1&0\\
 1&a&a-1&-1\\
 1&a+1&a-1&-2
 \end{array}
\end{equation}
For fixed $(e,j)$ the three values of $\ell$ are distinct, including the case $a=1$.
Each nonzero derived pushforward of a term \eqref{eq:ambient-Koszul-term}, with its
shift $-\ell$ included, is therefore a vector bundle in one of the degrees $-2,-1,0$
(locally free because relative cohomology of a homogeneous bundle on a projective
bundle concentrated in one degree is locally free and commutes with base change
\cite[Thm.~III.12.11]{Hartshorne}).  Push the finite stupid filtration of the
Koszul resolution through $R\pi_*$.  It expresses the required object as a finite
succession of extensions of these shifted vector bundles, and perfect complexes of
Tor-amplitude $[-2,0]$ are closed under extensions (tensor the triangle with any
quasicoherent module and use its cohomology sequence, see \cite[Tag 0DJC]{Stacks}).
Possible higher differentials in the associated spectral sequence change the maps in
a locally free representative but cannot enlarge the amplitude interval.
\end{proof}

For reference, the three graded pieces in the pushed filtration are, as
representations,
\begin{equation}\label{eq:two-row-bundles}
 \begin{array}{c|c|l}
 e&\text{degree}&\text{bundle on }Y\\ \hline
 0&0&\Lambda^j\cT\\
 0&-1&\Lambda^{a-j}\cG^\vee\otimes\det\cT\\
 0&-2&\det\cG^\vee\otimes\Lambda^{j+1}\cT\otimes\det\cT\\ \hline
 1&0&\Lambda^{a-j-1}\cG^\vee\otimes\det\cT\\
 1&-1&\Lambda^a\cG^\vee\otimes\Lambda^{j+1}\cT\otimes\det\cT\\
 1&-2&\det\cG^\vee\otimes\schur_{(2,1^j)}\cT\otimes\det\cT
 \end{array}
\end{equation}
(here $\schur_{(2,1^j)}$ is the hook Schur functor, and the top-cohomology factors follow from
$\omega_\pi=L^a\otimes\pi^*(\det\cT)^{-1}$ and relative Serre duality).  The table is
not a claim of a canonical splitting, only the degrees and local freeness are used.
The row $(e,j)=(0,0)$ has terms $\cT\otimes\scrD^{-1},\cG\otimes\scrD^{-1},\cO_Y$ in
degrees $-2,-1,0$, precisely the terms of the Hilbert--Burch resolution
\eqref{eq:HB}.  This is a consistency check, not used below.  The positions in
\eqref{eq:two-row-table} have been verified by an exact Bott-weight computation for
all $a\le50$.

\subsubsection{The structure sheaf of the determinantal scheme}\label{sec:p-structure}

\begin{lemma}\label{lem:B-normal}
The determinantal scheme $B$ is normal and irreducible.
\end{lemma}
\begin{proof}
$B$ is reduced and Cohen--Macaulay, and smooth wherever the even socle has rank
one (\cref{prop:HB}).  The complement is the locus with even socle rank at least
two.  If nonempty, \eqref{eq:BN-dim} with $(a,b)=(2,0)$ gives its codimension in $Y$
as $2\cdot3=6$, hence codimension four in $B$.  Thus $B$ is regular in codimension
one and satisfies $S_2$, and Serre's criterion \cite[Thm.~23.8]{Matsumura},
\cite[Tag 031S]{Stacks} gives normality.  Irreducibility is \cite[Thm.~1.3]{GGS}, and
normality component by component would suffice below.
\end{proof}

\begin{proposition}\label{prop:p-structure}
Writing $p=i\bar p$ with $\bar p:\Hecke\to B$,
\begin{equation}\label{eq:p-structure}
 Rp_*\cO_\Hecke\simeq i_*\cO_B,\qquad R\bar p_*\cO_\Hecke\simeq\cO_B .
\end{equation}
\end{proposition}
\begin{proof}
\Cref{prop:two-row-amplitude} with $(e,j)=(0,0)$ places $Rp_*\cO_\Hecke$ in
cohomological degrees at most zero.  As the derived pushforward of a sheaf it is
also in degrees at least zero, hence all higher direct images vanish.  The map
$\bar p$ is projective, by the kernel incidence description, and birational.  Indeed,
it is an isomorphism over the dense exact-corank-one locus \cite[Thm.~1.11]{GGS}, and
every irreducible component of $\Hecke$ dominates a component of $B$ by a
dimension count.  Over the stratum $\{s_0=k\}\subset Y$, of codimension $k(k+1)$
by \eqref{eq:BN-dim}, the fiber of $\bar p$ is $\PP(\soc_0)=\PP^{k-1}$, hence the
preimage of that stratum has dimension $2n-k(k+1)+(k-1)=2n-k^2-1<2n-2=\dim\Hecke$
for $k\ge2$.  Since $B$ is normal, Zariski's Main Theorem
\cite[Cor.~III.11.4]{Hartshorne} gives $\bar p_*\cO_\Hecke=\cO_B$ (alternatively,
use Stein factorization \cite[Cor.~III.11.5]{Hartshorne}, whose finite birational
factor over a normal integral scheme is an isomorphism).  As $i_*$ is
exact and faithful, the assertion about $R\bar p_*$ follows.
\end{proof}

\subsection{The negative window for the second projection}\label{sec:q-window}

\begin{proposition}\label{prop:q-window}
For the projection $q:\Hecke\to Y'$,
\begin{equation}\label{eq:q-window}
 Rq_*\cL^{-1}=Rq_*\cL^{-2}=0 ,
\end{equation}
and consequently $Rq_*(\cL^{-t}\otimes q^*G)=0$ for every $G\in\Perf(Y')$ and $t=1,2$.
\end{proposition}
\begin{proof}
Use the regular-section realization of \cref{prop:q-incidence}, with
$\pi':\PP_{\mathrm{quot},Y'}(\cV')\to Y'$.  The Koszul resolution of $\cO_\Hecke$ has
terms $\Lambda^\ell\pi'^*\cR'_0\otimes\cO(-\ell)$, $0\le\ell\le n-1$, in degrees
$-\ell$.  After twisting by $\cL^{-t}$ the line degrees are $-t-\ell$, and since the
ambient bundle has fibers $\PP^{n+1}$, every term is $\pi'$-acyclic for $t=1,2$,
because $-(n+1)\le-t-\ell\le-1$ is the full negative acyclic range on $\PP^{n+1}$.  Pushing
the finite resolution gives zero.  The proof takes place on the ambient projective
bundle and does not infer vanishing from the ordinary fibers of the nonflat map $q$.
The projection formula for the perfect complex $G$ \cite[Tag 08EU]{Stacks} gives
$Rq_*(\cL^{-t}\otimes q^*G)\simeq(Rq_*\cL^{-t})\Ltensor G=0$.
\end{proof}

\subsection{Concentration and the sheaf-level Hecke realization}\label{sec:concentration}

Define, for $e=0,1$,
\begin{equation}\label{eq:Ae-definition}
 A_e=Rp_*\bigl(p^*\scrD\otimes\cL^e\otimes q^*C'\bigr).
\end{equation}

\begin{proposition}[Concentration of both rows]\label{prop:Ae-concentrated}
$A_e\in\Perf^{[-n,0]}(Y)$, every $\cH^i(A_e)$ is supported on $|F_n|$, and
$\cH^i(A_e)=0$ for $i\ne0$.  In other words, both $A_0$ and $A_1$ are coherent sheaves.
\end{proposition}
\begin{proof}
The degree-$(-j)$ term of $q^*C'$ is $\Lambda^jq^*\cT'=\Lambda^j\cW$, $0\le j\le n-2$.
By \cref{prop:two-row-amplitude} its contribution to \eqref{eq:Ae-definition}, with
its degree included, has Tor-amplitude $[-j-2,-j]$, and the finite filtration by the
terms of $C'$ and the projection formula for $\scrD$ give
$A_e\in\Perf^{[-(n-2)-2,0]}(Y)=\Perf^{[-n,0]}(Y)$.  This is a calculation with
finite filtered perfect complexes, and no commutation of cohomology and pushforward is
implicit.  The complex $q^*C'$ is contractible outside the zero set of its section,
whose image under the proper map $p$ is contained in $F_n$ by \eqref{eq:qC-support},
hence the cohomology of $A_e$ is supported on $F_n$.  Since $\dim Y=2n$ and
$\dim F_n\le n$, \cref{lem:codim-amplitude} with $m=c=n$ proves concentration in
degree zero.
\end{proof}

\begin{theorem}[Identification of the two rows]\label{thm:Ae-identification}
\begin{equation}\label{eq:Ae-identification}
 A_0\simeq i_*(\scrD|_B\otimes\cO_Z),\qquad
 A_1\simeq i_*(\scrD|_B\otimes\cH^{-1}(C_B))\simeq\cM_n^+ .
\end{equation}
\end{theorem}
\begin{proof}
Tensor \eqref{eq:Koszul-filtration} with $p^*\scrD$ and apply $Rp_*$.  Its middle
term is
\[
 Rp_*\bigl(p^*\scrD\otimes\bar p^*C_B\bigr)
 \simeq i_*R\bar p_*\bigl(\bar p^*(\scrD|_B\otimes C_B)\bigr)
 \simeq i_*\bigl((\scrD|_B\otimes C_B)\Ltensor R\bar p_*\cO_\Hecke\bigr)
 \simeq i_*(\scrD|_B\otimes C_B),
\]
by \cref{prop:p-structure} and the projection formula \cite[Tag 08EU]{Stacks}.  The
projection formula applies because $C_B$ is a bounded complex of vector bundles on
$B$, hence perfect, and no flatness of $\bar p$ is required.  We obtain the
distinguished triangle
\begin{equation}\label{eq:Ae-triangle}
 A_1[1]\longrightarrow i_*(\scrD|_B\otimes C_B)\longrightarrow A_0\longrightarrow A_1[2].
\end{equation}
Both $A_e$ are sheaves by \cref{prop:Ae-concentrated}, and taking degrees $-1$ and $0$
gives $A_1\simeq i_*(\scrD|_B\otimes\cH^{-1}(C_B))$ and
$A_0\simeq i_*(\scrD|_B\otimes\cH^0(C_B))$, with $\cH^0(C_B)=\cO_Z$.  The
identification of the first sheaf with $\cM_n^+$ is \cref{prop:M-trace}.
\end{proof}

Substituting \eqref{eq:det-Hecke} into the definition of $A_1$ gives
\begin{equation}\label{eq:Hecke-realization-final}
 \cM_n^+\simeq Rp_*\bigl(p^*\scrD\otimes\cL\otimes q^*C'\bigr)
 \simeq Rp_*\bigl(\cL^{-1}\otimes q^*(\scrD'\otimes C')\bigr),
\end{equation}
which is the realization assertion of \cref{thm:realization}, and the negative window
\eqref{eq:q-window-intro} is \cref{prop:q-window}.  The amplitude-and-support
argument proves that the entire pushforward is a sheaf before it is identified with
$\cM_n^+$, and we never take the zeroth cohomology of an object already known to lie in a
null category.

\subsection{Proof of the vanishing theorem}\label{sec:final-proof}

\subsubsection{The common wall map}
The universal smaller algebra is a quotient of $S\otimes\cO_{Y'}$, and its odd part is an
$S_0$-linear quotient $S_1\otimes\cO_{Y'}\twoheadrightarrow\cR'_1$ of rank $n$, hence
defines a morphism to the Quot scheme, which factors through its reduction $Q_n$
since $Y'$ is reduced (\cref{prop:models-identified}(a)).  Denote it $h:Y'\to Q_n$.
On $\Hecke$ the odd quotients of the larger and the smaller cluster are the
\emph{same} quotient of $S_1$, not just isomorphic bundles.  Hence the square
\begin{equation}\label{eq:wall-square}
\begin{tikzcd}[column sep=large,row sep=large]
\Hecke \arrow[r,"q"] \arrow[d,"p"'] & Y' \arrow[d,"h"]\\
Y \arrow[r,"\rho"'] & Q_n
\end{tikzcd}
\end{equation}
commutes as a square of moduli morphisms.  It is not asserted to be Cartesian.

\subsubsection{Pushforward of the realized sheaf}
Put $G=\scrD'\otimes C'\in\Perf(Y')$.  By \eqref{eq:Hecke-realization-final},
\eqref{eq:wall-square} and \cref{prop:q-window},
\begin{equation}\label{eq:vanishing-computation}
 R\rho_*\cM_n^+
 \simeq R(\rho p)_*(\cL^{-1}\otimes q^*G)
 \simeq R(hq)_*(\cL^{-1}\otimes q^*G)
 \simeq Rh_*\bigl(G\Ltensor Rq_*\cL^{-1}\bigr)=0 .
\end{equation}
All compositions and projection formulas involve the indicated derived functors
\cite[Tag 08EU]{Stacks}, the line bundle $\cL$ is invertible, $G$ is perfect, and none
of $p,q,\rho$ is assumed flat.  This proves \cref{thm:vanishing} for $n\ge2$.  For
$n=1$, $\cR_0$ is the unit line and $(-y,x):\cR_0\to\cR_1^{\oplus2}$ is nowhere zero
(a cyclic balanced algebra has dimension two).  Its transpose is surjective, hence
the even dual-socle sheaf, and therefore $\cM_1^+$ by \cref{thm:dual-socle}, is zero.
\qed

\begin{remark}[Which map $\rho$]
The computation \eqref{eq:vanishing-computation} uses nothing about the target of
$\rho$ beyond $\rho p=hq$.  It proves $R\rho_*\cM_n^+=0$ for every proper
$\rho:Y\to Q$ that factors through the odd quotient
$S_1\otimes\cO_Y\twoheadrightarrow\cR_1$, in particular for the VGIT contraction
$Y_n=\mfM_\theta((n,n),(1,0))\to\mfM_{\theta_0}((n,n),(1,0))$, whose $Y'$-companion is
the analogous contraction of $\mfM_\theta((n-1,n),(1,0))$.
\end{remark}

\begin{remark}[What the proof does not use]\label{rem:hecke-inputs}
The trace Koszul complex $C_B$ is
carried along rather than replaced by its zeroth cohomology, hence only the
\emph{support} of $q^*C'$ enters, never its scheme structure, and no reducedness or
Cohen--Macaulayness of $F_n$ is used.  The Grassmannian acyclicity of
\cref{thm:BWB} reappears as the $(e,j)=(0,0)$ row of \eqref{eq:two-row-amplitude},
but the vanishing is a consequence of amplitude and codimension, not of a single
Borel--Weil--Bott group.  There are four places where ordinary and derived operations
must be kept separate.  First, $q^*C'$ is the termwise pullback, not $q^*\cO_{Z'}$.
Second, the projection formula for $\bar p$ uses only perfectness of $C_B$ and
$R\bar p_*\cO_\Hecke=\cO_B$.  Third, the two vanishings for $q$ come from a
regular-section resolution in a genuine projective bundle.  Finally, the last
computation uses only the equality of composed morphisms $\rho p=hq$.
\end{remark}

\subsubsection{Cohen--Macaulay consequences}
\begin{corollary}\label{cor:trace-CM}
The trace-defined scheme $Z\subset B$ is Cohen--Macaulay of pure dimension $n$, and
for $n\ge2$ the sheaf $\cM_n^+$ is nonzero, Cohen--Macaulay of pure dimension $n$,
with set-theoretic support $|Z|$.
\end{corollary}
\begin{proof}
By \cref{prop:Ae-concentrated,thm:Ae-identification}, both $i_*\cO_Z$ (up to an
invertible twist) and $\cM_n^+$ are sheaves with local projective dimension at most
$n$ over the regular ambient $Y$ and support of dimension at most $n$.  At a closed
point of the support the Auslander--Buchsbaum formula $\depth M+\pd M=\depth R$
\cite[Thm.~1.3.3]{BrunsHerzog} gives depth at least $2n-n=n$, hence equality holds.
Cohen--Macaulayness localizes and the same argument excludes components of smaller
dimension.  The scheme $Z$ is nonempty, since the complete intersection $S/(xy,x^2+y^{2n-2})$ of
\cref{rem:M-nonzero} is a point of it.  At a closed point of $Z$ the trace ideal in
the Cohen--Macaulay local ring of $B$ has grade $(2n-2)-n=n-2$, and its Koszul
complex has $n-1$ generators, hence grade sensitivity gives nonzero first homology.
By \cref{prop:M-trace}, $\cM_n^+$ is nonzero there, and its support equals $|Z|$.
\end{proof}

\begin{warning}
The corollary concerns the trace scheme $Z$ inside the first even Brill--Noether
locus.  It does not prove that the entire scheme-theoretic punctual fiber $F_n$ is
Cohen--Macaulay or reduced, and the vanishing theorem does not by itself identify
$R\rho_*\cI_n$ with the structure sheaf of the symbolic central fiber.
\end{warning}

\begin{corollary}[Strong excess vanishing and the perfect reduction]\label{cor:twists}
For any perfect complex $V$ on $Q_n$, $R\rho_*(\cM_n^+\Ltensor L\rho^*V)=0$.  In
particular, with $\cL_n=\rho^*\cO_{Q_n}(1)$,
\[
 R\rho_*(\cL_n^{\otimes d}\otimes\cM_n^+)=0\qquad(d\in\Z),
\]
and the positive-parity triangle \eqref{eq:positive-triangle} yields canonical
isomorphisms
\begin{equation}\label{eq:perfect-reduction}
 R\rho_*\bigl(\cL_n^{\otimes d}\otimes\cK_n^+\bigr)\xrightarrow{\ \sim\ }
 R\rho_*\bigl(\cL_n^{\otimes d}\otimes\cI_n\bigr)\qquad(d\in\Z).
\end{equation}
\end{corollary}
\begin{proof}
Use the projection formula and \cref{thm:vanishing}, and then push forward the twisted triangle
$\cM_n^+[1]\to\cK_n^+\to\cI_n\to$.
\end{proof}

Part (3) of Theorem~E is \cref{thm:vanishing}.  The isomorphism
\eqref{eq:perfect-reduction} replaces the incidence sheaf $\cI_n$, which can have
infinite projective dimension over the singular wall variety, by the pushforward of
the bounded complex of vector bundles $\cK_n^+$, whose terms will be shown in
\cref{sec:SD} to be acyclic for the pushforward over the completion at every wall
point.

\subsection{Rank two and rank six}\label{sec:low-rank}

At $n=2$ the smaller moduli space is the single algebra $S/(x,y)^2$, hence $Y'$ is a
point and $\Hecke=B\simeq\PP^2$.  The universal sequence has even part
$0\to\cL\to\cR_0|_B\to\cO_B\to0$ with $\cL=\cO_{\PP^2}(1)$, while $\cR_1|_B$ is
constant.  Hence $\scrD|_B=\cL^{-2}$ and $\cT|_B=\cL$.  Every cluster on $B$ is
punctual and the augmentation vanishes on $\cL$, hence $C_B\simeq\cO_B\oplus\cL[1]$ and
$C'=\cO_{Y'}$.  The realization formula gives
\[
 \cM_2^+\simeq i_*\cO_{\PP^2}(-1),\qquad R\rho_*\cM_2^+=R\Gamma(\PP^2,\cO(-1))=0 ,
\]
in agreement with a direct local computation of the module $\cM_2^+$ on charts.

At $n=6$, the first rank in which even socle rank two occurs
(\cref{cor:threshold}), we have $a=5$, $\rk\cG=6$, $\rk\cT'=4$.  The first ambient
space has fibers $\PP^4$ and the second $\PP^7$, the two regular sections have ranks
$6$ and $5$, the four smaller-trace exterior degrees extend the interval $[-2,0]$ to
$[-6,0]$, and punctual support has codimension at least six.  The $\Gr(3,5)$ wall
fiber and the rank-one strata meeting it are treated simultaneously, because the Hecke
scheme contains
all choices of one even socle line, including those inside a higher-dimensional even
socle, and no separate rank-two induction is needed.  The cancellation
$-(n-2)-2=-n$ is the reason the argument is uniform in $n$.

\section{Symbolic descent}\label{sec:SD}

In this section we prove Theorem~E(4):
\begin{equation}\label{eq:SD}
 R\rho_*\bigl(\cL_n^{\otimes d}\otimes\cI_n\bigr)\;\simeq\;
 \cO_{\cZ_n^{\mathrm{symb}}}(d),\qquad d\in\Z .
\end{equation}
The inputs are the two-term homology of \cref{thm:two-term}, the even-excess
vanishing of \cref{thm:vanishing}, and, at each wall point, the vector-bundle part
of the Bellamy--Craw normal form.  The latter says that on the completed
neighbourhood of a wall fiber the even cluster bundle is a trivial summand plus the
tautological quotient bundle of a cotangent Grassmannian, and the odd cluster bundle
is trivial.

The main idea is to compute the derived pushforward of the positive Haiman complex
$\cK_n^+$ instead of the derived pushforward of $\cO_{F_n}$.  The two agree by the
even-excess vanishing, and $\cK_n^+$ is a complex of vector bundles, to which
Borel--Weil--Bott applies.  Borel--Weil--Bott does two things.  It makes every term of
$\cK_n^+$ acyclic for the completed pushforward, hence the pushforward is computed
termwise.  A second vanishing, with one subbundle factor, makes multiplication of
pushed-forward sections surjective, and this identifies the image of the first
differential.  The pushed cluster algebra is then shown to be generated by the
original polynomials, and the elementary normalized-trace lemma of
\cref{sec:trace-quotient} computes its punctual quotient, nilpotents included.
Finally the even-excess vanishing forces the derived pushforward to be concentrated
in degree zero.  Nothing is assumed about a product decomposition of the cluster
algebra, about flatness of $\rho$, or about reducedness or Cohen--Macaulayness of the
punctual schemes.  The scalar trace functions retain their full dependence on the
cotangent and core coordinates throughout.

The section is organized as follows.  In \cref{sec:perfect-reduction,sec:SD-no-shortcut}
we define the canonical map $u_n$ and explain why the statement is not a formal
consequence of the vanishing theorems.  \Cref{sec:trace-quotient} contains the
normalized-trace lemma.  In \cref{sec:bott,sec:tensor} we collect the vanishing and
surjectivity results on the cotangent Grassmannian, and in
\cref{sec:bundle-model,sec:completion,sec:pushed} we transfer them to the formal
neighbourhood of a wall fiber.  The local computation and the proof of the theorem
are in \cref{sec:local-proof,sec:SD-proof}.  The case $n=2$ is worked out in closed
form in \cref{sec:n2-descent}.

\subsection{The canonical map and the perfect reduction}\label{sec:perfect-reduction}

Write $\mfa:=\mfm_0\cO_{Q_n}$, hence $\cZ_n^{\mathrm{symb}}=V(\mfa)$ and
$F_n=V(\mfa\cO_{Y_n})=Y_n\times_{Q_n}\cZ_n^{\mathrm{symb}}$.  These are ordinary
scheme-theoretic fibers, $\cO_{F_n}=\cO_{Y_n}/\mfa\cO_{Y_n}$, and
$\cI_n\cong\cO_{F_n}$ by \cref{prop:H0}.  Let
$j:\cZ_n^{\mathrm{symb}}\hookrightarrow Q_n$ and
$f=\rho|_{F_n}:F_n\to\cZ_n^{\mathrm{symb}}$, hence the square
\[
\begin{tikzcd}[column sep=large]
F_n\arrow[r,hook]\arrow[d,"f"'] & Y_n\arrow[d,"\rho"]\\
\cZ_n^{\mathrm{symb}}\arrow[r,hook,"j"'] & Q_n
\end{tikzcd}
\]
is Cartesian.  The adjunction unit $\cO_{\cZ}\to Rf_*\cO_{F_n}$ for the proper
morphism $f$, followed by the exact functor $j_*$, is a canonical $T$-equivariant
morphism
\begin{equation}\label{eq:u-n}
 u_n:\cO_{\cZ_n^{\mathrm{symb}}}\longrightarrow R\rho_*\cO_{F_n}=R\rho_*\cI_n .
\end{equation}
It exists as a unit for the actual morphism $f$, and nothing is inferred from
annihilation of cohomology.  Since $\cL_n=\rho^*\cO_{Q_n}(1)$, the projection formula
reduces \eqref{eq:SD} to the untwisted statement
\begin{equation}\label{eq:SD0}
 u_n:\ \cO_{\cZ_n^{\mathrm{symb}}}\xrightarrow{\ \sim\ }R\rho_*\cO_{F_n},
\end{equation}
which is what we shall prove.

\begin{proposition}[Perfect reduction]\label{prop:perfect-reduction}
The augmentation $\cK_n^+\to\cI_n$ of \eqref{eq:positive-triangle} induces a canonical
isomorphism $R\rho_*\cK_n^+\xrightarrow{\sim}R\rho_*\cI_n$.
\end{proposition}
\begin{proof}
Apply $R\rho_*$ to \eqref{eq:positive-triangle}.  Its first and last terms vanish by
\cref{thm:vanishing}.
\end{proof}

The sheaf $\cM_n^+$ does not vanish on $Y_n$ (\cref{cor:trace-CM}), only its derived
image does.  The proposition does not by itself identify $R\rho_*\cK_n^+$ with a
sheaf, let alone with $\cO_{\cZ}$, and that is the task of the rest of the section.  Its
advantage is that $\cK_n^+$ is a bounded complex of vector bundles on the smooth
$Y_n$, whereas $\cO_{F_n}$ can have infinite projective dimension over the singular
wall variety (\cref{prop:lci-fails}).

\subsection{Descent is not Tor-independent base change}\label{sec:SD-no-shortcut}

Let $j:\cZ_n^{\mathrm{symb}}\hookrightarrow Q_n$.  Since $Q_n$ has rational
singularities (\cref{prop:models-identified}(a)), $R\rho_*\cO_{Y_n}=\cO_{Q_n}$, and
the projection formula \cite[Tag 08EU]{Stacks} gives
\begin{equation}\label{eq:projection-unconditional}
 R\rho_*L\rho^*\cO_{\cZ_n^{\mathrm{symb}}}
 \simeq R\rho_*\cO_{Y_n}\Ltensor_{\cO_{Q_n}}\cO_{\cZ_n^{\mathrm{symb}}}
 \simeq\cO_{\cZ_n^{\mathrm{symb}}} .
\end{equation}
Since $F_n=\rho^{-1}(\cZ_n^{\mathrm{symb}})$ scheme-theoretically (both are cut out
by $\mfm_0$), $\cH^0(L\rho^*\cO_{\cZ_n^{\mathrm{symb}}})=\cO_{F_n}=\cI_n$ by
\cref{prop:H0}.  Thus \eqref{eq:SD0} is exactly the assertion that $\cI_n$ may be
substituted for $L\rho^*\cO_{\cZ_n^{\mathrm{symb}}}$, and it would follow at once if
the higher Tor sheaves $\Tor_i^{\cO_{Q_n}}(\cO_{Y_n},\cO_{\cZ_n^{\mathrm{symb}}})$,
$i>0$, vanished, for instance if $\cZ_n^{\mathrm{symb}}$ were a local complete
intersection of codimension $n$ in the Cohen--Macaulay $Q_n$.  This is the case for
$n=1$, where $Q_1\cong T^*\PP^1$ is smooth and $\cZ_1^{\mathrm{symb}}$ is a divisor.
It fails already at $n=2$.

\begin{proposition}\label{prop:lci-fails}
Let $z_J\in\cZ_2^{\mathrm{symb}}$ be the heart lying under the wall fiber
$F_J\cong\PP^2$.  Then
\[
 \mu\bigl(\mfm_0\cO_{Q_2,z_J}\bigr)=3\;>\;2=\htOp\bigl(\mfm_0\cO_{Q_2,z_J}\bigr),
\]
hence $\cZ_2^{\mathrm{symb}}$ is not a local complete intersection in $Q_2$ at $z_J$.
Moreover $\pd_{\cO_{Q_2,z_J}}\cO_{\cZ_2^{\mathrm{symb}},z_J}=\infty$, although
$\cZ_2^{\mathrm{symb}}$ itself is reduced, irreducible and Cohen--Macaulay of
dimension $2$.
\end{proposition}

\begin{proof}
Write $R_0=\C[u,v,w]/(uv-w^2)$, hence $R_1=R_0X+R_0Y$ with relations $vX=wY$ and
$uY=wX$.  The locus of $q\in\Quot^2_X(R_1)$ for which $(q(X),q(Y))$ is a basis of the
length-two quotient is open, contains $z_J$, and is rigidified by that basis.  A
point of it is a triple of pairwise commuting $2\times2$ matrices $U,V,W$, the
actions of $u,v,w$, with $UV=W^2$, $Ve_1=We_2$, $Ue_2=We_1$.  With $W=(w_{ij})$ the
two column conditions force
$U=\left(\begin{smallmatrix}u_1&w_{11}\\u_2&w_{21}\end{smallmatrix}\right)$,
$V=\left(\begin{smallmatrix}w_{12}&v_1\\w_{22}&v_2\end{smallmatrix}\right)$, hence the
chart is the closed subscheme of $\A^8$ cut out by the entries of $[U,V]$, $[U,W]$,
$[V,W]$ and $UV-W^2$, all homogeneous quadrics, and $z_J$ is the origin.  Hence
$I(Q_2)_{z_J}\subseteq\mfm_{z_J}^2$ and $\dim\mfm_{z_J}/\mfm_{z_J}^2=8$.  The support
morphism $Q_2\to\SymOp^2X$ is given by polarized power sums, hence
$\mfm_0\cO_{Q_2}$ contains $\Tr U=u_1+w_{21}$, $\Tr V=w_{12}+v_2$,
$\Tr W=w_{11}+w_{22}$, three linearly independent linear forms whose images in
$\mfm_{z_J}/\mfm_{z_J}^2$ remain independent since $I(Q_2)$ contains no linear form.
Thus $\mu(\mfm_0\cO_{Q_2,z_J})\ge3$, while its height is $\dim Q_2-\dim\cZ_2^{\mathrm{symb}}
=4-2=2$.  The remaining assertions are Gr\"obner computations in this chart.  The
chart is reduced, irreducible, Gorenstein of codimension four (Betti numbers
$1,9,16,9,1$ over $\C[\A^8]$), in agreement with the local model of
\cref{sec:local-model}, whose transverse singularity at $r=1$ is
$\mfN_{1,3}=\overline{\mathcal O}_{\min}\subset\mathfrak{sl}_3$.  The three linear
traces generate $\mfm_0\cO_{Q_2,z_J}$, hence $\mu=3$.  The ideal $I(\cZ_2^{\mathrm{symb}})$ has
projective dimension $6=\codim$ over $\C[\A^8]$, hence $\cZ_2^{\mathrm{symb}}$ is
Cohen--Macaulay, reduced and irreducible.  Finally, the minimal free resolution of
$\cO_{\cZ_2^{\mathrm{symb}},z_J}$ over $\cO_{Q_2,z_J}$ begins $1,3,6,18,66,246,918$.
Since $\cO_{Q_2,z_J}$ is Cohen--Macaulay of depth $4$, Auslander--Buchsbaum forces
$\pd\le4$ whenever it is finite, and a nonzero fifth syzygy proves $\pd=\infty$.
\end{proof}

The failure is one of depth over a singular base, not of $\cZ_n^{\mathrm{symb}}$
itself.  Note that $\mu=3=n+1$ exceeds the height by one, the same deviation as
$\rk\cE_n=n+1$ against $\codim F_n=n$ upstairs.  The higher-Tor defect must therefore
be controlled by additional structure, and the structure used below is the positivity of
the tautological quotient bundles on the wall fibers, together with the even-excess
vanishing.

\subsection{The normalized-trace quotient}\label{sec:trace-quotient}

In this subsection we prove an elementary lemma of commutative algebra.  It computes
the quotient of a parity-graded algebra by the ideal generated by the trace-zero
elements and by the odd elements times $x$ and $y$, and it identifies this quotient
with the quotient of the base ring by an ideal of traces.  In \cref{sec:local-proof}
the lemma is applied to the pushed-forward cluster algebra at a wall point, which is
not locally free, and this application dictates the generality of the statement.
Let us first fix the notation on $Y_n$.  For $n>0$ put
\begin{equation}\label{eq:normalized-trace}
 \epsilon:=\tfrac1n\Tr_{\cR_0}:\cR_0\longrightarrow\cO_{Y_n},\qquad
 \epsilon(1)=1,\qquad \cR_0=\cO_{Y_n}\cdot1\oplus\cT_n,\quad\cT_n=\ker\epsilon.
\end{equation}
For an even polynomial $f\in R_0=\C[x,y]^\Gamma$ we write $\bar f\in\cR_0$ for its
image in the universal cluster algebra, as in the proof of \cref{prop:H0}, and we put
$\tau_f:=\Tr_{\cR_0}(m_{\bar f})=n\,\epsilon(\bar f)\in\cO_{Y_n}$.
By \cref{lem:powersums} and Haiman's trace identity (proof of \cref{prop:H0}),
\begin{equation}\label{eq:trace-ideal}
 (\tau_f:\ f\in R_0,\ f(0)=0)=\mfm_0\cO_{Y_n}.
\end{equation}
Here $\Tr_{\cR_0}$ is the trace of an endomorphism of the rank-$n$ vector bundle
$\cR_0$.  We shall push $\epsilon$ forward, but we never take the trace of a module
which is not locally free.

In the lemma below and in \cref{sec:pushed,sec:local-proof} the letter $R$ denotes an
arbitrary base ring, which in the application is the completed local ring of $Q_n$
at a wall point.  It should not be confused with the polynomial ring $R=\C[x,y]$ of
\cref{sec:mckay}.  We write $x,y$ also for the images of the variables in $B_1$.

\begin{lemma}[Normalized-trace quotient]\label{lem:trace-quotient}
Let $R$ be a commutative $\C$-algebra and let $B=B_0\oplus B_1$ be a commutative
$\Z/2$-graded $R$-algebra.  Assume that the parity-preserving homomorphism of
$R$-algebras $R[x,y]\to B$, with $x,y$ odd, is surjective, and for a polynomial
$f\in R[x,y]$ write $\bar f\in B$ for its image.  Let $\epsilon:B_0\to R$ be an
$R$-linear map with $\epsilon(1)=1$ and put $T=\ker\epsilon$, hence
$B_0=R\cdot1\oplus T$.  Fix a nonzero integer $n$, and for an even polynomial
$f\in R[x,y]$ put $\tau_f:=n\,\epsilon(\bar f)\in R$.  Let
\[
 I=\bigl(\tau_f:\ f\in R[x,y]\ \text{even},\ f(0,0)=0\bigr)\subset R
\]
be the ideal generated by the elements $\tau_f$ for the even polynomials $f$ without
constant term.  Since $\epsilon$ is $R$-linear, $I$ is generated by the elements
$\tau_{x^ay^b}=n\,\epsilon(\overline{x^ay^b})$ with $a+b$ even and positive, and the
factor $n$ is immaterial for $I$.  Then the unit $R\to B_0$ induces an isomorphism of
$R$-algebras
\begin{equation}\label{eq:abstract-trace-quotient}
 R/I\;\xrightarrow{\ \sim\ }\;B_0/(B_0T+xB_1+yB_1),
\end{equation}
whose inverse is induced by $\epsilon$.  No multiplicativity of $\epsilon$, no local
freeness of $B$ and no reducedness of either ring is assumed.
\end{lemma}
\begin{proof}
\emph{Step 1, generators of $T$.}  Since $R[x,y]\to B$ is surjective and preserves
parity, $B_0$ is generated as an $R$-module by the images $\bar f$ of the even
monomials $f$, including $f=1$, and $B_1$ by the images of the odd monomials.  For an
even monomial $f$ put $g_f:=\bar f-\epsilon(\bar f)1=\bar f-\frac{\tau_f}n1\in T$.
Then $g_1=0$, and applying $b\mapsto b-\epsilon(b)1$ to an expression of $b\in T$ in
the generators $\bar f$ shows that $T$ is generated over $R$ by the $g_f$ with $f$ an
even monomial of positive degree.

\emph{Step 2, the quotient is generated by $[1]$.}  Each of $B_0T$, $xB_1$ and $yB_1$
is stable under multiplication by $B_0$, hence $B_0T+xB_1+yB_1$ is an ideal of $B_0$.
Since $T$ is killed and $B_0=R\cdot1\oplus T$, every element of the quotient is an
$R$-multiple of $[1]$, that is, the unit map $R\to B_0/(B_0T+xB_1+yB_1)$ is
surjective.

\emph{Step 3, $I$ maps to zero.}  A positive-degree even monomial $f$ is $x$ or $y$
times an odd monomial, hence $\bar f\in xB_1+yB_1$.  Thus both $\bar f$ and $g_f$
vanish in the quotient, and therefore $\frac{\tau_f}{n}[1]=[\bar f]-[g_f]=0$.  The
unit factors through $R/I$.

\emph{Step 4, $\epsilon$ kills the relations modulo $I$.}  For an odd monomial $a$
the monomials $xa$ and $ya$ are even of positive degree, hence
$\epsilon(x\bar a)=\epsilon(\overline{xa})=\tau_{xa}/n\in I$ and likewise
$\epsilon(y\bar a)\in I$.  For an even monomial $f$ and a positive-degree even
monomial $g$, $R$-linearity gives
\[
 \epsilon(\bar f g_g)=\epsilon(\overline{fg})-\tfrac{\tau_g}n\epsilon(\bar f)
 =\tfrac{\tau_{fg}}n-\tfrac{\tau_f\tau_g}{n^2}\in I .
\]
Indeed, for $f=1$ the expression vanishes because $\tau_1=n\,\epsilon(1)=n$, and for
$f\ne1$ the monomial $fg$ has positive degree.  Since $B_0T$ is spanned over $R$ by
the products $\bar fg_g$, and $xB_1$, $yB_1$ are spanned by the elements
$x\bar a$, $y\bar a$, we conclude that $\epsilon(B_0T+xB_1+yB_1)\subset I$.  Hence
$\epsilon$ induces an $R$-linear map $\bar\epsilon$ from the quotient to $R/I$.

\emph{Step 5, conclusion.}  The composite
$R/I\to B_0/(B_0T+xB_1+yB_1)\xrightarrow{\bar\epsilon}R/I$ is
$r\mapsto\epsilon(r1)=r$.  A surjection with a left inverse is an isomorphism, and
the inverse of an isomorphism of algebras is an algebra map.
\end{proof}

\begin{remark}[How the lemma is used]\label{rem:trace-quotient-upstairs}
The lemma is used twice.  First, on affine opens of $Y_n$ we take $B=\cP_n$ and
$\epsilon$ the normalized trace of \eqref{eq:normalized-trace}, hence $T=\cT_n$ and
$I=\mfm_0\cO_{Y_n}$ by \eqref{eq:trace-ideal}.  The lemma gives
$\cI_n=\cP_n/(\cT_n,x,y)\cP_n\cong\cO_{Y_n}/\mfm_0\cO_{Y_n}$, which is a second proof
of \cref{prop:H0}.  This proof exhibits the inverse map and therefore controls the
scheme structure without a Nakayama argument on a kernel.

The essential application is in \cref{sec:local-proof}.  There $R=\widehat{\cO}_{Q_n,z_J}$
is the completed local ring of the wall variety at a fixed point $z_J$, the algebra
$B$ is the completion of the pushed-forward cluster algebra $\rho_*\cP_n$, and
$\epsilon$ is the pushforward of the normalized trace.  At a wall point with a
positive-dimensional fiber $B_0$ is not locally free over $R$, its fiber at $z_J$ has
dimension larger than $n$, and $\epsilon$ is not the trace of the $R$-module $B_0$.
This is why the lemma assumes neither multiplicativity of $\epsilon$ nor local
freeness of $B$.  The hypotheses of the lemma are verified in \cref{sec:pushed}.
The surjectivity of $R[x,y]\to B$ is \cref{prop:polynomial-generation}, the splitting
$B_0=R\cdot1\oplus T$ is \eqref{eq:B-splitting}, and the equality $I=\mfm_0R$ is
\eqref{eq:B-trace}, which comes from Haiman's trace identity \eqref{eq:trace-ideal}.
The conclusion of the lemma identifies the cokernel of the first differential of the
pushed-forward Haiman complex, which is $B_0/(B_0T+xB_1+yB_1)$ by
\eqref{eq:actual-cokernel}, with $R/\mfm_0R=\widehat{\cO}_{\cZ_n^{\mathrm{symb}},z_J}$,
and it shows that the identification is induced by the unit, that is, by the descent
map $u_n$.  The lemma asserts an equality of ordinary quotients, not of radicals, and
this is what makes the nilpotent structure of $\cZ_n^{\mathrm{symb}}$ visible in the
computation.
\end{remark}

\subsection{Two Grassmannian vanishing theorems}\label{sec:bott}

In this subsection $G=\Gr(k,V)$ with $\dim V=N=k+r$, $k,r\ge1$, parametrizes
$k$-dimensional subspaces, with tautological sequence
\begin{equation}\label{eq:tautological-sequence}
 0\longrightarrow\cU\longrightarrow V\otimes\cO_G\longrightarrow\cQ\longrightarrow0,
 \qquad\rk\cU=k,\quad\rk\cQ=r
\end{equation}
($\cU$ is the bundle written $\cS$ in \cref{sec:cluster-socle}).  Thus
$T_G=\cU^\vee\otimes\cQ$, and for $X=T^*G$ with projection $p:X\to G$,
\begin{equation}\label{eq:cotangent-algebra}
  p_*\cO_X=\SymOp(T_G)=\bigoplus_{d\ge0}\SymOp^d(\cU^\vee\otimes\cQ).
\end{equation}
We give fiber-polynomial functions degree $1$ (the \emph{auxiliary grading}, which
is unrelated to the $q,t$-bigrading and is used only on the model).  Let
$\mu:X\to\mfN$ be the Springer contraction onto the reduced square-zero orbit closure
and $A=\Gamma(X,\cO_X)=\bigoplus_{d\ge0}A_d$, $A_0=\C$, with $A_+=\bigoplus_{d>0}A_d$.

\emph{Borel--Weil--Bott convention.}  For a dominant weight $\beta\in\Z^k$ we
interpret $\schur^\beta E=(\det E)^{\beta_k}\otimes\schur^{\beta-\beta_k(1,\dots,1)}E$,
which allows a last entry $\beta_k=-1$.  For dominant $\alpha\in\Z^r$, $\beta\in\Z^k$
the homogeneous bundle $\schur^\alpha\cQ\otimes\schur^\beta\cU^\vee$ has, in the
quotient-first convention, Bott weight
\begin{equation}\label{eq:general-bott-weight}
 \gamma=(\alpha_1,\dots,\alpha_r,\,-\beta_k,\dots,-\beta_1)\in\Z^N .
\end{equation}
Put $\varrho=(N-1,\dots,1,0)$.  If $\gamma+\varrho$ has a repeated entry the bundle is
acyclic, and otherwise its cohomology is concentrated in the degree equal to the length
of the permutation sorting $\gamma+\varrho$, in particular in degree $0$ when $\gamma$
is dominant \cite{Bott}.  The signs are calibrated by $\cQ$ (weight
$(1,0,\dots,0)$, $H^0=V$) and $\cU$ (block boundary $0,1$, acyclic).

\begin{theorem}[Positive and subbundle-tensor vanishing]\label{thm:bott}
For every partition $\lambda$ with at most $r$ parts,
\begin{align}
 H^i\bigl(X,p^*\schur^\lambda\cQ\bigr)&=0 &&(i>0),\label{eq:positive-vanishing}\\
 H^i\bigl(X,p^*(\cU\otimes\schur^\lambda\cQ)\bigr)&=0 &&(i>0),\label{eq:subbundle-vanishing}
\end{align}
in each fiber-polynomial degree separately.
\end{theorem}
\begin{proof}
\emph{Step 1.}  Since $p$ is affine, for a bundle $\cF$ on $G$
\begin{equation}\label{eq:cohom-degree}
 H^i(X,p^*\cF)=\bigoplus_{d\ge0}H^i\bigl(G,\cF\otimes\SymOp^d(\cU^\vee\otimes\cQ)\bigr),
\end{equation}
cohomology commuting with the direct sum (compute with a finite affine \v Cech cover
of $G$).

\emph{Step 2: the positive case.}  By the Cauchy formula
\begin{equation}\label{eq:cauchy}
 \SymOp^d(\cU^\vee\otimes\cQ)=\bigoplus_{\nu\vdash d,\ \ell(\nu)\le\min(k,r)}
 \schur^\nu\cU^\vee\otimes\schur^\nu\cQ .
\end{equation}
Decomposing $\schur^\lambda\cQ\otimes\schur^\nu\cQ$ by the Littlewood--Richardson rule
into bundles $\schur^\alpha\cQ$, $\alpha$ a partition, each summand
$\schur^\alpha\cQ\otimes\schur^\nu\cU^\vee$ has Bott weight
$(\alpha_1,\dots,\alpha_r,-\nu_k,\dots,-\nu_1)$, which is dominant since
$\alpha_r\ge0\ge-\nu_k$.  This proves \eqref{eq:positive-vanishing}.

\emph{Step 3: the subbundle factor.}  Write $E=\cU^\vee$.  Since
$E^\vee\cong\Lambda^{k-1}E\otimes(\det E)^{-1}$, the Pieri rule for exterior powers
shows that the weights occurring in $E^\vee\otimes\schur^\nu E$ are exactly the
dominant weights
\begin{equation}\label{eq:dual-pieri}
  \beta=\nu-\epsilon_j\qquad(1\le j\le k),
\end{equation}
because tensoring with $\Lambda^{k-1}E$ adds a vertical strip in $k-1$ distinct rows and
$(\det E)^{-1}$ subtracts one from every row.  Every entry of $\beta$ is nonnegative
except possibly $\beta_k=-1$.

\emph{Step 4: the three cases.}  A summand of
$\cU\otimes\schur^\lambda\cQ\otimes\SymOp^d(\cU^\vee\otimes\cQ)$ has weight
\[
 \gamma=(\alpha_1,\dots,\alpha_r,-\beta_k,\dots,-\beta_1),
\]
with $\alpha$ a partition and $\beta$ as in \eqref{eq:dual-pieri}.  If $\beta_k\ge0$, $\gamma$ is dominant.  If
$\beta_k=-1$ and $\alpha_r\ge1$, then $\alpha_r\ge-\beta_k$ and $\gamma$ is again
dominant.  If $\beta_k=-1$ and $\alpha_r=0$, the entries of $\gamma+\varrho$ in
positions $r$ and $r+1$ are $0+(N-r)=k$ and $1+(N-r-1)=k$, hence the shifted weight is
singular and all cohomology of the summand vanishes.  No regular non-dominant weight
occurs, which proves \eqref{eq:subbundle-vanishing}.
\end{proof}

\begin{corollary}\label{cor:positive-tensors}
Both vanishings hold with $\schur^\lambda\cQ$ replaced by a finite direct sum of tensor
powers of $\cQ$ or by a direct summand of such a bundle.  Since the target of $\mu$
is affine, $R^i\mu_*$ of the corresponding pulled-back bundles vanishes for $i>0$.
\end{corollary}
\begin{proof}
In characteristic zero tensor powers of $\cQ$ are finite sums of polynomial Schur
bundles, and cohomology is additive and respects split idempotents.
\end{proof}

The following variant, with one dual quotient factor, is used only in the proof of
\cref{prop:bundle-model}.

\begin{lemma}[Vanishing with one dual quotient factor]\label{lem:dual-quotient}
For every partition $\lambda$ with at most $r$ parts,
$H^i(X,p^*(\cQ^\vee\otimes\schur^\lambda\cQ))=0$ for $i>0$.  Consequently, for
$E_0=\cO_G^{\oplus m}\oplus\cQ$ with $m\ge0$,
\begin{equation}\label{eq:end-vanishing}
 H^i\bigl(G,\cEnd(E_0)\otimes\SymOp^dT_G\bigr)=0\qquad(i>0,\ d\ge0).
\end{equation}
\end{lemma}
\begin{proof}
By \eqref{eq:cohom-degree} and \eqref{eq:cauchy} it suffices to treat
$\cQ^\vee\otimes\schur^\lambda\cQ\otimes\schur^\nu\cQ\otimes\schur^\nu\cU^\vee$.
Decompose $\schur^\lambda\cQ\otimes\schur^\nu\cQ$ into $\schur^\mu\cQ$, $\mu$ a
partition, and then $\cQ^\vee\otimes\schur^\mu\cQ$ by \eqref{eq:dual-pieri} applied
to the quotient block.  The summands are $\schur^\alpha\cQ$ with $\alpha=\mu-\epsilon_j$
dominant, hence every entry of $\alpha$ is nonnegative except possibly $\alpha_r=-1$.
The Bott weight is $\gamma=(\alpha_1,\dots,\alpha_r,-\nu_k,\dots,-\nu_1)$.  If
$\alpha_r\ge0$, or if $\alpha_r=-1$ and $\nu_k\ge1$, it is dominant.  If $\alpha_r=-1$
and $\nu_k=0$, the entries of $\gamma+\varrho$ in positions $r,r+1$ are both
$k-1$, and the summand is acyclic.  For \eqref{eq:end-vanishing} write
$\cEnd(E_0)=\cO^{\oplus m^2}\oplus\cQ^{\oplus m}\oplus(\cQ^\vee)^{\oplus m}\oplus
(\cQ\otimes\cQ^\vee)$ and use \eqref{eq:positive-vanishing} with $\lambda=\varnothing,(1)$
for the first two summands and the vanishing just proved with
$\lambda=\varnothing,(1)$ for the last two.  Finally, $H^i(G,\cF\otimes\SymOp^dT_G)$ is the
degree-$d$ part of $H^i(X,p^*\cF)$ by \eqref{eq:cohom-degree}.
\end{proof}

\subsection{Tensor surjectivity and the degree-zero restriction}\label{sec:tensor}

For a bundle $\cF$ on $G$ put $M_\cF=\Gamma(X,p^*\cF)$, a finite $A$-module since
$\mu$ is proper.  There is a natural tensor-of-sections map
\begin{equation}\label{eq:tensor-map}
 m_{\cF,\cG}:M_\cF\otimes_AM_\cG\longrightarrow M_{\cF\otimes\cG},
\end{equation}
which in general need not be surjective.  It is not the multiplication of a cluster
algebra, the latter will be composed with it.

\begin{proposition}[Surjectivity with one quotient factor]\label{prop:one-factor}
For $\cF=\schur^\lambda\cQ$ the natural maps
\begin{equation}\label{eq:one-factor-surj}
 V\otimes_\C M_\cF\twoheadrightarrow M_{\cQ\otimes\cF},\qquad
 M_\cQ\otimes_AM_\cF\twoheadrightarrow M_{\cQ\otimes\cF}
\end{equation}
are surjective.
\end{proposition}
\begin{proof}
Tensor \eqref{eq:tautological-sequence} with $\cF$, pull back to $X$ (exact, as $\cF$ is
locally free and $p$ flat) and take sections.  This gives the exact sequence
$V\otimes M_\cF\to M_{\cQ\otimes\cF}\to H^1(X,p^*(\cU\otimes\cF))=0$, where the last
vanishing is \eqref{eq:subbundle-vanishing}.  The first map factors through the second via the
degree-zero sections $V\to M_\cQ$.
\end{proof}

\begin{proposition}[Tensor powers, sums and retracts]\label{prop:tensor-retracts}
The map \eqref{eq:tensor-map} is surjective for every pair of bundles on $X$ which
are direct summands of finite direct sums of the bundles $(p^*\cQ)^{\otimes a}$,
$a\ge0$, and the splitting idempotents need not be homogeneous for the auxiliary grading.
\end{proposition}
\begin{proof}
\Cref{prop:one-factor} holds for $\cF=\cQ^{\otimes b}$ by its Schur decomposition.
Applying it $a$ times gives a surjection $V^{\otimes a}\otimes M_{\cQ^{\otimes b}}
\twoheadrightarrow M_{\cQ^{\otimes(a+b)}}$, which factors through
$M_{\cQ^{\otimes a}}\otimes_AM_{\cQ^{\otimes b}}$, hence the tensor map is surjective,
and direct sums preserve this.  If $\cF,\cG$ are retracts of $\cF',\cG'$ with inclusions
$i_F,i_G$ and retractions $p_F,p_G$, then for a section $s$ of $\cF\otimes\cG$ include
it into $\cF'\otimes\cG'$, lift to $\sum_au_a\otimes v_a$, and apply $p_F,p_G$.
Naturality of \eqref{eq:tensor-map} gives
$m(\sum_a p_Fu_a\otimes p_Gv_a)=(p_F\otimes p_G)(i_F\otimes i_G)s=s$.  The argument is
ungraded.
\end{proof}

The first map in \eqref{eq:one-factor-surj} with $\cF=\cO_G$ is a surjection
\begin{equation}\label{eq:MQ-generated}
 A\otimes_\C V\twoheadrightarrow M_\cQ
\end{equation}
by generators of degree zero.

\begin{lemma}[Restriction of the quotient bundle]\label{lem:MQ-restriction}
Restriction to the zero section induces an isomorphism
\begin{equation}\label{eq:MQ-restriction}
 M_\cQ/A_+M_\cQ\xrightarrow{\ \sim\ }H^0(G,\cQ)=V .
\end{equation}
\end{lemma}
\begin{proof}
$M_\cQ=\bigoplus_d(M_\cQ)_d$ with $(M_\cQ)_0=H^0(G,\cQ)$ by \eqref{eq:cohom-degree}.
The surjection \eqref{eq:MQ-generated} is graded with $V$ in degree $0$, hence
$(M_\cQ)_{>0}=A_+M_\cQ$ and $M_\cQ/A_+M_\cQ=(M_\cQ)_0$, and restriction to the zero
section kills positive fiber degree and is the identity in degree zero.  Finally
$V\to H^0(G,\cQ)$ is an isomorphism (Borel--Weil for $\cQ$, or
\eqref{eq:tautological-sequence} with $\cU$ acyclic).
\end{proof}

This is an ordinary residue-field quotient of a specified finite module, not a base
change assertion for the nonflat $\mu$, and nothing is claimed about
$\Tor_i^A(M_\cQ,\C)$ or about derived pullback of $\cO_{\mfN}$.

\subsection{The formal bundle model at a wall point}\label{sec:bundle-model}

Fix a $T$-fixed point $z_J\in Q_n^T$ with $k=k_J\ge2$ and put $r=k-1$, $V=V_J$,
$N=\dim V=2r+1$ (\cref{thm:fiber-grass}), $F_J=\rho^{-1}(z_J)$.  From now on
$G=\Gr(k,V)$, $X=T^*G$ and $\cU,\cQ$ refer to this Grassmannian, and
$\cQ_J=\cQ$.  Recall $J'=R_1J$, $J''=(J:R_1)$, $V_J=J''/J'$.

\begin{lemma}[The scheme-theoretic fiber]\label{lem:fiber-scheme}
$F_J\cong\Gr(k,V_J)$ as schemes, and on it there are exact sequences
\begin{align}
 0&\longrightarrow\cU_J\longrightarrow V_J\otimes\cO_{F_J}\longrightarrow\cQ_J\longrightarrow0,
 \label{eq:fiber-taut}\\
 0&\longrightarrow\cU_J\longrightarrow(R_0/J')\otimes\cO_{F_J}\longrightarrow\cR_0|_{F_J}\longrightarrow0,
 \label{eq:fiber-even-seq}
\end{align}
and $\cR_1|_{F_J}=(R_1/J)\otimes\cO_{F_J}$.
\end{lemma}
\begin{proof}
$Q_n$ is the reduced Quot scheme of odd quotients (\cref{prop:models-identified}).
The odd quotient of the universal cluster is flat, hence $Y_n$ maps to the Quot scheme,
the map factors through $Q_n$, and the fiber over the closed point $z_J$ is the same
whether computed in the Quot scheme or in $Q_n$.  A $T_0$-point of the fiber is thus
a $T_0$-flat family of balanced clusters $I\subset R\otimes\cO_{T_0}$ whose odd part
is $J\otimes\cO_{T_0}$.  Its even part $I^+$ contains $R_1J\otimes\cO_{T_0}=J'\otimes\cO_{T_0}$,
and $I^+/J'\otimes\cO_{T_0}\subset(R_0/J')\otimes\cO_{T_0}$ is the kernel of a
surjection onto the locally free $\cR_0$ of rank $n$, hence a subbundle of rank
$k=\dim(R_0/J')-n$.  The condition $R_1I^+\subset J\otimes\cO_{T_0}$ says
$I^+\subset(J\otimes\cO_{T_0}:R_1)=J''\otimes\cO_{T_0}$, the colon commuting with
the flat extension of scalars from $\C$, hence the subbundle lies in $V_J\otimes\cO_{T_0}$.
Conversely every rank-$k$ subbundle of $V_J\otimes\cO_{T_0}$ gives an ideal, by
$\mfm_XV_J=0$, with flat balanced quotient.  These constructions are inverse and
functorial, including on nonreduced $T_0$, hence the fiber functor is the Grassmannian
functor.  The universal subbundle gives \eqref{eq:fiber-taut} and
\eqref{eq:fiber-even-seq}, and the odd quotient is constant by definition.
\end{proof}

By \eqref{eq:fiber-even-seq} and a vector-space splitting $R_0/J'=(R_0/J'')\oplus V_J$,
\begin{equation}\label{eq:fiber-bundles}
 \cR_0|_{F_J}\cong\cO^{\oplus(n-r)}\oplus\cQ_J,\qquad\cR_1|_{F_J}\cong\cO^{\oplus n},
\end{equation}
in agreement with \cref{prop:cluster-socle}(a) ($\dim R_0/J''=n-r$ by
\cref{thm:fiber-grass}(c)).

\emph{The formal normal form.}  By \cref{sec:local-model}, the local model of $\rho$
at $z_J$ is the Springer map $\mu:T^*\Gr(r+1,V)\to\mfN_{r,2r+1}$ times a smooth core
of dimension $2\ell$, $\ell=n-r(r+1)$.  We use the relative form of the Bellamy--Craw
normal form \cite[Thm.~3.2, eq.~(3.5), Thm.~3.7]{BellamyCraw} (\cref{thm:BC}(a),(b)).
Completing along the scheme-theoretic fibers, there is a commutative diagram of
formal schemes
\begin{equation}\label{eq:formal-diagram}
\begin{tikzcd}[column sep=large]
 (\widehat{Y_n})_{F_J}\arrow[r,"\sim"]\arrow[d,"\widehat\rho"'] &
 \widehat{(T^*G\times\A^{2\ell})}_{G\times\{0\}}\arrow[d,"\widehat{\mu\times\id}"]\\
 (\widehat{Q_n})_{z_J}\arrow[r,"\sim"] & \widehat{(\mfN_{r,2r+1}\times\A^{2\ell})}_{(0,0)} .
\end{tikzcd}
\end{equation}
It is obtained from the \'etale-local statement by completing at $z_J$, the \'etale
maps of \cref{thm:BC}(b) inducing isomorphisms of completions and the Cartesian
squares giving the top row.  Only the formal smoothness of the core is used.

\begin{remark}[The citation, checked against its hypotheses]\label{rem:BC-check}
Theorem~3.2 of \cite{BellamyCraw} is stated for the variation-of-GIT morphism
$f:\mfM_\theta(\mathbf v,\mathbf w)\to\mfM_{\theta_0}(\mathbf v,\mathbf w)$ with
$\theta_0$ in the boundary of the closure of the chamber of $\theta$, at a closed
point $x$ with $\theta_0$-polystable representative
$M_\infty\oplus M_0^{\oplus m_0}\oplus\dots\oplus M_k^{\oplus m_k}$.  Its
conclusion is the diagram of \cref{thm:BC}(a),(b), with Cartesian squares and
\'etale horizontal maps, and (3.5) of \emph{loc.\ cit.} is its completion.  In
our case $\theta$ is the cyclic chamber, $\theta_0=(0,1)$ lies on its wall, and at
$z_J$ the polystable representative is $E_{\mathrm{fr}}\oplus S_0^{\oplus r}$ by
\eqref{eq:polystable}, hence $k=0$, $M_\infty=E_{\mathrm{fr}}$ with
$\beta^{(\infty)}=(1,n-r,n)$ on the vertices $(\infty,0,1)$, $M_0=S_0$ with
$\beta^{(0)}=(0,1,0)$ and $m_0=r$.  For the framed doubled affine $A_1$ quiver
with one framing edge at the vertex $0$ the Cartan form is
$(\alpha,\alpha)=2(a^2+b^2+c^2)-2ab-4bc$ for $\alpha=(a,b,c)$, hence
\[
 \ell=p(\beta^{(\infty)})=1-\tfrac12(\beta^{(\infty)},\beta^{(\infty)})=n-r(r+1),\qquad
 \mathbf n_0=-(\beta^{(\infty)},\beta^{(0)})=2r+1,\qquad p(\beta^{(0)})=0.
\]
Thus the Ext-graph is one vertex without loops, framed by $2r+1$, the core has
dimension $2\ell$, and the local quiver varieties are
$\mfM_\varrho(r,2r+1)\to\mfM_0(r,2r+1)$, in agreement with
\cref{thm:fiber-grass}(c), \cref{cor:threshold} and \eqref{eq:local-model}.

Two points of the citation deserve comment.  First, \cite{BellamyCraw} takes
$\mfM_{\theta_0}(\mathbf v,\mathbf w)$ to be the GIT quotient of the
scheme-theoretic moment-map fiber, whereas $Q_n$ was given its reduced structure in
\cref{sec:quiver}, and the two agree.  Indeed, for the local model the dimension vector
$(1,r)$ of the deframed quiver (two vertices joined by $2r+1$ edges) lies in
Crawley-Boevey's set $\Sigma_0$ because $p(1,r)=r(r+1)>s(2r+1-s)=p(1,s)$ for
$0\le s<r$, hence the moment-map fiber and its quotient $\mfM_0(r,2r+1)$ carry their
reduced structure and the quotient is normal (\cite[Thm.~1.1]{CrawleyBoeveyNormality},
the reducedness of the natural scheme structures for dimension vectors in
$\Sigma_0$ is proved in \cite{CrawleyBoeveyMoment} and recalled in \S1 of
\cite{CrawleyBoeveyNormality}).  The
\'etale map $q$ of \cref{thm:BC}(b) then makes its source reduced, and the \'etale
map $p$ makes $\mfM_{\theta_0}(\mathbf v,\mathbf w)$ reduced, and in fact normal, at
every closed point (reducedness and normality ascend along flat local homomorphisms
with reduced, respectively normal, fibers and descend along all of them
\cite[Thm.~23.9 and its corollary]{Matsumura}).  This also reproves the normality of
$Q_n$ used in \cref{prop:models-identified}.  Second, the identification of the
fibers in \cite[Cor.~3.4]{BellamyCraw} is scheme-theoretic, since the fiber of an
\'etale map over a $\C$-point is a reduced point.  On the model the fiber over
$0$ is the zero section, because $BA=0$ with $A$ surjective forces $B=0$
(\cref{sec:local-model}), hence \eqref{eq:formal-diagram} carries $F_J$ to
$G\times\{0\}$, in agreement with \cref{lem:fiber-scheme}.  The tautological-bundle
statement \cite[Prop.~3.8]{BellamyCraw}, whose displayed formula omits the
trivial summands coming from $M_\infty$, is not used, since the underlying bundles are
determined below by their restriction to $F_J$.
\end{remark}

\begin{proposition}[Underlying tautological bundles]\label{prop:bundle-model}
On \eqref{eq:formal-diagram}, as vector bundles,
\begin{equation}\label{eq:bundle-model}
 \widehat{\cR_0}\cong\cO^{\oplus(n-r)}\oplus\widehat{p^*\cQ_J},\qquad
 \widehat{\cR_1}\cong\cO^{\oplus n},
\end{equation}
where $p:T^*G\times\A^{2\ell}\to G$ is the projection and the completion is along
$G\times\{0\}$.  The quotient bundle here is the positive $\cQ_J$, not $\cQ_J^\vee$,
and the formula includes the trivial summands contributed by the framed stable
factor.  No constancy of the multiplication of $\cP_n$ is asserted.
\end{proposition}
\begin{proof}
Write $\widehat Y=(\widehat{Y_n})_{F_J}$, let $\cI$ be the ideal of $F_J$ and
$Y_m=V(\cI^{m+1})$ its thickenings.  The isomorphism of formal schemes in
\eqref{eq:formal-diagram} identifies $\cI/\cI^2$ with the conormal bundle of
$G\times\{0\}$ in $T^*G\times\A^{2\ell}$, and $F_J\subset Y_n$ is a regular embedding,
hence
\begin{equation}\label{eq:conormal-powers}
 \cI^m/\cI^{m+1}\cong\SymOp^m(\cI/\cI^2)\cong\bigoplus_{d+e=m}\SymOp^dT_G\otimes\SymOp^e\C^{2\ell}
 \qquad(m\ge1).
\end{equation}

\emph{Lifting principle.}  Let $\cE,\cE'$ be vector bundles on $\widehat Y$ and
$\phi_0:\cE|_{F_J}\xrightarrow{\sim}\cE'|_{F_J}$, and suppose
\begin{equation}\label{eq:lifting-vanishing}
 H^1\bigl(F_J,\cHom(\cE|_{F_J},\cE'|_{F_J})\otimes\cI^m/\cI^{m+1}\bigr)=0\qquad(m\ge1).
\end{equation}
Then $\phi_0$ lifts to an isomorphism $\cE\cong\cE'$ on $\widehat Y$.  Indeed, given an
isomorphism $\phi_m$ over $Y_m$, the inclusion $Y_m\subset Y_{m+1}$ is a square-zero
thickening with ideal $\cI^{m+1}/\cI^{m+2}$ and both sheaves are locally free, hence
$\phi_m$ lifts locally, two local lifts differ by a section of
$\cHom(\cE|_{F_J},\cE'|_{F_J})\otimes\cI^{m+1}/\cI^{m+2}$, and the differences on
overlaps form a \v Cech $1$-cocycle whose class is the obstruction to a global
lift, and it vanishes by \eqref{eq:lifting-vanishing}.  The global lift $\phi_{m+1}$ is an
isomorphism because $\phi_0$ is and the thickening is nilpotent, and the compatible
system $(\phi_m)$ is an isomorphism of locally free sheaves on the formal scheme.

\emph{Application.}  Take $\cE=\widehat{\cR_0}$ and $\cE'=\widehat{p^*E_0}$ with
$E_0=\cO^{\oplus(n-r)}\oplus\cQ_J$.  Their restrictions to $F_J$ agree by
\eqref{eq:fiber-bundles}, and \eqref{eq:lifting-vanishing} is
\eqref{eq:end-vanishing} tensored with the constant factors $\SymOp^e\C^{2\ell}$.
For $\widehat{\cR_1}$ take $E_0=\cO^{\oplus n}$ and use
\eqref{eq:positive-vanishing} with $\lambda=\varnothing$.
\end{proof}

\begin{remark}\label{rem:associated-bundle}
\Cref{thm:BC}(c) yields the same isomorphism directly through the \'etale slice.  The
pullback of $\cR_i$ to the slice is the bundle associated with the restriction to
$H=GL(U)$ of the vertex representation, which is $(E_{\mathrm{fr}})_0\oplus U$ and
$(E_{\mathrm{fr}})_1$ by \eqref{eq:polystable}, and the bundle associated with the
standard representation $U$ is the positive quotient $\cQ_J$ because $A:V\to U$ is
surjective in the cluster chamber.  The proof above shows that this clause is not
needed, since the isomorphism type of the underlying bundles is forced by their restriction
to the fiber and Borel--Weil--Bott.
\end{remark}

\begin{warning}[Bundles, not algebras]\label{warn:bundles-not-algebras}
The isomorphisms \eqref{eq:bundle-model} are isomorphisms of vector bundles.  The
multiplication of $\cP_n$, the trace-zero projector and the differential of $\cK_n$
are transported to bundle maps on the formal model and may depend on every formal
coordinate.  On the big cell of \eqref{eq:bigcell}, completion retains all powers of
the cotangent block $D$ and of the core variables, and a scalar trace $\tau_f$ remains
a nonconstant formal function of $(u,BA)$.  Nothing below replaces it by its value on
the zero section.
\end{warning}

\subsection{Passage to the completed neighbourhood}\label{sec:completion}

Let $A=\Gamma(T^*G,\cO)$ as in \cref{sec:bott}, $A'=A[u_1,\dots,u_{2\ell}]$,
$\mathfrak n=(A_+,u_1,\dots,u_{2\ell})$ and $R_{\mathrm{mod}}=\widehat{A'_{\mathfrak n}}$.
The maps $A\to A'\to A'_{\mathfrak n}\to R_{\mathrm{mod}}$ are flat, the last
faithfully so \cite[Tags 00MB, 00MC]{Stacks}.  For a model bundle $\cF_X=p^*\cF$,
flat base change gives
\begin{equation}\label{eq:flat-cohom}
 H^i\bigl(X\times_{\SpecOp A}\SpecOp R_{\mathrm{mod}},\ \cF_X\otimes_AR_{\mathrm{mod}}\bigr)
 \cong H^i(X,\cF_X)\otimes_AR_{\mathrm{mod}}
\end{equation}
\cite[Tag 02KH]{Stacks} (tensor a finite affine \v Cech complex with the flat
algebra).  For a proper scheme over a complete noetherian local ring and a coherent
sheaf, the theorem on formal functions \cite[Tags 02OC, 0A0H]{Stacks} identifies the
cohomology, which is finite and hence complete, with the cohomology of the completed
sheaf on the formal completion along the closed fiber.  Applied to the model over
$R_{\mathrm{mod}}$ and to $Y_n\times_{Q_n}\SpecOp\widehat{\cO}_{Q_n,z_J}$ over
$\widehat{\cO}_{Q_n,z_J}$, and combined with the isomorphism of formal schemes
\eqref{eq:formal-diagram}, this identifies the completed higher direct images
$(R^i\rho_*\cF)^\wedge_{z_J}$ of a bundle $\cF$ on $Y_n$ with the cohomology of the
corresponding completed model bundle, together with all natural maps between such
groups.

\begin{lemma}[The completed positive-bundle package]\label{lem:completed-package}
After adding the core and completing at the origin, the following remain valid.
\textup{(i)} Higher direct images of positive tensor constructions in $p^*\cQ$
vanish.  \textup{(ii)} The tensor-of-sections maps for those bundles are surjective.
\textup{(iii)} The restriction isomorphism \eqref{eq:MQ-restriction} holds with $A_+$
replaced by the maximal ideal of the completed model base.  Moreover, \textup{(i)} and
\textup{(ii)} hold for arbitrary formal direct summands of these completed bundles,
even when the splitting maps are not homogeneous.
\end{lemma}
\begin{proof}
(i) is \eqref{eq:flat-cohom} applied to \cref{thm:bott} and \cref{cor:positive-tensors},
transferred to completed bundles by formal functions.  (ii) A surjection of finite
$A$-modules remains surjective after $\otimes_AR_{\mathrm{mod}}$.  By associativity of
tensor products its source is the tensor product over $R_{\mathrm{mod}}$ of the two
base-changed section modules, and naturality of flat base change and formal
functions identifies the base-changed map with the formal tensor-of-sections map.
The modules are finite over the complete $R_{\mathrm{mod}}$, hence their ordinary
tensor product is complete and equals the completed one.  (iii)
$(M_\cQ\otimes_AR_{\mathrm{mod}})\otimes_{R_{\mathrm{mod}}}\C=M_\cQ\otimes_A\C\cong H^0(G,\cQ)$
by \cref{lem:MQ-restriction}.  The retract argument of \cref{prop:tensor-retracts} is
natural and applies to formal bundle maps.
\end{proof}

\subsection{The pushed cluster algebra at a wall point}\label{sec:pushed}

Keep $z_J$ with $r\ge1$ and put
\[
 R=\widehat{\cO}_{Q_n,z_J},\qquad B_i=(\rho_*\cR_i)^\wedge_{z_J},\qquad
 B=B_0\oplus B_1,\qquad T=(\rho_*\cT_n)^\wedge_{z_J},
\]
finite $R$-modules, and $B$ is a finite commutative parity-graded $R$-algebra whose
multiplication is induced by that of $\cP_n$.

\begin{proposition}[Termwise acyclicity]\label{prop:termwise}
$\bigl(R^i\rho_*(\cK_n^{-a})^+\bigr)^\wedge_{z_J}=0$ for $i>0$ and $0\le a\le n+1$.
\end{proposition}
\begin{proof}
The projector $\id-1\cdot\epsilon$ of \eqref{eq:normalized-trace} exhibits $\cT_n$ as a
retract of $\cR_0$, hence $\cT_n^{\otimes a}$ is a retract of $\cR_0^{\otimes a}$, and the
alternating idempotent, which commutes with it, cuts out $\Lambda^a\cT_n$ as a direct
summand of $\cR_0^{\otimes a}$.  The even terms of $\cK_n$ are
\begin{equation}\label{eq:parity-terms}
 (\cK_n^{-a})^+\cong\cR_0\otimes\Lambda^a\cT_n\ \oplus\ \cR_1\otimes\Lambda^{a-1}\cT_n\otimes(\cO_q\oplus\cO_t)
 \ \oplus\ \cR_0\otimes\Lambda^{a-2}\cT_n\otimes\cO_{qt},
\end{equation}
and by \eqref{eq:bundle-model} tensor powers of $\widehat{\cR_0}$ are sums of positive
tensor powers of $p^*\cQ_J$ and trivial bundles while $\widehat{\cR_1}$ is trivial,
and character lines are trivial as ordinary bundles.  Hence every completed term is a
retract of a finite sum of positive tensor constructions, and
\cref{lem:completed-package}(i) applies.  The trace projector may mix auxiliary
degrees, but the retract argument is ungraded.
\end{proof}

\begin{proposition}[The tensor map for the first differential]\label{prop:critical-tensor}
The natural map $B_0\otimes_RT\to\bigl(\rho_*(\cR_0\otimes\cT_n)\bigr)^\wedge_{z_J}$ is
surjective.
\end{proposition}
\begin{proof}
The completed $\cR_0$ is $\cO^{n-r}\oplus p^*\cQ_J$, hence the tensor-of-sections map for
$\cR_0,\cR_0$ is surjective by \cref{prop:tensor-retracts,lem:completed-package}.  Since
$\cT_n$ is a retract of $\cR_0$ via the normalized trace, the lift-and-retract argument
replaces the second factor by $\cT_n$.
\end{proof}

\begin{lemma}[Residue-field restriction]\label{lem:taut-base-change}
For $i=0,1$ the restriction map $B_i/\mfm_RB_i\to H^0(F_J,\cR_i|_{F_J})$ is an
isomorphism.
\end{lemma}
\begin{proof}
Transport to the formal model by \cref{prop:bundle-model}.  For a trivial summand the
map is $R/\mfm_R=\C\to H^0(F_J,\cO)$, and for the summand $\widehat{p^*\cQ_J}$ it is the
base-changed map of \cref{lem:MQ-restriction,lem:completed-package}(iii).  The
identifications are functorial for the bundle isomorphisms, hence the result is the
original restriction map.
\end{proof}

\begin{proposition}[Polynomial generation after pushforward]\label{prop:polynomial-generation}
The evaluation maps $R\otimes_\C R_i\to B_i$ are surjective for $i=0,1$.  Equivalently,
$R[x,y]\to B$ is surjective, and finitely many polynomials generate each $B_i$ over $R$.
\end{proposition}
\begin{proof}
On the fiber, $H^0(F_J,\cR_1|_{F_J})=R_1/J$ is spanned by odd polynomials.  Taking
cohomology in \eqref{eq:fiber-even-seq} and using $H^0(F_J,\cU_J)=H^1(F_J,\cU_J)=0$
(the boundary-singular case $d=0$, $\lambda=\varnothing$ of \cref{thm:bott}, or
\eqref{eq:fiber-taut} with $H^0(\cQ_J)=V_J$) gives
$R_0/J'\xrightarrow{\sim}H^0(F_J,\cR_0|_{F_J})$, induced by the even polynomial sections.
Thus polynomial sections span the residue-field quotients of \cref{lem:taut-base-change}.
Choose monomial bases of $R_0/J'$ and $R_1/J$.  They define maps from finite free
$R$-modules to $B_0$, $B_1$ whose cokernels $C_i$ are finite with $C_i/\mfm_RC_i=0$, hence
$C_i=0$ by Nakayama.  Products are preserved because the multiplication of $B$ is
induced from $\cP_n$.
\end{proof}

At a positive-rank heart $\dim_\C(B_0/\mfm_RB_0)=\dim R_0/J'=n+r+1$ exceeds the generic
rank $n$.  Thus $B_0$ is not locally free, and no trace of its regular representation is
ever taken.  Instead, applying $\rho_*$ to \eqref{eq:normalized-trace}, using
$\rho_*\cO_{Y_n}=\cO_{Q_n}$ (\cref{prop:models-identified}(a)) and completing,
\begin{equation}\label{eq:B-splitting}
 B_0=R\cdot1\oplus T,\qquad\epsilon:B_0\to R,\quad\epsilon(1)=1,\quad T=\ker\epsilon,
\end{equation}
and for an even polynomial $f$
\begin{equation}\label{eq:B-trace}
 \epsilon(\bar f)=\tau_f/n,\qquad(\tau_f:\ f(0)=0)=\mfm_0R
\end{equation}
by \eqref{eq:trace-ideal}.  All hypotheses of \cref{lem:trace-quotient} therefore hold
for the actual finite $R$-algebra $B$, with the base ring $R$ of the lemma equal to the
complete local ring $\widehat{\cO}_{Q_n,z_J}$ and with $I=\mfm_0R$.

\subsection{The local computation}\label{sec:local-proof}

Let $(R\rho_*\cK_n^+)^\wedge_{z_J}$ denote derived pushforward followed by the flat
extension of scalars to $R$.  The hypercohomology spectral sequence
$E_1^{a,b}=(R^b\rho_*(\cK_n^a)^+)^\wedge_{z_J}\Rightarrow H^{a+b}$ of the bounded complex
has, by \cref{prop:termwise}, only the row $b=0$, whose differential is the pushforward
of the Haiman differential.  Equivalently, a bounded complex of $\rho_*$-acyclic
sheaves computes $R\rho_*$, and
\begin{equation}\label{eq:termwise-complex}
 (R\rho_*\cK_n^+)^\wedge_{z_J}\simeq\bigl(\rho_*\cK_n^{\bullet,+}\bigr)^\wedge_{z_J}\in D^{\le0}(R).
\end{equation}

\emph{Concentration.}  By \cref{prop:perfect-reduction}, $(R\rho_*\cK_n^+)^\wedge_{z_J}
\simeq(R\rho_*\cI_n)^\wedge_{z_J}\in D^{\ge0}(R)$, the derived pushforward of a sheaf
followed by a flat extension of scalars.  Hence
\begin{equation}\label{eq:intersection-t-structures}
  (R\rho_*\cK_n^+)^\wedge_{z_J}\in D^{\le0}(R)\cap D^{\ge0}(R)=\Coh(R),
\end{equation}
that is, the completed derived pushforward is a single module, namely the cokernel of the
first differential of \eqref{eq:termwise-complex}.  This is the one place where the
even-excess vanishing turns a cokernel computation into a computation of the entire
derived pushforward, and in particular it kills all higher direct images of $\cI_n$ at
$z_J$.

\emph{The first differential.}  By \eqref{eq:parity-terms},
$(\cK_n^0)^+=\cR_0$ and $(\cK_n^{-1})^+=\cR_0\otimes\cT_n\oplus\cR_1\otimes\cO_q\oplus\cR_1\otimes\cO_t$,
with differential $(b\otimes t,a,c)\mapsto bt+xa+yc$.  The pushforward of the first
summand maps to $B_0$ through the commutative diagram
\[
\begin{tikzcd}[column sep=huge]
 B_0\otimes_RT\arrow[r,two heads]\arrow[dr,"\mathrm{mult}"'] &
 \bigl(\rho_*(\cR_0\otimes\cT_n)\bigr)^\wedge_{z_J}\arrow[d,"\rho_*(\mathrm{mult})"]\\
 & B_0 ,
\end{tikzcd}
\]
whose top arrow is surjective by \cref{prop:critical-tensor}.  Hence its image is exactly
$B_0T$, and not a possibly larger submodule.  The two coordinate summands have images
$xB_1$ and $yB_1$.  Therefore
\begin{equation}\label{eq:actual-cokernel}
 H^0\bigl((R\rho_*\cK_n^+)^\wedge_{z_J}\bigr)=B_0/(B_0T+xB_1+yB_1).
\end{equation}
Acyclicity of the terms alone would not give this, the tensor surjectivity is what
identifies the image.

\emph{The quotient.}  By \cref{prop:polynomial-generation} the map $R[x,y]\to B$ is
surjective, and by \eqref{eq:B-splitting}--\eqref{eq:B-trace} the pushed-forward
normalized trace $\epsilon$ satisfies $\epsilon(1)=1$ and has trace ideal
$I=\mfm_0R$.  Hence \cref{lem:trace-quotient} applies to $B$ and gives
\begin{equation}\label{eq:local-quotient}
 R/\mfm_0R\xrightarrow{\ \sim\ }B_0/(B_0T+xB_1+yB_1),\qquad 1\longmapsto[1],
\end{equation}
with inverse induced by the normalized trace.  The kernel of the unit is exactly
$\mfm_0R$, not a saturation or radical.  Combining,
\begin{equation}\label{eq:completed-descent}
 (R\rho_*\cO_{F_n})^\wedge_{z_J}\cong R/\mfm_0R=\widehat{\cO}_{\cZ,z_J}.
\end{equation}

\emph{The map.}  The unit $\cO_{Y_n}\to\cP_n\to\cI_n$ sends $1$ to $[1]$, and so does
the augmentation $\cK_n^+\to\cI_n$ on the degree-zero representative $1\in\cR_0$.  By
naturality of the adjunction unit for the Cartesian square of
\cref{sec:perfect-reduction}, the map induced by $u_n$ on $H^0$ of the completions is
the map \eqref{eq:local-quotient}.  Both sides are concentrated in degree zero, hence
$(u_n)^\wedge_{z_J}$ is an isomorphism.  There is no undetermined scalar, the inverse
sends $[1]$ to $1$.

\emph{Rank zero.}  At a fixed point with $k_J=1$ the fiber is a single reduced point
(\cref{lem:fiber-scheme}), hence $\rho$ is quasi-finite over an open neighbourhood of
$z_J$.  Being proper and birational onto the normal $Q_n$, it is an isomorphism there
by Zariski's Main Theorem \cite[Cor.~III.11.4]{Hartshorne}.  Hence $F_n\to\cZ$ is an
isomorphism near $z_J$ and $u_n$ is the identity.  (Equivalently, the Bellamy--Craw
local contraction is an isomorphism when the local Ext-quiver has no multiplicity
vertex.)

\subsection{Proof of symbolic descent}\label{sec:SD-proof}

\begin{lemma}[Detection at the fixed hearts]\label{lem:reduction-fixed}
Let $D_n=\Cone\bigl(u_n:\cO_{\cZ_n^{\mathrm{symb}}}\to R\rho_*\cO_{F_n}\bigr)$.  If
$D_n\ne0$, its completion at some $T$-fixed closed point of $\cZ_n^{\mathrm{symb}}$ is
nonzero.  Consequently $u_n$ is an isomorphism as soon as its completion is an
isomorphism at every $z_J\in Q_n^T$.
\end{lemma}
\begin{proof}
$D_n$ is a bounded $T$-equivariant complex with coherent cohomology supported on the
projective scheme $\cZ_n^{\mathrm{symb}}=\pi_Q^{-1}(0)$.  Let $W$ be the reduced
support of a nonzero cohomology sheaf, a nonempty closed $T$-stable subvariety.  It
contains a $T$-fixed point.  Indeed, embed $\cZ$ equivariantly in $\PP(V)$ by a power of the
$T$-linearized ample $\cO(1)$, decompose a lift of a point of $W$ into finitely many
weight vectors, and take the limit under a one-parameter subgroup separating those
weights.  The limit lies in $W$ and is fixed by $T$, being a projective point in a
single weight space.  At that point the stalk of the sheaf is nonzero, and completion
of a noetherian local ring is faithfully flat and commutes with cohomology
\cite[Tag 00MC]{Stacks}, hence the completed complex is nonzero.
\end{proof}

\begin{theorem}[Symbolic descent]\label{thm:SD}
For every $n\ge0$ the canonical map $u_n$ of \eqref{eq:u-n} is an isomorphism, and for
every $d\in\Z$
\[
 R\rho_*\bigl(\cL_n^{\otimes d}\otimes\cI_n\bigr)\;\simeq\;\cO_{\cZ_n^{\mathrm{symb}}}(d),
\]
canonically and $T$-equivariantly.
\end{theorem}
\begin{proof}
For $n=0$ everything is a point.  For $n\ge1$, at every fixed point $z_J$ with
$k_J\ge2$ the completion of $u_n$ is an isomorphism by \eqref{eq:completed-descent}
and the identification of the map, and at every fixed point with $k_J=1$ by
the rank-zero argument.  Hence $D_n$ has zero completion at every fixed point, and
$D_n=0$ by \cref{lem:reduction-fixed}.  The local computations forget the torus
grading, since the auxiliary grading of the model need not be compatible with it.
This suffices, because $u_n$ is equivariant and forgetting equivariance is
conservative.  Finally $\cL_n^{\otimes d}=\rho^*\cO_{Q_n}(d)$ with $\cO_{Q_n}(d)$
invertible, hence the projection formula \cite[Tag 01E6]{Stacks} gives
$R\rho_*(\cL_n^{\otimes d}\otimes\cO_{F_n})\simeq\cO_{Q_n}(d)\otimes R\rho_*\cO_{F_n}
\simeq\cO_{Q_n}(d)\otimes\cO_{\cZ}=\cO_{\cZ}(d)$, for negative $d$ as well.
\end{proof}

This completes the proof of Theorem~E.

\begin{corollary}\label{cor:higher-direct-images}
$\rho_*\cO_{F_n}=\cO_{\cZ_n^{\mathrm{symb}}}$, $R^i\rho_*\cO_{F_n}=0$ for $i>0$, and
\[
 R\Gamma\bigl(F_n,\cL_n^{\otimes d}|_{F_n}\bigr)\cong R\Gamma\bigl(\cZ_n^{\mathrm{symb}},\cO(d)\bigr)
 \qquad(d\in\Z).
\]  In particular $H^{>0}(F_n,\cL_n^{\otimes d})=0$ for $d\gg0$ by Serre
vanishing on the projective $\cZ_n^{\mathrm{symb}}$.
\end{corollary}

The two statements of the corollary have different sources.  The trace quotient
identifies the degree-zero image, and the intersection of the two cohomological
bounds in \eqref{eq:intersection-t-structures} kills the other degrees.  The first
alone would not imply the second.

\begin{remark}[What the proof does not assume]\label{rem:SD-not-assumed}
The proof does not require $\rho$ to be flat, the punctual scheme
$\cZ_n^{\mathrm{symb}}$ to be a local complete intersection in $Q_n$ or its structure
sheaf to be perfect over $\cO_{Q_n}$ (both fail at $n=2$, \cref{prop:lci-fails}),
either punctual scheme to be reduced or Cohen--Macaulay, or ordinary restriction to a
wall fiber to agree with derived pullback.  The only product description used is the
formal description \eqref{eq:formal-diagram} of the ambient contraction with its
underlying bundles.  The multiplication, the trace-zero projector and the first
Haiman differential remain their actual maps, and the scalar traces remain in the
ordinary ideal with their full dependence on all formal coordinates.
\end{remark}

\subsection{The case $n=2$ in closed form}\label{sec:n2-descent}

As an illustration of \cref{thm:SD} in the first nontrivial rank, and as an
independent check, we describe the geometry of $\rho$ over the unique wall point of
$Q_2$ with positive-dimensional fiber completely.  Every ingredient of the general
proof is visible here in closed form.

\begin{theorem}[The geometry at $n=2$]\label{thm:n2-descent}
Let $z_J\in Q_2$ be the wall point with
$F_J\cong\PP^2$, hence $J=R_1\cap\mfm^3$ and $V:=V_J=\langle x^2,xy,y^2\rangle=\SymOp^2\C^2$,
and let $U\subset Q_2$ be the open set of quotients $J$ for which $x,y$ is a basis
of $R_1/J$.  Then the following hold.
\begin{enumerate}[label=\textup{(\alph*)}]
\item $\rho^{-1}(U)\cong T^*\PP^2$, $\PP^2=\Gr(2,V)$, equivariantly for $T$ and for
the group $GL_2$ acting on $\langle x,y\rangle$, and $\rho|_{\rho^{-1}(U)}$ is the
affinization map.  Thus $U\cong\overline{\mathcal O}_{\min}\subset\mathfrak{sl}(V)$ and
$\rho$ is the Springer map.
\item $\mfm_0\cO_{Q_2}|_U$ is generated by the three trace functions
$\Tr U,\Tr V,\Tr W$ of \cref{prop:lci-fails}, which under \textup{(a)} are the
components of the moment map of the $SL_2$-action on $T^*\PP^2$, and
$\cZ_2^{\mathrm{symb}}\cap U$ is the cone over the rational normal quartic,
$\cO(\cZ_2^{\mathrm{symb}}\cap U)=\bigoplus_{d\ge0}\SymOp^{4d}\C^2$, i.e.\ the
quotient singularity $\C^2/\mu_4$.
\item $F_2\cap\rho^{-1}(U)$ is a reduced local complete intersection, equal to
$\PP^2\cup N^*_C$, where $C\subset\PP^2$ is the conic of squares and
$N^*_C\cong\Tot\,\cO_{\PP^1}(-4)$ is its conormal bundle.
\end{enumerate}
Consequently $R\rho_*\cO_{F_2}\simeq\cO_{\cZ_2^{\mathrm{symb}}}$ via $u_2$, in
accordance with \cref{thm:SD}.
\end{theorem}

\begin{proof}
By \cref{cor:threshold} and \cref{app:census} the only fixed point of $Q_2$ with
$k\ge2$ is $z_J$, and by \cref{lem:reduction-fixed} together with the rank-zero
case it suffices to prove \eqref{eq:SD0} on a neighbourhood of $z_J$.  The group $GL_2$ acting on $\langle x,y\rangle$ commutes with $\Gamma$,
hence acts on $Y_2$, $Q_2$, $S_2$ compatibly with $\rho$, $\pi_Q$, the polarized
power sums, and the torus $T\subset GL_2$.  The symplectic form of $Y_2$ transforms
by $\det$, hence it is $SL_2$-invariant.

\emph{Step 1: $U$ is the attracting chart of $z_J$.}  $U$ is open, $GL_2$-stable,
and contains $z_J$.  Its coordinate ring is generated by the entries of the
matrices of $u,v,w$ acting on $R_1/J$ in the basis $x,y$ (the chart of
\cref{prop:lci-fails}), and their $T$-weights are
$q^2,\ q^3t^{-1},\ t^2,\ t^3q^{-1},\ qt,\ t^2,\ q^2,\ qt$, all of weight $s^2$
under the central one-parameter subgroup $(q,t)=(s,s)$.  Hence $\C[U]$ is positively
graded by this subgroup with $\C[U]_0=\C$.  Thus every point of $U$ flows to $z_J$, and
every point of $Q_2$ flowing to $z_J$ lies in $U$, $U$ being open and $T$-stable.

\emph{Step 2: $\rho^{-1}(U)\cong T^*\PP^2$.}  Every point of $\rho^{-1}(U)$ has a
limit under the central subgroup, by properness of $\rho$, and the limit lies in
$\rho^{-1}(z_J)=\PP^2$.  The fixed locus of the central subgroup in $\rho^{-1}(U)$
is exactly $\PP^2$, since it maps to the unique fixed point $z_J$ of $U$, and the
subgroup acts trivially on $\Gr(2,V)$ because it acts on $V$ by the scalar $s^2$.
By the Bia{\l}ynicki-Birula decomposition \cite[Thm.~4.1]{BB} the smooth variety
$\rho^{-1}(U)$ is the attracting set of $\PP^2$, a Zariski locally trivial
fibration in affine spaces on which the subgroup acts with the weights of the
normal bundle.  At the three fixed points $(2,2)$, $(3,1)$, $(2,1,1)$ of $\PP^2$
the tangent weights of $Y_2$ (\cref{lem:localization}) are $q^{\pm1}t^{\mp1}$ along
$\PP^2$ and $q^2,t^2$ in the normal directions.  The normal weight is $s^2$
throughout, hence the transition functions of the fibration commute with a single
scaling and are linear, and $\rho^{-1}(U)\cong N_{\PP^2/Y_2}$ equivariantly.
Since $\PP^2$ is Lagrangian and the symplectic form is $SL_2$-invariant,
$N_{\PP^2/Y_2}\cong T^*\PP^2$ $SL_2$-equivariantly.  Finally
$U=\SpecOp\Gamma(\rho^{-1}(U),\cO)$, because $\rho_*\cO_{Y_2}=\cO_{Q_2}$ and $U$
is affine, and $\Gamma(T^*\PP^2,\cO)=\bigoplus_dH^0(\PP^2,\SymOp^dT_{\PP^2})$ is the
coordinate ring of the minimal nilpotent orbit closure
$\overline{\mathcal O}_{\min}\subset\mathfrak{sl}(V)$, which is normal with ideal
generated by the $2\times2$ minors \cite{KraftProcesi79}, $\rho$ being its
Springer resolution.  This is (a).

\emph{Step 3: the trace ideal.}  By \cref{prop:H0}, $\mfm_0\cO_{Q_2}|_U$ is
generated by the functions $\Tr(m_f\,|\,R_1/J)$, $f$ a $\Gamma$-invariant monomial
of degree $2d>0$.  These are homogeneous of degree $d$ in the matrix entries, and
$f\mapsto\Tr(m_f)$ is $GL_2$-equivariant.  Under (a), $\C[U]_1=H^0(\PP^2,T_{\PP^2})
=\mathfrak{sl}(V)\cong\SymOp^2\oplus\SymOp^4$ as an $SL_2$-module, and the traces of
$u,v,w$ span a nonzero submodule isomorphic to $\SymOp^2$ (nonzero since $\Tr U\ne0$),
hence exactly the image of $\mathfrak{sl}_2\to\mathfrak{sl}(V)$.  In other words, they are the
vector fields $\xi_E,\xi_F,\xi_H$ of the $SL_2$-action, viewed as fiber-linear
functions on $T^*\PP^2$, i.e.\ the components of the moment map of the
$SL_2$-action.  The ideal of $\overline{\mathcal O}_{\min}$ is generated by
quadrics spanning $\mathfrak{sl}(V)\oplus\C$, namely $X\mapsto\Tr(MX^2)$ for
$M\in\mathfrak{sl}(V)$ and $\Tr(X^2)$.  Restricted to the linear subspace
$\SymOp^4=\mathfrak{sl}_2^\perp$ they vanish on the cone $\{\ell^4\}$ over the
rational normal quartic (the $SL_2$-orbit closure of the highest root vector
$E_\theta\in\SymOp^4$), hence they lie in the six-dimensional space
$\SymOp^4\oplus\SymOp^0\subset\SymOp^2(\SymOp^4)^*$ of quadrics vanishing on that
cone.  Both isotypic components occur, the invariant one because the trace form is
nondegenerate on $\SymOp^4$, the other because $\Tr(E_\theta X^2)=2$ for
$X=E_\theta+E_{-\theta}+\mathrm{diag}(1,-2,1)\in\SymOp^4$.  Since the ideal of the
rational normal quartic is generated by quadrics,
\[
 \C[U]/(\xi_E,\xi_F,\xi_H)=\C[\SymOp^4]/(\text{quadrics of the cone})
 =\bigoplus_{d\ge0}\SymOp^{4d}\C^2 ,
\]
the coordinate ring of the cone over the quartic, which is the quotient
$\C^2/\mu_4$.  For $d\ge1$ every $SL_2$-equivariant map $\SymOp^{2d}\to\SymOp^{4d}$
is zero, hence all higher traces lie in $(\xi_E,\xi_F,\xi_H)$.  This is (b).

\emph{Step 4: the punctual fiber.}  Thus $F_2\cap\rho^{-1}(U)=V(\xi_E,\xi_F,\xi_H)
\subset T^*\PP^2$, with $\pi_*\cO_{F_2}=\SymOp^\bullet T_{\PP^2}/(\xi)$.  The
$SL_2$-orbits on $\PP^2=\PP(\SymOp^2)$ are the conic $C=\{[\ell^2]\}$ and its
complement.  The three fields generate the subsheaf $T_{\PP^2}(-\log C)\subset
T_{\PP^2}$ of fields tangent to $C$.  Indeed, they lie in it, and they generate its fibers,
off $C$ because the orbit is open, at $w\in C$ because the stabilizer $\mathfrak b$
of $w$ acts on the normal line $N_{C/\PP^2,w}$ with a nonzero weight.  In local
coordinates $(c,s)$ with $C=\{c=0\}$ and dual fiber coordinates $(p_c,p_s)$ this
ideal is $(p_s,\,cp_c)$, hence $F_2\cap\rho^{-1}(U)$ is the reduced local complete
intersection $\PP^2\cup N^*_C$, and $N^*_C=\Tot\,\cO_{\PP^1}(-4)$ since
$N_{C/\PP^2}=\cO_{\PP^2}(2)|_C$.  This is (c).

\emph{Step 5: conclusion.}  Since $(p_s,p_c)\cap(p_s,c)=(p_s,cp_c)$ there is an
exact sequence $0\to\cO_{F_2}\to\cO_{\PP^2}\oplus\cO_{N^*_C}\to\cO_C\to0$ on
$\rho^{-1}(U)$.  Now $R\rho_*\cO_{\PP^2}=R\Gamma(\PP^2,\cO)=\C_{z_J}$ and
$R\rho_*\cO_C=R\Gamma(\PP^1,\cO)=\C_{z_J}$.  The complement $N^*_C\setminus C$ is a
single $SL_2$-orbit, with stabilizer $U\rtimes\mu_4$ ($U$ the unipotent radical of
$\mathfrak b$, the torus acting on the conormal line by a character of order four),
mapping equivariantly onto the orbit $SL_2\cdot E_\theta$, whose stabilizer is the
same group.  Hence $\rho|_{N^*_C}$ is birational onto $\cZ_2^{\mathrm{symb}}\cap U$.
Moreover $\Gamma(N^*_C,\cO)=\bigoplus_dH^0(\PP^1,\cO(4d))=\bigoplus_d\SymOp^{4d}$, hence
the injective pullback $\cO(\cZ_2^{\mathrm{symb}}\cap U)\to\Gamma(N^*_C,\cO)$ is an
isomorphism degree by degree, and $H^1(N^*_C,\cO)=\bigoplus_dH^1(\PP^1,\cO(4d))=0$.
Therefore $R\rho_*\cO_{N^*_C}=\cO_{\cZ_2^{\mathrm{symb}}}$ on $U$.  In the pushed-forward
sequence the map $\C\oplus\cO_{\cZ}\to\C$ is $(a,f)\mapsto a-f(z_J)$, surjective
with kernel $\{(f(z_J),f)\}\cong\cO_{\cZ}$.  Therefore
$R\rho_*\cO_{F_2}\simeq\cO_{\cZ_2^{\mathrm{symb}}}$ on $U$, and the isomorphism is
the descent map $u_2$, which sends $f$ to $(f(z_J),f|_{N^*_C})$.  Twists follow by
the projection formula.
\end{proof}

\begin{remark}[Exact check]\label{rem:n2-check}
With $\cF_d:=(\SymOp^\bullet T_{\PP^2}/(\xi_E,\xi_F,\xi_H))_d$, a direct computation in
Macaulay2 gives $H^1(\PP^2,\cF_d)=H^2(\PP^2,\cF_d)=0$ and $h^0(\PP^2,\cF_d)=4d+1$ for
$0\le d\le9$, and the ideal of $\mathfrak{sl}_3$ generated by the $2\times2$
minors, the trace, and the three pairings with $\mathfrak{sl}_2$ is prime of
dimension two and degree four with Hilbert function $4d+1$, whereas a generic triple of
linear vector fields gives different numbers.  This confirms (b) and (c) and the
vanishing of $R^{>0}\rho_*\cO_{F_2}$ independently of Steps 3--5, and it is the
$n=2$ instance of \cref{cor:higher-direct-images}.
\end{remark}

\begin{remark}
The germ $(\cZ_2^{\mathrm{symb}},z_J)\cong\C^2/\mu_4$ is Cohen--Macaulay and not
Gorenstein, in agreement with \cref{prop:lci-fails}, and it is resolved by the
non-contracted component $N^*_C$ of $F_2$ while the contracted component $\PP^2$
contributes only the constants.  At a wall point of rank $r$ and dimension $n>r(r+1)$
the transverse core is positive-dimensional, the chart is no longer an attracting
cone, and the trace functions depend on the core coordinates as well as on the
cotangent block.  The general proof works on the formal neighbourhood instead, where
this dependence is retained (\cref{warn:bundles-not-algebras}).
\end{remark}

\section{The all-twists localization polynomial}\label{sec:localization}

In this section we compute the $q,t$-characters.  Equivariant localization on $Y_n$
expresses the Euler characteristic of $\cL_n^{\otimes m}\otimes\cK_n$ as the
fixed-point sum $F_{n,m}$ of \eqref{eq:F-def}, and the determinant-parity projection
$P_{n,m}$ is the Euler characteristic of $\cL_n^{\otimes m}\otimes\cI_n$, by the
two-term homology and the even-excess vanishing.  Symbolic descent identifies
$P_{n,m}$ with the Euler characteristic of $\cO(m)$ on $\cZ_n^{\mathrm{symb}}$, and
Serre vanishing turns it into the Hilbert series of $\mfS^{\mathrm{symb}}_{n,m}$ for
$m\gg0$ (\cref{thm:character}).

The main idea of the principal specialization $t=q^{-1}$ (\cref{thm:principal}) is
that only the hook partitions survive the specialization, because every partition
which is not a hook contains the box $(1,1)$, whose factor $1-qt$ vanishes.  The hook
contributions pair up, and the finite $q$-binomial theorem gives the closed form.
This proves Theorem~C and, together with \cref{thm:character}, the dimension count
$\binom{n(m+1)}{n}$ used in \cref{sec:cherednik,sec:collapse}.

\subsection{The fixed-point formula}

\begin{lemma}\label{lem:localization}
The $T$-fixed points of $Y_n$ are the monomial balanced ideals
$I_\lambda$, $\lambda\in\Bal(2n)$, the tangent weights at $I_\lambda$ are the
$2n$ hook weights
$q^{a(s)+1}t^{-l(s)},\ q^{-a(s)}t^{l(s)+1}$ over the boxes $s\in\lambda$ with
$a(s)+l(s)$ odd, and
\[
 [\cP_n]_\lambda=B_\lambda,\qquad
 [\cL_n]_\lambda=\ell_\lambda,\qquad
 \sum_{a\ge0}(-1)^a\bigl[{\textstyle\bigwedge^a}\cE_n\bigr]_\lambda
 =\prod_{w\ \in\ \wt(\cE_n|_\lambda)}(1-w)
 =(1-q)(1-t)\,\Pi^{(0)}_\lambda,
\]
in the notation \eqref{eq:statistics}.  Consequently, for every $m\in\Z$,
\begin{equation}\label{eq:F-localization}
 F_{n,m}(q,t):=\chi^T\bigl(Y_n,\cL_n^{\otimes m}\otimes\cK_n\bigr)
 =\sum_{\lambda\in\Bal(2n)}
 \frac{\ell_\lambda^{\,m}B_\lambda(1-q)(1-t)\Pi^{(0)}_\lambda}{D_\lambda},
\end{equation}
agreeing with \eqref{eq:F-def}, and the determinant-parity projection
$P_{n,m}$ of \eqref{eq:F-def} satisfies
\begin{equation}\label{eq:corrected-parity}
 P_{n,m}=\chi^T\bigl(Y_n,\cL_n^m\otimes\cI_n\bigr)
 -\chi^T\bigl(Y_n,\cL_n^m\otimes\cM_n^+\bigr)
 =\chi^T\bigl(Y_n,\cL_n^m\otimes\cI_n\bigr),
\end{equation}
the last equality by \cref{cor:twists}.
\end{lemma}

\begin{proof}
The fixed points of $\Hilb^{2n}(\A^2)^T$ are monomial ideals, and the
$\Gamma$-fixed component $Y_n$ selects the balanced ones.  The tangent space of the
Hilbert scheme at $I_\lambda$ has the standard hook-weight basis, and the
$\Gamma$-invariant part consists of the weights of even total degree, i.e.\
$a+l$ odd.  The fiber of $\cP_n$ is spanned by the standard monomials
($B_\lambda$), the line bundle $\cL_n=\det\cR_1$ gives the product of odd-box weights, and
$[\cE_n]_\lambda=(B_\lambda^{\mathrm{even}}-1)+q+t$, hence the Koszul factor
$\sum_a(-1)^a[\bigwedge^a\cE_n]_\lambda=(1-q)(1-t)\Pi^{(0)}_\lambda$.  Equivariant
localization (Thomason) gives \eqref{eq:F-localization}.  The class
$[\cK_n]=[\cI_n]-[\cM_n]$ (\cref{thm:two-term}) splits under the
determinant-parity projector $\Pi^{(m)}f=\tfrac12(f(q,t)+(-1)^{nm}f(-q,-t))$ into
the even part $[\cI_n]-[\cM_n^+]$ (note $\ell_\lambda(-q,-t)=(-1)^{n}\ell_\lambda$
and $\cI_n$ is even), giving the first equality of
\eqref{eq:corrected-parity}, and the second is \cref{cor:twists}.
\end{proof}

\subsection{Hook reduction and the principal specialization}

\begin{theorem}\label{thm:principal}
Theorem C holds.  For all $n\ge1$, $m\ge0$, with $M=n(m+1)$, $Q=q^2$, we have
\begin{equation}\label{eq:F-principal}
 F_{n,m}(q,q^{-1})
 =q^{-mn^2}\qbinom{M}{n}_{Q}-q^{1-mn^2}[m]_{Q^n}\qbinom{M}{n-1}_{Q},
\end{equation}
with the two summands of opposite exponent parity.  Consequently
$P_{n,m}(q,q^{-1})=q^{-mn^2}\qbinom{M}{n}_{Q}$,
$N_{n,m}(q,q^{-1})=q^{1-mn^2}[m]_{q^{2n}}\qbinom{M}{n-1}_{q^2}$, and the scalar
values of Theorem C follow at $q\to1$.
\end{theorem}

\begin{proof}
\emph{Step 1: only hooks survive.}  If $\lambda$ is not a hook it contains the box
$(1,1)$, of even parity, contributing the factor $1-qt$ to
$\Pi^{(0)}_\lambda$, which vanishes at $t=q^{-1}$, while
$D_\lambda(q,q^{-1})=\prod(1-q^{\pm h})\neq0$ in $\Q(q)$.  Every hook of size $2n$
is balanced, and the hooks are
\[
 \lambda_j^+=(2j{+}1,1^{2n-2j-1}),\ 0\le j\le n-1;\qquad
 \lambda_j^-=(2j,1^{2n-2j}),\ 1\le j\le n .
\]

\emph{Step 2: a general hook.}  For $\lambda=(a+1,1^b)$, $a+b=2n-1$, put
$u=\lfloor a/2\rfloor$, $v=\lfloor b/2\rfloor$.  Then
\[
 B_\lambda(q,q^{-1})=q^{-b}\frac{1-q^{2n}}{1-q},\qquad
 \Pi^{(0)}_\lambda(q,q^{-1})=\prod_{r=1}^{u}(1-Q^r)\prod_{s=1}^{v}(1-Q^{-s}),
\]
and, since the even hook lengths are $2n$ (origin) and
$2,4,\dots,2u$ and $2,4,\dots,2v$ along the two arms,
\[
 D_\lambda(q,q^{-1})=(1-Q^n)(1-Q^{-n})
 \prod_{r=1}^{u}(1-Q^r)(1-Q^{-r})\prod_{s=1}^{v}(1-Q^s)(1-Q^{-s}).
\]
Using $\prod_{r\le u}(1-Q^{-r})=(-1)^uQ^{-u(u+1)/2}(Q;Q)_u$ and
$\frac{(1-q^{2n})(1-q^{-1})}{(1-Q^n)(1-Q^{-n})}=\frac{q^{2n-1}(1-q)}{1-Q^n}$, the
contribution is
\begin{equation}\label{eq:general-hook}
 C^{(m)}_{a,b}
 =(-1)^u\,\ell_\lambda(q,q^{-1})^m\,q^{-b+2n-1+u(u+1)}\,
 \frac{1-q}{(1-Q^n)(Q;Q)_u(Q;Q)_v}.
\end{equation}

\emph{Step 3: the two families.}  For both families the odd boxes have total
exponent $j^2-(n-j)^2$, hence
$\ell_{\lambda_j^\pm}(q,q^{-1})=q^{2nj-n^2}$.  Substituting $(a,b,u,v)$ and using
$\frac1{(1-Q^n)(Q;Q)_j(Q;Q)_{n-j-1}}=\frac1{(Q;Q)_n}\qbinom{n-1}{j}_Q$,
\begin{align*}
 C^{(m)}(\lambda_j^+)&=\frac{1-q}{(Q;Q)_n}(-1)^j
 q^{-mn^2+j^2+(2nm+3)j}\qbinom{n-1}{j}_Q,\\
 C^{(m)}(\lambda_j^-)&=\frac{1-q}{(Q;Q)_n}(-1)^{j+1}
 q^{-mn^2-1+j^2+(2nm+1)j}\qbinom{n-1}{j-1}_Q .
\end{align*}

\emph{Step 4: pairing and the $q$-binomial theorem.}  Pair $\lambda_k^+$ with
$\lambda_{k+1}^-$ ($0\le k\le n-1$).  Signs and Gaussian coefficients match, the
second exponent exceeding the first by $2nm+1$, and we obtain
\[
 F_{n,m}(q,q^{-1})=\frac{(1-q)(1+q^{2nm+1})q^{-mn^2}}{(Q;Q)_n}
 \sum_{k=0}^{n-1}(-1)^kq^{k^2+(2nm+3)k}\qbinom{n-1}{k}_Q .
\]
With $q^{k^2+(2nm+3)k}=Q^{\binom k2}(Q^{nm+2})^k$ and the finite $q$-binomial
theorem $\sum_k(-1)^kQ^{\binom k2}\qbinom{n-1}kz^k=(z;Q)_{n-1}$ at $z=Q^{nm+2}$,
\[
 F_{n,m}(q,q^{-1})
 =q^{-mn^2}(1-q)(1+q^{2nm+1})\,
 \frac{(Q;Q)_{n(m+1)}}{(Q;Q)_n(Q;Q)_{nm+1}}
 =q^{-mn^2}\qbinom{M}{n}_{Q}\frac{(1-q)(1+q^{2nm+1})}{1-Q^{nm+1}} .
\]

\emph{Step 5: parity separation.}  The identity
$\frac{(1-q)(1+q^{2nm+1})}{1-Q^{nm+1}}=1-q\,\frac{1-Q^{nm}}{1-Q^{nm+1}}$ and
$\qbinom{M}{n}_Q\frac{1-Q^{nm}}{1-Q^{nm+1}}=[m]_{Q^n}\qbinom{M}{n-1}_Q$ (from
$\qbinom{M}{n-1}_Q\big/\qbinom{M}{n}_Q=\frac{1-Q^n}{1-Q^{nm+1}}$) yield
\eqref{eq:F-principal}.  Exponents in the first summand are $\equiv nm$, in the
second $\equiv nm+1\pmod2$.  Under $t=q^{-1}$ total bidegree parity is preserved
($a+b\equiv a-b$), hence $\Pi^{(m)}$ selects the first summand.  Let $q\to1$ for the
scalar identities, and note that $m\binom{M}{n-1}=\frac{nm}{nm+1}\binom Mn$ gives
$F_{n,m}(1,1)=\frac1{nm+1}\binom Mn$.
\end{proof}

\begin{remark}[Independent verification]\label{rem:verification}
The contents of this section are elementary and finite, hence machine-checkable.
Implementing \eqref{eq:F-def} verbatim in exact rational arithmetic, we have
confirmed for all $n\le4$ and $0\le m\le3$ that the fixed-point sum
\eqref{eq:F-localization} simplifies to a Laurent polynomial, the full identity
\eqref{eq:F-principal}, the Gaussian specialization
of $P_{n,m}$, and the three scalar identities
$P_{n,m}(1,1)=\binom{n(m+1)}n$, $N_{n,m}(1,1)=m\binom{n(m+1)}{n-1}$,
$F_{n,m}(1,1)=\frac1{nm+1}\binom{n(m+1)}n$.  Likewise all statements of
\cref{ex:lowrank}, and every assertion of
\cref{app:census} \textup(length duality
$|C|+|D|=2n+1$ and $N=2k-1$ at \emph{every} heart, the fiber counts
$|\rho^{-1}(z_J)^T|=\binom Nk$, the census profile, and the thresholds
$n=m(m-1)$ for $\Gr(m,2m-1)$\textup) have been verified for all $n\le8$.
\end{remark}

\begin{example}\label{ex:lowrank}
For $n=2$ we have $P_{2,1}=[5]_{q,t}+qt$ and $N_{2,1}=qt[4]_{q,t}$.
For $n=3$ we have $P_{3,1}=[10]+qt[6]+qt[4]$ and $N_{3,1}=qt[9]+q^2t^2[5]+q^4t^4$.
For $n=4$ we have $P_{4,1}=[17]+qt[13]+qt[11]+q^2t^2[9]+qt[9]+q^3t^3[5]+q^2t^2[5]+q^4t^4$
and $N_{4,1}=qt[16]+q^2t^2[12]+q^2t^2[10]+q^3t^3[8]+q^4t^4[6]+q^4t^4[4]$.  Here
$[k]_{q,t}=\sum_{i+j=k-1}q^it^j$.  The positive
parts agree with Stump's tables of $\mathrm{Cat}^{(1)}(B_n;q,t)$
\cite[Cor.~3.18 and App.~A]{Stump}, and
$(P_{n,1}(1,1),F_{n,1}(1,1))=\bigl(\binom{2n}n,\mathrm{Cat}_n\bigr)$ for
$n=2,3,4$.
\end{example}


\begin{theorem}\label{thm:character}
For each $n$ there is $m_0(n)$ such that for all
$m\ge m_0(n)$ we have
$H^{i}(\cZ_n^{\mathrm{symb}},\cO(m))=0$ for $i>0$, the canonical map
$\widetilde\mfS^{\mathrm{symb}}_{n,m}\to H^0(\cZ_n^{\mathrm{symb}},\cO(m))$ is an
isomorphism, and
\[
 \Hilbs\widetilde\mfS^{\mathrm{symb}}_{n,m}=q^{mn^2}P_{n,m}(q,t),
 \qquad
 \chf_{q,t}\,\mfS^{\mathrm{symb}}_{n,m}=P_{n,m}(q,t).
\]
In particular $\dim\mfS^{\mathrm{symb}}_{n,m}=\binom{n(m+1)}n$ for
$m\ge m_0(n)$.
\end{theorem}

\begin{proof}
By \eqref{eq:corrected-parity} and \cref{thm:SD},
$P_{n,m}=\chi^T(\cZ_n^{\mathrm{symb}},\cO(m))$ for every $m$.  The morphism
$\pi_Q$ is projective with $\cO(1)$ relatively ample
(\cref{prop:models-identified}).  Serre vanishing on the projective
central fiber gives the cohomology vanishing, and relative Serre vanishing
$R^{>0}\pi_{Q*}\cO(m)=0$ ($m\gg0$) with cohomology-and-base-change at
$0\in S_n$ identifies $\pi_{Q*}\cO(m)\otimes k(0)$, which is
$\widetilde\mfS^{\mathrm{symb}}_{n,m}$ by \eqref{eq:sections}, with
$H^0$ of the central fiber.  Hence
$P_{n,m}=\chf_TH^0(\cZ,\cO(m))=q^{-mn^2}\Hilbs\widetilde\mfS_{n,m}$, where the
normalization shift is \eqref{eq:two-normalizations} (the equivariant structure of
$\cO(1)$ normalizes the generator $\Delta$ to degree zero).  The dimension is
\cref{thm:principal} at $q=t=1$.  The bound $m_0(n)$ supplied by Serre vanishing is
not effective.
\end{proof}

\section{The Cherednik identification}\label{sec:cherednik}

In this section we prove Theorem~D.  The comparison between the symbolic model and
the spherical Cherednik module $eL_{\frac1{2n}+m}(\triv)$ is made through the shift
bimodules of the rational Cherednik algebra.  Kivinen's associated-graded theorem
identifies the associated graded of the iterated shift bimodule with
$e(\Delta^mI^{(m)})$, and the Procesi bundle theory of Losev and of Boixeda
Alvarez--Losev, transported to the wall variety $Q_n$ by \cref{thm:polarization},
produces a canonical bigraded surjection from $\mfS^{\mathrm{symb}}_{n,m}$ onto
$\gr_F\,eL_{\frac1{2n}+m}(\triv)$.

The main idea is a dimension count.  The target has dimension $\binom{n(m+1)}{n}$
by the theorem of Berest, Etingof and Ginzburg on the finite-dimensional
representations, which we make explicit in type $B_n$ in \cref{thm:typeB-dim}, and
the source has the same dimension for $m\gg0$ by \cref{thm:character}.  A surjection
between spaces of equal dimension is an isomorphism.  The first subsection fixes the
bigrading and checks that the parameters are spherical, and the remaining
subsections assemble the surjection.

\subsection{Recollections and the bigrading}\label{sec:cherednik-recollections}

Let $H_c=H_c(W(B_n))$ be the rational Cherednik algebra at constant parameter
$c$ (equal parameters on the two conjugacy classes of reflections), $U_c=eH_ce$.
Its Euler element $\mathbf h$ satisfies $[\mathbf h,x]=x$, $[\mathbf h,y]=-y$,
$[\mathbf h,w]=0$ for $x\in\frh^*$, $y\in\frh$, $w\in W$, giving the Euler grading
$\degh(x)=1$, $\degh(y)=-1$, $\degh(w)=0$.  Two filtrations occur.  The PBW
total-degree filtration has $\degt(x)=\degt(y)=1$, $\degt(w)=0$, with
$\gr^{\mathrm{tot}}H_c\cong S\#W$, and the order filtration has
$\ord(x)=\ord(w)=0$, $\ord(y)=1$, and is the ordinary order filtration after the
Dunkl embedding, with $\gr^{\mathrm{ord}}U_c\cong A$.  For a PBW monomial
$x^\alpha wy^\beta$ one has $\degt=|\alpha|+|\beta|$, $\degh=|\alpha|-|\beta|$,
$\ord=|\beta|$, hence
\begin{equation}\label{eq:keyrelation}
 \degt=\degh+2\ord,
 \qquad
 F^{\mathrm{ord}}_bH_c[\lambda]=F^{\mathrm{tot}}_{\lambda+2b}H_c[\lambda]
\end{equation}
on the Euler eigenspace of eigenvalue $\lambda$ (the PBW basis gives equality of the
two filtered subspaces).

\begin{definition}[Gordon--Stump bidegree]\label{def:bigrading}
Let $M$ be a filtered $H_c$-module whose filtration is stable under $\mathbf h$.  A
class in $\gr M$ has bidegree $(a,b)$ if it has filtration degree $a+b$ and Euler
degree $a-b$, equivalently $a=\frac{\degt+\degh}2$, $b=\frac{\degt-\degh}2$.  Under
the PBW isomorphism, $\deg x=(1,0)$, $\deg y=(0,1)$, $\deg w=(0,0)$, which is the
convention of \cite[\S4.2.2]{Stump}.  For $L^{(m)}:=L_{\frac1h+m}(\triv)$, the
filtration $F$ is the \emph{iterated tensor-product filtration}, which is defined as
follows.  Writing $H^{(j)}=H_{\frac1h+j}$, the shift isomorphisms give
\begin{equation}\label{eq:iteratedshift}
 eL^{(m)}\cong eH^{(m)}e_-\otimes_{eH^{(m-1)}e}eH^{(m-1)}e_-\otimes\cdots\otimes
 eH^{(1)}e_-\otimes_{eH^{(0)}e}\C ,
\end{equation}
and $F$ is defined recursively by $F_s(A\otimes_RB)=\sum_{i+j=s}F_iA\otimes_RF_jB$,
starting from the PBW filtrations on the shift factors and placing $\C$ in degree
zero.  The bidegree-$(a,b)$ piece of $\gr^G eL^{(m)}$ is
$\gr^F_{a+b}(eL^{(m)}[a-b])$.  This is the bigrading for which Stump proves the
surjection $DR^{(m)}(W)\otimes\eps\twoheadrightarrow\gr^GL^{(m)}$
\cite[Thm.~4.5]{Stump}.
\end{definition}

By \eqref{eq:keyrelation}, on a single filtered Cherednik factor the order and
total-degree filtrations determine each other in each Euler eigenspace.  This is the
bridge between Kivinen's associated-graded theorem, stated for the order filtration,
and the Gordon--Stump bidegree.  The bidegree obtained in this way is intrinsic.  By
\cite[Rem.~3.6]{BoixedaLosev} the shift bimodules carry a $(\C^\times)^2$-action, the
two factors scaling $\frh$ and $\frh^*$, and in the associated graded the Dunkl
operator $D_y$ becomes the coordinate $y$ \cite[proof of Thm.~4.3]{Kivinen}.  Hence the
symbol of $x^\alpha D^\beta\Delta^m$ is $x^\alpha y^\beta\Delta^m$, of bidegree
$(|\alpha|+mn^2,|\beta|)$, and the Gordon normalization places $\Delta^m$ in
bidegree $(0,0)$, which is the shift $q^{mn^2}$ of \eqref{eq:two-normalizations}.

Set $c_0=\frac1h=\frac1{2n}$.  The module $L_{c_0}(\triv)$ is one-dimensional
\cite[Thm.~1.4, Prop.~1.7]{BEG}, and $L_0:=eL_{c_0}(\triv)$ has $\gr L_0\cong A/\mfm$.

\begin{proposition}[Sphericity at the parameters used]\label{prop:sphericity}
For $c=m+\frac1h$ the parameter is $e$-spherical for $m\ge0$ and $e_-$-spherical for
$m>0$.
\end{proposition}
\begin{proof}
\cite[Prop.~2.12]{BoixedaLosev}, quoted there from \cite[Lemma~4.5]{Gordon}.
Independently, for $W(B_n)=G(2,1,n)$ the classification of aspherical parameters
\cite[Thm.~3.4]{DunklGriffeth}, applied with $(c_0,d_0,d_1)=(c,c,-c)$, shows that
$\frac1{2n}+j$ avoids the aspherical locus for all $j\ge0$.  ``Constant'' is a real
hypothesis in type $B_n$, since a parameter is a pair
$(c_{\mathrm{short}},c_{\mathrm{long}})$, and we are on the diagonal.
\end{proof}

\subsection{Kivinen's associated-graded theorem}

Following Kivinen \cite[\S4]{Kivinen}, define the shift bimodules
$Q^{c+1}_c=eH_{c+1}e_-\Delta$ and their iterates
$B_{m0}=Q^{c_0+m}_{c_0+m-1}\cdots Q^{c_0+1}_{c_0}$, a
$(U_{c_0+m},U_{c_0})$-bimodule, with the order filtration.

\begin{theorem}[{Kivinen \cite[Thm.~4.3(2), Cor.~4.5(2)]{Kivinen}}]\label{thm:kivinen}
Let $c$ be constant and spherical.  Then, for $N_k=B_{k0}eH_c$,
$\ogr N_k=I^{(k)}\Delta^k$ as graded submodules of $S$, and $\gr B_c$ is the $\Z$-algebra of $\bigoplus_de(\Delta^dI^{(d)})$.  In particular
$\gr_FB_{m0}\cong e(\Delta^mI^{(m)})$.
\end{theorem}

Both hypotheses hold at our parameters $c=m+\frac1{2n}$ by \cref{prop:sphericity}.
In type $A$, where $I^{(m)}=J^m$, this recovers \cite[Prop.~6.5]{GS1}.

\subsection{The Procesi package on the wall}\label{sec:procesi}

The shift bimodules of Boixeda Alvarez--Losev live on a Procesi sheaf over the wall
variety, and we assemble the one we need from Losev's classification.

Write $W_n=W(B_n)$, $H=S\#W_n$ for the skew group algebra,
$e=\frac1{|W_n|}\sum_{w\in W_n}w$ for the spherical idempotent and
$f=\frac1{|W_{n-1}|}\sum_{w\in W_{n-1}}w$ for the parabolic averaging idempotent
of $W_{n-1}=\mu_2\wr S_{n-1}\subset W_n$.

\begin{convention}[Side of the idempotents]\label{conv:side}
$H$ acts on Procesi sheaves on the \emph{left}, so that taking invariants under a
subgroup is left multiplication by the averaging idempotent, that is, $e\sP=\sP^{W_n}$
and $f\sP=\sP^{W_{n-1}}$.  This is the convention of \cite{BoixedaLosev}, who write
$\eps\widetilde\sP\cong\cO_{\widetilde X}$.
\end{convention}

Losev \cite[Thm.~1.1]{LosevProcesi} classifies the normalized Procesi bundles
$\widetilde\sP$ on a smooth Hamiltonian-reduction resolution, that is, those with
$\End(\widetilde\sP)\cong H$, $\Ext^{>0}(\widetilde\sP,\widetilde\sP)=0$ and
$e\widetilde\sP\cong\cO$, and proves the following theorem.

\begin{theorem}[{Losev \cite[Thm.~1.2]{LosevProcesi}}]\label{thm:losev}
For every generic stability chamber there is a normalized Procesi bundle
$\widetilde\sP_n$ on $Y_n$ with $\widetilde\sP_n^{W_{n-1}}\cong\cT^\theta$, the
tautological bundle of the wreath-product quiver variety.
\end{theorem}

For the two-vertex McKay quiver of $\mu_2$ the tautological bundle is the cluster
bundle $\cP_n=\cR_0\oplus\cR_1$, hence in \cref{conv:side}
\begin{equation}\label{eq:parabolic-upstairs}
 f\widetilde\sP_n\cong\cP_n .
\end{equation}
This fixes the choice between the tautological and dual-tautological normalizations.

Boixeda Alvarez--Losev \cite[\S2.2]{BoixedaLosev} call a maximal Cohen--Macaulay
sheaf $\widetilde\sP$ on a symplectic variety $Z$ over $S_n=\SpecOp A$ a
\emph{Procesi sheaf} if (i) $\End(\widetilde\sP)\cong H$ as graded algebras, (ii)
$\End(\widetilde\sP)$ is maximal Cohen--Macaulay, (iii)
$H^i(Z,\End(\widetilde\sP))=0$ for $i>0$, and (iv)
$e\widetilde\sP\cong\cO_Z$, $\C^\times$-equivariantly.  All four can be taken
$(\C^\times)^2$-equivariant, which is what makes the package compatible with the
$q,t$-bigrading.  Their Lemma~2.3 states that for their partial resolution
$\bar\rho:\widetilde X\to X$ and a Procesi sheaf $\widetilde\sP$ on $\widetilde X$,
the pushforward $\bar\rho_*\widetilde\sP$ is a Procesi sheaf on $X$, maximal
Cohen--Macaulay, with the isomorphism
$\End(\widetilde\sP)\to\End(\bar\rho_*\widetilde\sP)$ produced in the proof.  By
\cref{thm:polarization}, in type $B_n$ their $X$ is $Q_n$ and $\bar\rho=\rho$.

\begin{theorem}[Procesi package on the type-$B_n$ wall]\label{thm:package}
Let $\widetilde\sP_n$ be the normalized Procesi bundle of \cref{thm:losev}, and set
\[
  \sP_n:=\rho_*\widetilde\sP_n,\qquad\sA_n:=\cEnd_{Q_n}(\sP_n).
\]
Then the following hold.
\begin{enumerate}[label=\textup{(P\arabic*)}]
\item\label{P1} $\sP_n$ and $\sA_n$ are $(\C^\times)^2$-equivariant maximal
Cohen--Macaulay sheaves, and $\Gamma(Q_n,\sA_n)\cong S\#W_n$ as bigraded
$A$-algebras, compatibly with the central map $A\to\Gamma(Q_n,\sA_n)$ induced by
$\pi_Q$;
\item\label{P2} $e\sP_n\cong\cO_{Q_n}$;
\item\label{P3} $f\sP_n\cong\rho_*\cP_n$;
\item\label{P4} $H^i(Q_n,\sA_n)=0$ for every $i>0$.
\end{enumerate}
Every idempotent summand of $\sP_n$ is maximal Cohen--Macaulay.
\end{theorem}

\begin{proof}
By \cref{thm:polarization}, $\rho:Y_n\to Q_n$ is the morphism
$\bar\rho:\widetilde X\to X$ of \cite{BoixedaLosev} after choosing the adjacent generic
chamber whose smooth quiver variety is $Y_n$.  Their Lemma~2.3 therefore applies to
$\widetilde\sP_n$ and gives \ref{P1}, \ref{P2}, \ref{P4}.  For \ref{P3}, note that $f$ is an
idempotent acting on $\widetilde\sP_n$, hence
$\widetilde\sP_n=f\widetilde\sP_n\oplus(1-f)\widetilde\sP_n$.  The direct-image functor
is additive, hence commutes with the image of the idempotent, and
\eqref{eq:parabolic-upstairs} gives $f\sP_n=f\rho_*\widetilde\sP_n\cong\rho_*(f\widetilde\sP_n)
\cong\rho_*\cP_n$.  No vanishing of higher direct images is used in this identity.
A direct summand of a maximal Cohen--Macaulay sheaf is maximal Cohen--Macaulay.
\end{proof}

\subsection{The shift bimodule and the canonical quotient}

Let $B_m$ be the classical Procesi shift bimodule of \cite[\S3]{BoixedaLosev},
\[
 B_m=\Gamma\bigl(Q_n^{\mathrm{reg}},(\sP_n^{\mathrm{reg}})^\vee\otimes\cO^{\mathrm{reg}}(m)
 \otimes\sP_n^{\mathrm{reg}}\bigr),
\]
an $(H,H)$-bimodule, which by \cite[Prop.~3.1]{BoixedaLosev} (applicable since
$\cO_{Q_n}(1)$ is relatively ample and $\codim_{Q_n}Q_n^{\mathrm{sing}}=2k(k-1)\ge4$
on the $\Gr(k,2k-1)$ strata, by semismallness) extends across the singular locus as
a maximal Cohen--Macaulay sheaf with no higher cohomology.  Set
$M_m:=eB_m$ and $P_m:=eB_me$.  By \ref{P2} both corners are computed by
$e\sP_n^{\mathrm{reg}}\cong\cO$, hence $P_m\cong\Gamma(Q_n^{\mathrm{reg}},\cO(m))$, and
\eqref{eq:sections} with normality gives
\begin{equation}\label{eq:Pm-symbolic}
 P_m\cong e\bigl(\Delta^mI^{(m)}\bigr),
\end{equation}
the $e$--$e$ corner of \cref{thm:kivinen}.  Boixeda Alvarez--Losev construct a
$\C[\hbar]$-flat Rees deformation $B_{\hbar,c_0+m\leftarrow c_0}$ of $B_m$ whose
$\hbar=1$ fiber $B_{c_0+m\leftarrow c_0}$ is an $(H_{c_0+m},H_{c_0})$-bimodule with
$\gr B_{c_0+m\leftarrow c_0}=B_m$.  The bimodule $B_{c_0+m\leftarrow c_0}$ is built from the shift bimodules $eH_{c+1}e_-$
\cite[\S3, Lemma~3.7]{BoixedaLosev}, which in the Dunkl realization are Kivinen's
$Q^{c+1}_c$ up to multiplication by $\Delta$.  Hence $eB_{c_0+m\leftarrow c_0}e$ is
$B_{m0}$ up to the Vandermonde normalization, and its order-associated graded is
$P_m$.

Let $\mathfrak n=S_+$ and $\C_{\triv}$ the one-dimensional $H$-module on which
$S$ acts through $S/\mathfrak n$ and $W$ trivially, hence $H/H\mathfrak n\cong\C W$.

\begin{definition}\label{def:two-specializations}
The \emph{central symbolic specialization} is
$\mfS^{\mathrm{cent}}_{n,m}:=P_m/\mfm_0P_m\cong\widetilde\mfS^{\mathrm{symb}}_{n,m}$,
and the \emph{augmentation specialization} is
$\mfS^{\mathrm{aug}}_{n,m}:=M_m\otimes_H\C_{\triv}$.
\end{definition}

\begin{proposition}[Canonical quotient]\label{prop:canonical-quotient}
There is a canonical bigraded surjection
\begin{equation}\label{eq:canonical-quotient}
 \pi_{n,m}:\mfS^{\mathrm{cent}}_{n,m}\twoheadrightarrow\mfS^{\mathrm{aug}}_{n,m},
 \qquad
 \mfS^{\mathrm{aug}}_{n,m}\cong\frac{P_m}{(M_m\mathfrak n)e},\qquad
 \ker\pi_{n,m}=\frac{(M_m\mathfrak n)e}{\mfm_0P_m}.
\end{equation}
\end{proposition}
\begin{proof}
$M_m\otimes_H\C_{\triv}\cong(M_m/M_m\mathfrak n)\otimes_{\C W}\C_{\triv}$, and for
every right $\C W$-module $N$ averaging gives $N\otimes_{\C W}\C_{\triv}\xrightarrow{\sim}Ne$,
$v\otimes1\mapsto ve$.  Since $M_m\mathfrak n$ is $W$-stable,
$(M_m/M_m\mathfrak n)e\cong P_m/(M_m\mathfrak n)e$.  Finally $\mfm_0\subset\mathfrak n$
is central, hence $\mfm_0P_m\subseteq(M_m\mathfrak n)e$.
\end{proof}

\begin{theorem}[{Boixeda Alvarez--Losev \cite[Thm.~1.2, Prop.~6.6, Prop.~2.11]{BoixedaLosev}}]
\label{thm:BAL-aug}
Let $\mathfrak g$ be simple with Weyl group $W$ and Coxeter number $h$, and $m\ge0$.
The tensor-product filtration gives a surjection
\begin{equation}\label{eq:BAL-surjection}
 B_m\otimes_H\C_{\triv}\twoheadrightarrow\gr_FL_{m+1/h}(\triv),
\end{equation}
and $L_{m+1/h}(\triv)$ is the unique finite-dimensional simple $H_{m+1/h}$-module, of
dimension $(mh+1)^{\dim\frh}$ and isomorphic as a $W$-module to
$\C[\Lambda_0/(mh+1)\Lambda_0]$, $\Lambda_0$ the coroot lattice.
\end{theorem}
\begin{proof}[Source]
The dimension and $W$-module structure are \cite[Prop.~2.11]{BoixedaLosev}, quoted
from \cite[Thm.~1.4, Prop.~1.7]{BEG}.  The surjection is the proof of
\cite[Prop.~6.6]{BoixedaLosev}.  Flatness gives $\gr B_{m+1/h\leftarrow1/h}=B_m$, hence
a surjection $B_m\otimes_H\C_{\triv}\twoheadrightarrow\gr(B_{m+1/h\leftarrow1/h}\otimes_{H_{1/h}}L_{1/h})$,
and $B_{m+1/h\leftarrow1/h}$ is a Morita equivalence bimodule (by induction on
$m$ using \cref{prop:sphericity} and
$e_-B_{\hbar,c+m+1\leftarrow c}\cong eB_{\hbar,c+m\leftarrow c}$
\cite[Lem.~3.7(2)]{BoixedaLosev}), hence
$B_{m+1/h\leftarrow1/h}\otimes_{H_{1/h}}L_{1/h}\cong L_{m+1/h}$ by uniqueness.  (Their
Theorem~1.2 upgrades \eqref{eq:BAL-surjection} to an isomorphism using affine Springer
fibers, but we do not need this, since the dimension count below forces it.)
\end{proof}

\begin{lemma}[Spherical tensor identity]\label{lem:spherical-tensor}
Let $c=\frac1h$ and $m\ge0$.  Then, compatibly with the tensor-product filtrations,
\[
 eL_{c+m}\cong\bigl(eB_{c+m\leftarrow c}e\bigr)\otimes_{U_c}\bigl(eL_c\bigr),\qquad\dim eL_c=1 .
\]
Consequently $F_b(eL_{c+m})$ is the image of $F_b(eB_{c+m\leftarrow c}e)$.
\end{lemma}
\begin{proof}
By the proof of \cite[Prop.~6.6]{BoixedaLosev}, $B\otimes_{H_c}L_c\cong L_{c+m}$
for $B=B_{c+m\leftarrow c}$.  The module $L_c$ is the trivial one-dimensional representation of
$W$ on which $\frh,\frh^*$ act by zero, hence $eL_c=L_c$, and $c$ is $e$-spherical
(\cref{prop:sphericity}), hence $L_c\cong H_ce\otimes_{U_c}eL_c$.  Therefore
$eL_{c+m}=e(B\otimes_{H_c}L_c)=(eB)\otimes_{H_c}(H_ce\otimes_{U_c}eL_c)=(eBe)\otimes_{U_c}(eL_c)$,
all steps filtered, the filtration on $eL_c$ being trivial.
\end{proof}

Taking the left spherical part of \eqref{eq:BAL-surjection} and composing with
\cref{prop:canonical-quotient} gives a canonical bigraded surjection
\begin{equation}\label{eq:symbolic-to-Cherednik}
 \Phi_{n,m}:\ \mfS^{\mathrm{cent}}_{n,m}=\frac{e(\Delta^mI^{(m)})}{\mfm_0\,e(\Delta^mI^{(m)})}
 \twoheadrightarrow\mfS^{\mathrm{aug}}_{n,m}\cong\gr_FeL_{m+1/h}(\triv),
\end{equation}
where the last isomorphism is \eqref{eq:BAL-surjection} restricted to spherical parts
and \cref{lem:spherical-tensor}, and bigradedness is \eqref{eq:keyrelation} together
with the intrinsic bigrading of the shift bimodule.  In the normalization
\eqref{eq:two-normalizations} this is a bigraded surjection
$\mfS^{\mathrm{symb}}_{n,m}\twoheadrightarrow\gr_FeL_{\frac1{2n}+m}(\triv)$.

\subsection{The type-\texorpdfstring{$B_n$}{Bn} dimension}

For $W(B_n)$, $h=2n$ and $N_m:=mh+1=2nm+1$ is odd.  The coroot lattice is
$\Lambda_0=\{a\in\Z^n:\sum a_i\ \text{even}\}$ (coroots $\pm e_i\pm e_j$ and $2e_i$),
and inclusion into $\Z^n$ induces a $W(B_n)$-equivariant isomorphism
$\Lambda_0/N_m\Lambda_0\xrightarrow{\sim}(\Z/N_m)^n$.  Indeed, the index $[\Z^n:\Lambda_0]=2$ is
invertible modulo the odd $N_m$, hence one of $b$, $b+N_me_1$ lies in $\Lambda_0$ for
every $b\in\Z^n$, and $N_mb\in\Lambda_0$ forces $\sum b_i$ even.  (Whether one reads
$\Lambda_0$ as the root or the coroot lattice, for $\mathfrak{so}_{2n+1}$ or
$\mathfrak{sp}_{2n}$, all readings give the same $W$-module because $N_m$ is odd.)

\begin{theorem}[Type-$B_n$ target dimension]\label{thm:typeB-dim}
For every $m\ge0$,
\[
 \dim\gr_FeL_{\frac1{2n}+m}(\triv)=\binom{n(m+1)}n
 =\prod_{i=1}^n\frac{2i+2nm}{2i},
\]
with an ungraded basis indexed by weakly increasing sequences
$0\le a_1\le\dots\le a_n\le nm$, i.e.\ by partitions in an $n\times nm$ rectangle.
\end{theorem}
\begin{proof}
By \cref{thm:BAL-aug} the spherical dimension is the number of $W(B_n)$-orbits on
$(\Z/N_m)^n$.  The sign orbits in one coordinate are $\{0\},\{\pm1\},\dots,\{\pm nm\}$,
hence there are $nm+1$ coordinate types, and a signed-permutation orbit is a multiset of
size $n$ chosen from these, which gives $\binom{(nm+1)+n-1}n=\binom{n(m+1)}n$.  Orbit sums give the
basis.  The product form is \cite[Thm.~4.3]{Stump}.
\end{proof}

\subsection{The asymptotic isomorphism}

\begin{theorem}\label{thm:cherednik}
Theorem D holds.  For $m\ge m_0(n)$, the canonical
surjection \eqref{eq:symbolic-to-Cherednik} is an isomorphism of bigraded vector
spaces for the iterated tensor-product filtration of \cref{def:bigrading},
\[
 \mfS^{\mathrm{symb}}_{n,m}\xrightarrow{\ \sim\ }\gr_F\,eL_{\frac1{2n}+m}(\triv),
 \qquad
 P_{n,m}(q,t)=\chf^G_{q,t}\gr_F\,eL_{\frac1{2n}+m}(\triv),
\]
and $(M_m\mathfrak n)e=\mfm_0P_m$, that is, taking the spherical corner commutes with
specialization at the origin in the Serre range.
\end{theorem}
\begin{proof}
By \cref{thm:character} the source has dimension $\binom{n(m+1)}n$, and by
\cref{thm:typeB-dim} the target has the same dimension.  A surjection of
finite-dimensional spaces of equal dimension is an isomorphism, and it is bigraded
since all constructions preserve the tensor-product filtration and the Euler
grading.  The character identity is \cref{thm:character}, and the last assertion is
the vanishing of the kernel in
\eqref{eq:canonical-quotient}.
\end{proof}

\begin{remark}[The two filtrations at higher rungs]\label{rem:strictness}
At $m=1$ there is a single shift factor in \eqref{eq:iteratedshift}, hence the
tensor-product filtration coincides with the ambient order filtration of the
already-multiplied Dunkl operators, and the bidegree-$(a,b)$ piece of $\gr^GeL^{(1)}$
is $\cF_b(n^2+a-b)/\cF_{b-1}(n^2+a-b)$ with
$\cF_b(r)=\bigl(e\Span\{x^\alpha D^\beta\Delta:|\beta|\le b\}\bigr)_{\degpol=r}$.
The parity $\degt\equiv\degh\pmod2$ of \eqref{eq:keyrelation} is what makes the
quotient by $F_{b-1}$ the correct total-degree piece.  For $m\ge2$ multiplication of
shift factors is filtered but need not be strict ($\ord(PQ)\le\ord P+\ord Q$), hence the
ambient order of a product may be smaller than its tensor degree.  The intrinsic
bigrading is the tensor-product one of \cref{def:bigrading}, and a reconstruction from
the ambient order of $x^\alpha D^\beta\Delta^m$ is not part of the convention.  This
is visible already at $n=2$, $m=2$.
\end{remark}

\section{The spherical collapse and Stump's conjecture}\label{sec:collapse}

In this section we prove Theorems~A and~B.  The collapse $e_mJ^m=e_mI^m=e_mI^{(m)}$
for $m\gg0$ is deduced from Theorem~D by a Nakayama argument.  Stump's generalized
Gordon surjection, from the minimal generating space of $J^m$ onto
$\gr_FL_{\frac1h+m}(\triv)$, is compatible with the surjection of
\cref{sec:cherednik}, since both are induced by the same filtered multiplication of
shift bimodules.  Once the second surjection is known to be an isomorphism, the first
forces $e_mJ^m$ to generate $e_mI^{(m)}$ modulo $\mfm$, and graded Nakayama gives the
equality.  Theorem~B then follows by comparing the two punctual quotients of
\cref{sec:two-quotients}.  We close the section with a list of open problems.

\subsection{The collapse}

\begin{theorem}[{Stump \cite[Thm.~4.5 with \S4.2.1]{Stump}}]\label{thm:stump-surjection}
The filtered iterated-shift construction \eqref{eq:iteratedshift} induces a bigraded
surjection from the determinant component of the generalized diagonal coinvariants
(equivalently, by \cite[Thm.~3.5]{Stump}, from the minimal generating space of
$J^m$, up to the twist $\eps^{\otimes(1-m)}$) onto $\gr_FL_{\frac1h+m}(\triv)$.
Its spherical part is a surjection $\overline A_m\twoheadrightarrow T_m$, where
$\overline A_m=e_mJ^m/\mfm\,e_mJ^m$ and $T_m=\gr_FeL_{c_0+m}(\triv)$.
\end{theorem}

\begin{theorem}\label{thm:collapse}
Theorem A holds.  For $m\ge m_1(n)$ we have
$e_mJ^m=e_mI^m=e_mI^{(m)}$, and the three spherical models
\eqref{eq:three-models} coincide in type $B_n/C_n$.
\end{theorem}

\begin{proof}
Set $A_m=e_mJ^m\subseteq D_m=e_mI^m\subseteq C_m=e_mI^{(m)}$, finite graded
$A$-modules ($S$ is finite over $A$), and $\overline A_m=A_m/\mfm A_m$,
$\overline C_m=C_m/\mfm C_m$, $T_m=\gr_FeL_{c_0+m}(\triv)$.

\emph{Two maps.}  The map $\phi^{\mathrm{symb}}_m:\overline C_m\twoheadrightarrow T_m$ is
\eqref{eq:symbolic-to-Cherednik}.  The map
$\phi^{J}_m:\overline A_m\twoheadrightarrow T_m$ is Stump's generalized Gordon
surjection, \cref{thm:stump-surjection}, that is, the same filtered iterated-shift
construction applied to the subalgebra generated by products of alternants.  Its
spherical part is $\overline A_m$, since products of $m$ alternating polynomials
span $e_mJ^m$ over $A$ ($e_m(fa_1\cdots a_m)=e(f)\,a_1\cdots a_m$).

\emph{Compatibility.}  The diagram
\[
\begin{tikzcd}[column sep=large]
\overline A_m\arrow[r,"\overline\iota_m"]\arrow[dr,"\phi^J_m"']
 &\overline C_m\arrow[d,"\phi^{\mathrm{symb}}_m"]\\
 &T_m
\end{tikzcd}
\]
commutes, because both maps are induced by the associated graded of one and the same
filtered multiplication
$Q^{c_0+m}_{c_0+m-1}\otimes\cdots\otimes Q^{c_0+1}_{c_0}\otimes eL_{c_0}\to
eL_{c_0+m}$.  Under \cref{thm:kivinen},
$A_m\subseteq C_m=\gr B_{m0}$ is the submodule generated by the products of the
distinguished alternating symbols, and $\phi^J_m$ is by construction the
restriction of $\phi^{\mathrm{symb}}_m$.

\emph{Nakayama.}  For $m\ge m_0(n)$, $\phi^{\mathrm{symb}}_m$ is an isomorphism
(\cref{thm:cherednik}).  Surjectivity of $\phi^J_m$ then forces
$\overline\iota_m$ surjective, i.e.\ $C_m=A_m+\mfm C_m$, and graded Nakayama on the
finite $A$-module $C_m/A_m$ gives $A_m=C_m$, hence also $=D_m$.  For the Proj
statement compare $r$-th Veronese subalgebras for any $r\ge m_1(n)$, noting
$(\Delta S^\eps)^m=e(\Delta^mJ^m)$ and the sandwich
$e(\Delta^mJ^m)\subseteq(e(\Delta I))^m\subseteq e(\Delta^mI^m)\subseteq
e(\Delta^mI^{(m)})$, all equal in the range, and relative Proj is insensitive to
Veronese.
\end{proof}

\begin{remark}
The collapse is a statement about the \emph{spherical} pieces only.  Kivinen's
type-$B_3$ computation \cite[Ex.~B.2]{Kivinen} shows $J\subsetneq I$ while $eJ=eI$,
and the isospectral models remain distinct.
\end{remark}

\subsection{Stump's quotient versus the invariant quotient, and the proof of Theorem B}

\begin{theorem}\label{thm:qtCat}
Theorem B holds, and moreover
$\pi_m:\mathsf Q^{\mathrm{inv}}_m(J^m)\twoheadrightarrow\mathsf Q^{\mathrm{full}}_m(J^m)$
of \eqref{eq:two-quotient-surj} is an isomorphism for $m\gg0$.  In particular Stump's
conjecture $\dim e_\eps DR^{(m)}(B_n)=\binom{n(m+1)}n$ holds for all $m\gg0$.
\end{theorem}

\begin{proof}
Let $m$ be in the ranges of \cref{thm:character,thm:collapse}.  Then
$\dim\mathsf Q^{\mathrm{inv}}_m(J^m)=\dim\mfS^{\mathrm{symb}}_{n,m}=\binom{n(m+1)}n$ by
the collapse and \cref{thm:character}, hence
$\dim\mathsf Q^{\mathrm{full}}_m(J^m)\le\binom{n(m+1)}n$.  On the other hand Stump's
surjection (\cref{thm:stump-surjection}, with target of dimension
$\binom{n(m+1)}n$ by \cref{thm:typeB-dim}) factors through
$\mathsf Q^{\mathrm{full}}_m(J^m)$, giving $\dim\mathsf Q^{\mathrm{full}}_m(J^m)\ge\binom{n(m+1)}n$.
Hence equality holds, $\pi_m$ is an isomorphism, and
\[
 \mathrm{Cat}^{(m)}(B_n;q,t)=\Hilbs\mathsf Q^{\mathrm{full}}_m(J^m)
 =\Hilbs\mathsf Q^{\mathrm{inv}}_m(J^m)=\chf_{q,t}\mfS^{\mathrm{symb}}_{n,m}
 =P_{n,m}(q,t)
\]
by \cref{thm:character}.  (The determinant twist $\eps^{\otimes(1-m)}$ in
\eqref{eq:stump-def} does not affect the bigrading.)
\end{proof}

\begin{remark}[Type $A$ versus type $B$]
The binomial $\binom{n(m+1)}n=\prod_{i=1}^n\frac{2i+2nm}{2i}$ is the Fuss--Catalan
number of $W(B_n)$.  For the literal type-$A_{n-1}$ diagonal ideal
$\bigcap_{i<j}(x_i-x_j,y_i-y_j)^m$ the corresponding count is
$\frac1{mn+1}\binom{(m+1)n}n$, and the coordinate-root and sum-diagonal factors of
\eqref{eq:Bn-ideals} are responsible for the difference.  Compare
$F_{n,m}(1,1)=\frac{1}{nm+1}\binom{n(m+1)}{n}$ in Theorem C, the \emph{virtual}
(unprojected) count, which is precisely the type-$A$ Fuss--Catalan number.
\end{remark}

\subsection{Open problems}\label{sec:open-problems}

Let us put the results of the paper in the context of the conjectures of Haiman on
diagonal coinvariants.  For $W=S_n$ Haiman \cite{HaimanConj} conjectured, and later
proved \cite{HaimanNfact,Haiman}, that the diagonal coinvariant ring
$DR(W)=S/\mfm S$ has dimension $(n+1)^{n-1}$ and that its determinantal part
$e_\eps DR(W)$ has dimension the Catalan number $\frac1{n+1}\binom{2n}n$, with
bigraded Hilbert series the $q,t$-Catalan number of \cite{HaimanConj,GarsiaHaiman}.

The generalized diagonal coinvariants $DR^{(m)}(W)$ of Garsia and Haiman have
dimension $(mn+1)^{n-1}$, and their determinantal part has dimension the
Fuss--Catalan number $\frac1{mn+1}\binom{(m+1)n}n$ and Hilbert series the
$q,t$-Fuss--Catalan number \cite{Haiman}.  For a general Weyl group $W$ of rank $\ell$
and Coxeter number $h$, Haiman \cite[\S7]{HaimanConj} computed the dimension of
$DR(W)$ for $B_4$, $B_5$ and $D_4$, found it different from $(h+1)^\ell$ (it is
$9^4+1$ for $B_4$). Haiman conjectured that $DR(W)$ has a natural quotient of dimension
$(h+1)^\ell$ \cite[Conj.~7.1.2]{HaimanConj} whose twist by the sign character is the
permutation representation of $W$ on $Q/(h+1)Q$, $Q$ the root lattice.

The number of
$W$-orbits on $Q/(h+1)Q$ is the Catalan number
$\Cat(W)=\prod_{i=1}^{\ell}\frac{h+1+e_i}{1+e_i}$ of $W$ \cite[\S7]{HaimanConj}, which
is $\binom{2n}n$ for $W(B_n)$ (compare \cref{thm:typeB-dim}), hence the conjecture
predicts that the determinantal part of the quotient has dimension $\Cat(W)$.  Gordon
\cite{GordonDiag} constructed the quotient as $\gr L_{1+\frac1h}(\triv)$, which is the
case $m=1$ of the surjection of \cref{thm:stump-surjection}.  Haiman's conjecture on
the determinantal part of the diagonal coinvariants \cite[Conj.~7.2.5]{HaimanCDM},
recorded for all $m$ as \cite[Conj.~4.7]{Stump}, is that the kernel of the surjection
$DR^{(m)}(W)\otimes\eps\twoheadrightarrow\gr L_{m+\frac1h}(\triv)$ contains no copy of
the trivial representation, that is, that the determinantal part of the generalized
diagonal coinvariants maps isomorphically onto $e\,\gr L_{m+\frac1h}(\triv)$.

For
$W=W(B_n)$ this is the statement that the spherical part of the surjection of
\cref{thm:stump-surjection} is an isomorphism, and \cref{thm:qtCat} proves it for all
$m\gg0$, together with Conjectures~3.4 and~3.9 of \cite{Stump} on the dimension and on
the principal specialization of $\mathrm{Cat}^{(m)}(W;q,t)$.  The dimension of the
full diagonal coinvariant ring of $W(B_n)$ exceeds $(2n+1)^n$ for every $n\ge4$, by a
quantity which grows at least quadratically in $n$ \cite{AjilaGriffeth}, hence it is
only the determinantal part which retains the Catalan behaviour.

The polynomials $P_{n,m}$ of Theorem~B contain more information than the
Fuss--Catalan numbers, and the questions below concern the low twists, the
effectivity of the bounds, and the combinatorics of $P_{n,m}$ and of the
opposite-parity part $N_{n,m}=P_{n,m}-F_{n,m}$ of Theorem~C.

The conjectures and theorems of this paper lead to a collection of natural questions which is list below.

\begin{question}(All twists).  Do Theorems B and D hold for every $m\ge1$?  By
\cref{thm:character} this reduces to
$H^{>0}(\cZ_n^{\mathrm{symb}},\cO(m))=0$ and base change in the finitely many low
degrees.  A Cohen--Macaulay property of the punctual section ring
$\bigoplus_m\widetilde\mfS_{n,m}$ with $a$-invariant $\le0$ would give all
$m\ge1$ at once.
\end{question}

\begin{question}(Effectivity.)  Give an effective bound for
$m_0(n)$ (a Castelnuovo--Mumford regularity bound for
$\cZ_n^{\mathrm{symb}}$).
\end{question}

\begin{question}(The first rung.)  At $m=1$,
\eqref{eq:symbolic-to-Cherednik} still provides the surjection
$\mfS^{\mathrm{symb}}_{n,1}\twoheadrightarrow\gr eL_{\frac1{2n}+1}(\triv)$, whose
target has dimension $\binom{2n}n=\Cat(B_n)$.  Its
bijectivity for all $n$, equivalently
$H^{>0}$-vanishing at $m=1$, would give the exact
$q,t$-Fuss--Catalan character at the classical rung, and it would prove Haiman's
conjecture \cite[Conj.~7.2.5]{HaimanCDM} on the determinantal part of the diagonal
coinvariants of $W(B_n)$.  It is verified for
$n\le3$ by exact computation of the contravariant form on $L_{\frac1{2n}+1}(\triv)$,
in the bigrading of \cref{rem:strictness}.
\end{question}

\begin{question}(Combinatorics of $P_{n,m}$ and $N_{n,m}$.)  All polynomials in
\cref{ex:lowrank} are positive linear combinations of the strings $(qt)^a[k]_{q,t}$.
Stump \cite[Conj.~3.14]{Stump} conjectured that $\mathrm{Cat}^{(m)}(W;q,1)$ is the
coheight generating function of the positive regions of the extended Shi arrangement
of $W$, and in type $B_n$ the specialization $\mathrm{Cat}^{(1)}(B_n;1,t)$ is expected
to be the area generating function of the shifted Dyck paths of length $n$, that is,
of the lattice paths with unit north and east steps which start on the anti-diagonal,
end at $(n,n)$ and stay above the diagonal \cite[Conj.~1.17]{Lu}.

Is there a pair of
statistics on lattice paths which realizes $P_{n,m}(q,t)$, and is there a geometric
or representation-theoretic meaning of the opposite-parity part $N_{n,m}$?  A
conjectural answer to the first question is given in \cref{sec:shuffle}.  For $m=1$
the polynomials $N_{n,1}$ have appeared, up to the factor $qt$, in the thesis of
Lingxi Lu \cite{Lu}.  That thesis attaches to the shifted Dyck paths of length $n$
with base $k$, which start on the anti-diagonal $y=1-k-x$, polynomials $B_{n,k}(q,t)$
written as positive combinations of the strings $(qt)^a[l]_{q,t}$ and normalized by
$B_{n,1}(q,t)=\mathrm{Cat}^{(1)}(B_n;q,t)$ and by the two specializations
$B_{n,k}(1,t)=\sum_\pi t^{\mathrm{area}(\pi)}$ and
$q^{n(n+k-1)}B_{n,k}(q,q^{-1})=\qbinom{2n+k-1}{n}_{q^2}$, and it computes them for
$n\le3$.  One has
\[
 N_{n,1}(q,t)=qt\,B_{n-1,3}(q,t)\qquad(2\le n\le4),
\]
compare \cref{ex:lowrank} with \cite[Ch.~3, Ch.~4 and App.~C]{Lu}, and for every $n$
the specializations $N_{n,1}(1,1)=\binom{2n}{n-1}$ and
$N_{n,1}(q,q^{-1})=q^{1-n^2}\qbinom{2n}{n-1}_{q^2}$ of Theorem~C agree with those of
$qt\,B_{n-1,3}(q,t)$.  It would be interesting to know whether the equality holds for
all $n$ and what the base of the paths means on $Y_n$.
\end{question}

\section{Shuffle conjecture in type $B$}\label{sec:shuffle}

In this section we formulate a combinatorial model for the polynomials $P_{n,m}$
of Theorem~B, which answers, conjecturally, the first question of
\cref{sec:open-problems}.  The model is a sum over the $\binom{(m+1)n}{n}$ lattice
paths of the $n\times mn$ rectangle, and it uses three statistics of such a path,
the higher diagonal inversion number $\dinv_m$ of Loehr \cite{Loehr}, in the form of
Lee, Li and Loehr \cite{LLL}, and two area statistics $A^{(m)}_\mp$ which measure the excursions of the path below and
above the line $y=x/m$.  We prove that the resulting polynomial $\widehat P_{n,m}$
shares with $P_{n,m}$ the positivity, the $q\leftrightarrow t$ symmetry and the two
specializations of Theorem~C, we prove the equality $P_{n,m}=\widehat P_{n,m}$ in
ranks one and two for every $m$, and we verify it by exact computation in a range
of small ranks.

The main idea behind the name of the section is the following.  Let us write the
localization polynomial as a pairing in the ring of two-colour symmetric functions.
The wreath Macdonald polynomials $\wtH_\lambda$, $\lambda\in\Bal(2n)$, introduced by
Haiman \cite{HaimanCDM} and constructed by Bezrukavnikov and Finkelberg \cite{BF},
are the bigraded Frobenius characters of the fibers of the normalized Procesi bundle
of \cref{thm:losev} at the fixed points $I_\lambda$, and the fixed-point data of
\cref{lem:localization} are pairings of $\wtH_\lambda$ with explicit symmetric
functions.  The Cauchy identity for the $\wtH_\lambda$ then converts the fixed-point
sum \eqref{eq:F-def} into the pairing
\[
 P_{n,m}=(-1)^{n-1}\Bigl\langle\nabla_1^{m}\nabla_0\,
 p_n\Bigl[\tfrac{(1+qt)X_0+(q+t)X_1}{(1-q^2)(1-t^2)}\Bigr],\
 h_{n-1}[X_0]\,p_1\bigl[(1+qt)X_0-(q+t)X_1\bigr]\Bigr\rangle ,
\]
where $\nabla_0,\nabla_1$ are the two-colour nabla operators
(\cref{prop:plethystic}).  In type $A$ the pairing $\langle\nabla^me_n,e_n\rangle$ is
the sum over the $m$-Dyck paths of the $n\times mn$ rectangle of
$q^{\dinv_m}t^{\mathrm{area}}$, and the pairing $(-1)^{n-1}\langle\nabla p_n,e_n\rangle$
is a sum over all lattice paths of the $n\times n$ square, by the shuffle theorem of
Carlsson and Mellit \cite{CarlssonMellit}, conjectured in \cite{HHLRU}, its rational
version by Mellit \cite{Mellit} and the square paths theorem of Sergel \cite{Sergel}.
\Cref{conj:shuffleB} is the type-$B$ analogue of these statements at the Catalan
level, with a two-colour nabla applied to a power sum on the left and a sum over all
paths of the rectangle on the right.  One of the inputs of the pairing formula, the
evaluation of $\wtH_\lambda$ at the colour-zero power sum, is not available in the
literature for two colours, and we prove it in \cref{thm:hooks} from an
exterior-compatibility property of the normalized Procesi bundle, which we deduce
from Losev's formal restriction theorem.

Throughout the section $m\ge1$, $D:=mn^2$, $[x]_+:=\max(x,0)$, and
$[k]_{q,t}=\sum_{i+j=k-1}q^it^j$ as in \cref{ex:lowrank}, hence $[k]_{q,t}$ has $k$
terms and $[k]_{q,t}(q,q^{-1})=q^{1-k}[k]_{q^2}$.

\subsection{Paths in the rectangle and three statistics}\label{sec:path-stats}

A lattice path $\pi$ from $(0,0)$ to $(mn,n)$ with unit north and east steps is
determined by the numbers $e_1\le e_2\le\dots\le e_n$ of east steps which precede
its first, second, \dots, $n$-th north step, $0\le e_i\le mn$.  Respectively, $\pi$
is determined by the partition
\[
 \alpha(\pi)=(e_n,e_{n-1},\dots,e_1)\subseteq(mn)^n ,
\]
the partition with $n$ rows and parts at most $mn$ which fills the region between
the path and the $y$-axis, and $|\alpha(\pi)|=e_1+\dots+e_n$.  We set
\begin{equation}\label{eq:heights}
 b_i:=mi-e_i\qquad(1\le i\le n),
\end{equation}
the height of the path above the line $y=x/m$ after its $i$-th north step, measured
in units of $1/m$.  Let us define the kernel
\begin{equation}\label{eq:kernel}
 s_m(d)=\begin{cases} m+d,&1-m\le d\le0,\\ m+1-d,&1\le d\le m,\\ 0,&\text{otherwise},
 \end{cases}
\end{equation}
and the three statistics
\begin{equation}\label{eq:pathstats}
\begin{gathered}
 \dinv_m(\pi)=\sum_{i<j}s_m(b_i-b_j),\\
 A^{(m)}_-(\pi)=\sum_{i=1}^n\bigl([m-b_i]_++[-b_i]_+\bigr),\qquad
 A^{(m)}_+(\pi)=\sum_{i=1}^n\bigl([b_i]_++[b_i-m]_+\bigr).
\end{gathered}
\end{equation}
For $m=1$ the kernel is $s_1(d)=1$ for $d\in\{0,1\}$ and $s_1(d)=0$
otherwise, thus $\dinv_1$ is the usual diagonal inversion number of the sequence
$(b_1,\dots,b_n)$, and $A^{(1)}_-=\sum_{b_i\le0}(1-2b_i)$,
$A^{(1)}_+=\sum_{b_i>0}(2b_i-1)$.  The kernel $s_m$ is the one of the higher
$q,t$-Catalan numbers \cite{Loehr}, \cite[\S1(b)]{LLL}, where it is applied to the
$m$-Dyck paths, that is to the paths with all $b_i\ge0$.  Here it is applied to all paths of the
rectangle, and the two area statistics count a step at height $b$ in the strip
$0<b<m$ on both sides.

\begin{definition}\label{def:Phat}
For $n,m\ge1$ let
\begin{equation}\label{eq:Phat-path}
 \widehat P_{n,m}(q,t)=\sum_{\pi}(qt)^{\dinv_m(\pi)}\,q^{A^{(m)}_-(\pi)}\,t^{A^{(m)}_+(\pi)},
\end{equation}
the sum over all $\binom{(m+1)n}{n}$ lattice paths from $(0,0)$ to $(mn,n)$.
\end{definition}

The statistics have a second form, on the partition $\alpha=\alpha(\pi)$.  For a
box $s=(i,j)$ of $\alpha$ in row $i$ and column $j$, both counted from $1$ in
English notation, let $a_\alpha(s)=\alpha_i-j$ and $l_\alpha(s)=\alpha'_j-i$ be its
arm and leg, and put
\begin{equation}\label{eq:Hc}
 H_m(\alpha)=\#\{s\in\alpha:\ 0\le a_\alpha(s)-m\,l_\alpha(s)\le m\},\qquad
 c_{m,n}(\alpha)=\sum_{i=1}^n\bigl[\,2\bigl(\alpha_i-m(n+1-i)\bigr)-1\,\bigr]_+ .
\end{equation}
The hook statistic $H_m$ is the statistic $c_m$ of \cite[\S1(a)]{LLL}, applied here to
all partitions of the rectangle.  It does not depend on $n$, and the boundary
correction $c_{m,n}$ vanishes exactly on the partitions inside the dilated staircase
$\alpha_i\le m(n+1-i)$.

\begin{theorem}[Hook--path identity]\label{thm:hookpath}
For every path $\pi$ with partition $\alpha$ and heights $b$ as in
\eqref{eq:heights},
\begin{equation}\label{eq:HcA}
 H_m(\alpha)+c_{m,n}(\alpha)=\dinv_m(\pi)+A^{(m)}_-(\pi),\quad
 D-2|\alpha|+H_m(\alpha)+c_{m,n}(\alpha)=\dinv_m(\pi)+A^{(m)}_+(\pi).
\end{equation}
Consequently
\begin{equation}\label{eq:Phat-part}
 \widehat P_{n,m}(q,t)=\sum_{\alpha\subseteq(mn)^n}
 q^{\,H_m(\alpha)+c_{m,n}(\alpha)}\;t^{\,D-2|\alpha|+H_m(\alpha)+c_{m,n}(\alpha)} .
\end{equation}
\end{theorem}

\begin{proof}
Let us number the rows of $\alpha$ from the bottom, hence row $i$ has length
$L_i:=e_i=mi-b_i$ and $L_0:=0\le L_1\le\dots\le L_n$.  A box $(i,j)$ of row $i$ has
arm $L_i-j$, and its leg is the number of rows below row $i$ of length at least $j$,
which equals $i-1-k$ for $L_k<j\le L_{k+1}$, $0\le k\le i-1$.  For such a box the
hook condition $0\le\mathrm{arm}-m\,\mathrm{leg}\le m$ reads
\[
 0\le L_i-j-m(i-1-k)\le m\quad\Longleftrightarrow\quad mk-b_i\le j\le m(k+1)-b_i .
\]
For integers $A\le B+1$ and $C\le D'$ the number of integers $j$ with $A\le j\le B$
and $C\le j\le D'$ is $[B-C+1]_+-[B-D']_+-[A-C]_++[A-1-D']_+$, since
$\#\{j\in[C,D']:j\le X\}=[X-C+1]_+-[X-D']_+$ for every integer $X$, and one
subtracts the values at $X=B$ and $X=A-1$.  With $A=L_k+1$, $B=L_{k+1}$,
$C=mk-b_i$ and $D'=m(k+1)-b_i$ the number of hook boxes of row $i$ in the columns
$L_k<j\le L_{k+1}$ is therefore
\[
 [m+1+b_i-b_{k+1}]_+-[b_i-b_{k+1}]_+-[b_i-b_k+1]_++[b_i-b_k-m]_+ .
\]
We sum over $k=0,\dots,i-1$ and reindex the first two terms by $j=k+1\in[1,i]$ and
the last two by $j=k\in[0,i-1]$.  The indices $1\le j\le i-1$ occur in both groups
and contribute, with $d:=b_j-b_i$,
\[
 [m+1-d]_+-[-d]_+-[1-d]_++[-m-d]_+=s_m(d),
\]
as one checks on the four ranges $d\ge m+1$, $1\le d\le m$, $1-m\le d\le0$ and
$d\le-m$, where the left side equals $0$, $m+1-d$, $(m+1-d)+d-(1-d)=m+d$ and
$(m+1-d)+d-(1-d)+(-m-d)=0$.  The index $j=i$ of the first group contributes
$[m+1]_+-[0]_+=m+1$, and the index $j=0$ of the second group contributes
$-[b_i+1]_++[b_i-m]_+$, since $b_0=0$.  Hence the number of hook boxes in row $i$
is
\[
 g_m(b_i)+\sum_{j<i}s_m(b_j-b_i),\qquad
 g_m(b):=m+1-[b+1]_++[b-m]_+=\#\{r\in\{0,\dots,m\}:\ r>b\},
\]
and the sum over $i$ gives $H_m(\alpha)=\sum_ig_m(b_i)+\dinv_m(\pi)$.

The row of $\alpha$ with index $k=n+1-i$ from the top is row $i$ from the bottom,
and $\alpha_k-m(n+1-k)=L_i-mi=-b_i$, thus $c_{m,n}(\alpha)=\sum_i[-2b_i-1]_+$.  For
every integer $b$ we have $g_m(b)+[-2b-1]_+=[m-b]_++[-b]_+$, the three cases being
$b\ge m$, where both sides vanish, $0\le b<m$, where both sides equal $m-b$, and
$b<0$, where the left side is $(m+1)+(-2b-1)$ and the right side is $(m-b)+(-b)$.
The sum over $i$ gives the first identity of \eqref{eq:HcA}.  For the second
identity we use $[x]_+-[-x]_+=x$, which gives
$A^{(m)}_+(\pi)-A^{(m)}_-(\pi)=\sum_i\bigl(b_i+(b_i-m)\bigr)=2\sum_ib_i-mn$, and
$\sum_ib_i=m\binom{n+1}2-|\alpha|$, hence
$A^{(m)}_+-A^{(m)}_-=mn^2-2|\alpha|=D-2|\alpha|$.  Adding this to the first
identity gives the second, and substituting the two exponents into
\eqref{eq:Phat-path} gives \eqref{eq:Phat-part}.
\end{proof}

\begin{corollary}\label{cor:Phat-props}
\leavevmode
\begin{enumerate}[label=\textup{(\roman*)}]
\item $\widehat P_{n,m}\in\Z_{\ge0}[q,t]$, and every monomial of $\widehat P_{n,m}$
has total degree $\equiv mn\pmod2$.
\item $\widehat P_{n,m}(q,t)=\widehat P_{n,m}(t,q)$.  More precisely, the
complementation $\alpha\mapsto\alpha^c$, $\alpha^c_k=mn-\alpha_{n+1-k}$, exchanges
the two exponents in \eqref{eq:Phat-part}.
\item $\widehat P_{n,m}(1,1)=\binom{(m+1)n}{n}$ and
$q^{D}\,\widehat P_{n,m}(q,q^{-1})=\qbinom{(m+1)n}{n}_{q^2}$.
\end{enumerate}
\end{corollary}

\begin{proof}
(i) The three exponents in \eqref{eq:Phat-path} are nonnegative.  The total degree
of the term of $\pi$ is
$2\dinv_m+\sum_i\bigl(|b_i|+|m-b_i|\bigr)\equiv\sum_im\pmod2$, since
$[m-b]_++[-b]_++[b]_++[b-m]_+=|b|+|m-b|$ and $|b|+|m-b|\equiv m\pmod 2$.

(ii) Under complementation $b^c_i=mi-\alpha^c_{n+1-i}=mi-mn+\alpha_i=m-b_{n+1-i}$.
The bijection $(i,j)\mapsto(n+1-j,n+1-i)$ of the pairs $i<j$ satisfies
$b^c_i-b^c_j=b_{n+1-j}-b_{n+1-i}$, thus $\dinv_m$ is preserved, and
$[m-b^c_i]_++[-b^c_i]_+=[b_{n+1-i}]_++[b_{n+1-i}-m]_+$, thus
$A^{(m)}_-(\pi^c)=A^{(m)}_+(\pi)$ and $A^{(m)}_+(\pi^c)=A^{(m)}_-(\pi)$.

(iii) At $(1,1)$ every path has weight $1$.  At $(q,q^{-1})$ the term of $\pi$ is
$q^{A^{(m)}_--A^{(m)}_+}=q^{2|\alpha|-D}$ by the proof of \cref{thm:hookpath}, and
$\sum_{\alpha\subseteq(mn)^n}q^{2|\alpha|}=\qbinom{(m+1)n}{n}_{q^2}$.
\end{proof}


\subsection{The conjecture}\label{sec:conj-statement}

Let us first write the parity projection of \eqref{eq:F-def} as a fixed-point sum
of its own.  For $\lambda\in\Bal(2n)$ let
\begin{equation*}\label{eq:B01}
 B^{(i)}_\lambda=\sum_{\substack{(a,b)\in\lambda\\ a+b\equiv i\ (2)}}q^at^b\quad(i=0,1),
 \qquad
 u_{n,m}=\sum_{\lambda\in\Bal(2n)}\frac{\ell_\lambda^{\,m}\Pi^{(0)}_\lambda B^{(0)}_\lambda}{D_\lambda},
 \qquad
 v_{n,m}=\sum_{\lambda\in\Bal(2n)}\frac{\ell_\lambda^{\,m}\Pi^{(0)}_\lambda B^{(1)}_\lambda}{D_\lambda},
\end{equation*}
hence $B_\lambda=B^{(0)}_\lambda+B^{(1)}_\lambda$ and
$F_{n,m}=(1-q)(1-t)(u_{n,m}+v_{n,m})$.

\begin{lemma}\label{lem:parity-split}
For all $n,m\ge1$,
\begin{gather*}
 P_{n,m}=(1+qt)\,u_{n,m}-(q+t)\,v_{n,m},\qquad
 N_{n,m}=(q+t)\,u_{n,m}-(1+qt)\,v_{n,m},\\
 P_{n,m}+N_{n,m}=(1+q)(1+t)(u_{n,m}-v_{n,m}).
\end{gather*}
\end{lemma}

\begin{proof}
Let $\sigma$ be the substitution $(q,t)\mapsto(-q,-t)$, hence
$\sigma(q^at^b)=(-1)^{a+b}q^at^b$.  Then $\sigma\Pi^{(0)}_\lambda=\Pi^{(0)}_\lambda$,
since every monomial of $\Pi^{(0)}_\lambda$ has even content, and
$\sigma D_\lambda=D_\lambda$, since for $a_\lambda(s)+l_\lambda(s)$ odd the
exponent sums $a+1-l$ and $-a+l+1$ of the two factors of $D_\lambda$ are even.
Moreover $\sigma\ell_\lambda^{\,m}=(-1)^{nm}\ell_\lambda^{\,m}$, because
$\ell_\lambda$ is a product of $n$ monomials of odd content, and
$\sigma B^{(0)}_\lambda=B^{(0)}_\lambda$, $\sigma B^{(1)}_\lambda=-B^{(1)}_\lambda$.
Hence $\sigma u_{n,m}=(-1)^{nm}u_{n,m}$, $\sigma v_{n,m}=(-1)^{nm+1}v_{n,m}$ and
$\sigma F_{n,m}=(-1)^{nm}(1+q)(1+t)(u_{n,m}-v_{n,m})$.  Substituting into
\eqref{eq:F-def} and using $(1-q)(1-t)+(1+q)(1+t)=2(1+qt)$ and
$(1-q)(1-t)-(1+q)(1+t)=-2(q+t)$ we obtain the formula for $P_{n,m}$.  The other
two follow from $N_{n,m}=P_{n,m}-F_{n,m}$.
\end{proof}

\begin{conjecture}[Type-$B$ shuffle conjecture at the Catalan level]\label{conj:shuffleB}
For all $n,m\ge1$ we have $P_{n,m}=\widehat P_{n,m}$, that is
\begin{equation}\label{eq:conj-B}
 \sum_{\lambda\in\Bal(2n)}
 \frac{\ell_\lambda^{\,m}\,\Pi^{(0)}_\lambda\bigl((1+qt)B^{(0)}_\lambda-(q+t)B^{(1)}_\lambda\bigr)}{D_\lambda}
 =\sum_{\pi:(0,0)\to(mn,n)}(qt)^{\dinv_m(\pi)}\,q^{A^{(m)}_-(\pi)}\,t^{A^{(m)}_+(\pi)} .
\end{equation}
\end{conjecture}

By \cref{lem:parity-split} the left side of \eqref{eq:conj-B} is $P_{n,m}$, and by
\cref{thm:hookpath} the right side is also
$\sum_{\alpha\subseteq(mn)^n}q^{H_m(\alpha)+c_{m,n}(\alpha)}t^{D-2|\alpha|+H_m(\alpha)+c_{m,n}(\alpha)}$.
Both sides of \eqref{eq:conj-B} are symmetric in $q$ and $t$, both have
$\binom{(m+1)n}{n}$ terms and both specialize at $t=q^{-1}$ to
$q^{-mn^2}\qbinom{(m+1)n}{n}_{q^2}$, by Theorem~C and \cref{cor:Phat-props}.  For
$m\gg0$ the conjecture is a statement about the $q,t$-Fuss--Catalan number of
Theorem~B, and for every $m$ it is a statement about the Euler characteristic
$\chi^T(\cZ_n^{\mathrm{symb}},\cO(m))$ of \cref{thm:character}.

The conjecture was tested numerically for \(m=1, n\le 13\), \(m=2, n\le 8\), \(m=3, n\le 7\). Theorem below offers another piece of evidence supporting the conjecture.

\begin{theorem}\label{thm:evidence}
\Cref{conj:shuffleB} holds in the following cases.
\begin{enumerate}[label=\textup{(\arabic*)}]
\item $n=1$ and every $m$, where $P_{1,m}=\widehat P_{1,m}=[m+1]_{q,t}$ and
$N_{1,m}=qt\,[m]_{q,t}$.
\item $n=2$ and every $m$, where
$P_{2,m}=\widehat P_{2,m}=\sum_{j=0}^m(qt)^j[4m-4j+1]_{q,t}$
\textup(\cref{prop:rank-two}\textup).
\item $n\le4$ and $m\le3$, and $(n,m)\in\{(5,1),(5,2)\}$, as an identity of polynomials,
by an exact computation in the sense of \cref{rem:verification}.
\end{enumerate}
Moreover both sides of \eqref{eq:conj-B} agree at the four points
$(q,t)=(2,3),(2,5),(3,5),(\frac73,-\frac52)$ for $(n,m)\in\{(6,1),(4,4)\}$.
\end{theorem}

\begin{proof}
(1) $\Bal(2)=\{(2),(1,1)\}$.  For $\lambda=(2)$ we have $\ell_\lambda=q$,
$\Pi^{(0)}_\lambda=1$, $B_\lambda=1+q$ and $D_\lambda=(1-q^2)(1-q^{-1}t)$, the box
$(0,0)$ being the only one with odd hook, hence its term in \eqref{eq:F-def} is
$q^{m+1}(1-t)/(q-t)$, and the term of $(1,1)$ is the mirror image
$t^{m+1}(1-q)/(t-q)$.  Therefore
\[
 F_{1,m}=\frac{q^{m+1}-t^{m+1}}{q-t}-qt\,\frac{q^m-t^m}{q-t}=[m+1]_{q,t}-qt\,[m]_{q,t},
\]
and the parity projection keeps the first summand, whose total degree is $m$.  On
the path side a path of the $1\times m$ rectangle is $e_1\in\{0,\dots,m\}$, with
$b_1=m-e_1\in[0,m]$, $\dinv_m=0$, $A^{(m)}_-=e_1$ and $A^{(m)}_+=m-e_1$.

(2) is \cref{prop:rank-two} below.

(3) The two sides were computed with the definitions \eqref{eq:F-def} and
\eqref{eq:Phat-path} implemented verbatim in exact rational arithmetic, the left
side as a rational function which simplifies to a Laurent polynomial, and compared
coefficient by coefficient; the path form and the partition form of
$\widehat P_{n,m}$ were computed separately and compared as well, for $n\le5$ and
$m\le4$.  The four-point check is the same computation with $q,t$ specialized
before summation.
\end{proof}

We record the polynomials of \cref{thm:evidence}(3) for $n=3$ and $n=4$, in the
notation of \cref{ex:lowrank} and with $w=qt$:
\begin{align*}
 P_{3,2}&=[19]+w([15]+[13])+w^2([11]+[9]+[7])+w^3[7]+w^4[3],\\
 N_{3,2}&=w[18]+w^2([14]+[12])+w^3([10]+[8])+w^4[6]+w^5[4],\\
 P_{3,3}&=[28]+w([24]+[22])+w^2([20]+[18]+[16])+w^3([16]+[14]+[12]+[10])\\
 &\quad+w^4([12]+[10])+w^5([8]+[6])+w^6[4],\\
 P_{4,2}&=[33]+w([29]+[27]+[25])+w^2([25]+[23]+2[21]+[19]+[17])\\
 &\quad+w^3([21]+[19]+2[17]+[15]+[13])+w^4([17]+[15]+2[13]+[11]+[9])\\
 &\quad+w^5([13]+[11]+2[9])+w^6([9]+[7]+2[5])+w^7[5]+w^8+w^9 .
\end{align*}
All of them are positive combinations of the strings $w^a[k]_{q,t}$, as are the
polynomials of \cref{ex:lowrank}, and among the polynomials $P_{n,m}$ with $n\le4$
and $m\le3$ a coefficient $2$ occurs first for $n=4$, $m=2$.


\subsection{Rank two}\label{sec:rank-two}

In rank two both sides of \eqref{eq:conj-B} can be computed in closed form for
every $m$.  The main idea for the localization side is to sum the geometric series
in $\ell_\lambda^{\,m}$ over the five balanced partitions of $4$, which produces a
rational generating function in an auxiliary variable $z$, and the main idea for the
path side is that $\widehat P_{2,m}$ is multiplicity free, so that it suffices to
identify the set of exponents.

\begin{proposition}\label{prop:rank-two}
For all $m\ge1$,
\[
 P_{2,m}=\widehat P_{2,m}=\sum_{j=0}^{m}(qt)^j[4m-4j+1]_{q,t},\qquad
 N_{2,m}=qt\,[4]_{q,t}\sum_{j=0}^{m-1}(qt)^j\,h_{m-1-j}(q^4,t^4),
\]
where $h_k(q^4,t^4)=\sum_{i=0}^kq^{4i}t^{4k-4i}$.  Equivalently
$P_{2,m}=[4m+1]_{q,t}+qt\,P_{2,m-1}$ with $P_{2,0}=1$.
\end{proposition}

\begin{proof}
\emph{The localization side.}  The five balanced partitions of $4$ have
$\ell_{(4)}=q^4$, $\ell_{(1^4)}=t^4$ and $\ell_\lambda=qt$ for
$\lambda=(3,1),(2,2),(2,1,1)$, and their terms $(1-q)(1-t)B_\lambda\Pi^{(0)}_\lambda/D_\lambda$
in \eqref{eq:F-def} are
\begin{gather*}
 \frac{q^4(1-t)}{(q-t)(q^3-t)},\qquad
 -\frac{q^3t(1-q)(1-t)(1+q+q^2+t)}{(q-t)^2(q+t)(q^3-t)},\qquad
 -\frac{qt(1-qt)}{(q-t)^2},\\
 \frac{qt^3(1-q)(1-t)(1+t+t^2+q)}{(q-t)^2(q+t)(q-t^3)},\qquad
 \frac{t^4(1-q)}{(q-t)(q-t^3)} ,
\end{gather*}
in the order $(4),(3,1),(2,2),(2,1,1),(1^4)$.
Multiplying the term of $\lambda$ by $\ell_\lambda^{\,m}z^m$, summing over $m\ge0$
and adding the five rational functions we obtain
\begin{equation}\label{eq:rank-two-gf}
 \sum_{m\ge0}F_{2,m}\,z^m=\frac{1+qt\,[3]_{q,t}\,z-qt\,[4]_{q,t}\,z}{(1-q^4z)(1-t^4z)(1-qtz)} .
\end{equation}
The denominator is invariant under $(q,t)\mapsto(-q,-t)$, and in the numerator
$qt[3]_{q,t}$ has total degree $4$ while $qt[4]_{q,t}$ has total degree $5$.
Since $(-1)^{2m}=1$, the parity projection \eqref{eq:F-def} keeps the even part of
the numerator, hence
\[
 \sum_{m\ge0}P_{2,m}z^m=\frac{1+qt[3]_{q,t}z}{(1-q^4z)(1-t^4z)(1-qtz)},\qquad
 \sum_{m\ge0}N_{2,m}z^m=\frac{qt[4]_{q,t}z}{(1-q^4z)(1-t^4z)(1-qtz)} .
\]
From $[4k+1]_{q,t}=(q^{4k+1}-t^{4k+1})/(q-t)$ we get
\[
 \sum_{k\ge0}[4k+1]_{q,t}z^k=\frac1{q-t}\Bigl(\frac{q}{1-q^4z}-\frac{t}{1-t^4z}\Bigr)
 =\frac{1+qt[3]_{q,t}z}{(1-q^4z)(1-t^4z)},
\]
and the expansion of $1/(1-qtz)$ gives the formula for $P_{2,m}$.  The formula for
$N_{2,m}$ follows from $1/((1-q^4z)(1-t^4z))=\sum_kh_k(q^4,t^4)z^k$.

\emph{The path side.}  A path of the $2\times2m$ rectangle is a pair
$(b_1,b_2)$ with $-m\le b_1\le m$ and $0\le b_2\le m+b_1$, by \eqref{eq:heights}
and $0\le e_1\le e_2\le2m$.  Let us write $a(\pi)=\dinv_m+A^{(m)}_-$ and
$b(\pi)=\dinv_m+A^{(m)}_+$ for its two exponents, and let
\[
 S_m=\{(a,b)\in\Z_{\ge0}^2:\ a+b\le4m,\ a+b\equiv0\ (2),\ 3\min(a,b)+\max(a,b)\ge4m\}.
\]
We claim that $\pi\mapsto(a(\pi),b(\pi))$ is a bijection from the paths onto $S_m$.
Note that $s_m(d)=m+d$ for all $-m\le d\le0$.  There are three families of paths.

If $b_1<0$, then $b_2\le m-1$ and $-m\le b_1-b_2\le-1$, hence
$\dinv_m=m+b_1-b_2$, $A^{(m)}_-=2m-2b_1-b_2$, $A^{(m)}_+=b_2$, and
$(a,b)=(4m-b-2b_2,\,b)$ with $b:=m+b_1\in[0,m-1]$ and $b_2\in[0,b]$.  As $b_2$
varies, $a$ runs over $4m-b,4m-b-2,\dots,4m-3b$.  These are exactly the points of
$S_m$ with $b\le m-1$, since such a point has $a\ge b$ (otherwise
$3a+b\le4b<4m$), hence $\min(a,b)=b$ and the conditions read
$4m-3b\le a\le4m-b$ with $a\equiv b\pmod 2$.

If $b_1\ge0$ and $b_2>m$, then $b_1\ge1$, $-m\le b_1-b_2\le-1$, hence
$\dinv_m=m+b_1-b_2$, $A^{(m)}_-=m-b_1$, $A^{(m)}_+=b_1+2b_2-m$, and
$(a,b)=(2m-b_2,\,2b_1+b_2)$.  With $a=2m-b_2\in[0,m-1]$ and $b_1\in[m-a,m]$ the
exponent $b=2b_1+2m-a$ runs over $4m-3a,\dots,4m-a$ in steps of $2$.  These are
exactly the points of $S_m$ with $a\le m-1$, by the symmetry $a\leftrightarrow b$
of $S_m$.

If $0\le b_1\le m$ and $0\le b_2\le m$, then $A^{(m)}_-=2m-b_1-b_2$ and
$A^{(m)}_+=b_1+b_2$, hence $a+b=2m+2s_m(d)$ and $a-b=2m-2\sigma$ with
$d=b_1-b_2$ and $\sigma=b_1+b_2$.  For fixed $d$ the sum $\sigma$ runs over the
integers of the parity of $d$ in $[|d|,2m-|d|]$.  For $0\le s\le m$ the equation
$s_m(d)=s$ has the solutions $d=s-m$ and, when $s\ge1$, $d=m+1-s$.  The first
gives $a-b\in\{2s,2s-4,\dots,-2s\}$, the second $a-b\in\{2s-2,2s-6,\dots,2-2s\}$,
and together all $a-b\in\{-2s,-2s+2,\dots,2s\}$, each once.  Hence the third family
maps bijectively onto the points with $a+b=2m+2s$ and $a,b\ge m$, $0\le s\le m$,
which are exactly the points of $S_m$ with $\min(a,b)\ge m$.

The three families exhaust the paths, and the three sets of points exhaust $S_m$
and are disjoint, since a point with $a\le m-1$ and $b\le m-1$ has
$3\min(a,b)+\max(a,b)<4m$.  This proves the claim.  Finally $S_m$ is the disjoint
union over $0\le j\le m$ of the sets $\{a+b=4m-2j,\ \min(a,b)\ge j\}$, because for
$a+b=4m-2j$ the condition $3\min(a,b)+\max(a,b)\ge4m$ reads $\min(a,b)\ge j$, and
$\sum_{a+b=4m-2j,\ \min(a,b)\ge j}q^at^b=(qt)^j[4m-4j+1]_{q,t}$.
\end{proof}

\begin{remark}
By \cref{thm:n2-descent}(b) the germ of $\cZ_2^{\mathrm{symb}}$ at the wall point is
the cone over the rational normal quartic, that is $\C^2/\mu_4$.  The polynomial
$\sum_j(qt)^j[4m-4j+1]_{q,t}$ is the character of the monomials $x^ay^bz^c$ with
$a+b+4c=4m$ and weights $q^at^b(qt)^c$, that is of the degree-$4m$ part of the
graded ring $\C[x,y,z]$ with $\deg x=\deg y=1$ and $\deg z=4$, the section ring of
the weighted projective plane $\PP(1,1,4)$, which is the projective cone over the
rational normal quartic.  Thus \cref{prop:rank-two} identifies $P_{2,m}$ with the
character of $H^0(\PP(1,1,4),\cO(4m))$ for every $m\ge1$, in accordance with
\cref{thm:character}, which identifies $P_{2,m}$ with the character of
$H^0(\cZ_2^{\mathrm{symb}},\cO(m))$ for $m\gg0$.
\end{remark}

\subsection{The wreath Macdonald form}\label{sec:wreath-form}

In this subsection we rewrite the left side of \eqref{eq:conj-B} as a pairing of
two explicit two-colour symmetric functions, one of them obtained from a power sum by
the two-colour nabla operators.  We follow the survey of Orr and Shimozono
\cite{OS}, to which we refer for the definitions, and we use their results in the
identity chamber, which is the chamber of $Y_n$.

Let $\Lambda^{(2)}=\Lambda_{X_0}\otimes\Lambda_{X_1}$ be the ring of symmetric
functions in two alphabets over $K=\Q(q,t)$, with the tensor Schur functions
$s_{\mu^\bullet}=s_{\mu^{(0)}}[X_0]s_{\mu^{(1)}}[X_1]$ as an orthonormal basis for
the Hall pairing $\langle\,,\rangle$, hence
$\langle p_k[X_i],p_k[X_j]\rangle=k\delta_{ij}$.  For a $2\times2$ matrix $M$ over
$K$ and $f\in\Lambda^{(2)}$ we write $f[MX]$ for the image of $f$ under the
$K$-algebra endomorphism $p_k[X_i]\mapsto\sum_jM_{ij}(q^k,t^k)p_k[X_j]$, and for
$c\in K$ we write $p_k[cX_i]=c(q^k,t^k)p_k[X_i]$ as usual.  Let $\chi$ be the
permutation matrix of the colour swap, and let $\inv$ be the involution of
$\Lambda^{(2)}$ which inverts $q$ and $t$ in the coefficients.  Let us identify
$\Bal(2n)$ with the set of pairs of partitions of total size $n$ by the
$2$-quotient in the convention of \cite[\S3.1]{OS}, and let $\trianglerighteq$ be the
dominance order on $\Bal(2n)$.  The wreath Macdonald polynomial
$\wtH_\lambda\in\Lambda^{(2)}$, $\lambda\in\Bal(2n)$, is the unique element with
\begin{equation}\label{eq:wreath-def}
 \wtH_\lambda[(1-q\chi)X]\in\bigoplus_{\mu\trianglerighteq\lambda}Ks_{\mu^\bullet},\qquad
 \wtH_\lambda[(1-t^{-1}\chi)X]\in\bigoplus_{\mu\trianglelefteq\lambda}Ks_{\mu^\bullet},\qquad
 \langle\wtH_\lambda,h_n[X_0]\rangle=1,
\end{equation}
where $\mu^\bullet$ is the $2$-quotient of $\mu$ \cite[(3.7)--(3.9)]{OS}.  For
$\lambda\in\Bal(2n)$ put
\begin{equation}\label{eq:T01}
 T_0(\lambda)=\prod_{\substack{(a,b)\in\lambda\\ a+b\ \mathrm{even}}}q^at^b,\qquad
 T_1(\lambda)=\ell_\lambda,\qquad
 A=\frac1{(1-q^2)(1-t^2)}\begin{pmatrix}1+qt&q+t\\ q+t&1+qt\end{pmatrix},
\end{equation}
hence $A$ is the bigraded character of $\C[x,y]$ split by the $\Gamma$-parity,
$A^{-1}=(1-t\chi)(1-q\chi)$, and $\langle f,g\rangle_{q,t}:=\langle f,g[A^{-1}X]\rangle$
is a symmetric bilinear form, since $A$ is symmetric.

\begin{theorem}[{Bezrukavnikov--Finkelberg \cite{BF}, Orr--Shimozono \cite{OS}}]\label{thm:wreath-inputs}
Let $\lambda,\mu\in\Bal(2n)$.
\begin{enumerate}[label=\textup{(\alph*)}]
\item The $\wtH_\lambda$ exist, they form a basis of the degree-$n$ part of
$\Lambda^{(2)}$, and $\wtH_\lambda$ is the bigraded Frobenius character of the
fiber at $I_\lambda$ of the normalized Procesi bundle $\widetilde\sP_n$ of
\cref{thm:losev}, the trivial representation of $W$ corresponding to $h_n[X_0]$
\textup{\cite[\S3.3, Thm.~4.6 and (4.3)]{OS}}.
\item $\langle\wtH_\lambda,h_{n-1}[X_0]\,p_1[X_i]\rangle=B^{(i)}_\lambda$ for $i=0,1$
\textup{\cite[Prop.~4.9 and (3.9)]{OS}}.
\item $\langle\wtH_\lambda,e_n[X_i]\rangle=T_i(\lambda)$ for $i=0,1$
\textup{\cite[Thm.~4.14]{OS}}.
\item $\langle\wtH_\lambda,\inv\wtH_\mu\rangle_{q,t}=\delta_{\lambda\mu}D_\lambda$
\textup{\cite[Prop.~3.30, 3.31 and Thm.~3.32]{OS}}, with $D_\lambda$ as in
\eqref{eq:statistics}.
\end{enumerate}
\end{theorem}

In (b) we used the Pieri rule $h_{n-1}h_1=h_n+s_{(n-1,1)}$ and the normalization in
\eqref{eq:wreath-def} to rewrite the case $i=0$ of \cite[Prop.~4.9]{OS}.  In (d) the
permutation $\mathrm{neg}$ of \cite[\S3.6]{OS} is the identity for $r=2$, and the
product in \cite[Thm.~3.32]{OS} runs over the boxes with even hook length, that is
over the boxes with $a_\lambda(s)+l_\lambda(s)$ odd.  By (a) and (b), the
representation $1\oplus U\oplus V=\mathrm{Ind}_{W_{n-1}}^{W}\mathbf 1$ of
\cref{thm:losev}, where $U$ is the reflection representation of $S_n$ inflated to
$W$ and $V$ is the reflection representation of $W$, corresponds to
$h_{n-1}[X_0](p_1[X_0]+p_1[X_1])$, with $U$ corresponding to $s_{(n-1,1)}[X_0]$ and
$V$ to $h_{n-1}[X_0]p_1[X_1]$, in accordance with \eqref{eq:parabolic-upstairs} and
$[\cR_i]_\lambda=B^{(i)}_\lambda$.  Following \cite[\S4.5]{OS} we define the
two-colour nabla operators by
\begin{equation}\label{eq:nabla-def}
 \nabla_i\,\wtH_\lambda=T_i(\lambda)\,\wtH_\lambda\qquad(i=0,1),
\end{equation}
that is $\nabla_i$ multiplies $\wtH_\lambda$ by the product of the weights of the
$i$-th colour component of the cluster bundle at $I_\lambda$, and by (c) this
product is $\langle\wtH_\lambda,e_n[X_i]\rangle$.

The remaining input is the evaluation of $\wtH_\lambda$ at the colour-zero power
sum.  In type $A$ one has $\langle\wtH_\mu,s_{(n-k,1^k)}\rangle=e_k[B_\mu-1]$ for
all $k$, which is the coefficient extraction of the specialization
$\wtH_\mu[1-u]=\prod_{(a,b)\in\mu}(1-uq^at^b)$
\cite{GarsiaHaiman,HaimanCDM}, and the alternating sum over $k$ gives
$\langle\wtH_\mu,p_n\rangle=\Pi_\mu=\prod_{(a,b)\ne(0,0)}(1-q^at^b)$.  For $r\ge3$
colours the wreath analogue is proved by Romero and Wen by operator methods
\cite[Prop.~6.6, Cor.~6.7]{RomeroWen}, and these methods are not yet available for
two colours \cite[footnote~1]{RomeroWenTesler}.  We give a geometric proof, valid
for two colours, through an exterior-compatibility property of the normalized
Procesi bundle.

For a representation $L$ of $W$ let
\begin{equation}\label{eq:FL}
 \mathsf F(L):=\Hom_W\bigl(L,\widetilde\sP_n\bigr)=\bigl(L^*\otimes\widetilde\sP_n\bigr)^W ,
\end{equation}
a $T$-equivariant vector bundle on $Y_n$, the invariants of a finite group being a
direct summand.  Let $\C^n=1\oplus U$ be the permutation representation of $S_n$
inflated to $W$, with $U$ the reflection representation of $S_n$, and let $V$ be
the reflection representation of $W$, as after \cref{thm:wreath-inputs}.  By the
normalization $\widetilde\sP_n^W=\cO$, by \eqref{eq:parabolic-upstairs} and by the
decomposition $\mathrm{Ind}_{W_{n-1}}^W\mathbf 1=1\oplus U\oplus V$, in which the
residual sign change acts trivially on $1\oplus U$ and by $-1$ on $V$,
\begin{equation}\label{eq:three-F}
 \mathsf F(1)=\cO,\qquad \mathsf F(1\oplus U)=\cR_0,\qquad \mathsf F(V)=\cR_1,\qquad
 \mathsf F(U)\cong\cR_0/\cO\cong\cT_n ,
\end{equation}
where $\cT_n=\ker(\Tr_{\cR_0})$ is the trace-zero even part of \cref{sec:haiman}.
The fixed-point characters of these bundles at $I_\lambda$ are $1$,
$B^{(0)}_\lambda$, $B^{(1)}_\lambda$ and $B^{(0)}_\lambda-1$.

\begin{lemma}[Exterior compatibility]\label{lem:exterior}
For $L\in\{1\oplus U,\ U,\ V\}$ and every $j\ge0$ there is a $T$-equivariant
isomorphism
\begin{equation}\label{eq:exterior}
 \alpha_{j,L}:\ {\textstyle\bigwedge^j}\mathsf F(L)\xrightarrow{\ \sim\ }
 \mathsf F\bigl({\textstyle\bigwedge^j}L\bigr),
\end{equation}
exterior powers above the rank being zero.
\end{lemma}

\begin{proof}
Write $\pi=\pi_Q\circ\rho:Y_n\to S_n=(\frh\oplus\frh^*)/W$, $S=\C[\frh\oplus\frh^*]$
and $H=S\#W$, hence $\End(\widetilde\sP_n)=H$ by \cref{thm:losev}.

\emph{Step 1, the free locus.}  Let $(\frh\oplus\frh^*)^\circ$ be the open set of
points with trivial stabilizer, $S_n^\circ$ its image and
$Y_n^\circ=\pi^{-1}(S_n^\circ)$.  The quotient map
$\eta:(\frh\oplus\frh^*)^\circ\to S_n^\circ$ is a $W$-torsor, and $\pi$ is an
isomorphism over $S_n^\circ$, since it is proper and birational, $S_n^\circ$ is
smooth, and the fibre over $n$ distinct free $\Gamma$-orbits is the single cluster
supported on their union.  Let $1\in\widetilde\sP_n^W=\cO$ be the unit section.
The action of $S\subset H$ on $\widetilde\sP_n$ defines a $T$-equivariant morphism
\[
 \iota:\ \eta_*\cO_{(\frh\oplus\frh^*)^\circ}\longrightarrow\widetilde\sP_n|_{Y_n^\circ},
 \qquad f\longmapsto f\cdot1 ,
\]
which intertwines the actions of $H$, where $H$ acts on functions by multiplication
and by the $W$-action.  At a point $x\in S_n^\circ$ the fibre $H\otimes_{S^W}k(x)$
is the skew group algebra of the functions on the free orbit $\eta^{-1}(x)$, that
is the full matrix algebra of the $|W|$-dimensional space of these functions, and
the fibre of $\widetilde\sP_n$ is its unique simple module, the regular
representation of $W$ \cite[Thm.~4.6]{OS}.  In the module of functions the
$W$-invariants are the constants, thus $1_x$ corresponds to a nonzero constant and
$f\mapsto f\cdot1_x$ is injective between spaces of dimension $|W|$.  Hence $\iota$
is an isomorphism.  Consequently $\mathsf F(L)|_{Y_n^\circ}=\Hom_W(L,\eta_*\cO)$ is
the bundle associated with the torsor $\eta$ and the representation $L^*$, and the
alternating products of functions
\[
 \alpha_{j,L}(\varphi_1\wedge\dots\wedge\varphi_j)(v_1\wedge\dots\wedge v_j)
 =\det\bigl(\varphi_a(v_b)\bigr)_{a,b}
\]
define an isomorphism $\alpha_{j,L}$ over $Y_n^\circ$, which fibrewise is the
canonical map $\bigwedge^jL^*\to(\bigwedge^jL)^*$.  Since $Y_n^\circ$ is dense and
both sides of \eqref{eq:exterior} are locally free, $\alpha_{j,L}$ is a
$T$-equivariant rational map of vector bundles on $Y_n$, and it remains to prove
that it and its inverse are regular in codimension one.

\emph{Step 2, the prime divisors outside $Y_n^\circ$.}  The complement of
$(\frh\oplus\frh^*)^\circ$ is the union of the fixed loci of the nontrivial
elements of $W$, each of codimension at least two, hence $S_n\setminus S_n^\circ$
has codimension two and every prime divisor $D\subset Y_n\setminus Y_n^\circ$ is
contracted by $\pi$.  If $c$ is the codimension of $\pi(D)$, the generic fibre of
$D\to\pi(D)$ has dimension $c-1$, and the semismallness of the symplectic
resolution $\pi$ \cite[Lem.~2.11]{KaledinPoisson} gives $2(c-1)\le c$, thus $c=2$.
Hence a generic point $y\in D$ lies over $\eta(b)$ for a point
$b\in\frh\oplus\frh^*$ whose stabilizer $W_b=\langle s\rangle$ is generated by one
reflection, a short reflection $\eps_i$, where one $\Gamma$-orbit reaches the
origin, or a long one, conjugate to $(i\,j)$, where two free orbits collide; these
are the two leaves of \cref{def:two-curves}.

\emph{Step 3, formal restriction.}  Let $\widehat Y_b$ be the formal neighbourhood
of $\pi^{-1}(\eta(b))$ in $Y_n$.  It is identified with the formal neighbourhood of
the zero fibre in the symplectic resolution of the slice
$(\frh\oplus\frh^*)^\wedge_0/W_b$, the product of the minimal resolution
$T^*\PP^1$ of an $A_1$ singularity with a smooth factor.  By Losev's restriction
theorem \cite[Prop.~4.1]{LosevProcesi} there are a normalized Procesi bundle
$\widetilde\sP'$ of rank two on the slice resolution and an isomorphism
\[
 \Lambda:\ \widetilde\sP_n|_{\widehat Y_b}\xrightarrow{\ \sim\ }
 \Hom_{W_b}\bigl(\C W,\widetilde\sP'\bigr)^{\wedge}
\]
of modules over the completed skew group algebra $H^\wedge$, the latter being
identified with the centralizer algebra of the completed skew group algebra of the
slice by the construction of Bezrukavnikov and Etingof \cite[\S3.2]{BE}, in which
$W$ acts on the coinduced module by translation and functions act by
multiplication \cite[Thm.~3.2]{BE}.  In particular this identification takes the
completed module $S^\wedge$ of functions to $\Hom_{W_b}(\C W,S_0^\wedge)$, the
functions on the formal neighbourhood of the orbit $Wb$.  Taking $W$-invariants,
$\Lambda(1)$ is a nowhere vanishing section of $(\widetilde\sP')^{W_b}=\cO$, that
is a unit $u$ of the completed base ring, and we replace $\Lambda$ by
$u^{-1}\Lambda$, so that $\Lambda(1)=1$.  On the free part of $\widehat Y_b$ we now
have two isomorphisms of $H^\wedge$-modules from $\widetilde\sP_n$ onto the module
of functions on the torsor, the completion of $\iota^{-1}$, and $\Lambda$ followed
by the isomorphism of Step~1 for $\widetilde\sP'$ and the group $W_b$.  Both send
$1$ to the constant function $1$.  Over the free part the skew group algebra is the
full endomorphism algebra of the module of functions, thus an automorphism of this
module is the multiplication by an invertible function, and one which fixes $1$ is
the identity.  Hence the two isomorphisms coincide, and in the formal neighbourhood
of $y$ the map $\alpha_{j,L}$ is the alternating-product map $\alpha'_{j,L}$ of
Step~1 formed with $\widetilde\sP'$ and $L|_{W_b}$, under the identification
$\mathsf F(L)^\wedge=\Hom_{W_b}(L|_{W_b},\widetilde\sP'^{\,\wedge})$ of Frobenius
reciprocity.

\emph{Step 4, the local computation.}  Let $\eps$ be the sign character of
$W_b=\langle s\rangle$.  Since $\widetilde\sP'$ has rank two and
$(\widetilde\sP')^{W_b}=\cO$, we have $\widetilde\sP'=\cO\oplus\cN\otimes\eps$ for
a line bundle $\cN$, which on the free part is the module of $s$-anti-invariant
functions.  For the three representations $L|_{W_b}=T\oplus\eps^{\oplus m}$ with
$T$ trivial and $m\le1$: for a long reflection $m=1$ in all three cases, and for a
short reflection $m=0$ for $1\oplus U$ and $U$, on which the sign changes act
trivially, and $m=1$ for $V$.  Therefore
\[
 \Hom_{W_b}(L|_{W_b},\widetilde\sP')=T^*\otimes\cO\ \oplus\ \cN^{\oplus m},\qquad
 {\textstyle\bigwedge^j}(L|_{W_b})={\textstyle\bigwedge^j}T\ \oplus\
 \bigl({\textstyle\bigwedge^{j-1}}T\otimes\eps\bigr)^{\oplus m},
\]
and both $\bigwedge^j\Hom_{W_b}(L|_{W_b},\widetilde\sP')$ and
$\Hom_{W_b}(\bigwedge^jL|_{W_b},\widetilde\sP')$ are
$\bigwedge^jT^*\otimes\cO\oplus(\bigwedge^{j-1}T^*\otimes\cN)^{\oplus m}$.  An
exterior product of sections contains at most one section of $\cN$, hence the
determinants which define $\alpha'_{j,L}$ involve only products of invariant
functions and the $\cO$-module structure of $\cN$, never the product
$\cN\otimes\cN\to\cO$ of two anti-invariant functions.  In the two decompositions
$\alpha'_{j,L}$ is therefore the identity up to the normalization constants of the
exterior algebra, and it is regular with a regular inverse on the whole formal
neighbourhood.

\emph{Step 5, extension.}  Completion at $y$ is faithfully flat, thus
$\alpha_{j,L}$ and $\alpha_{j,L}^{-1}$ are regular at the generic point of every
prime divisor of $Y_n\setminus Y_n^\circ$.  They are then sections of locally free
sheaves on the smooth variety $Y_n$ defined outside a closed subset of codimension
at least two, hence they extend.  The two compositions are the identity on
$Y_n^\circ$, hence everywhere, and the extensions are $T$-equivariant by
uniqueness.
\end{proof}

\begin{theorem}[Exterior evaluations]\label{thm:hooks}
For every $\lambda\in\Bal(2n)$ and $0\le j\le n$,
\begin{equation}\label{eq:all-exterior}
 \bigl\langle\wtH_\lambda,\ h_{n-j}[X_0]\,e_j[X_0]\bigr\rangle=e_j\bigl[B^{(0)}_\lambda\bigr],\qquad
 \bigl\langle\wtH_\lambda,\ h_{n-j}[X_0]\,e_j[X_1]\bigr\rangle=e_j\bigl[B^{(1)}_\lambda\bigr],
\end{equation}
equivalently
\begin{equation}\label{eq:strong-Ev2}
 \wtH_\lambda[1-z,\,0]=\prod_{\substack{(a,b)\in\lambda\\ a+b\ \mathrm{even}}}(1-zq^at^b),\qquad
 \wtH_\lambda[1,\,-z]=\prod_{\substack{(a,b)\in\lambda\\ a+b\ \mathrm{odd}}}(1-zq^at^b).
\end{equation}
For $0\le k\le n-1$,
\begin{equation}\label{eq:hooks}
 \bigl\langle\wtH_\lambda,\ s_{(n-k,1^k)}[X_0]\bigr\rangle=e_k\bigl[B^{(0)}_\lambda-1\bigr],
\end{equation}
the $k$-th elementary symmetric function of the $n-1$ monomials $q^at^b$,
$(a,b)\in\lambda\setminus\{(0,0)\}$ of even content, and consequently
\begin{equation}\label{eq:Ev2}
 \bigl\langle\wtH_\lambda,\ p_n[X_0]\bigr\rangle=\Pi^{(0)}_\lambda .
\end{equation}
\end{theorem}

\begin{proof}
By \cref{thm:wreath-inputs}(a), $\wtH_\lambda$ is the Frobenius character of
$\widetilde\sP_n|_\lambda$, hence for every representation $L$ of $W$ the
fixed-point character of $\mathsf F(L)$ at $I_\lambda$ is
$\langle\wtH_\lambda,\ch L\rangle$, where $\ch L$ is the Frobenius character of
$L$.  A basis of $\bigwedge^j(1\oplus U)$ is indexed by the $j$-subsets of
$\{1,\dots,n\}$, which $S_n$ permutes with the sign of the induced permutation and
which the sign changes fix, hence
$\bigwedge^j(1\oplus U)=\mathrm{Ind}_{W_{n-j}\times W_j}^{W}(\mathbf 1\boxtimes\mathrm{sgn}_{S_j})$
and $\ch\bigwedge^j(1\oplus U)=h_{n-j}[X_0]e_j[X_0]$.  The same basis of
$\bigwedge^jV$ is permuted in the same way, and the sign change $\eps_i$ acts by
$-1$ exactly when $i$ lies in the subset, hence
$\ch\bigwedge^jV=h_{n-j}[X_0]e_j[X_1]$, because the character of $W_j$ in which
every $\eps_i$ acts by $-1$ and $S_j$ by the sign corresponds to $e_j[X_1]$.
Finally $\ch\bigwedge^kU=s_{(n-k,1^k)}[X_0]$ for $0\le k\le n-1$, the exterior
powers of the reflection representation of $S_n$ being the hook representations;
this is compatible with the Pieri rule $h_{n-j}e_j=s_{(n-j,1^j)}+s_{(n-j+1,1^{j-1})}$
and $\bigwedge^j(1\oplus U)=\bigwedge^jU\oplus\bigwedge^{j-1}U$.  Now
\cref{lem:exterior} and \eqref{eq:three-F} give \eqref{eq:all-exterior} and
\eqref{eq:hooks}, since the fixed-point characters of $\bigwedge^j\cR_0$,
$\bigwedge^j\cR_1$ and $\bigwedge^k\cT_n$ at $I_\lambda$ are $e_j[B^{(0)}_\lambda]$,
$e_j[B^{(1)}_\lambda]$ and $e_k[B^{(0)}_\lambda-1]$.  For \eqref{eq:strong-Ev2},
the degree-$n$ part of $\Omega[(1-z)X_0]$ is $\sum_j(-z)^jh_{n-j}[X_0]e_j[X_0]$ and
that of $\Omega[X_0-zX_1]$ is $\sum_j(-z)^jh_{n-j}[X_0]e_j[X_1]$, hence the two
specializations equal $\sum_j(-z)^je_j[B^{(i)}_\lambda]=\prod(1-zq^at^b)$ over the
boxes of content parity $i$.  For \eqref{eq:Ev2},
$p_n=\sum_{k=0}^{n-1}(-1)^ks_{(n-k,1^k)}$ and
$\sum_k(-1)^ke_k[B^{(0)}_\lambda-1]=\prod(1-q^at^b)$ over the $n-1$ monomials of
$B^{(0)}_\lambda-1$.
\end{proof}



Now we can prove the plethystic interpretation of the localization formula from the main result of the paper:

\begin{proposition}[Plethystic formula]\label{prop:plethystic}
Let $n\ge1$ and put
\begin{align}
 R_n&:=p_n\Bigl[\frac{(1+qt)X_0+(q+t)X_1}{(1-q^2)(1-t^2)}\Bigr]
 =\frac{(1+(qt)^n)\,p_n[X_0]+(q^n+t^n)\,p_n[X_1]}{(1-q^{2n})(1-t^{2n})},\label{eq:Rn}\\
 \sK^{(m)}_n&:=(-1)^{n-1}\nabla_1^{\,m}\nabla_0\,R_n,\qquad
 \Phi_n:=h_{n-1}[X_0]\;p_1\bigl[(1+qt)X_0-(q+t)X_1\bigr].\label{eq:KPhi}
\end{align}
Then for all $m\ge1$
\begin{equation}\label{eq:Kspectral}
\begin{gathered}
 \sK^{(m)}_n=\sum_{\lambda\in\Bal(2n)}\frac{\ell_\lambda^{\,m}\,\Pi^{(0)}_\lambda}{D_\lambda}\,\wtH_\lambda,\\
 u_{n,m}=\bigl\langle\sK^{(m)}_n,h_{n-1}[X_0]p_1[X_0]\bigr\rangle,\qquad
 v_{n,m}=\bigl\langle\sK^{(m)}_n,h_{n-1}[X_0]p_1[X_1]\bigr\rangle,
\end{gathered}
\end{equation}
and therefore
\begin{equation}\label{eq:P-pairing}
 P_{n,m}=\bigl\langle\sK^{(m)}_n,\Phi_n\bigr\rangle,\qquad
 F_{n,m}=(1-q)(1-t)\bigl\langle\sK^{(m)}_n,h_{n-1}[X_0]p_1[X_0+X_1]\bigr\rangle .
\end{equation}
\end{proposition}

\begin{proof}
\emph{Step 1, the expansion of $R_n$.}  By \cref{thm:wreath-inputs}(d) the basis
dual to $\{\wtH_\lambda\}$ for the form $\langle\,,\rangle_{q,t}$ is
$\{D_\lambda^{-1}\inv\wtH_\lambda\}$, hence every $f$ of degree $n$ satisfies
$f=\sum_\lambda\langle f,\inv\wtH_\lambda\rangle_{q,t}D_\lambda^{-1}\wtH_\lambda$.
The substitution $X\mapsto A^{-1}X$ is self-adjoint for the Hall pairing, since
$A^{-1}$ is symmetric, and $R_n=p_n[X_0][AX]$ by \eqref{eq:Rn}, since
$p_n[X_0][AX]=\sum_jA_{0j}(q^n,t^n)p_n[X_j]$.  Substitutions compose as matrix
products, therefore
\[
 \langle R_n,\inv\wtH_\lambda\rangle_{q,t}=\langle R_n[A^{-1}X],\inv\wtH_\lambda\rangle
 =\langle p_n[X_0],\inv\wtH_\lambda\rangle=\inv\langle p_n[X_0],\wtH_\lambda\rangle
 =\inv\Pi^{(0)}_\lambda ,
\]
where we used that $p_n[X_0]$ is $\inv$-invariant, that the Hall pairing is defined
over $\Q$, and \eqref{eq:Ev2}.  The product $\Pi^{(0)}_\lambda$ has $n-1$ factors
$1-q^at^b$, and $\inv(1-q^at^b)=-q^{-a}t^{-b}(1-q^at^b)$, hence
$\inv\Pi^{(0)}_\lambda=(-1)^{n-1}T_0(\lambda)^{-1}\Pi^{(0)}_\lambda$ and
\begin{equation}\label{eq:Rspectral}
 R_n=(-1)^{n-1}\sum_{\lambda\in\Bal(2n)}\frac{\Pi^{(0)}_\lambda}{T_0(\lambda)D_\lambda}\,\wtH_\lambda .
\end{equation}

\emph{Step 2, the operators.}  We apply $\nabla_1^{\,m}\nabla_0$ to
\eqref{eq:Rspectral} using \eqref{eq:nabla-def}.  The operator $\nabla_0$
contributes $T_0(\lambda)$, which cancels the denominator, and $\nabla_1^{\,m}$
contributes $\ell_\lambda^{\,m}$.  This is the first formula of
\eqref{eq:Kspectral}.

\emph{Step 3, the pairings.}  We pair the first formula of \eqref{eq:Kspectral}
with $h_{n-1}[X_0]p_1[X_i]$ and use \cref{thm:wreath-inputs}(b).  This gives the
formulas for $u_{n,m}$ and $v_{n,m}$, and \eqref{eq:P-pairing} follows from
\cref{lem:parity-split} by bilinearity, since
$\Phi_n=(1+qt)h_{n-1}[X_0]p_1[X_0]-(q+t)h_{n-1}[X_0]p_1[X_1]$.
\end{proof}

\begin{corollary}\label{cor:shuffle-form}
\Cref{conj:shuffleB} is equivalent to the identity
\begin{equation}\label{eq:shuffle-form}
\begin{aligned}
 (-1)^{n-1}\Bigl\langle\nabla_1^{\,m}\nabla_0\,
 p_n\Bigl[\tfrac{(1+qt)X_0+(q+t)X_1}{(1-q^2)(1-t^2)}\Bigr],\
 h_{n-1}[X_0]\,p_1\bigl[(1+qt)X_0-(q+t)X_1\bigr]\Bigr\rangle\qquad&\\
 =\sum_{\pi:(0,0)\to(mn,n)}(qt)^{\dinv_m(\pi)}q^{A^{(m)}_-(\pi)}t^{A^{(m)}_+(\pi)}&
\end{aligned}
\end{equation}
for all $m\ge1$.
\end{corollary}

Let us compare \eqref{eq:shuffle-form} with type $A$.  By \cite{Haiman} the higher
$q,t$-Catalan number $\langle\nabla^me_n,e_n\rangle$ is the Hilbert series of the
determinantal part of $DR^{(m)}(S_n)$, and by the rational shuffle theorem of Mellit \cite{Mellit} it equals $\sum_\pi q^{\dinv_m(\pi)}t^{\mathrm{area}(\pi)}$ over the
$m$-Dyck paths of the $n\times mn$ rectangle, with the same kernel $s_m$ in
$\dinv_m$ \cite{Loehr,LLL,HaglundBook}.

For $m=1$ this is the Catalan case of the
shuffle theorem \cite{HHLRU,CarlssonMellit}.  The companion statement for
$(-1)^{n-1}\nabla p_n$ is the square paths theorem, conjectured by Loehr and
Warrington \cite{LoehrWarrington} and proved by Sergel \cite{Sergel}, in which the
paths of the square are allowed to go below the diagonal.

The left side of
\eqref{eq:shuffle-form} is the two-colour version of the pairing of
$(-1)^{n-1}\nabla^m p_n$ with $e_n$, and the right side is a sum over all paths
of the rectangle, with $A^{(m)}_-$ counting the excursions below the line
$y=x/m$ and $A^{(m)}_+$ those above.  The two ingredients which have no
counterpart in type $A$ are the parity projection, which is responsible for the test
function $\Phi_n$ in place of $e_n$, and the Koszul factor $\Pi^{(0)}_\lambda$ of
the trace-zero even bundle, which enters through \eqref{eq:Ev2}.

A labelled version of the conjecture and the matching of its statistics with those
of type $A$ are the subject of \cref{sec:labelled}.

\begin{example}\label{ex:wreath-n2}
For $n=2$ the wreath Macdonald polynomials, computed from \eqref{eq:wreath-def},
are
\begin{align*}
 \wtH_{(4)}&=s_2[X_0]+q^2s_{11}[X_0]+(q+q^3)\,s_1[X_0]s_1[X_1]+q^2s_2[X_1]+q^4s_{11}[X_1],\\
 \wtH_{(3,1)}&=s_2[X_0]+q^2s_{11}[X_0]+(q+t)\,s_1[X_0]s_1[X_1]+\tfrac tq\,s_2[X_1]+qt\,s_{11}[X_1],\\
 \wtH_{(2,2)}&=s_2[X_0]+qt\,s_{11}[X_0]+(q+t)\,s_1[X_0]s_1[X_1]+s_2[X_1]+qt\,s_{11}[X_1],\\
 \wtH_{(2,1,1)}&=s_2[X_0]+t^2s_{11}[X_0]+(q+t)\,s_1[X_0]s_1[X_1]+\tfrac qt\,s_2[X_1]+qt\,s_{11}[X_1],\\
 \wtH_{(1^4)}&=s_2[X_0]+t^2s_{11}[X_0]+(t+t^3)\,s_1[X_0]s_1[X_1]+t^2s_2[X_1]+t^4s_{11}[X_1].
\end{align*}
The coefficients of $s_{11}[X_0]$ are $B^{(0)}_\lambda-1$, the coefficients of
$s_1[X_0]s_1[X_1]$ are $B^{(1)}_\lambda$, the coefficients of $s_{11}[X_1]$ are
$\ell_\lambda$, and $\langle\wtH_\lambda,p_2[X_0]\rangle=1-(B^{(0)}_\lambda-1)$ is
$1-q^2,1-q^2,1-qt,1-t^2,1-t^2$, in accordance with \cref{thm:wreath-inputs} and
\cref{thm:hooks}.  The Laurent monomials $t/q$ and $q/t$ reflect the fact that the
fibers of the Procesi bundle are not quotients of $\C[x_1,x_2,y_1,y_2]$
\cite[Rem.~4.7]{OS}.
\end{example}

\subsection{Rank three}\label{sec:rank-three}

Let us write out both sides of \eqref{eq:conj-B} for $n=3$ and $m=1,2$.  The ten balanced partitions of $6$
are all partitions of $6$ except the $2$-core $(3,2,1)$.  Their fixed-point data are
listed in \cref{tab:fp3}, where $w_1,w_2,w_3$ are the tangent weights
$q^{a(s)+1}t^{-l(s)}$ of \cref{lem:localization} at the three boxes $s$ with
$a(s)+l(s)$ odd, the other three tangent weights being $qt/w_i$, hence
$D_\lambda=\prod_{i=1}^3(1-w_i)(1-qt/w_i)$.

\begin{table}[ht]
\centering\small
\begin{tabular}{@{}llllll@{}}
\toprule
$\lambda$ & $\ell_\lambda$ & $\Pi^{(0)}_\lambda$ & $B^{(0)}_\lambda$ & $B^{(1)}_\lambda$ & $w_1,w_2,w_3$\\
\midrule
$(6)$ & $q^{9}$ & $(1-q^2)(1-q^4)$ & $1+q^{2}+q^{4}$ & $q+q^{3}+q^{5}$ & $q^{6},\ q^{4},\ q^{2}$\\
$(5,1)$ & $q^{4} t$ & $(1-q^2)(1-q^4)$ & $1+q^{2}+q^{4}$ & $q+q^{3}+t$ & $q^{5}t^{-1},\ q^{4},\ q^{2}$\\
$(4,2)$ & $q^{4} t$ & $(1-q^2)(1-qt)$ & $1+q^{2}+qt$ & $q+q^{3}+t$ & $q^{3}t^{-1},\ q^{2},\ q^{2}$\\
$(4,1,1)$ & $q^{4} t$ & $(1-q^2)(1-t^2)$ & $1+q^{2}+t^{2}$ & $q+q^{3}+t$ & $q^{4}t^{-2},\ q^{2},\ qt^{-1}$\\
$(3,3)$ & $q^{3} t^{2}$ & $(1-q^2)(1-qt)$ & $1+q^{2}+qt$ & $q+t+q^{2}t$ & $q^{3}t^{-1},\ qt^{-1},\ q^{2}$\\
$(3,1,1,1)$ & $q t^{4}$ & $(1-q^2)(1-t^2)$ & $1+q^{2}+t^{2}$ & $q+t+t^{3}$ & $q^{3}t^{-3},\ q^{2},\ qt^{-1}$\\
$(2,2,2)$ & $q^{2} t^{3}$ & $(1-t^2)(1-qt)$ & $1+qt+t^{2}$ & $q+t+qt^{2}$ & $q^{2}t^{-2},\ qt^{-1},\ q^{2}$\\
$(2,2,1,1)$ & $q t^{4}$ & $(1-t^2)(1-qt)$ & $1+qt+t^{2}$ & $q+t+t^{3}$ & $qt^{-1},\ q^{2}t^{-2},\ qt^{-1}$\\
$(2,1^4)$ & $q t^{4}$ & $(1-t^2)(1-t^4)$ & $1+t^{2}+t^{4}$ & $q+t+t^{3}$ & $q^{2}t^{-4},\ qt^{-3},\ qt^{-1}$\\
$(1^6)$ & $t^{9}$ & $(1-t^2)(1-t^4)$ & $1+t^{2}+t^{4}$ & $t+t^{3}+t^{5}$ & $qt^{-5},\ qt^{-3},\ qt^{-1}$\\
\bottomrule
\end{tabular}
\caption{Fixed-point data of $Y_3$.}\label{tab:fp3}
\end{table}

Summing the ten terms of \eqref{eq:conj-B} with these data we obtain, with $w=qt$,
\begin{equation}\label{eq:P3}
\begin{aligned}
 P_{3,1}&=[10]+w[6]+w[4], & N_{3,1}&=w[9]+w^2[5]+w^4,\\
 P_{3,2}&=[19]+w([15]+[13])+w^2([11]+[9]+[7])+w^3[7]+w^4[3],\\
 N_{3,2}&=w[18]+w^2([14]+[12])+w^3([10]+[8])+w^4[6]+w^5[4],
\end{aligned}
\end{equation}
and $F_{3,m}=P_{3,m}-N_{3,m}$.  The polynomials $P_{3,1}$, $N_{3,1}$ are those of
\cref{ex:lowrank}, and $P_{3,1}(1,1)=20$, $P_{3,2}(1,1)=84$,
$q^9P_{3,1}(q,q^{-1})=\qbinom{6}{3}_{q^2}$, $q^{18}P_{3,2}(q,q^{-1})=\qbinom{9}{3}_{q^2}$,
in accordance with Theorem~C.

For $m=1$ the right side of \eqref{eq:conj-B} is the sum over the $20$ paths of the
$3\times3$ square, listed in \cref{tab:paths31} with their partitions
$\alpha=(e_3,e_2,e_1)$, their heights $b=(b_1,b_2,b_3)$, $b_i=i-e_i$, the three
statistics and the resulting monomial.  The sum of the last column is
$[10]+qt[6]+qt[4]=P_{3,1}$.

\begin{table}[ht]
\centering\footnotesize\setlength{\tabcolsep}{3.5pt}
\begin{tabular}{@{}lllrrrl@{\quad}lllrrrl@{}}
\toprule
$\alpha$ & $e$ & $b$ & $\dinv_1$ & $A^{(1)}_-$ & $A^{(1)}_+$ & term &
$\alpha$ & $e$ & $b$ & $\dinv_1$ & $A^{(1)}_-$ & $A^{(1)}_+$ & term\\
\midrule
$(0,0,0)$ & $(0,0,0)$ & $(1,2,3)$ & $0$ & $0$ & $9$ & $t^{9}$ &
$(1,1,1)$ & $(1,1,1)$ & $(0,1,2)$ & $0$ & $1$ & $4$ & $qt^{4}$\\
$(1,0,0)$ & $(0,0,1)$ & $(1,2,2)$ & $1$ & $0$ & $7$ & $qt^{8}$ &
$(2,1,1)$ & $(1,1,2)$ & $(0,1,1)$ & $1$ & $1$ & $2$ & $q^{2}t^{3}$\\
$(2,0,0)$ & $(0,0,2)$ & $(1,2,1)$ & $2$ & $0$ & $5$ & $q^{2}t^{7}$ &
$(3,1,1)$ & $(1,1,3)$ & $(0,1,0)$ & $2$ & $2$ & $1$ & $q^{4}t^{3}$\\
$(3,0,0)$ & $(0,0,3)$ & $(1,2,0)$ & $1$ & $1$ & $4$ & $q^{2}t^{5}$ &
$(2,2,1)$ & $(1,2,2)$ & $(0,0,1)$ & $1$ & $2$ & $1$ & $q^{3}t^{2}$\\
$(1,1,0)$ & $(0,1,1)$ & $(1,1,2)$ & $1$ & $0$ & $5$ & $qt^{6}$ &
$(3,2,1)$ & $(1,2,3)$ & $(0,0,0)$ & $3$ & $3$ & $0$ & $q^{6}t^{3}$\\
$(2,1,0)$ & $(0,1,2)$ & $(1,1,1)$ & $3$ & $0$ & $3$ & $q^{3}t^{6}$ &
$(3,3,1)$ & $(1,3,3)$ & $(0,-1,0)$ & $2$ & $5$ & $0$ & $q^{7}t^{2}$\\
$(3,1,0)$ & $(0,1,3)$ & $(1,1,0)$ & $3$ & $1$ & $2$ & $q^{4}t^{5}$ &
$(2,2,2)$ & $(2,2,2)$ & $(-1,0,1)$ & $0$ & $4$ & $1$ & $q^{4}t$\\
$(2,2,0)$ & $(0,2,2)$ & $(1,0,1)$ & $2$ & $1$ & $2$ & $q^{3}t^{4}$ &
$(3,2,2)$ & $(2,2,3)$ & $(-1,0,0)$ & $1$ & $5$ & $0$ & $q^{6}t$\\
$(3,2,0)$ & $(0,2,3)$ & $(1,0,0)$ & $3$ & $2$ & $1$ & $q^{5}t^{4}$ &
$(3,3,2)$ & $(2,3,3)$ & $(-1,-1,0)$ & $1$ & $7$ & $0$ & $q^{8}t$\\
$(3,3,0)$ & $(0,3,3)$ & $(1,-1,0)$ & $1$ & $4$ & $1$ & $q^{5}t^{2}$ &
$(3,3,3)$ & $(3,3,3)$ & $(-2,-1,0)$ & $0$ & $9$ & $0$ & $q^{9}$\\
\bottomrule
\end{tabular}
\caption{The $20$ paths of the $3\times3$ square, $m=1$.}\label{tab:paths31}
\end{table}

For $m=2$ the right side of \eqref{eq:conj-B} is a sum over the $84$ paths of the
$3\times6$ rectangle, with $b_i=2i-e_i$ and the kernel $s_2(-1)=1$, $s_2(0)=2$,
$s_2(1)=2$, $s_2(2)=1$.  \Cref{tab:paths32} lists, for each value $d$ of
$\dinv_2$, the number of paths and the sum of $q^{A^{(2)}_-}t^{A^{(2)}_+}$ over
them, hence $\widehat P_{3,2}=\sum_d(qt)^d\,r_d$ with $r_d$ the polynomial in the
last column.  The result is the polynomial $P_{3,2}$ of \eqref{eq:P3}.  For
instance the seven paths with $\dinv_2=0$ are the paths $e_1=e_2=e_3=e$ with three
consecutive north steps, and the seven paths with $\dinv_2=6$ are the paths with
$b_1\ge b_2\ge b_3\ge b_1-1$, whose three pairs contribute $s_2(0)=s_2(1)=2$ each.

\begin{table}[ht]
\centering\footnotesize
\begin{tabular}{@{}rrp{12.4cm}@{}}
\toprule
$d$ & paths & $r_d=\sum_{\dinv_2=d}q^{A^{(2)}_-}t^{A^{(2)}_+}$\\
\midrule
$0$ & $7$ & $t^{18} + qt^{13} + q^{2}t^{8} + q^{5}t^{5} + q^{8}t^{2} + q^{13}t + q^{18}$\\
$1$ & $14$ & $t^{16} + t^{14} + qt^{11} + qt^{9} + q^{2}t^{8} + q^{2}t^{6} + q^{3}t^{5} + q^{5}t^{3} + q^{6}t^{2} + q^{8}t^{2} + q^{9}t + q^{11}t + q^{14} + q^{16}$\\
$2$ & $21$ & $t^{14} + t^{12} + t^{10} + 2qt^{9} + qt^{7} + 2q^{2}t^{6} + q^{2}t^{4} + q^{3}t^{5} + q^{3}t^{3} + q^{4}t^{2} + q^{5}t^{3} + 2q^{6}t^{2} + q^{7}t + 2q^{9}t + q^{10} + q^{12} + q^{14}$\\
$3$ & $14$ & $t^{12} + t^{10} + 2qt^{7} + q^{2}t^{6} + q^{2}t^{4} + 2q^{3}t^{3} + q^{4}t^{2} + q^{6}t^{2} + 2q^{7}t + q^{10} + q^{12}$\\
$4$ & $14$ & $t^{10} + t^{8} + qt^{7} + qt^{5} + 2q^{2}t^{4} + 2q^{3}t^{3} + 2q^{4}t^{2} + q^{5}t + q^{7}t + q^{8} + q^{10}$\\
$5$ & $7$ & $t^{8} + qt^{5} + q^{2}t^{4} + q^{3}t^{3} + q^{4}t^{2} + q^{5}t + q^{8}$\\
$6$ & $7$ & $t^{6} + qt^{5} + q^{2}t^{4} + q^{3}t^{3} + q^{4}t^{2} + q^{5}t + q^{6}$\\
\bottomrule
\end{tabular}
\caption{The $84$ paths of the $3\times6$ rectangle, $m=2$, grouped by $\dinv_2$.}\label{tab:paths32}
\end{table}

\section{The labelled conjecture and a monomial basis}\label{sec:labelled}

In this section we lift \cref{conj:shuffleB} from an identity of polynomials to an
identity of symmetric functions.  In type $A$ the $q,t$-Fuss--Catalan identity
$\langle\nabla^me_n,e_n\rangle=\sum_{\pi}q^{\dinv_m}t^{\mathrm{area}}$ is the
$\langle\,\cdot\,,e_n\rangle$-coefficient of the shuffle theorem of Carlsson and
Mellit \cite{CarlssonMellit}, conjectured in \cite{HHLRU}.  There the paths acquire
labels, each labelled path contributes a fundamental quasisymmetric function, and the
numerical statement is recovered from the coefficient of one Schur function.  We
formulate the type-$B$ analogue of this passage.

The main idea is to let the labels interact with the kernel $s_m$ of
\eqref{eq:kernel}.  The support of $s_m$ splits into two halves, and an inversion of
the labels cancels one unit of $\dinv_m$ on one half, a non-inversion on the other.
The resulting labelled series $\mathcal H^+_{n,m}$ is a sum of fundamental
quasisymmetric functions, and nothing in its definition says that it is symmetric.

To prove that it is symmetric and Schur positive we follow the type-$A$ strategy of
Haglund, Haiman, Loehr, Remmel and Ulyanov \cite[Thm.~3.1.3 and \S5]{HHLRU}, see
also \cite[Ch.~6]{HaglundBook}.  The contribution of one path is identified with an
LLT polynomial of Lascoux, Leclerc and Thibon \cite{LLT97} (\cref{thm:llt}), the
symmetry of LLT polynomials is a theorem, and their Schur positivity, which for the
tuples we need rests on the parabolic Kazhdan--Lusztig polynomials
\cite[Prop.~5.3.1]{HHLRU}, \cite{LT,KT}, gives the Schur positivity of
$\mathcal H^+_{n,m}$ (\cref{prop:labelled-props}).

The type-$B$ features are that
the heights of a path take both signs and that the kernel $s_m$ is supported on $2m$
values, and the identification requires splitting every height into a content and a
colour.

At $m=1$ we prove a type-$C$ analogue of the schedule formula of Haglund and Loehr
\cite{HL}, \cite[Thm.~5.3]{HaglundBook}.  The series $\mathcal H^+_{n,1}$ is a sum
over the permutations of $\{0,1,\dots,n\}$ of a monomial in $q,t$ times a product
of $qt$-integers (\cref{thm:rooted-insertion}), and the two extreme sectors of the
sum are the type-$A$ formula itself (\cref{prop:typeA-sectors}).  In type $A$ the
schedule formula was the key to the monomial basis of the diagonal coinvariants found
by Carlsson and Oblomkov \cite[Thm.~B]{CO}, whose specialization $x_i\mapsto q$,
$y_i\mapsto t$ is the schedule formula \cite[\S4]{HagSergel}.  This is the origin
of the second conjecture below.

There are two conjectures.  The first, \cref{conj:labelled}, is a symmetric-function
identity whose $s_n$-coefficient is exactly \cref{conj:shuffleB}
(\cref{prop:sn-coefficient}).  Conjecturally it computes the bigraded Frobenius
character of the $G_n$-invariant part of Gordon's module, of which the spherical part
of \cref{thm:cherednik} is the $S_n$-invariant part.  The second,
\cref{conj:rooted-basis}, is a monomial basis for that module at $m=1$, indexed by
the permutations of $\{0,1,\dots,n\}$ and by a box of independent schedule choices
attached to each of them.  It is the type-$B$ analogue of the basis of Carlsson and
Oblomkov \cite[Thm.~B]{CO} for the type-$A$ diagonal coinvariants, and it implies the
Hilbert-series half of the first conjecture at $m=1$.

Neither conjecture implies the other, and neither is proved here.  What we prove is
the comparison with \cref{sec:shuffle} and the elementary properties of the labelled
model.  The evidence is collected in \cref{thm:labelled-evidence} and
\cref{rem:geometric-evidence}.

\subsection{Labelled paths}\label{sec:labelled-paths}

Let us fix the notation.  We keep $m\ge1$, $D=mn^2$, the heights \eqref{eq:heights}
and the statistics \eqref{eq:pathstats}.  For a set $S\subseteq\{1,\dots,n-1\}$ let
\[
 F_{n,S}(X)=\sum_{\substack{i_1\le\dots\le i_n\\ i_j<i_{j+1}\ \text{if}\ j\in S}}
 x_{i_1}\cdots x_{i_n}
\]
be Gessel's fundamental quasisymmetric function, and for $w\in S_n$ let
$\operatorname{ides}(w)=\{j:w^{-1}(j)>w^{-1}(j+1)\}$ be the inverse descent set of
$w$.  We use the expansion $s_\lambda=\sum_{Q}F_{n,\mathrm{Des}(Q)}$ of the Schur
functions, the sum over the standard Young tableaux $Q$ of shape $\lambda$
\cite[Prop.~2.4.1]{HHLRU}, and the linear independence of the $F_{n,S}$.

\begin{definition}\label{def:labelled}
A \emph{labelled path} of the $n\times mn$ rectangle is a pair $P=(e,\sigma)$ with
$0\le e_1\le\dots\le e_n\le mn$ and $\sigma\in S_n$, subject to
\[
 e_i=e_{i+1}\ \Longrightarrow\ \sigma_i<\sigma_{i+1}.
\]
We write $\mathsf{PF}^+_{n,m}$ for the set of labelled paths.  It has $(mn+1)^n$
elements.  The \emph{reading word} $w(P)$ lists the labels $\sigma_i$ in increasing
order of $(b_i,i)$.  The \emph{defect} and the \emph{labelled $\dinv$} are
\begin{equation}\label{eq:label-defect}
 \operatorname{def}_m(P)=\sum_{i<j}\Bigl(
 \mathbf 1_{1-m\le b_i-b_j\le0}\mathbf 1_{\sigma_i>\sigma_j}
 +\mathbf 1_{1\le b_i-b_j\le m}\mathbf 1_{\sigma_i<\sigma_j}\Bigr),
 \qquad
 \dinv^{\mathrm{lab}}_m=\dinv_m-\operatorname{def}_m ,
\end{equation}
and the \emph{labelled path series} is
\begin{equation}\label{eq:Hplus}
 \mathcal H^+_{n,m}(X;q,t)=\sum_{P\in\mathsf{PF}^+_{n,m}}
 q^{A^{(m)}_-(P)+\dinv^{\mathrm{lab}}_m(P)}\,
 t^{A^{(m)}_+(P)+\dinv^{\mathrm{lab}}_m(P)}\,F_{n,\operatorname{ides}(w(P))}(X).
\end{equation}
\end{definition}

The three path statistics of \eqref{eq:pathstats} depend only on $e$, and the labels
enter only through $\operatorname{def}_m$ and $\operatorname{ides}$.  The two summands
of \eqref{eq:label-defect} are the two halves of the support of the kernel $s_m$ of
\eqref{eq:kernel}.  Namely, a pair $i<j$ with $b_i-b_j\in[1-m,0]$ is an attack in
which the later north step is weakly higher, and it is cancelled by an inversion of
the labels, while a pair with $b_i-b_j\in[1,m]$ is an attack in which the earlier
north step is higher, and it is cancelled by a non-inversion.

Since $s_m(d)=[m+d]_+$
for $d\le0$ and $s_m(d)=[m+1-d]_+$ for $d\ge1$, we have $s_m(d)\ge1$ exactly on
$1-m\le d\le m$, and the two summands of \eqref{eq:label-defect} partition this
range.  Hence $\operatorname{def}_m\le\dinv_m$ pair by pair, and
$\dinv^{\mathrm{lab}}_m\ge0$.  At $m=1$ the kernel is an indicator function and
\eqref{eq:Hplus} simplifies to
\begin{gather}
 \mathcal H^+_{n,1}=\sum_P q^{A_-}t^{A_+}(qt)^{D(P)}F_{n,\operatorname{ides}(w(P))},
 \notag\\
 \label{eq:D-one}
 D(P)=\#\{i<j:b_i=b_j,\ \sigma_i<\sigma_j\}
 +\#\{i<j:b_i=b_j+1,\ \sigma_i>\sigma_j\}.
\end{gather}

\subsection{Worked examples}\label{sec:labelled-examples}

In this subsection we compute the statistics on three labelled paths, and we write
out the whole of $\mathcal H^+_{2,1}$.  We draw a labelled path as the lattice path
from $(0,0)$ to $(mn,n)$ whose $i$-th north step joins $(e_i,i-1)$ to $(e_i,i)$.  The
label $\sigma_i$ is circled in the cell immediately east of that step, the dashed
line is $y=x/m$, and the column of small numbers at the left lists the heights
$b_1,\dots,b_n$, read upwards.  Thus the number in the $i$-th row is the height $b_i$
of the $i$-th north step, which by \eqref{eq:heights} is the height of the path above
the line after that step, measured in units of $1/m$.  Unlike in type $A$ the path is
an arbitrary lattice path of the rectangle, hence the $b_i$ take both signs and the
line is crossed in both directions.  This is what the statistics have to accommodate.

\begin{example}\label{ex:pathA}
Let $n=4$, $m=1$, $e=(0,3,3,4)$ and $\sigma=(3,1,4,2)$.  The picture is
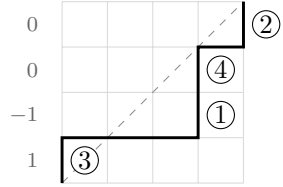
\[
\begin{tikzpicture}[scale=0.6,baseline=(current bounding box.center)]
  \draw[step=1,black!15,very thin] (0,0) grid (4,4);
  \draw[black!45,dashed] (0,0) -- (4,4);
  \draw[line width=1.1pt] (0,0) -- (0,1) -- (3,1) -- (3,2) -- (3,3) -- (4,3) -- (4,4);
  \node[draw,circle,inner sep=0.8pt,fill=white,font=\small] at (0.5,0.5) {$3$};
  \node[draw,circle,inner sep=0.8pt,fill=white,font=\small] at (3.5,1.5) {$1$};
  \node[draw,circle,inner sep=0.8pt,fill=white,font=\small] at (3.5,2.5) {$4$};
  \node[draw,circle,inner sep=0.8pt,fill=white,font=\small] at (4.5,3.5) {$2$};
  \node[font=\scriptsize,black!60,anchor=east] at (-0.30,0.5) {$1$};
  \node[font=\scriptsize,black!60,anchor=east] at (-0.30,1.5) {$-1$};
  \node[font=\scriptsize,black!60,anchor=east] at (-0.30,2.5) {$0$};
  \node[font=\scriptsize,black!60,anchor=east] at (-0.30,3.5) {$0$};
\end{tikzpicture}
\]
Here $b=(1,-1,0,0)$.  The only equality among the $e_i$ is $e_2=e_3$, and there
$\sigma_2=1<4=\sigma_3$.  Hence $P$ is admissible, the labels increase up the vertical
run.

\emph{Areas.}  The summands of $A^{(1)}_-$ are
$[1-b_i]_++[-b_i]_+=0,3,1,1$ and those of $A^{(1)}_+$ are
$[b_i]_++[b_i-1]_+=1,0,0,0$, hence $A^{(1)}_-=5$ and $A^{(1)}_+=1$.

\emph{Kernel and $\dinv$.}  At $m=1$ the kernel \eqref{eq:kernel} is the indicator
of $d\in\{0,1\}$, hence $\dinv_1$ counts the pairs $i<j$ with $b_i-b_j\in\{0,1\}$.
These are $(1,3)$ and $(1,4)$, with $d=1$, and $(3,4)$, with $d=0$.  Thus
$\dinv_1=3$.  The pairs $(1,2)$, $(2,3)$, $(2,4)$, with $d=2,-1,-1$, do not attack.

\emph{Defect.}  A pair with $d=0$ is counted by \eqref{eq:label-defect} when its
labels are in decreasing order, a pair with $d=1$ when they are in increasing
order.  Here $\sigma_1=3<4=\sigma_3$ (counted), $\sigma_1=3>2=\sigma_4$ (not
counted), and $\sigma_3=4>2=\sigma_4$ (counted).  Therefore $\operatorname{def}_1=2$
and $\dinv^{\mathrm{lab}}_1=1$.

\emph{Reading word.}  Sorting by increasing $(b_i,i)$ visits the steps in the order
$2,3,4,1$, hence $w(P)=1\,4\,2\,3$.  In this word $4$ precedes $3$ and no other
letter precedes its predecessor, hence $\operatorname{ides}(w(P))=\{3\}$.  The
contribution of $P$ to \eqref{eq:Hplus} is therefore $q^{6}t^{2}F_{4,\{3\}}$.
\end{example}

The next example has almost the same heights but $m=2$, which widens the attacking
window from two values of $d$ to $2m=4$ and lets the kernel exceed $1$.

\begin{example}\label{ex:pathB}
Let $n=3$, $m=2$, $e=(1,5,6)$ and $\sigma=(2,3,1)$.  The picture is
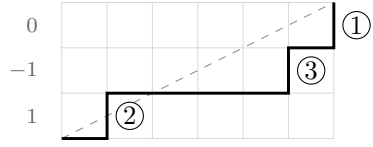
\[
\begin{tikzpicture}[scale=0.6,baseline=(current bounding box.center)]
  \draw[step=1,black!15,very thin] (0,0) grid (6,3);
  \draw[black!45,dashed] (0,0) -- (6,3);
  \draw[line width=1.1pt] (0,0) -- (1,0) -- (1,1) -- (5,1) -- (5,2) -- (6,2) -- (6,3);
  \node[draw,circle,inner sep=0.8pt,fill=white,font=\small] at (1.5,0.5) {$2$};
  \node[draw,circle,inner sep=0.8pt,fill=white,font=\small] at (5.5,1.5) {$3$};
  \node[draw,circle,inner sep=0.8pt,fill=white,font=\small] at (6.5,2.5) {$1$};
  \node[font=\scriptsize,black!60,anchor=east] at (-0.30,0.5) {$1$};
  \node[font=\scriptsize,black!60,anchor=east] at (-0.30,1.5) {$-1$};
  \node[font=\scriptsize,black!60,anchor=east] at (-0.30,2.5) {$0$};
\end{tikzpicture}
\]
Here $b=(1,-1,0)$.  Now $s_2(d)=2$ for $d\in\{0,1\}$, $s_2(d)=1$ for $d\in\{-1,2\}$
and $s_2(d)=0$ otherwise.  The three pairs have $d=2,1,-1$, hence all three attack
and $\dinv_2=1+2+1=4$.  At $m=1$ the pairs with $d=2$ and $d=-1$ would not have
attacked at all.

The pairs with $d=2$ and $d=1$ lie in the second range of
\eqref{eq:label-defect} and are counted when the labels increase, the pair with
$d=-1$ lies in the first range and is counted when they decrease.  Here
$\sigma_1=2<3=\sigma_2$ (counted), $\sigma_1=2>1=\sigma_3$ (not counted), and
$\sigma_2=3>1=\sigma_3$ (counted).  Therefore $\operatorname{def}_2=2$ and
$\dinv^{\mathrm{lab}}_2=2$.  Note that the pair $(1,3)$ contributes $2$ to
$\dinv_2$ and $0$ to $\operatorname{def}_2$.  The inequality
$\operatorname{def}_m\le\dinv_m$ holds pair by pair, but it is far from tight.

The
areas are $A^{(2)}_-=1+4+2=7$ and $A^{(2)}_+=1+0+0=1$, the reading word is
$w(P)=3\,1\,2$ with $\operatorname{ides}=\{2\}$, and $P$ contributes
$q^{9}t^{3}F_{3,\{2\}}$.
\end{example}

The complementation of \cref{lem:complementation} is visible in the picture.  It
turns the path by a half-turn about the centre of the rectangle and replaces each
label $c$ by $n+1-c$.  On \cref{ex:pathA} the two operations on the labels cancel.

\begin{example}\label{ex:pathC}
Let $P$ be as in \cref{ex:pathA}.  Then $e^c_i=4-e_{5-i}$ gives $e^c=(0,1,1,4)$ and
$\sigma^c_i=5-\sigma_{5-i}$ gives $\sigma^c=(3,1,4,2)=\sigma$.  The two pictures are
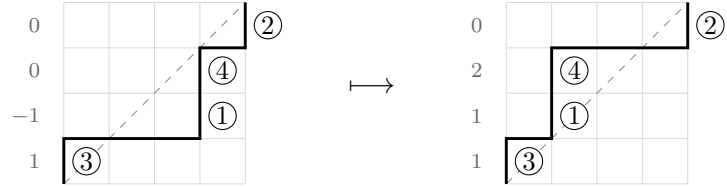
\[
\begin{tikzpicture}[scale=0.6,baseline=(current bounding box.center)]
  \draw[step=1,black!15,very thin] (0,0) grid (4,4);
  \draw[black!45,dashed] (0,0) -- (4,4);
  \draw[line width=1.1pt] (0,0) -- (0,1) -- (3,1) -- (3,2) -- (3,3) -- (4,3) -- (4,4);
  \node[draw,circle,inner sep=0.8pt,fill=white,font=\small] at (0.5,0.5) {$3$};
  \node[draw,circle,inner sep=0.8pt,fill=white,font=\small] at (3.5,1.5) {$1$};
  \node[draw,circle,inner sep=0.8pt,fill=white,font=\small] at (3.5,2.5) {$4$};
  \node[draw,circle,inner sep=0.8pt,fill=white,font=\small] at (4.5,3.5) {$2$};
  \node[font=\scriptsize,black!60,anchor=east] at (-0.30,0.5) {$1$};
  \node[font=\scriptsize,black!60,anchor=east] at (-0.30,1.5) {$-1$};
  \node[font=\scriptsize,black!60,anchor=east] at (-0.30,2.5) {$0$};
  \node[font=\scriptsize,black!60,anchor=east] at (-0.30,3.5) {$0$};
\end{tikzpicture}
\qquad\longmapsto\qquad
\begin{tikzpicture}[scale=0.6,baseline=(current bounding box.center)]
  \draw[step=1,black!15,very thin] (0,0) grid (4,4);
  \draw[black!45,dashed] (0,0) -- (4,4);
  \draw[line width=1.1pt] (0,0) -- (0,1) -- (1,1) -- (1,2) -- (1,3) -- (4,3) -- (4,4);
  \node[draw,circle,inner sep=0.8pt,fill=white,font=\small] at (0.5,0.5) {$3$};
  \node[draw,circle,inner sep=0.8pt,fill=white,font=\small] at (1.5,1.5) {$1$};
  \node[draw,circle,inner sep=0.8pt,fill=white,font=\small] at (1.5,2.5) {$4$};
  \node[draw,circle,inner sep=0.8pt,fill=white,font=\small] at (4.5,3.5) {$2$};
  \node[font=\scriptsize,black!60,anchor=east] at (-0.30,0.5) {$1$};
  \node[font=\scriptsize,black!60,anchor=east] at (-0.30,1.5) {$1$};
  \node[font=\scriptsize,black!60,anchor=east] at (-0.30,2.5) {$2$};
  \node[font=\scriptsize,black!60,anchor=east] at (-0.30,3.5) {$0$};
\end{tikzpicture}
\]
and $b^c=(1,1,2,0)=(1-b_4,1-b_3,1-b_2,1-b_1)$.  The attacking pairs are
transported by $(i,j)\mapsto(5-j,5-i)$.  The pair $(3,4)$ of $P$, of type $d=0$ and
counted, becomes the pair $(1,2)$ of $P^c$, again of type $d=0$ and again counted,
and the two pairs with $d=1$ become the two pairs $(1,4)$ and $(2,4)$ with $d=1$, of
which again exactly one is counted.  Thus $\dinv_1=3$ and $\operatorname{def}_1=2$
are unchanged, while $A^{(1)}_-$ and $A^{(1)}_+$ exchange the values $5$ and $1$
and the weight becomes $q^{2}t^{6}$.  The reading word is $w(P^c)=2\,3\,1\,4$ with
$\operatorname{ides}=\{1\}$, and indeed $1=n-3$.
\end{example}

Finally we write out the whole of $\mathcal H^+_{2,1}$, which the reader can compare
with \cref{ex:lowrank}.  The nine labelled paths of the $2\times2$ square, with
$F_{2,\varnothing}=s_2$ and $F_{2,\{1\}}=s_{11}$, are
\[\small
\begin{tabular}{@{}ccccccccc@{}}\toprule
$e$ & $\sigma$ & $b$ & $A^{(1)}_-$ & $A^{(1)}_+$ & $\dinv_1$ &
$\operatorname{def}_1$ & $w(P)$ & contribution\\\midrule
$00$ & $12$ & $1,2$ & $0$ & $4$ & $0$ & $0$ & $12$ & $t^{4}s_2$\\
$01$ & $12$ & $1,1$ & $0$ & $2$ & $1$ & $0$ & $12$ & $qt^{3}s_2$\\
$01$ & $21$ & $1,1$ & $0$ & $2$ & $1$ & $1$ & $21$ & $t^{2}s_{11}$\\
$02$ & $12$ & $1,0$ & $1$ & $1$ & $1$ & $1$ & $21$ & $qt\,s_{11}$\\
$02$ & $21$ & $1,0$ & $1$ & $1$ & $1$ & $0$ & $12$ & $q^{2}t^{2}s_2$\\
$11$ & $12$ & $0,1$ & $1$ & $1$ & $0$ & $0$ & $12$ & $qt\,s_2$\\
$12$ & $12$ & $0,0$ & $2$ & $0$ & $1$ & $0$ & $12$ & $q^{3}t\,s_2$\\
$12$ & $21$ & $0,0$ & $2$ & $0$ & $1$ & $1$ & $21$ & $q^{2}s_{11}$\\
$22$ & $12$ & $-1,0$ & $4$ & $0$ & $0$ & $0$ & $12$ & $q^{4}s_2$\\
\bottomrule
\end{tabular}
\]
Summing the last column we obtain
\[
 \mathcal H^+_{2,1}
 =\bigl(q^4+q^3t+q^2t^2+qt^3+t^4+qt\bigr)s_2+\bigl(q^2+qt+t^2\bigr)s_{11}.
\]
Its $s_2$-coefficient is $\widehat P_{2,1}$, in accordance with
\cref{lem:labelled-basic}\textup{(i)}, and setting $q=t=1$ gives
$6s_2+3s_{11}=h_2[3X]$, as in \cref{lem:labelled-basic}\textup{(iii)}.

\begin{lemma}\label{lem:labelled-basic}
Let $n,m\ge1$.
\begin{enumerate}[label=\textup{(\roman*)}]
\item Every lattice path $\pi$ of the rectangle carries exactly one labelling $P$
with $\operatorname{ides}(w(P))=\varnothing$, namely the one whose reading word is the
identity, and for it $\operatorname{def}_m(P)=0$.  Consequently
\[
 [s_n]\,\mathcal H^+_{n,m}=\widehat P_{n,m}.
\]
\item $\langle\mathcal H^+_{n,m},h_1^n\rangle
 =\sum_{P}q^{A^{(m)}_-+\dinv^{\mathrm{lab}}_m}t^{A^{(m)}_++\dinv^{\mathrm{lab}}_m}$,
the weight enumerator of all $(mn+1)^n$ labelled paths.
\item $\mathcal H^+_{n,m}(X;1,1)=h_n[(mn+1)X]$, the Frobenius character of the
permutation representation of $S_n$ on the functions $[n]\to\{0,1,\dots,mn\}$.
\end{enumerate}
\end{lemma}

\begin{proof}
(i) Fix $\pi$, hence $e$ and $b$.  The reading order is the linear order on
$\{1,\dots,n\}$ given by $i\prec j$ iff $(b_i,i)<(b_j,j)$ lexicographically.  The
condition $\operatorname{ides}(w)=\varnothing$ holds iff $w^{-1}$ has no descent, that
is iff $w$ is the identity, that is iff $\sigma$ is increasing along $\prec$.  This
determines $\sigma$ uniquely, and this $\sigma$ is admissible.  Indeed, if
$e_i=e_{i+1}$ then $b_{i+1}=b_i+m>b_i$, hence $i\prec i+1$ and
$\sigma_i<\sigma_{i+1}$.

For this $\sigma$ and $i<j$ we have $\sigma_i>\sigma_j$ iff
$j\prec i$ iff $b_j<b_i$, since ties in $b$ are broken by the index.  In the first sum
of \eqref{eq:label-defect} the two conditions $b_i-b_j\le0$ and $b_i-b_j>0$ are
incompatible, and in the second sum the two conditions $b_i-b_j\ge1$ and
$b_i-b_j\le0$ are incompatible.  Hence $\operatorname{def}_m(P)=0$, and the weight of
$P$ is $(qt)^{\dinv_m(\pi)}q^{A^{(m)}_-(\pi)}t^{A^{(m)}_+(\pi)}$, the term of $\pi$ in
\eqref{eq:Phat-path}.

Finally, $s_\lambda=\sum_{Q}F_{n,\mathrm{Des}(Q)}$ over the
standard Young tableaux $Q$ of shape $\lambda$, and the only shape with a tableau of
empty descent set is the single row.  Therefore a symmetric $f=\sum_Sc_SF_{n,S}$ has
$c_\varnothing=[s_n]f$.

(ii) Each $F_{n,S}$ contains the monomial $x_1\cdots x_n$ exactly once, and
$\langle\,\cdot\,,h_1^n\rangle$ is the coefficient of $x_1\cdots x_n$.

(iii) Fix a path $e$, let $R_1,\dots,R_k$ be its maximal vertical runs, the maximal
blocks of consecutive indices $i$ with equal $e_i$, and let $C_s\subseteq\{1,\dots,n\}$
be the set of positions in the reading order occupied by the steps of $R_s$.  Along a
run the heights increase, hence the steps of $R_s$ appear in the reading order in the
order of their indices, and a labelling $\sigma$ is admissible iff its labels increase
along each run.  Therefore the words $w(e,\sigma)$, $\sigma$ admissible for $e$, are
exactly the permutations $u\in S_n$ which are increasing on each block $C_s$.

For a
permutation $u$ the function $F_{n,\operatorname{ides}(u)}$ is the sum of the
monomials $x_{v_1}\cdots x_{v_n}$ over the sequences $v_1\le\dots\le v_n$ with
$v_c<v_{c+1}$ whenever $u^{-1}(c)>u^{-1}(c+1)$.

We claim that the pairs $(u,v)$
correspond bijectively to the functions $g\colon\{1,\dots,n\}\to\{1,2,\dots\}$, by
$g(p)=v_{u(p)}$, and that under this correspondence $u$ is increasing on a block $C$
exactly when $g$ is weakly increasing on $C$.

Indeed, given $g$ let $u$ be its
standardization with ties broken from left to right, that is $u(p)<u(p')$ iff
$g(p)<g(p')$, or $g(p)=g(p')$ and $p<p'$, and let $v_c=g(u^{-1}(c))$.  Then $v$ is
weakly increasing, and if $u^{-1}(c)>u^{-1}(c+1)$ then $v_c\ne v_{c+1}$, since an
equality would be broken the other way, hence $v_c<v_{c+1}$.

Conversely, given
$(u,v)$ put $g(p)=v_{u(p)}$.  If $u(p)<u(p')$ then $g(p)\le g(p')$, and in case of
equality all $v_c$ with $u(p)\le c\le u(p')$ coincide, hence no $c$ in
$\{u(p),\dots,u(p')-1\}$ lies in $\operatorname{ides}(u)$, hence $p<p'$.  Thus $u$ is
the standardization of $g$, and the two constructions are inverse to each other.

Finally, if $g$ is weakly increasing on $C$ then its standardization is increasing on
$C$, and if $u$ is increasing on $C$ then $g=v\circ u$ is weakly increasing on $C$.
This proves the claim, and it gives
\[
 \sum_{\sigma}F_{n,\operatorname{ides}(w(e,\sigma))}(X)
 =\sum_{g}\ \prod_{p=1}^nx_{g(p)}=\prod_{s=1}^kh_{|R_s|}(X),
\]
the middle sum over the functions $g$ which are weakly increasing on each block
$C_s$.

A path $e$ is a multiset of $n$ values from $\{0,1,\dots,mn\}$, and the sizes
of its runs are the multiplicities $a_0,\dots,a_{mn}$ of the values.  Summing over
the paths we obtain
\[
 \mathcal H^+_{n,m}(X;1,1)=\sum_{a_0+\dots+a_{mn}=n}\ \prod_{j=0}^{mn}h_{a_j}(X)
 =h_n\bigl[(mn+1)X\bigr],
\]
by the expansion $h_n[(y_0+\dots+y_{mn})X]=\sum_a\prod_jy_j^{a_j}h_{a_j}[X]$ at
$y_j=1$.  The permutation representation of $S_n$ on the functions
$[n]\to\{0,\dots,mn\}$ is the direct sum over the multisets of the representations
induced from the trivial representation of the Young subgroups $\prod_jS_{a_j}$,
whose Frobenius character is the same product.
\end{proof}

Since $A^{(m)}_\pm$ and $\dinv_m$ depend only on $e$, the sum \eqref{eq:Hplus}
groups by path,
\begin{equation}\label{eq:path-grouping}
 \mathcal H^+_{n,m}(X;q,t)=\sum_{e}q^{A^{(m)}_-(e)}\,t^{A^{(m)}_+(e)}\,
 (qt)^{\dinv_m(e)}\;L_e\bigl(X;(qt)^{-1}\bigr),
\end{equation}
where $e$ runs over the $\binom{(m+1)n}{n}$ paths of the rectangle and
\begin{equation}\label{eq:Le}
 L_e(X;z)=\sum_{\sigma}z^{\operatorname{def}_m(e,\sigma)}\,
 F_{n,\operatorname{ides}(w(e,\sigma))}(X),
\end{equation}
the inner sum over the labellings admissible for $e$.  The polynomial $L_e$ is the
part of $\mathcal H^+_{n,m}$ contributed by the labels of one path.  It is a
quasisymmetric function of degree $n$ with coefficients in $\mathbb N[z]$, and nothing
in its definition says that it is symmetric.

The aim of the rest of this subsection is to identify $L_e$ with an LLT polynomial.
This is the type-$B$ counterpart of the LLT expansion of the type-$A$ shuffle formula
\cite[Thm.~3.1.3]{HHLRU}, \cite[Ch.~6]{HaglundBook}, and it is the tool by which the
two properties of a Frobenius character, symmetry and Schur positivity, are proved
for $\mathcal H^+_{n,m}$ in \cref{prop:labelled-props} below.  The main idea is to
split every height into its quotient and its remainder modulo $m$, $b_i=mq_i+r_i$.
The quotient $q_i$ becomes the content of a cell and the remainder $r_i$ orders the
components, and the $2m$ values of $b_i-b_j$ on which the kernel $s_m$ is supported
collapse to the two values of the attacking relation of LLT theory.

We use the LLT polynomials in the form of the definition of Blasiak, Haiman, Morse,
Pun and Seelinger \cite[Def.~4.1.1]{BHMPS}.  Let
$\boldsymbol\nu=(\nu^{(1)},\dots,\nu^{(k)})$ be a tuple of skew diagrams.  A cell
of $\boldsymbol\nu$ is a cell of one of the components, the components being kept
apart, hence $\boldsymbol\nu$ has $|\nu^{(1)}|+\dots+|\nu^{(k)}|$ cells.  Let $c(u)$
be the content of a cell $u$, that is its column index minus its row index, and for
$u\in\nu^{(s)}$ put
\[
 \tilde c(u)=c(u)+\varepsilon s,\qquad 0<\varepsilon<1/k,
\]
the \emph{adjusted content}.

The \emph{reading order} lists all the cells of
$\boldsymbol\nu$ by increasing $\tilde c$, that is by increasing content, the cells of
one content in the order of their components, and cells of equal content in one
component, which do not occur for the tuples of rows below, in any order.

Two cells
$u$ and $v$ of $\boldsymbol\nu$, in the same component or in different ones, form
an \emph{attacking pair} $(u,v)$ if $0<\tilde c(v)-\tilde c(u)<1$.  The pair is
ordered, $u$ being the cell which comes first in the reading order, and the condition
says that $v$ follows $u$ by less than one unit of content.  Explicitly, $(u,v)$ is
attacking if either $c(v)=c(u)$ and $v$ lies in a later component than $u$, or
$c(v)=c(u)+1$ and $v$ lies in an earlier component than $u$.  Two cells of the same
component never attack.

A \emph{semistandard filling} $T$ of $\boldsymbol\nu$ is a
semistandard filling of each component, and an attacking pair $(u,v)$ is an
\emph{inversion} of $T$ if the earlier cell carries the larger entry, $T(u)>T(v)$.
Writing $\operatorname{inv}(T)$ for the number of inversions of $T$ and $x^T$ for
the product of the $x_{T(u)}$ over the cells, the LLT polynomial is
$G_{\boldsymbol\nu}(X;z)=\sum_Tz^{\operatorname{inv}(T)}x^T$, summed over the
semistandard fillings, and it is a symmetric function \cite[\S4.1]{BHMPS}.

When every component is a single row, a semistandard filling is a weakly increasing
filling of each row, and a standard filling, one with the entries $1,\dots,n$, is a
filling which increases along each row.

\begin{construction}[The tuple of rows of a path]\label{def:nu-of-e}
Let $e$ be a path of the $n\times mn$ rectangle with heights $b_i=mi-e_i$.  The tuple
$\boldsymbol\nu(e)$ is built in three steps.
\begin{enumerate}
\item[(a)] Split every height as $b_i=mq_i+r_i$ with $q_i=\lfloor b_i/m\rfloor$ and
 $0\le r_i<m$.  We call $q_i$ the \emph{content} and $r_i$ the \emph{colour} of the
 $i$-th north step.  At $m=1$ every colour is $0$ and $q_i=b_i$.
\item[(b)] Cut the sequence of north steps into the maximal vertical runs
 $R_1,\dots,R_k$ of the path, the maximal blocks of consecutive indices $i$ with the
 same $e_i$.  Along a run $b$ increases by $m$ at each step, hence $q$ increases by
 $1$ and the colour is constant.  We write $r(R)$ for the colour of the run $R$.
\item[(c)] Replace the run $R=\{i_1<\dots<i_l\}$ by a single row of $l$ cells, the
 cell of the step $i$ being placed at content $q_i$, hence the row occupies the
 consecutive contents $q_{i_1},q_{i_1}+1,\dots,q_{i_1}+l-1$.  The row has as many
 cells as the run has north steps.  Order the rows by
 increasing colour, and the rows of one colour by increasing first index, that is in
 the order in which they occur along the path.
\end{enumerate}
The resulting ordered tuple of rows is $\boldsymbol\nu(e)=(\nu^{(1)},\dots,\nu^{(k)})$.
It has one component for each vertical segment of the path, that is $k$ is the number
of distinct values among $e_1,\dots,e_n$, the sizes $|\nu^{(1)}|,\dots,|\nu^{(k)}|$ are
the lengths of the segments, that is the multiplicities of these values, and they add
up to $n$.  Every cell is the cell of one north step $i$, and $c(i)=q_i$.
\end{construction}

Every component of $\boldsymbol\nu(e)$ is a row, hence a horizontal strip.  A
standard filling of $\boldsymbol\nu(e)$ assigns the entries $1,\dots,n$ to the north
steps so that they increase along each run, and by \cref{def:labelled} this is exactly
a labelling $\sigma$ admissible for $e$, the label $\sigma_i$ being placed in the cell
of the step $i$.  The tuple forgets the heights themselves, but it keeps what the
kernel and the defect see, namely the contents $q_i$ and the colours $r_i$, which
together decide whether two steps attack and in which order.

\begin{example}\label{ex:tuple-big}
Let $n=6$, $m=2$ and $e=(3,3,4,8,8,8)$, a path of the $6\times12$ rectangle with
three vertical segments $R_1=\{1,2\}$, $R_2=\{3\}$ and $R_3=\{4,5,6\}$, at the
abscissae $3$, $4$ and $8$,
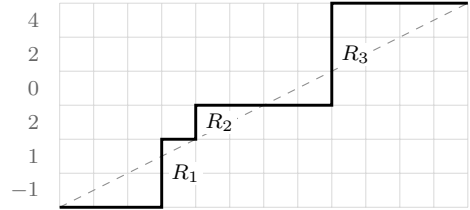
\[
\begin{tikzpicture}[scale=0.45,baseline=(current bounding box.center)]
  \draw[step=1,black!15,very thin] (0,0) grid (12,6);
  \draw[black!45,dashed] (0,0) -- (12,6);
  \draw[line width=1.1pt] (0,0) -- (3,0) -- (3,2) -- (4,2) -- (4,3) -- (8,3) -- (8,6) -- (12,6);
  \node[font=\scriptsize,black!60,anchor=east] at (-0.30,0.5) {$-1$};
  \node[font=\scriptsize,black!60,anchor=east] at (-0.30,1.5) {$1$};
  \node[font=\scriptsize,black!60,anchor=east] at (-0.30,2.5) {$2$};
  \node[font=\scriptsize,black!60,anchor=east] at (-0.30,3.5) {$0$};
  \node[font=\scriptsize,black!60,anchor=east] at (-0.30,4.5) {$2$};
  \node[font=\scriptsize,black!60,anchor=east] at (-0.30,5.5) {$4$};
  \node[fill=white,inner sep=1pt,font=\scriptsize] at (3.7,1.0) {$R_1$};
  \node[fill=white,inner sep=1pt,font=\scriptsize] at (4.7,2.5) {$R_2$};
  \node[fill=white,inner sep=1pt,font=\scriptsize] at (8.7,4.5) {$R_3$};
\end{tikzpicture}
\]
The heights $b_i=2i-e_i$, listed at the left, and their quotients and remainders
modulo $2$ are
\[
\begin{tabular}{@{}lcccccc@{}}\toprule
$i$ & $1$ & $2$ & $3$ & $4$ & $5$ & $6$\\ \midrule
$e_i$ & $3$ & $3$ & $4$ & $8$ & $8$ & $8$\\
$b_i$ & $-1$ & $1$ & $2$ & $0$ & $2$ & $4$\\
$q_i$ & $-1$ & $0$ & $1$ & $0$ & $1$ & $2$\\
$r_i$ & $1$ & $1$ & $0$ & $0$ & $0$ & $0$\\ \bottomrule
\end{tabular}
\]
Along each segment the content rises by one and the colour is constant, as step (b)
says.

The segments $R_1$, $R_2$, $R_3$ have the colours $1$, $0$, $0$.  The two
segments of colour $0$ come first, in the order of their first indices, and the
segment of colour $1$ comes last.  Hence $\nu^{(1)}$ is the row of $R_2$, one cell at
content $1$, $\nu^{(2)}$ is the row of $R_3$, three cells at contents $0,1,2$, and
$\nu^{(3)}$ is the row of $R_1$, two cells at contents $-1,0$,
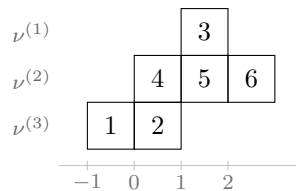
\[
\begin{tikzpicture}[scale=0.62,baseline=(current bounding box.center)]
  \draw[black!25,very thin] (-1.6,-0.35) -- (3.6,-0.35);
  \foreach \c in {-1,0,1,2}
    \node[font=\scriptsize,black!55] at (\c,-0.7) {$\c$};
  \foreach \c in {-1,0,1,2} \draw[black!25,very thin] (\c,-0.45) -- (\c,-0.25);
  \draw (1,2) rectangle (2,3);
  \node[font=\small] at (1.5,2.5) {$3$};
  \node[left,font=\scriptsize,black!60] at (-1.55,2.5) {$\nu^{(1)}$};
  \draw (0,1) rectangle (1,2); \draw (1,1) rectangle (2,2); \draw (2,1) rectangle (3,2);
  \node[font=\small] at (0.5,1.5) {$4$}; \node[font=\small] at (1.5,1.5) {$5$};
  \node[font=\small] at (2.5,1.5) {$6$};
  \node[left,font=\scriptsize,black!60] at (-1.55,1.5) {$\nu^{(2)}$};
  \draw (-1,0) rectangle (0,1); \draw (0,0) rectangle (1,1);
  \node[font=\small] at (-0.5,0.5) {$1$}; \node[font=\small] at (0.5,0.5) {$2$};
  \node[left,font=\scriptsize,black!60] at (-1.55,0.5) {$\nu^{(3)}$};
\end{tikzpicture}
\]
where each cell carries the index of its north step and the horizontal position of a
cell is its content.

The sizes of the components are $1,3,2$, the lengths of the
segments, and they add up to $n=6$.  A labelling $\sigma$ writes $\sigma_i$ in the
cell of the step $i$, and it is admissible exactly when every row increases from left
to right.

Note that the row of $R_3$ precedes the row of $R_1$ in the tuple although
$R_3$ comes later along the path, because the colour is compared first.
\end{example}

At $m=1$ all colours are $0$ and the rows stay in their order along the path.  We
return to the running example, in which the attacking pairs can also be seen.

\begin{example}\label{ex:tuple}
We take the labelled path of \cref{ex:pathA}, with $n=4$, $m=1$, $e=(0,3,3,4)$,
$\sigma=(3,1,4,2)$ and $b=(1,-1,0,0)$.  Since $m=1$ every $r_i$ is $0$ and
$q_i=b_i$.  The maximal vertical runs are $\{1\}$, $\{2,3\}$, $\{4\}$.  They all have
$r=0$, hence they are ordered by their first index.  Thus $\boldsymbol\nu(e)$ is the
triple consisting of one cell of content $1$, a row of two cells of contents $-1,0$,
and one cell of content $0$, carrying the labels $3$, $(1,4)$ and $2$,
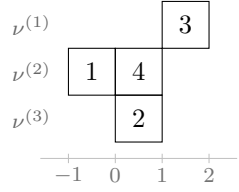
\[
\begin{tikzpicture}[scale=0.62,baseline=(current bounding box.center)]
  \draw[black!25,very thin] (-1.6,-0.35) -- (2.6,-0.35);
  \foreach \c in {-1,0,1,2}
    \node[font=\scriptsize,black!55] at (\c,-0.7) {$\c$};
  \foreach \c in {-1,0,1,2} \draw[black!25,very thin] (\c,-0.45) -- (\c,-0.25);
  \draw (1,2) rectangle (2,3);
  \node[font=\small] at (1.5,2.5) {$3$};
  \node[left,font=\scriptsize,black!60] at (-1.15,2.5) {$\nu^{(1)}$};
  \draw (-1,1) rectangle (0,2); \draw (0,1) rectangle (1,2);
  \node[font=\small] at (-0.5,1.5) {$1$}; \node[font=\small] at (0.5,1.5) {$4$};
  \node[left,font=\scriptsize,black!60] at (-1.15,1.5) {$\nu^{(2)}$};
  \draw (0,0) rectangle (1,1);
  \node[font=\small] at (0.5,0.5) {$2$};
  \node[left,font=\scriptsize,black!60] at (-1.15,0.5) {$\nu^{(3)}$};
\end{tikzpicture}
\]
where the horizontal position of a cell is its content.  With $\varepsilon$ small the
adjusted contents are $\tilde c=1+\varepsilon$ for the label $3$, $-1+2\varepsilon$
and $2\varepsilon$ for the labels $1$ and $4$, and $3\varepsilon$ for the label
$2$.

An ordered pair attacks when the difference of the adjusted contents lies in
$(0,1)$, which here happens for three pairs, namely the labels $4$ and $2$, with
difference $\varepsilon$, the labels $4$ and $3$, with difference $1-\varepsilon$,
and the labels $2$ and $3$, with difference $1-2\varepsilon$.  In terms of the
north steps these are the pairs $(3,4)$, $(1,3)$ and $(1,4)$ of \cref{ex:pathA},
which are exactly the three pairs counted there by $\dinv_1$.  The pair $(2,3)$ of
the same component and the pairs $(1,2)$, $(2,4)$ do not attack, again as there.

Of
the three attacking pairs, the first two carry the larger entry first and are
inversions, the third does not.  Hence $\operatorname{inv}=2=\operatorname{def}_1(P)$,
as \cref{thm:llt} below asserts.
\end{example}

The next theorem says that the defect of an admissible labelling is the inversion
number of the corresponding standard filling, that the reading orders of the path and
of the tuple agree, and consequently that $L_e$ is the LLT polynomial of
$\boldsymbol\nu(e)$.

\begin{theorem}\label{thm:llt}
For all $n,m\ge1$ and every path $e$ of the $n\times mn$ rectangle, the polynomial
$L_e$ of \eqref{eq:Le} is the LLT polynomial of the tuple of rows $\boldsymbol\nu(e)$
of \cref{def:nu-of-e},
\[
 L_e(X;z)=G_{\boldsymbol\nu(e)}(X;z).
\]
In particular $L_e$ is a symmetric function, and by \eqref{eq:path-grouping}
$\mathcal H^+_{n,m}$ is a sum of LLT polynomials, evaluated at $z=(qt)^{-1}$, with
monomial coefficients.
\end{theorem}

\begin{proof}
The main points are in Steps $2$ and $3$, and we first describe them in words.  The
defect \eqref{eq:label-defect} sees a pair $i<j$ of north steps when $b_i-b_j$ lies
in $[1-m,m]$, a range of $2m$ integers, and it counts the pair when the labels form
an inversion, $\sigma_i>\sigma_j$, on the half $[1-m,0]$ and a non-inversion,
$\sigma_i<\sigma_j$, on the half $[1,m]$.  The inversion number of a filling of
$\boldsymbol\nu(e)$ sees a pair of cells when their adjusted contents differ by less
than one, which leaves only two possibilities for their contents, equal or differing
by one, and it counts the pair when the earlier cell in the reading order carries the
larger entry.

Step $2$ shows that the two descriptions agree.  Writing
$b_i-b_j=m\Delta+\rho$ with $\Delta$ the difference of the contents and $\rho$ the
difference of the colours, the range $[1-m,m]$ corresponds to $\Delta\in\{0,1\}$
when $\rho\le0$ and to $\Delta\in\{-1,0\}$ when $\rho\ge1$, where the sign of
$\rho$ is the order of the two components by Step $1$.  In each case the cell read
first is the cell of $i$ on the half $[1-m,0]$, where $b_i\le b_j$, and the cell of
$j$ on the half $[1,m]$, where $b_i>b_j$, so that an inversion of the labels on the
first half and a non-inversion on the second half both say that the earlier cell
carries the larger entry.

Step $3$ shows that the reading order of
$\boldsymbol\nu(e)$, by increasing adjusted content, is the reading order of the
path, by increasing $(b_i,i)$.  Hence the reading words of a labelling and of the
corresponding standard filling have the same inverse descent set, and the passage
from standard to semistandard fillings is the usual standardization.

Write $\pi(i)$
for the index of the run containing $i$.  Then the cell of the label $i$ has
$\tilde c(i)=q_i+\varepsilon\pi(i)$, and for $i<j$ we put
\[
 \Delta=q_i-q_j,\qquad \rho=r_i-r_j\in[1-m,m-1],\qquad b_i-b_j=m\Delta+\rho.
\]

\emph{Step $1$, the component order.}  Let $i<j$ lie in different runs.  If
$r_i<r_j$ then $\pi(i)<\pi(j)$ by the first coordinate of the ordering.  If
$r_i=r_j$ then the two runs are distinct blocks of consecutive indices, hence the one
containing $i$ lies entirely to the left of the one containing $j$ and has the
smaller minimum, whence again $\pi(i)<\pi(j)$.  If $r_i>r_j$ then $\pi(i)>\pi(j)$.
Thus for $i<j$ in different runs,
\begin{equation}\label{eq:pi-vs-r}
 \pi(i)<\pi(j)\iff r_i\le r_j\iff\rho\le0 .
\end{equation}

\emph{Step $2$, attacking pairs and inversions.}  Suppose first that $i<j$ lie in
the same run, say $j=i+s$ with $s\ge1$.  Then $q_j=q_i+s$ and $\pi(i)=\pi(j)$, hence
$\tilde c(i)-\tilde c(j)=-s\le-1$ and $\tilde c(j)-\tilde c(i)=s\ge1$, and neither
ordered pair is attacking.  On the other side $b_i-b_j=-ms\le-m<1-m$, hence the pair
lies in neither range of \eqref{eq:label-defect} and $\operatorname{def}_m$ ignores
it.  Both statistics ignore the pair.

Now let $i<j$ lie in different runs and put $\eta=\pi(j)-\pi(i)\ne0$, hence
\[
 \tilde c(j)-\tilde c(i)=-\Delta+\varepsilon\eta,\qquad |\varepsilon\eta|<1 .
\]
Assume first $\rho\le0$, hence $\eta>0$ by \eqref{eq:pi-vs-r}.  There are four
cases.  If $\Delta=0$ then $\tilde c(j)-\tilde c(i)=\varepsilon\eta\in(0,1)$, hence
$(i,j)$ is attacking and it is an inversion iff $\sigma_i>\sigma_j$.  On the other
side $b_i-b_j=\rho\in[1-m,0]$, and the pair is counted by $\operatorname{def}_m$ iff
$\sigma_i>\sigma_j$.  If $\Delta=1$ then
$\tilde c(i)-\tilde c(j)=1-\varepsilon\eta\in(0,1)$, hence $(j,i)$ is attacking and
it is an inversion iff $\sigma_j>\sigma_i$.  On the other side $b_i-b_j=m+\rho$ lies
in $[1,m]$ because $1-m\le\rho\le0$, and the pair is counted iff
$\sigma_i<\sigma_j$.

If $\Delta\ge2$ then $\tilde c(i)-\tilde c(j)
=\Delta-\varepsilon\eta>2-1=1$ and neither ordered pair is attacking, while
$b_i-b_j=m\Delta+\rho\ge2m+(1-m)=m+1$ lies outside both ranges.  If $\Delta\le-1$
then $\tilde c(j)-\tilde c(i)=-\Delta+\varepsilon\eta>1$ and neither ordered pair is
attacking, while $b_i-b_j=m\Delta+\rho\le-m+0=-m$ lies outside both ranges.

Assume now $\rho\ge1$, hence $\eta<0$.  If $\Delta=0$ then
$\tilde c(i)-\tilde c(j)=-\varepsilon\eta\in(0,1)$, hence $(j,i)$ is attacking and
it is an inversion iff $\sigma_j>\sigma_i$.  On the other side
$b_i-b_j=\rho\in[1,m-1]\subseteq[1,m]$, and the pair is counted iff
$\sigma_i<\sigma_j$.  If $\Delta=-1$ then
$\tilde c(j)-\tilde c(i)=1+\varepsilon\eta\in(0,1)$, hence $(i,j)$ is attacking and
it is an inversion iff $\sigma_i>\sigma_j$.  On the other side
$b_i-b_j=-m+\rho\in[1-m,-1]\subseteq[1-m,0]$, and the pair is counted iff
$\sigma_i>\sigma_j$.

If $\Delta\ge1$ then
$\tilde c(i)-\tilde c(j)=\Delta-\varepsilon\eta>1$, and $b_i-b_j=m\Delta+\rho\ge m+1$
lies outside both ranges.  If $\Delta\le-2$ then
$\tilde c(j)-\tilde c(i)=-\Delta+\varepsilon\eta>2-1=1$, and
$b_i-b_j=m\Delta+\rho\le-2m+(m-1)=-m-1$ lies outside both ranges.

In every case the ordered pair attacks exactly when $\operatorname{def}_m$ sees the
pair, and the inversion condition is exactly the condition counted by
$\operatorname{def}_m$.  Hence for a standard filling $S$ of $\boldsymbol\nu(e)$,
corresponding to the admissible labelling $\sigma$,
\begin{equation}\label{eq:inv-is-def}
 \operatorname{inv}(S)=\operatorname{def}_m(e,\sigma).
\end{equation}

\emph{Step $3$, the reading order and the quasisymmetric expansion.}  The reading
order of $\boldsymbol\nu(e)$ sorts the labels by increasing
$\tilde c(i)=q_i+\varepsilon\pi(i)$, that is by increasing $q_i$ with ties broken by
increasing $\pi(i)$.  Two labels with equal $q$ lie in different runs, since $q$ is
injective on a run, hence \eqref{eq:pi-vs-r} applies to them and breaks the tie by
increasing $(r_i,i)$.  The reading order is therefore the order of increasing
$(q_i,r_i,i)$, which is the order of increasing $(b_i,i)$, the reading order of the
path.  Thus $\operatorname{ides}(S)=\operatorname{ides}(w(e,\sigma))$.

Let $T$ be a semistandard filling of $\boldsymbol\nu(e)$ and let $S=\operatorname{std}(T)$
be its standardization, obtained by replacing the entries by $1,\dots,n$ preserving
their order and breaking ties by the reading order.  Along a row the reading order
increases, hence the weak increase of $T$ along rows becomes a strict increase in
$S$.  Thus $S$ is a standard filling, and every standard filling occurs.

If $(u,v)$
is attacking then $T(u)>T(v)$ iff $S(u)>S(v)$.  Indeed, the two inequalities agree
when $T(u)\ne T(v)$, and when $T(u)=T(v)$ both fail, the second because the tie is
broken by the reading order and $\tilde c(u)<\tilde c(v)$.  Thus $\operatorname{inv}$
is constant on the fibres of $\operatorname{std}$.

Finally, fix $S$ and write
$v_c$ for the entry of $T$ in the cell where $S$ takes the value $c$.  Then
$\operatorname{std}(T)=S$ if and only if $v_1\le v_2\le\dots\le v_n$ with
$v_c<v_{c+1}$ whenever the cell carrying $c+1$ precedes the cell carrying $c$ in the
reading order, that is whenever $c\in\operatorname{ides}(S)$.  Such a $T$ is
automatically semistandard, because $v$ is weakly increasing and $S$ is increasing
along each row.

Summing the monomials of one fibre gives
$F_{n,\operatorname{ides}(S)}$ by the definition of the fundamental quasisymmetric
function.  Therefore
\[
 G_{\boldsymbol\nu(e)}(X;z)=\sum_{S}z^{\operatorname{inv}(S)}F_{n,\operatorname{ides}(S)}(X)
 =\sum_{\sigma}z^{\operatorname{def}_m(e,\sigma)}F_{n,\operatorname{ides}(w(e,\sigma))}(X)
 =L_e(X;z),
\]
where the middle equality is \eqref{eq:inv-is-def} together with the identification
of the reading orders.
\end{proof}

Schur positivity of $G_{\boldsymbol\nu}$ for an arbitrary tuple of skew diagrams is
proved only in the unpublished manuscript of Grojnowski and Haiman
\cite[Cor.~6.9]{GH}.  We do not use it, because the tuples of \cref{def:nu-of-e} lie
in a smaller class for which the statement is published.

\begin{lemma}[Rotation]\label{lem:rotation}
Write $\nu^{(s)}[1]$ for $\nu^{(s)}$ translated so that every content increases by
$1$, and set
$\boldsymbol\nu^{\circlearrowright}
=\bigl(\nu^{(k)}[1],\nu^{(1)},\dots,\nu^{(k-1)}\bigr)$.  Then
$G_{\boldsymbol\nu^{\circlearrowright}}=G_{\boldsymbol\nu}$.  Translating every
component by the same amount also leaves $G_{\boldsymbol\nu}$ unchanged.
\end{lemma}

\begin{proof}
Both operations leave the shape of each component, hence the semistandard condition,
unchanged.  Therefore it suffices to show that they preserve the attacking pairs
together with the order of the two cells in each.  A common translation adds a
constant to every $\tilde c$, which changes no difference.

For the rotation write
$\tilde c'$ for the adjusted content afterwards.  If $u$ and $v$ both lie in
$\nu^{(k)}$, or neither of them does, then their contents and their component indices
are shifted by the same amounts, hence $\tilde c'(u)-\tilde c'(v)=\tilde c(u)-\tilde c(v)$.

Let
$u\in\nu^{(k)}$ and $v\in\nu^{(s)}$ with $s<k$, and put $\delta=c(u)-c(v)$.  Before
the rotation $\tilde c(u)-\tilde c(v)=\delta+\varepsilon(k-s)$ with
$0<\varepsilon(k-s)<1$, hence the ordered pair $(v,u)$ is attacking exactly when
$\delta=0$, and $(u,v)$ exactly when $\delta=-1$.  Afterwards $u$ carries the
component index $1$ and $v$ the index $s+1$, hence
$\tilde c'(u)-\tilde c'(v)=\delta+1-\varepsilon s$ with $0<\varepsilon s<1$, and
again $(v,u)$ is attacking exactly when $\delta=0$ and $(u,v)$ exactly when
$\delta=-1$.
\end{proof}

\begin{lemma}\label{lem:core-form}
For every path $e$ there is a partition $\mu$ such that a tuple obtained from
$\boldsymbol\nu(e)$ by rotations and a common translation is
$\operatorname{quot}_k(\mu)$ over the inner shape $\operatorname{core}_k(\mu)$, in the
normalization of \cite[\S5.1]{HHLRU}.
\end{lemma}

\begin{proof}
By \cite[\S5.1]{HHLRU} a $k$-core $\nu$ is determined by its content sequence
$(s_0,\dots,s_{k-1})$, $s_i\equiv i\pmod k$, the contents of the $k$ ribbons which
can be added to $\nu$.  The partitions $\mu$ with $\operatorname{core}_k(\mu)=\nu$
correspond to the $k$-tuples of partitions $(\mu^{(0)},\dots,\mu^{(k-1)})$, and the
adjusted content of $x\in\mu^{(i)}$ is $\tilde c(x)=k\,c(x)+s_i$ \cite[(71)]{HHLRU}.
Writing $s_i=ka_i+i$, the $i$-th component of the resulting tuple is the partition
$\mu^{(i)}$ with its contents shifted by $a_i$.

In the abacus description of cores
the vectors $(a_0,\dots,a_{k-1})$ which arise from the $k$-cores are exactly the
integer vectors of sum $0$ \cite[Bijection~2]{GKS}.  The components of
$\boldsymbol\nu(e)$ are the one-row partitions $(l_1),\dots,(l_k)$ with contents
shifted by $d_1,\dots,d_k$, hence all that is needed is $\sum_sd_s=0$.  A common
translation changes that sum by a multiple of $k$ and one rotation raises it by $1$.
Thus at most $k-1$ rotations followed by a translation achieve it, and by
\cref{lem:rotation} neither move changes $G$.
\end{proof}

\begin{corollary}\label{cor:llt-positive}
$G_{\boldsymbol\nu(e)}(X;z)\in\bigoplus_\lambda\mathbb N[z]\,s_\lambda$ for every path
$e$.
\end{corollary}

\begin{proof}
Let $\mu$ and $\nu=\operatorname{core}_k(\mu)$ be as in \cref{lem:core-form}.  We
recall the objects of \cite[\S5.1--5.2]{HHLRU}.  A $k$-ribbon is a connected skew
shape of $k$ cells which contains no $2\times2$ square, and a semistandard $k$-ribbon
tableau $T$ of $\mu/\nu$ is a tiling of $\mu/\nu$ by $k$-ribbons together with a
filling of the ribbons which is weakly increasing along the rows and the columns of
$\mu/\nu$ and whose level sets are horizontal ribbon strips with the tiling
prescribed there.  The quotient $\operatorname{quot}_k$ identifies these tableaux
with the semistandard fillings of $\operatorname{quot}_k(\mu)$.

The spin of a ribbon
is one less than its number of rows, hence a horizontal ribbon has spin $0$ and a
vertical one spin $k-1$, and the spin $\operatorname{sp}(T)$ of the tableau is half
the sum of the spins of its ribbons, normalized by subtracting the minimum over all
tableaux of the shape, hence a nonnegative integer.  It measures how far the tiling
is from the flattest tiling of the shape.  For instance the shape $(2,2)$ with $k=2$
has two tilings, by two horizontal dominoes, of spin $0$, and by two vertical
dominoes, of spin $1$.

Let $G_{\mu/\nu}(z;u)=\sum_Tu^{\operatorname{sp}(T)}z^T$ be
the spin generating function of the semistandard $k$-ribbon tableaux of $\mu/\nu$,
where $z^T$ records the entries.  By \cite[Cor.~5.2.4]{HHLRU} there is a constant
$c$ with $u^{c}G_{\mu/\nu}(z;u^{-1})=G_{\operatorname{quot}_k(\mu)}(z;u)$.  This is
the semistandard form, \cite[Lem.~5.2.3]{HHLRU}, of the identity
$\operatorname{sp}=c-\operatorname{inv}\circ\operatorname{quot}_k$ proved for
standard tableaux in \cite[Lem.~5.2.2]{HHLRU}.

The two statistics are
complementary, every unit of spin of the ribbon tableau being traded for an inversion
of the corresponding filling of the quotient, as in the example, where the tiling of
spin $0$ corresponds to the filling with one inversion and the tiling of spin $1$ to
the filling with none.

By \cite[Prop.~5.3.1]{HHLRU}, whose
hypothesis is exactly that $\nu$ be $\operatorname{core}_k(\mu)$ and which
\cref{lem:core-form} supplies, the Schur coefficients of $G_{\mu/\nu}$ are, up to a
power of the variable and the substitution $u\mapsto u^2$, the parabolic
Kazhdan--Lusztig polynomials $P^-_{\mu+\rho,\,\nu+k\lambda+\rho}$, and these have
non-negative coefficients \cite{KT}.  Inverting the variable permutes the degrees
within each Schur coefficient and hence preserves non-negativity.
\end{proof}

\begin{remark}\label{rem:positivity-scope}
The hypothesis of \cite[Prop.~5.3.1]{HHLRU} cannot be dropped.  The same paper
records the case of $\nu$ not a $k$-core as a conjecture, observing that the
Kazhdan--Lusztig expressions then acquire negative terms.  That case is
\cite[Cor.~6.9]{GH}, where the authors note that empty $\nu$ is
\cite[Thm.~4.1]{LT} and $\nu$ a $k$-core is \cite[Prop.~5.3.1]{HHLRU}.  Their
manuscript is unpublished, and nothing here depends on it.  The published case of
empty $\nu$ alone would not suffice, because it forces all the $a_i$ to be equal,
whereas the offsets $d_s$ of \cref{def:nu-of-e} are the values $q_{i_1}$ at the first
labels of the runs and are unequal for most paths.  At $(n,m)=(5,1)$ only $47$ of
the $252$ paths have all offsets equal, and at $(4,2)$ only $107$ of $495$.
\end{remark}

\begin{lemma}[Complementation]\label{lem:complementation}
For $P=(e,\sigma)\in\mathsf{PF}^+_{n,m}$ put
\[
 e^c_i=mn-e_{n+1-i},\qquad \sigma^c_i=n+1-\sigma_{n+1-i}\qquad(1\le i\le n).
\]
Then $P\mapsto P^c=(e^c,\sigma^c)$ is an involution of $\mathsf{PF}^+_{n,m}$ under
which
\[
 \dinv_m(e^c)=\dinv_m(e),\quad \operatorname{def}_m(P^c)=\operatorname{def}_m(P),
 \quad A^{(m)}_\mp(e^c)=A^{(m)}_\pm(e),
\]
and $\operatorname{ides}(w(P^c))=\{\,n-c:c\in\operatorname{ides}(w(P))\,\}$.
\end{lemma}

\begin{proof}
Both formulas are visibly involutive.  From $b_i=mi-e_i$ we get
\[
 b^c_i=mi-e^c_i=mi-mn+e_{n+1-i}=mi-mn+m(n+1-i)-b_{n+1-i}=m-b_{n+1-i},
\]
hence for $i<j$, writing $i'=n+1-i$ and $j'=n+1-j$, we have $j'<i'$ and
\begin{equation}\label{eq:comp-diff}
 b^c_i-b^c_j=b_{j'}-b_{i'} .
\end{equation}
The map $(i,j)\mapsto(j',i')$ is a bijection of the pairs $i<j$.  Hence
\eqref{eq:comp-diff} gives $\dinv_m(e^c)=\dinv_m(e)$ at once, and
$[m-b^c_i]_++[-b^c_i]_+=[b_{i'}]_++[b_{i'}-m]_+$ gives $A^{(m)}_-(e^c)=A^{(m)}_+(e)$
and $A^{(m)}_+(e^c)=A^{(m)}_-(e)$, which is
\cref{cor:Phat-props}\textup{(ii)}.

For the labels,
$\sigma^c_i>\sigma^c_j$ means $n+1-\sigma_{i'}>n+1-\sigma_{j'}$, that is
$\sigma_{j'}>\sigma_{i'}$.  Thus the pair $(i,j)$ of $P^c$ is counted in the first
sum of \eqref{eq:label-defect} exactly when the pair $(j',i')$ of $P$ is, and
likewise for the second sum.  Hence
$\operatorname{def}_m(P^c)=\operatorname{def}_m(P)$.  The pair $P^c$ is admissible,
since $e^c_i=e^c_{i+1}$ means $e_{n-i}=e_{n+1-i}$, whence
$\sigma_{n-i}<\sigma_{n+1-i}$ and therefore $\sigma^c_i<\sigma^c_{i+1}$.

For the reading words, the order of increasing $(b^c_i,i)$ is, by
$b^c_i=m-b_{n+1-i}$, the order of decreasing $(b_{n+1-i},n+1-i)$.  Hence $w(P^c)$ is
the reverse of the word obtained from $w(P)$ by replacing each letter $v$ by
$n+1-v$.  Writing $\operatorname{pos}_w(v)$ for the position of $v$ in $w$, this says
$\operatorname{pos}_{w(P^c)}(v)=n+1-\operatorname{pos}_{w(P)}(n+1-v)$.  Hence
$c\in\operatorname{ides}(w(P^c))$, that is
$\operatorname{pos}_{w(P^c)}(c+1)<\operatorname{pos}_{w(P^c)}(c)$, holds if and only
if $\operatorname{pos}_{w(P)}(n-c)>\operatorname{pos}_{w(P)}(n+1-c)$, that is if and
only if $n-c\in\operatorname{ides}(w(P))$.
\end{proof}

\begin{proposition}\label{prop:labelled-props}
$\mathcal H^+_{n,m}$ is a symmetric function, it is Schur positive, it satisfies
$\mathcal H^+_{n,m}(X;q,t)=\mathcal H^+_{n,m}(X;t,q)$, and
\[
 q^{D}\,\mathcal H^+_{n,m}(X;q,q^{-1})=h_n\bigl[(1+q^2+\dots+q^{2mn})X\bigr].
\]
\end{proposition}
\begin{proof}
By \eqref{eq:path-grouping} and \cref{thm:llt},
\[
 \mathcal H^+_{n,m}(X;q,t)=\sum_{e}q^{A^{(m)}_-(e)}t^{A^{(m)}_+(e)}
 (qt)^{\dinv_m(e)}\,G_{\boldsymbol\nu(e)}\bigl(X;(qt)^{-1}\bigr).
\]
By \cref{cor:llt-positive} each $G_{\boldsymbol\nu(e)}$ is symmetric, with Schur
coefficients in $\mathbb N[z]$.  Its degree in $z$ is
$\max_\sigma\operatorname{def}_m(e,\sigma)$, at most $\dinv_m(e)$ because
$\operatorname{def}_m\le\dinv_m$ pair by pair.  Hence each summand
\[
 (qt)^{\dinv_m(e)}\,G_{\boldsymbol\nu(e)}\bigl(X;(qt)^{-1}\bigr)
\]
has coefficients in $\mathbb N[qt]$, and the sum is symmetric and Schur positive.

For the third assertion, we apply \cref{lem:complementation} and substitute
$P\mapsto P^c$ in \eqref{eq:Hplus}, which gives
\[
 \mathcal H^+_{n,m}(X;q,t)
 =\sum_{P}q^{A^{(m)}_+(P)+\dinv^{\mathrm{lab}}_m(P)}\,
 t^{A^{(m)}_-(P)+\dinv^{\mathrm{lab}}_m(P)}\,F_{n,\operatorname{ides}(w(P))^{\mathrm{rev}}}(X),
\]
where $S^{\mathrm{rev}}=\{n-c:c\in S\}$.  Its bidegree-$(a,b)$ part is therefore
obtained from the bidegree-$(a,b)$ part of $\mathcal H^+_{n,m}(X;t,q)$, which is the
same sum with $F_{n,\operatorname{ides}(w(P))}$ in place of
$F_{n,\operatorname{ides}(w(P))^{\mathrm{rev}}}$, by replacing every $F_{n,S}$ by
$F_{n,S^{\mathrm{rev}}}$.

That operation fixes every symmetric function.  Indeed,
let us truncate to $N\ge n$ variables and substitute $x_i\mapsto x_{N+1-i}$.  Writing
$k_j=N+1-i_{n+1-j}$ turns a weakly increasing index sequence $i_1\le\dots\le i_n$ into
a weakly increasing $k_1\le\dots\le k_n$, and the strictness required at $j\in S$ into
strictness at $n-j$.  Hence the substitution carries $F_{n,S}(x_1,\dots,x_N)$ to
$F_{n,S^{\mathrm{rev}}}(x_1,\dots,x_N)$.  A symmetric $f=\sum_Sc_SF_{n,S}$ is fixed by
the substitution, which is a permutation of the variables, hence
$\sum_Sc_SF_{n,S}=\sum_Sc_SF_{n,S^{\mathrm{rev}}}$ in $N$ variables.

For $N\ge n$ the
$F_{n,S}(x_1,\dots,x_N)$, $S\subseteq\{1,\dots,n-1\}$, are linearly independent.
Indeed, the coefficient of $x_1^{\alpha_1}x_2^{\alpha_2}\cdots$ in $F_{n,S}$ is $1$
when $S$ is contained in the set of partial sums of the composition $\alpha$ and $0$
otherwise, hence the matrix of coefficients on the monomials indexed by compositions
is unitriangular for inclusion.  Therefore $c_S=c_{S^{\mathrm{rev}}}$ for every $S$.
Each bidegree part of $\mathcal H^+_{n,m}(X;t,q)$ is symmetric by the first
assertion, hence the two agree.

For the last assertion, $\dinv^{\mathrm{lab}}_m$ occurs with the same exponent in
$q$ and in $t$, hence it cancels at $t=q^{-1}$ and the term of $P$ becomes
$q^{A^{(m)}_-(e)-A^{(m)}_+(e)}$.  Since
$[m-b]_+-[b-m]_+=m-b$ and $[-b]_+-[b]_+=-b$ for every $b$,
\[
 A^{(m)}_--A^{(m)}_+=\sum_{i=1}^n(m-2b_i)=mn-2m\tbinom{n+1}2+2\sum_ie_i
 =2\sum_ie_i-D .
\]
Hence
\[
 q^{D}\,\mathcal H^+_{n,m}(X;q,q^{-1})
 =\sum_Pq^{2\sum_ie_i}\,F_{n,\operatorname{ides}(w(P))}(X).
\]

For a fixed path $e$ the sum over the admissible labellings is
$\prod_sh_{|R_s|}(X)$, the product over the maximal vertical runs of $e$, by the
proof of \cref{lem:labelled-basic}\textup{(iii)}, and $\sum_ie_i=\sum_jja_j$, where
$a_j$ is the number of north steps of $e$ at abscissa $j$.  Therefore the sum equals
$\sum_{a}\prod_jq^{2ja_j}h_{a_j}(X)=h_n\bigl[(1+q^2+\dots+q^{2mn})X\bigr]$, by the
expansion $h_n[(y_0+\dots+y_{mn})X]=\sum_a\prod_jy_j^{a_j}h_{a_j}[X]$ at $y_j=q^{2j}$.
\end{proof}

The specialization at $t=q^{-1}$ refines \cref{thm:principal}.  Taking $[s_n]$ and
using \cref{lem:labelled-basic}(i) returns
$q^{D}\widehat P_{n,m}(q,q^{-1})=\qbinom{(m+1)n}{n}_{q^2}$ of
\cref{cor:Phat-props}(iii), because $[s_n]h_n[(1+q^2+\dots+q^{2mn})X]$ is the
generating function of the partitions in the $n\times mn$ rectangle by twice the
size.

\subsection{The conjecture}\label{sec:labelled-conj}

In this subsection we state the labelled conjecture and we relate it to its operator
form and to its representation-theoretic form.  The main idea is to promote each
fixed-point weight of \eqref{eq:conj-B} to the colour-zero pure part of the
corresponding wreath Macdonald polynomial, with the coefficients inverted.

Recall
$\ell_\lambda$, $\Pi^{(0)}_\lambda$, $D_\lambda$ from \eqref{eq:statistics},
$B^{(i)}_\lambda$ from \cref{sec:conj-statement}, the wreath Macdonald polynomials
$\wtH_\lambda$ of \eqref{eq:wreath-def} and the involution $\inv$ of
\cref{sec:wreath-form}.  Write $\wtH_\lambda[X,0]\in\Lambda_{X_0}$ for the
colour-zero pure part.  Put
\begin{equation}\label{eq:Clambda}
 C_\lambda=(1+qt)B^{(0)}_\lambda-(q+t)B^{(1)}_\lambda
 =\bigl\langle\wtH_\lambda,\Phi_n\bigr\rangle ,
\end{equation}
where the second equality holds by \cref{thm:wreath-inputs}(b) and \eqref{eq:KPhi}.
Thus $P_{n,m}=\sum_\lambda\ell^{\,m}_\lambda\Pi^{(0)}_\lambda C_\lambda D_\lambda^{-1}$
is the left side of \eqref{eq:conj-B}.

\begin{conjecture}[Type-$B$ shuffle conjecture, labelled form]\label{conj:labelled}
For all $n,m\ge1$,
\begin{equation}\label{eq:conj-labelled}
 \sum_{\lambda\in\Bal(2n)}
 \frac{\ell_\lambda^{\,m}\,\Pi^{(0)}_\lambda\,C_\lambda}{D_\lambda}\;
 \inv\bigl(\wtH_\lambda[X,0]\bigr)
 \;=\;\mathcal H^+_{n,m}(X;q,t).
\end{equation}
\end{conjecture}

Here $\inv$ is the involution of \cref{sec:wreath-form}, which inverts $q$ and $t$
in the coefficients and fixes the Schur functions.  Applied to the one-alphabet
symmetric function $\wtH_\lambda[X,0]=\sum_\mu c_\mu(q,t)s_\mu$ it gives
$\sum_\mu c_\mu(q^{-1},t^{-1})s_\mu$.  For instance \cref{ex:wreath-n2} gives
$\wtH_{(3,1)}[X,0]=s_2+q^2s_{11}$ and $\inv(\wtH_{(3,1)}[X,0])=s_2+q^{-2}s_{11}$.

The inversion is not a normalization.  It is what converts the colour-one pure part,
which is the one produced by the operator form below, into the colour-zero pure
part, which is the one normalized by \eqref{eq:wreath-def}.  The precise statement
is the following.

\begin{lemma}\label{lem:pure-symmetry}
For every $\lambda\in\Bal(2n)$,
\begin{equation}\label{eq:iota-sym}
 \inv\wtH_\lambda=\ell_\lambda^{-1}\,\Omega\,\wtH_\lambda,\qquad
 \Omega=\chi\,(\omega\otimes\omega),
\end{equation}
where $\chi$ exchanges the two alphabets, and consequently
\begin{equation}\label{eq:pure-symmetry}
 \wtH_\lambda[0,X]=\ell_\lambda\,\omega\inv\bigl(\wtH_\lambda[X,0]\bigr).
\end{equation}
\end{lemma}

\begin{proof}
Orr and Shimozono give two symmetries of the wreath Macdonald basis, each only up
to a scalar, which they write $\equiv$.  The inversion symmetry is
$\downarrow\wtH^w_{\mu^\bullet}\equiv\wtH^{w_0w}_{w_0\mu^\bullet}$
\cite[Prop.~3.23]{OS}, where
$\downarrow\,=\inv\circ\operatorname{neg}\circ\,\omega$ and $\omega$ is applied to
both alphabets, and the rotational symmetry is
$\chi\wtH^w_{\mu^\bullet}\equiv\wtH^{\chi w}_{\chi\mu^\bullet}$
\cite[Thm.~3.26]{OS}.  At $r=2$ the three permutations of $I=\Z/2\Z$ that occur
degenerate, namely $\operatorname{neg}$ is the identity, and $\chi$ and $w_0$ are
both the transposition.

Let us take $w=1$, which is the chamber attached to an empty
$2$-core.  The first symmetry gives $\inv\omega\wtH_\lambda\equiv\wtH^{w_0}_{w_0\mu^\bullet}$,
and the second, applied with $w=w_0$ to $\wtH^{w_0}_{w_0\mu^\bullet}$, gives
$\chi\inv\omega\wtH_\lambda\equiv\wtH^{\chi w_0}_{\chi w_0\mu^\bullet}=\wtH_\lambda$,
because $\chi w_0=1$.  Since $\omega$ commutes with the colour permutations and with
$\inv$, this says $\Omega\inv\wtH_\lambda=c\,\wtH_\lambda$ for a scalar $c$.

It remains to evaluate $c$.  Orr and Shimozono say of the symmetries of their
\S3.5 that ``the constant of proportionality is a Laurent monomial which may easily
be deduced using the normalization condition~(3.9)'' \cite[\S3.5]{OS}, and they carry
this out for the inversion symmetry.  Namely, \cite[Rem.~3.25]{OS} records that the
$c$ with $\downarrow\wtH^w_{\mu^\bullet}=c\,\wtH^{w_0w}_{w_0\mu^\bullet}$ is
$\langle s_n[X^{(0)}],\downarrow\wtH^w_{\mu^\bullet}\rangle
=\inv\langle e_n[X^{(0)}],\wtH^w_{\mu^\bullet}\rangle$.

The constant of
\cite[Thm.~3.26]{OS} is not written out there, and \eqref{eq:iota-sym} needs the
composite.  What follows is their computation applied to $\Omega\inv$ at $r=2$,
where it is the colour-one rather than the colour-zero monomial that appears.

We
pair both sides with $s_n[X_0]$ and use the normalization
$\langle\wtH_\lambda,s_n[X_0]\rangle=1$ of \eqref{eq:wreath-def}.  On the right this
gives $c$.  On the left, $\chi$ exchanges the alphabets, $\omega$ is self-adjoint for
the Hall pairing and carries $s_n[X_1]$ to $e_n[X_1]$, and $\inv$ acts on
coefficients only.  Hence
\[
 \bigl\langle\Omega\inv\wtH_\lambda,\;s_n[X_0]\bigr\rangle
 =\bigl\langle\inv\wtH_\lambda,\;e_n[X_1]\bigr\rangle
 =\inv\bigl\langle\wtH_\lambda,e_n[X_1]\bigr\rangle=\inv(\ell_\lambda)
 =\ell_\lambda^{-1},
\]
where the third equality is \cite[Thm.~4.14]{OS}, attributed there to Haiman, which
identifies $\langle\wtH_\lambda,e_n[X_i]\rangle$ with the product of the box weights
$q^at^b$ of $\lambda$ over the boxes of residue $i$.  At $r=2$ the residue is the
parity of $a+b$, hence the value at $i=1$ is $\ell_\lambda$.

Here residue $1$ is the
odd runner, which carries the alphabet $X_1$ in the convention of
\cref{sec:wreath-form} fixed by \cref{ex:wreath-n2}.  Orr and Shimozono weight $x$
by $t$ and $y$ by $q$, and \eqref{eq:iota-sym} was checked directly in the present
conventions (\cref{thm:labelled-evidence}(9)), hence the difference does not affect
it.

Since $\Omega$ is an involution, \eqref{eq:iota-sym} follows.  Taking the
colour-zero pure part of \eqref{eq:iota-sym} and using
$(\Omega f)[X,0]=\omega\bigl(f[0,X]\bigr)$ gives \eqref{eq:pure-symmetry}.
\end{proof}

Passing from $\omega\wtH_\lambda[0,X]$ to $\inv(\wtH_\lambda[X,0])$ absorbs
one factor $\ell_\lambda$.  This is why the exponent is $\ell_\lambda^{\,m}$ in
\eqref{eq:conj-labelled} and $\ell_\lambda^{\,m-1}$ in \eqref{eq:operator-form}
below.  Dropping $\inv$ from \eqref{eq:conj-labelled} leaves
\cref{prop:sn-coefficient} intact, since $[s_n]$ of the extra factor is $1$ either
way, but it destroys the conjecture in every other Schur coordinate.  At $n=2$, $m=1$
it gives $[s_{11}]=(q^2+qt+t^2)(q^4+qt+t^4)$ in place of the correct
$[s_{11}]\mathcal H^+_{2,1}=q^2+qt+t^2$.

\begin{remark}[A conjectural proportionality of the pure parts]\label{rem:pure-proportional}
Write $\ell^{(i)}_\lambda=\prod_{a+b\equiv i\ (2)}q^at^b$ for the product of the box
weights of residue $i$, hence $\ell^{(1)}_\lambda=\ell_\lambda$.  Computation
suggests that the two pure parts are proportional, namely that for every
$\lambda\in\Bal(2n)$
\begin{equation}\label{eq:pure-proportional}
 \wtH_\lambda[0,X]=\frac{\ell^{(1)}_\lambda}{\ell^{(0)}_\lambda}\,\wtH_\lambda[X,0].
\end{equation}
This was verified symbolically over $\Q(q,t)$ for $n\le4$ and at rational points for
every one of the $110$ diagrams with $n\le7$, but it is not proved.

It is not a
consequence of \cref{lem:pure-symmetry}, which combined with
\eqref{eq:pure-proportional} would say that the colour-zero pure part has the
type-$A$ self-duality $\inv(\wtH_\lambda[X,0])=\ell^{(0)\,-1}_\lambda\,\omega\,
\wtH_\lambda[X,0]$.

Granting \eqref{eq:pure-proportional},
\eqref{eq:conj-labelled} takes the $\inv$-free form
\[
 \sum_{\lambda\in\Bal(2n)}
 \frac{\ell_\lambda^{\,m}\,\Pi^{(0)}_\lambda\,C_\lambda}{D_\lambda\,\ell^{(0)}_\lambda}\;
 \omega\,\wtH_\lambda[X,0]\;=\;\mathcal H^+_{n,m}(X;q,t),
\]
which is a statement about the $G_n$-invariant part of the Procesi bundle alone.
Equivalently, with $\nabla_i$ the diagonal operators
$\nabla_i\wtH_\lambda=\ell^{(i)}_\lambda\wtH_\lambda$, the identity
\eqref{eq:pure-proportional} says that $f\mapsto f[0,X]$ agrees with
$f\mapsto(\nabla_1\nabla_0^{-1}f)[X,0]$ on all of $\Lambda^{(2)}_n$.
\end{remark}

The left side of \eqref{eq:conj-labelled} is the left side of \eqref{eq:conj-B} with
each fixed-point weight promoted to a symmetric function.  Namely, the scalar term of
$\lambda$ is multiplied by $\inv(\wtH_\lambda[X,0])$, whose $s_n$-coefficient is
$1$.  Hence we obtain the following.

\begin{proposition}\label{prop:sn-coefficient}
The $s_n$-coefficient of the left side of \eqref{eq:conj-labelled} is $P_{n,m}$, and
the $s_n$-coefficient of the right side is $\widehat P_{n,m}$.  Consequently
\cref{conj:labelled} implies \cref{conj:shuffleB} in the same rank and twist.
\end{proposition}

\begin{proof}
For a one-alphabet symmetric function $f$ of degree $n$ we have
$[s_n]f=\langle f,h_n\rangle$, and the third condition of \eqref{eq:wreath-def} gives
$[s_n]\wtH_\lambda[X,0]=\langle\wtH_\lambda,h_n[X_0]\rangle=1$.  The involution
$\inv$ acts on coefficients only, hence $[s_n]\inv(\wtH_\lambda[X,0])=\inv(1)=1$.
The first assertion is now \cref{lem:parity-split}, and the second is
\cref{lem:labelled-basic}(i).
\end{proof}

\begin{remark}[The operator form]\label{rem:operator-form}
Let $\operatorname{pr}_1f=f[0,X]$ be the projection to the colour-one pure part.  The
expected operator form of \cref{conj:labelled} is
\begin{equation}\label{eq:operator-form}
 \omega\,\operatorname{pr}_1\nabla_1^{\,m}e_n[X_1]\;=\;\mathcal H^+_{n,m}(X;q,t),
\end{equation}
the two-colour analogue of $\nabla^me_n=\sum_{\mathrm{PF}}q^{\dinv}t^{\mathrm{area}}F$.

The two left sides are related as follows.  Expanding $e_n[X_1]$ in the basis
$\{\wtH_\lambda\}$ by \cref{thm:wreath-inputs}(d) gives
$e_n[X_1]=\sum_\lambda E_\lambda\ell_\lambda^{-1}D_\lambda^{-1}\wtH_\lambda$ with
$E_\lambda=\wtH_\lambda[1+qt,-(q+t)]$, and
$\wtH_\lambda[0,X]=\ell_\lambda\,\omega\inv(\wtH_\lambda[X,0])$ by
\cref{lem:pure-symmetry}.  Hence the left side of \eqref{eq:operator-form}
is $\sum_\lambda\ell_\lambda^{\,m}E_\lambda D_\lambda^{-1}\inv(\wtH_\lambda[X,0])$.
The two left sides therefore differ by
\begin{equation}\label{eq:cancellation}
 \sum_{u}u^{\,m}\sum_{\lambda:\,\ell_\lambda=u}
 \frac{E_\lambda-\Pi^{(0)}_\lambda C_\lambda}{D_\lambda}\,\inv(\wtH_\lambda[X,0]),
\end{equation}
which vanishes for every $m$ if and only if each inner sum vanishes.

The inner sums
do not vanish termwise, since $E_\lambda\ne\Pi^{(0)}_\lambda C_\lambda$ for $6$ of
the $10$ diagrams at $n=3$ and for $15$ of the $20$ at $n=4$.  Their vanishing is a
separate conjecture, verified symbolically for $n\le6$ and at rational points for
$n=7,8$.  Under it, \eqref{eq:conj-labelled} and \eqref{eq:operator-form} are
equivalent.  Without it they are independent statements, and it is
\eqref{eq:conj-labelled}, not \eqref{eq:operator-form}, which refines
\cref{conj:shuffleB}.
\end{remark}

\begin{remark}[The representation-theoretic form]\label{rem:cherednik-form}
Let $L^{(m)}=L_{\frac1{2n}+m}(\triv)$ with the bigrading of \cref{def:bigrading} and
let
\[
 DH^+_{n,m}:=\bigl(\gr_FL^{(m)}\bigr)^{G_n},\qquad DH^+_n:=DH^+_{n,1},
\]
an $S_n$-module by the residual action, with $e\,DH^+_{n,m}=\gr_F\,eL^{(m)}$.  The
representation-theoretic form of the conjecture is
\begin{equation}\label{eq:cherednik-form}
 \operatorname{Frob}^{q,t}_{S_n}\bigl(DH^+_{n,m}\bigr)=\mathcal H^+_{n,m}(X;q,t).
\end{equation}
Applying $[s_n]$ to \eqref{eq:cherednik-form} gives
$\ch\gr_F\,eL^{(m)}=\widehat P_{n,m}$ by \cref{lem:labelled-basic}(i), which is
\cref{thm:cherednik} combined with \cref{conj:shuffleB}.  For $m\gg0$ the first half
is a theorem.

At $q=t=1$, \eqref{eq:cherednik-form} is a theorem for all $n,m$.
Indeed, $(L^{(m)}\otimes\eps)^{G_n}\cong\C[(\Z/(2mn+1))^n]^{G_n}$ is the permutation
module on the functions $[n]\to\{0,\dots,mn\}$ by \cite{GordonDiag}, of character
$h_n[(mn+1)X]$, which is \cref{lem:labelled-basic}(iii).
\end{remark}

For $n=2$ and $n=3$ at $m=1$ the right side of \eqref{eq:conj-labelled} is, in the
notation of \cref{ex:lowrank} and with $w=qt$,
\begin{align*}
 \mathcal H^+_{2,1}&=\bigl([5]+w\bigr)s_2+[3]\,s_{11},\\
 \mathcal H^+_{3,1}&=\bigl([10]+w[6]+w[4]\bigr)s_3
 +\bigl([8]+[6]+w[4]+w[2]\bigr)s_{21}+[4]\,s_{111},
\end{align*}
whose $s_n$-coefficients are the polynomials $P_{2,1}$ and $P_{3,1}$ of
\cref{ex:lowrank}, as \cref{prop:sn-coefficient} requires.

\subsection{Rooted schedules}\label{sec:rooted-schedules}

In this subsection we prove the schedule formula for $\mathcal H^+_{n,1}$.  At $m=1$
the labelled paths can be reorganized into boxes of independent choices, in the way
that Haglund and Loehr \cite{HL} reorganize the ordinary parking functions.  The
main idea is to build a labelled path by inserting its labels one at a time into a
word with two sentinels, the labels above the diagonal from the right and the labels
on or below it from the left.  The index set is $S_{\{0,1,\dots,n\}}$, and the role of
the extra symbol $0$ is to mark the line $y=x$, which the paths of the square cross in
both directions.

\begin{definition}\label{def:rooted}
Let $\tau$ be a permutation of $\{0,1,\dots,n\}$ with maximal increasing runs
$R_1|\dots|R_s$, and let $0\in R_r$.  Since $0$ is the smallest symbol it begins its
run.  For an ordinary label $c\in R_j$ put $h_\tau(c)=r-j$ and $B_h=\{c:h_\tau(c)=h\}$.
The \emph{schedule bounds} are
\begin{equation}\label{eq:schedule}
 s_\tau(c)=\begin{cases}
 \#\{d\in B_h:d>c\}+\#\{d\in B_{h-1}:d<c\},&h\ge2,\\
 1+\#\{d\in B_1:d>c\},&h=1,\\
 1+\#\{d\in B_0:d<c\}+\#\{d\in B_1:d>c\},&h=0,\\
 \#\{d\in B_h:d<c\}+\#\{d\in B_{h+1}:d>c\},&h\le-1,
 \end{cases}\qquad(c\in B_h),
\end{equation}
all of them positive, and the \emph{schedule box} is
$\mathcal B_\tau=\prod_{c=1}^n\{0,1,\dots,s_\tau(c)-1\}$.  Put
\[
 a_-(\tau)=\sum_{h_\tau(c)\le0}\bigl(1-2h_\tau(c)\bigr),\qquad
 a_+(\tau)=\sum_{h_\tau(c)\ge1}\bigl(2h_\tau(c)-1\bigr),
\]
and for $\mathbf k\in\mathcal B_\tau$ define the \emph{rooted monomial}
\begin{equation}\label{eq:rooted-monomial}
 M_{\tau,\mathbf k}
 =\prod_{h_\tau(c)\ge1}x_c^{k_c}y_c^{2h_\tau(c)-1+k_c}
 \prod_{h_\tau(c)\le0}x_c^{1-2h_\tau(c)+k_c}y_c^{k_c}
 =g^B_\tau\prod_{c=1}^n(x_cy_c)^{k_c},
\end{equation}
of bidegree $(a_-(\tau)+|\mathbf k|,\,a_+(\tau)+|\mathbf k|)$, where
$g^B_\tau=\prod_{h\ge1}y_c^{2h-1}\prod_{h\le0}x_c^{1-2h}$.
\end{definition}

Every variable pair of $M_{\tau,\mathbf k}$ has odd total exponent, and the exponents
$(a_c,b_c)$ recover $h_\tau(c)=(b_c-a_c+1)/2$ and $k_c=\min(a_c,b_c)$, hence they
recover $\tau$.  Thus the $M_{\tau,\mathbf k}$ are pairwise distinct monomials, and
they span a space of dimension $(n+1)^n$.  The baseline $g^B_\tau$ is an odd
root-centred Garsia--Stanton monomial, in the sense that $h_\tau(c)=d_\tau(c)-d_\tau(0)$,
where $d_\tau(c)$ is the number of descents of $\tau$ at or after the position of $c$.

\begin{theorem}[Rooted insertion]\label{thm:rooted-insertion}
There is a bijection
$\bigsqcup_{\tau}\mathcal B_\tau\to\mathsf{PF}^+_{n,1}$, given by a two-sentinel
insertion algorithm, under which
\[
 A_-(P)=a_-(\tau),\qquad A_+(P)=a_+(\tau),\qquad D(P)=|\mathbf k|,
\]
and under which $\operatorname{ides}(w(P))=\mathcal D_\tau(\mathbf k)$, where
$c\in\mathcal D_\tau(\mathbf k)$ iff $h_\tau(c+1)<h_\tau(c)$, or
$h_\tau(c+1)=h_\tau(c)\ge1$ and $k_c\le k_{c+1}$, or $h_\tau(c+1)=h_\tau(c)\le0$ and
$k_c\ge k_{c+1}$.  Consequently
\begin{equation}\label{eq:rooted-forms}
 \mathcal H^+_{n,1}=\sum_{\tau\in S_{\{0,\dots,n\}}}q^{a_-(\tau)}t^{a_+(\tau)}
 \sum_{\mathbf k\in\mathcal B_\tau}(qt)^{|\mathbf k|}F_{n,\mathcal D_\tau(\mathbf k)}(X),
 \qquad
 \bigl\langle\mathcal H^+_{n,1},h_1^n\bigr\rangle
 =\sum_\tau q^{a_-(\tau)}t^{a_+(\tau)}\prod_{c=1}^n[s_\tau(c)]_{qt},
\end{equation}
with $[r]_{qt}=1+qt+\dots+(qt)^{r-1}$.
\end{theorem}

\begin{remark}[Relation to the type-$A$ schedule formula]\label{rem:typeA-schedules}
\Cref{thm:rooted-insertion} is a type-$C$ form of a formula of Haglund and Loehr.
In type $A$ one takes $\tau\in S_n$ with descents at $i_1<\dots<i_k$, divides it into
the maximal increasing runs $R_1|\dots|R_{k+1}$, lets $\operatorname{cars}(\tau)$ be
the set of parking functions whose cars in the rows of length $k+1-j$ are the
elements of $R_j$, and for $\tau_i\in R_j$ sets
\begin{equation}\label{eq:haglund-w}
 w_i(\tau)=\#\{d\in R_j:d>\tau_i\}+\#\{d\in R_{j+1}:d<\tau_i\},
\end{equation}
where, in Haglund's words, ``it will prove convenient to call element $0$ the $k+2$nd
run of $\tau$'' \cite[p.~79]{HaglundBook}.  Then
\begin{equation}\label{eq:haglund-53}
 \sum_{P\in\operatorname{cars}(\tau)}q^{\dinv(P)}t^{\operatorname{area}(P)}
 =t^{\operatorname{maj}(\tau)}\prod_{i=1}^n[w_i(\tau)]_q,
 \qquad
 \sum_{P\in\mathcal P_n}q^{\dinv(P)}t^{\operatorname{area}(P)}
 =\sum_{\tau\in S_n}t^{\operatorname{maj}(\tau)}\prod_{i=1}^n[w_i(\tau)]_q,
\end{equation}
which are \cite[Thm.~5.3 and Cor.~5.3.1]{HaglundBook}.

The first of these is due to
Haglund and Loehr \cite{HL}.  It is Theorem~$1$ of the Melbourne~2002 preliminary
version of their paper, cited as such in \cite[(56)]{HHLRU}, where it is restated as
$H(\sigma;q,t)=t^{\operatorname{comaj}(\sigma)}\prod_{i=2}^n[v(\sigma,i)+\chi(i\le r_1)]_q$
with a different bookkeeping of the runs.

The word \emph{schedule} is later.  It was
introduced by Garsia and Hicks while Hicks's thesis \cite{Hicks} was being written
\cite[Historical Remark~3.2]{HagSergel}.  Both identities in \eqref{eq:haglund-53}
became theorems with the shuffle theorem \cite{CarlssonMellit}, and a second proof is
the Carlsson--Oblomkov basis \cite{CO} specialized at $x_i\mapsto q$, $y_i\mapsto t$
\cite[\S4]{HagSergel}.

Three features of \eqref{eq:schedule} are visible already in \eqref{eq:haglund-w},
and one is not.
\begin{enumerate}
\item[(i)] \emph{The bounds.}  By construction $B_h=R_{r-h}$, hence $B_{h-1}$ is the
 run immediately to the right of $B_h$, and the first line of \eqref{eq:schedule} is
 \eqref{eq:haglund-w} verbatim.
\item[(ii)] \emph{The root is a sentinel.}  Haglund's convention that element $0$ is
 the $(k+2)$nd run contributes to a car $c$ in the last run the single count
 $\#\{d\in\{0\}:d<c\}=1$.  That is the $+1$ in the second line of
 \eqref{eq:schedule}.  The left sentinel $(0,0)$ of the insertion below is Haglund's
 element $0$.  In \cite[\S3.1]{HagSergel} the same device is described as appending
 a $0$ to $\tau$, where ``the appended $0$ accounts for the extra point of insertion
 in the special case $k=0$''.
\item[(iii)] \emph{The insertion.}  The proof of \eqref{eq:haglund-53} in
 \cite[pp.~80--81]{HaglundBook} builds $\operatorname{cars}(\tau)$ from right to left
 and identifies $w_i(\tau)$ with ``the number of ways to insert a row containing car
 $\tau_i$ into a partial parking function containing cars $\tau_{i+1},\dots,\tau_n$,
 in rows of the appropriate length, and still obtain a partial parking function'',
 the $w_i(\tau)$ choices raising $\dinv$ by $0,1,\dots,w_i(\tau)-1$.  Steps~$2$
 and~$5$ of the proof below are that argument carried out with a second sentinel.
\item[(iv)] \emph{What is not there.}  A type-$A$ path stays on one side of $y=x$,
 hence all its heights have the same sign and the insertion has one direction.  A
 path in $\mathsf{PF}^+_{n,1}$ crosses the diagonal.  The index set is
 $S_{\{0,\dots,n\}}$ rather than $S_n$ with a $0$ appended, the insertion proceeds
 from the right for $h\ge1$ and from the left for $h\le0$, the third line of
 \eqref{eq:schedule} counts in both directions at once and has no type-$A$
 counterpart, and a unit of $D$ raises both degrees, hence the bracket is
 $[\,\cdot\,]_{qt}$ and the prefactor is a monomial in two variables instead of
 $t^{\operatorname{maj}(\tau)}$.
\end{enumerate}
Chapter~5 of \cite{HaglundBook} (pp.~77--90) carries the type-$A$ story in full.
The $\dinv$ statistic is in Chapter~3 (pp.~48--52), the shuffle conjecture and the
LLT expansion used in \cref{thm:llt} above are in Chapter~6 (pp.~91--99), and the
$m$-parameter of \cref{rem:m-one} below is in Chapter~6 (pp.~107--112).
\end{remark}

\begin{proof}
The main points are in Steps $2$, $5$ and $6$.  They are the count of admissible
slots, which must come out equal to \eqref{eq:schedule} and must not depend on the
earlier choices, the identification of $D$ with $|\mathbf k|$, and the comparison of
the anchor sets of two labels of equal height, which is what produces the descent
rule.  The other three steps are bookkeeping.

Let us write a labelled path $P=(e,\sigma)$ as the word
\[
 W(P)=(\sigma_0,b_0)(\sigma_1,b_1)\cdots(\sigma_n,b_n)(\sigma_{n+1},b_{n+1}),
 \qquad (\sigma_0,b_0)=(0,0),\quad (\sigma_{n+1},b_{n+1})=(n+1,1),
\]
each entry being a pair (label, height), with $b_i=i-e_i$ for $1\le i\le n$.  The
two extreme entries are the \emph{sentinels}, and they are the first and the last
entry of every word considered below.  We say that a word of pairs satisfies the
\emph{local rule} if for every two consecutive entries $(u,\beta)$, $(v,\gamma)$ we
have
\[
 \gamma\le\beta+1,\qquad\text{and}\qquad \gamma=\beta+1\ \Longrightarrow\ u<v.
\]

\emph{Step $0$, the local rule is the definition.}  For $1\le i<n$ the inequality
$b_{i+1}\le b_i+1$ is $e_i\le e_{i+1}$, and $b_{i+1}=b_i+1$ is $e_i=e_{i+1}$, in
which case the local rule demands $\sigma_i<\sigma_{i+1}$.  These are exactly the
conditions of \cref{def:labelled} for the inner pairs.  At the left sentinel,
$b_1\le b_0+1=1$ is $e_1\ge0$, and if $b_1=1$ the local rule demands
$\sigma_0=0<\sigma_1$, which holds because $0$ is smaller than every label.  At the
right sentinel, $b_{n+1}=1\le b_n+1$ is $e_n\le n$, and if $b_n=0$ the local rule
demands $\sigma_n<n+1=\sigma_{n+1}$, which holds because $n+1$ is larger than every
label.  Thus $P\in\mathsf{PF}^+_{n,1}$ if and only if $W(P)$ satisfies the local
rule at all $n+1$ consecutive pairs.

\emph{Step $1$, the heights are a rooted permutation.}  Let $P$ be a labelled path
and let $B_h=\{c:h(c)=h\}$ be the set of ordinary labels of height $h$, where
$h(\sigma_i)=b_i$.  If $B_h\ne\varnothing$ for some $h\ge2$ then
$B_{h-1}\ne\varnothing$.  Indeed, the word starts at height $b_0=0$ and the heights
increase by at most one from one entry to the next, hence some entry strictly
between the left sentinel and an entry of height $h$ has height $h-1$.  It is not the
left sentinel, and it is not the right sentinel either, that being the last entry of
the word.  Moreover, at the first pair at which the height rises from $h-1$ to $h$
the local rule gives $d<c$ with $d\in B_{h-1}$ and $c\in B_h$, hence
\begin{equation}\label{eq:block-descent}
 \max B_h>\min B_{h-1}\qquad(h\ge2).
\end{equation}

Symmetrically, if $B_h\ne\varnothing$ for some $h\le-1$ then $B_{h+1}\ne\varnothing$.
Indeed, the word ends at height $b_{n+1}=1$ and the heights increase by at most one
from one entry to the next, hence some entry strictly to the right of an entry of
height $h$ has height $h+1\le0$.  That entry is not the right sentinel, whose height
is $1$, and it is not the left sentinel, which is the first entry of the word and
lies to the right of nothing.  At the first pair at which the height rises from $h$
to $h+1$ the local rule gives
\begin{equation}\label{eq:block-descent-neg}
 \max B_{h+1}>\min B_h\qquad(h\le-1).
\end{equation}

Let $H=\max\{h:B_h\ne\varnothing\}$, let $L=\min\{h:B_h\ne\varnothing\}$ if
$B_h\ne\varnothing$ for some $h\le 0$ and $L=1$ otherwise, and set
\[
 \tau=\bigl(B_H\mid B_{H-1}\mid\dots\mid B_1\mid 0,B_0\mid B_{-1}\mid\dots\mid B_L\bigr),
\]
each block written in increasing order.

Each block is increasing, and each junction
is a descent.  Namely, between $B_h$ and $B_{h-1}$ for $h\ge2$ this is
\eqref{eq:block-descent}, between $B_1$ and the block $0,B_0$ it holds because
$\max B_1>0$, between $0,B_0$ and $B_{-1}$ it is \eqref{eq:block-descent-neg} with
$h=-1$, whose left side is $\max B_0$, and between $B_h$ and $B_{h-1}$ for $h\le-1$
it is \eqref{eq:block-descent-neg} applied with $h-1$ in place of $h$.  The last two
junctions occur only when $B_0\ne\varnothing$, and if $B_0=\varnothing$ then
$B_{-1}=\varnothing$ as well, hence no junction is left unaccounted for.

Therefore
the blocks are the maximal increasing runs of $\tau$.  The run containing $0$ is the
$(H+1)$-st, and $c\in B_h$ lies in the $(H+1-h)$-th, hence $h_\tau(c)=h$ in the sense
of \cref{def:rooted}.  Conversely the blocks of any $\tau$ satisfy
\eqref{eq:block-descent} and \eqref{eq:block-descent-neg}, because its runs are
maximal.  Thus $P\mapsto\tau$ is the map recording the heights, and the possible
height functions are exactly the rooted permutations.

\emph{Step $2$, the insertion and the count of slots.}  Fix $\tau$.  Starting from
the word consisting of the two sentinels, we insert the ordinary labels in the
following order.  First the labels of height $1$, then those of height $2$, and so on
up to $H$, with the labels of one height inserted in decreasing order.  Then the
labels of height $0$, then those of height $-1$, and so on down to $L$, with the
labels of one height inserted in increasing order.  A label $c$ of height $h\ge1$ is
inserted immediately to the right of an entry, a label of height $h\le0$ immediately
to the left of an entry.  The sentinels are never displaced, hence the chosen entry is
not the right sentinel in the first case and not the left sentinel in the second.

Let $c\in B_h$ with $h\ge1$, and suppose that the labels of $B_{h'}$ for $h'<h$ and
the labels of $B_h$ greater than $c$ have been inserted, and no others.  All entries
then present have heights in $\{0,1,\dots,h\}$, the value $0$ occurring only at the
left sentinel and the value $1$ also at the right sentinel.

Inserting $c$ between
two consecutive entries $(u,\beta)$ and $(v,\gamma)$ destroys the pair $(u,v)$ and
creates the pairs $(u,c)$ and $(c,v)$, and the local rule at all other pairs is
unaffected.  For $(c,v)$ the rule requires $\gamma\le h+1$, which holds because
$\gamma\le h$, and the implication is vacuous because $\gamma=h+1$ never occurs.  For
$(u,c)$ the rule requires $h\le\beta+1$, that is $\beta\ge h-1$, hence
$\beta\in\{h-1,h\}$, and when $\beta=h-1$ it requires $u<c$.

Therefore the
admissible positions are exactly those immediately to the right of an entry $u$ with
either $\beta=h$, or $\beta=h-1$ and $u<c$, the right sentinel excluded.  The entries
with $\beta=h$ are the labels of $B_h$ greater than $c$, together with the right
sentinel when $h=1$, which is excluded.  The entries with $\beta=h-1$ are the labels
of $B_{h-1}$, all of which are present, when $h\ge2$, together with the right sentinel
when $h=2$, whose label $n+1$ is not smaller than $c$, and when $h=1$ they reduce to
the left sentinel, whose label $0$ is smaller than $c$, since $B_0$ has not yet been
inserted.  The number of admissible positions is therefore
\begin{gather*}
 \#\{d\in B_h:d>c\}+\#\{d\in B_{h-1}:d<c\}\quad(h\ge2),\\
 1+\#\{d\in B_1:d>c\}\quad(h=1),
\end{gather*}
which are the first two lines of \eqref{eq:schedule}.

Let now $c\in B_h$ with $h\le0$, and suppose that the labels of $B_{h'}$ for $h'>h$
and the labels of $B_h$ smaller than $c$ have been inserted, and no others.  All
entries present have heights in $\{h,h+1,\dots,H\}\cup\{0,1\}$, and in particular
all are $\ge h$.

For the pair $(u,c)$ created by inserting $c$ immediately to the
left of $(v,\gamma)$, the rule requires $h\le\beta+1$, which holds because
$\beta\ge h$, and the implication is vacuous because $\beta=h-1$ never occurs.  For
the pair $(c,v)$ the rule requires $\gamma\le h+1$, hence $\gamma\in\{h,h+1\}$, and
when $\gamma=h+1$ it requires $c<v$.

Therefore the admissible positions are exactly
those immediately to the left of an entry $v$ with either $\gamma=h$, or $\gamma=h+1$
and $v>c$, the left sentinel excluded.  The entries with $\gamma=h$ are the labels of
$B_h$ smaller than $c$, together with the left sentinel when $h=0$, which is
excluded.  The entries with $\gamma=h+1$ are the labels of $B_{h+1}$, all of which
are present, together with the right sentinel when $h=0$, whose label $n+1$ is
larger than $c$.  The left sentinel is not among them, since its height is $0=h+1$
only for $h=-1$, and then its label $0$ is not larger than $c$.  The number of
admissible positions is therefore
\begin{gather*}
 1+\#\{d\in B_0:d<c\}+\#\{d\in B_1:d>c\}\quad(h=0),\\
 \#\{d\in B_h:d<c\}+\#\{d\in B_{h+1}:d>c\}\quad(h\le-1),
\end{gather*}
which are the last two lines of \eqref{eq:schedule}.

In all four cases the set of admissible entries, which we call the set
$\mathcal A(c)$ of \emph{anchors} of $c$, is given uniformly by
\begin{equation}\label{eq:anchors}
 \mathcal A(c)=\begin{cases}
 \{d\in B_h:d>c\}\ \cup\ \{d\in\widehat B_{h-1}:d<c\},&h\ge1,\\
 \{d\in B_h:d<c\}\ \cup\ \{d\in\widehat B_{h+1}:d>c\},&h\le0,
 \end{cases}
\end{equation}
where $\widehat B_{h-1}=B_{h-1}$ for $h\ge2$ and $\widehat B_0=\{0\}$ is the left
sentinel for $h=1$, and where $\widehat B_{h+1}=B_{h+1}$ for $h\le-1$ and
$\widehat B_1=B_1\cup\{n+1\}$ adjoins the right sentinel for $h=0$.

Since the labels
of a word are distinct we read $\mathcal A(c)$ as a set of labels.  It depends only
on $\tau$ and $c$, not on the positions chosen for the earlier insertions.  Only the
left-to-right order of its members in the current word depends on those choices.  We
number the anchors of $c$ from right to left by $0,1,2,\dots$ when $h\ge1$ and from
left to right when $h\le0$, and we let $k_c$ be the number of the anchor used.

This
identifies the possible insertion sequences for the given $\tau$ with
$\mathcal B_\tau$, and each bound $s_\tau(c)$ is positive.  Indeed, for $h\in\{0,1\}$
the summand $1$ is present.  For $h\ge2$ the first term is nonzero unless
$c=\max B_h$, and for that $c$ the second term is nonzero by \eqref{eq:block-descent}.
For $h\le-1$ the first term is nonzero unless $c=\min B_h$, and for that $c$ the
second term is nonzero by \eqref{eq:block-descent-neg}.  By Step $0$ the word
obtained at the end is $W(P)$ for a labelled path $P$, whose heights are those of
$\tau$ by construction.

\emph{Step $3$, the inverse.}  Let $P$ be a labelled path and let $\tau$ be its
rooted permutation from Step $1$.  We delete the ordinary labels in the reverse
order.  First the labels of height $L$ in decreasing order, then those of height
$L+1$, and so on up to height $0$.  Then the labels of height $H$ in increasing
order, then those of height $H-1$, and so on down to height $1$.

We check that each
deletion is the inverse of an admissible insertion.  Let $c\in B_h$ with $h\ge1$ be
the label to be deleted, hence the labels present are those of $B_{h'}$ with $h'<h$
and those of $B_h$ greater than or equal to $c$, and let $(u,\beta)$ and $(v,\gamma)$
be the entries adjacent to $c$.  The local rule at $(u,c)$ gives $\beta\ge h-1$, and
$\beta\le h$ because all present heights are at most $h$.  If $\beta=h-1$ the rule
gives $u<c$.  Thus $u$ is an anchor of $c$, and $c$ was inserted immediately to its
right.

Deleting $c$ restores the pair $(u,v)$, which satisfies the local rule.
Indeed $\gamma\le h\le\beta+1$, and if $\gamma=\beta+1$ then $\gamma=h$ and
$\beta=h-1$, whence $u<c$ from the rule at $(u,c)$ and $c<v$ because $v$ is either
the right sentinel or a label of $B_h$ present at this moment, hence $u<v$.
Counting the anchors of $c$ lying to its right recovers $k_c$.

For $c\in B_h$ with
$h\le0$ the argument is the mirror image.  The present heights are all at least $h$,
and the local rule at $(c,v)$ gives $\gamma\le h+1$ and $\gamma\ge h$, hence $v$ is
an anchor and $c$ was inserted immediately to its left.  The restored pair $(u,v)$
satisfies $\gamma\le h+1\le\beta+1$, and if $\gamma=\beta+1$ then $\beta=h$ and
$\gamma=h+1$, whence $c<v$ from the rule at $(c,v)$ and $u<c$ because $u$ is either
the left sentinel or a label of $B_h$ present at this moment, hence $u<v$.  Counting
the anchors of $c$ lying to its left recovers $k_c$.  Deletion is therefore inverse
to insertion, and the map of Step $2$ is a bijection.

\emph{Step $4$, the two area statistics.}  At $m=1$ the kernel $[1-b]_++[-b]_+$
equals $1-2b$ for $b\le0$ and $0$ for $b\ge1$, and $[b]_++[b-1]_+$ equals $0$ for
$b\le0$ and $2b-1$ for $b\ge1$.  Summing over the ordinary labels and using
$h(\sigma_i)=b_i$ gives $A_-(P)=\sum_{h(c)\le0}(1-2h(c))=a_-(\tau)$ and
$A_+(P)=\sum_{h(c)\ge1}(2h(c)-1)=a_+(\tau)$.

\emph{Step $5$, $D=|\mathbf k|$.}  Recall from \cref{def:labelled} that at $m=1$
\[
 D(P)=\#\{i<j:\ b_i=b_j,\ \sigma_i<\sigma_j\}
 +\#\{i<j:\ b_i=b_j+1,\ \sigma_i>\sigma_j\},
\]
a sum over the unordered pairs of ordinary labels which depends only on their
heights, their labels and their left-to-right order.  An insertion does not change
the relative order, the heights or the labels of the entries already present, hence
it does not change the contribution of any pair already present.  It suffices to show
that inserting $c$ at the anchor number $k_c$ creates exactly $k_c$ new
contributions.

Let $h=h(c)\ge1$ and let $d$ be an ordinary label already present, hence
$h(d)\in\{1,\dots,h\}$ and, if $h(d)=h$, then $d>c$.  If $h(d)\le h-2$ the pair
$\{c,d\}$ contributes to neither sum, since the two heights are neither equal nor
consecutive.  If $h(d)=h$, hence $d>c$, the pair contributes to the first sum if and
only if $c$ precedes $d$, because the earlier of the two must carry the smaller
label, and it never contributes to the second sum, the heights being equal.  If
$h(d)=h-1$ the pair contributes to the second sum if and only if $c$ precedes $d$ and
$c>d$, since the earlier of the two must have the larger height, which forces it to
be $c$, and it never contributes to the first sum.

Thus a new contribution arises
exactly from the labels $d$ lying to the right of the insertion point with
$d\in B_h$, $d>c$, or with $d\in B_{h-1}$, $d<c$, that is, exactly from the anchors
of $c$ lying to the right of the insertion point.

The two counts agree even though
$D$ counts only ordinary labels while $\mathcal A(c)$ may contain a sentinel.  The
only sentinel it can contain is the left one, at $h=1$, and that is the first entry
of the word, hence never to the right of the insertion point.  Since $c$ is inserted
immediately to the right of the anchor number $k_c$ and the anchors are numbered from
the right, there are exactly $k_c$ of them.

Let now $h=h(c)\le0$ and let $d$ be an ordinary label already present, hence
$h(d)\ge h$ and, if $h(d)=h$, then $d<c$.  If $h(d)\ge h+2$ the pair contributes to
neither sum.  If $h(d)=h$, hence $d<c$, the pair contributes to the first sum if and
only if $d$ precedes $c$, the earlier of the two carrying the smaller label.  If
$h(d)=h+1$ the pair contributes to the second sum if and only if $d$ precedes $c$ and
$d>c$, the earlier of the two having the larger height.

Thus a new contribution
arises exactly from the anchors of $c$ lying to the left of the insertion point.
Again the only sentinel that $\mathcal A(c)$ can contain is the right one, at $h=0$,
and that is the last entry of the word, hence never to the left of the insertion
point.  Since $c$ is inserted immediately to the left of the anchor number $k_c$ and
the anchors are numbered from the left, there are exactly $k_c$ of them.  Summing
over the insertions, $D(P)=\sum_ck_c=|\mathbf k|$.

\emph{Step $6$, the inverse descents.}  The reading word lists the ordinary labels
by increasing $(b_i,i)$, that is by increasing height and, within one height, from
left to right.  Hence $c\in\operatorname{ides}(w(P))$, which means that $c+1$
precedes $c$ in the reading word, holds if and only if $h(c+1)<h(c)$, or
$h(c+1)=h(c)$ and $c+1$ lies to the left of $c$ in $W(P)$.  It remains to treat the
case $h(c+1)=h(c)=h$.

We isolate the mechanism as a separate observation.  Let $V$ be a word, let $X$ be a
set of entries of $V$, and let $V'$ be obtained from $V$ by inserting a new entry $e$
immediately to the right of the member of $X$ whose right-to-left index in $V$ is
$j$.  Then in $V'$ the members of $X$ retain their relative order, and the
right-to-left index of $e$ inside $X\cup\{e\}$ is again $j$.  Indeed, nothing
separates $e$ from its anchor, hence the members of $X$ lying to the right of $e$ in
$V'$ are exactly those lying strictly to the right of the anchor in $V$, namely the
$j$ members of index $0,\dots,j-1$.  Reading $V'$ from the right one meets those $j$
first, then $e$, then the rest of $X$.  The mirror statement holds with ``right'' and
``left'' exchanged throughout.  We refer to this as the \emph{index lemma}.

Let $h\ge1$.  Since the labels of $B_h$ are inserted in decreasing order and
$c,c+1\in B_h$ are consecutive integers, $c+1$ is the immediate predecessor of $c$ in
the insertion order, hence no label at all is inserted between them.  We claim that
\begin{equation}\label{eq:anchor-step}
 \mathcal A(c)=\mathcal A(c+1)\sqcup\{c+1\}.
\end{equation}

Indeed, by \eqref{eq:anchors} the first constituent of $\mathcal A(c)$ is
$\{d\in B_h:d>c\}=\{d\in B_h:d>c+1\}\cup\{c+1\}$, the union being disjoint because
$c+1\in B_h$.  The second constituents agree,
$\{d\in\widehat B_{h-1}:d<c\}=\{d\in\widehat B_{h-1}:d<c+1\}$, because the two sets
can differ only in the element $c$, and $c\notin\widehat B_{h-1}$.  Namely, for
$h\ge2$ the height function is single-valued and $c\in B_h$, while for $h=1$ one has
$\widehat B_0=\{0\}$ and $c\ge1$.  Finally $c+1\notin\mathcal A(c+1)$, because the
first constituent of $\mathcal A(c+1)$ has all its elements strictly greater than
$c+1$, and $c+1\notin\widehat B_{h-1}$ by the argument just given with $c+1$ in place
of $c$.  This proves \eqref{eq:anchor-step}.

We apply the index lemma with $V$ the word just before $c+1$ is inserted,
$X=\mathcal A(c+1)$, $e=c+1$ and $j=k_{c+1}$.  Since no label is inserted between
$c+1$ and $c$, the word $V'$ it produces is exactly the word just before $c$ is
inserted, and by \eqref{eq:anchor-step} the set $X\cup\{e\}$ is exactly
$\mathcal A(c)$.  Hence at that moment the right-to-left index of the entry $c+1$
inside $\mathcal A(c)$ is $k_{c+1}$.

Write $\mathcal A(c)=\{b_0,b_1,\dots,b_r\}$ in
that right-to-left order, hence $b_{k_{c+1}}=c+1$.  The label $c$ is inserted
immediately to the right of $b_{k_c}$, hence it lies to the right of
$b_{k_c},b_{k_c+1},\dots,b_r$ and to the left of $b_{k_c-1},\dots,b_0$, and no later
insertion changes the relative order of the entries already present.  Therefore $c+1$
lies to the left of $c$ in $W(P)$ if and only if $k_{c+1}\ge k_c$, which is the
second clause of the rule.

Let $h\le0$.  Now the labels of $B_h$ are inserted in increasing order, hence $c$ is
the immediate predecessor of $c+1$ in the insertion order and no label is inserted
between them.  This time
\begin{equation}\label{eq:anchor-step-neg}
 \mathcal A(c+1)=\mathcal A(c)\sqcup\{c\},
\end{equation}
by the mirrored computation.  Namely, $\{d\in B_h:d<c+1\}=\{d\in B_h:d<c\}\cup\{c\}$
disjointly because $c\in B_h$, while
$\{d\in\widehat B_{h+1}:d>c+1\}=\{d\in\widehat B_{h+1}:d>c\}$ because the two sets
can differ only in the element $c+1$, and $c+1\notin\widehat B_{h+1}$.  Indeed, for
$h\le-1$ the height function is single-valued and $c+1\in B_h$, while for $h=0$ one
has $\widehat B_1=B_1\cup\{n+1\}$, and $c+1\notin B_1$ for the same reason and
$c+1\le n<n+1$.

Also $c\notin\mathcal A(c)$, both constituents of $\mathcal A(c)$
being defined by inequalities strict at $c$.

By the mirrored index lemma, applied
with $V$ the word just before $c$ is inserted, $X=\mathcal A(c)$, $e=c$ and $j=k_c$,
the left-to-right index of the entry $c$ inside $\mathcal A(c+1)$, in the word just
before $c+1$ is inserted, is $k_c$.

Writing
$\mathcal A(c+1)=\{b'_0,\dots,b'_{r'}\}$ in left-to-right order, hence
$b'_{k_c}=c$, the label $c+1$ is inserted immediately to the left of $b'_{k_{c+1}}$,
hence it lies to the left of $b'_{k_{c+1}},\dots,b'_{r'}$ and to the right of
$b'_{k_{c+1}-1},\dots,b'_0$.  Therefore $c+1$ lies to the left of $c$ if and only if
$k_c\ge k_{c+1}$, which is the third clause of the rule.

In the remaining case $h(c+1)>h(c)$ the label $c+1$ follows $c$ in the reading word,
hence $c\notin\operatorname{ides}(w(P))$, in agreement with the rule.

Steps $2$ and $3$ together give the bijection.  Step $2$ produces from each
$(\tau,\mathbf k)$ a labelled path whose rooted permutation is $\tau$, Step $3$
produces from each labelled path a pair $(\tau,\mathbf k)$, and each construction
undoes the other insertion by insertion.  The first formula of
\eqref{eq:rooted-forms} now follows by summing \eqref{eq:Hplus} over the fibres of
the bijection, using Steps $4$, $5$ and $6$ for the three statistics and the descent
set.  The second formula follows from the first by
\cref{lem:labelled-basic}\textup{(ii)}, which identifies
$\langle\mathcal H^+_{n,1},h_1^n\rangle$ with the plain weight enumerator, together
with the independence of the choices $k_c$, each contributing $[s_\tau(c)]_{qt}$.
\end{proof}

\begin{example}\label{ex:rooted}
Let $n=4$ and $\tau=3\,0\,2\,4\,1$, with maximal increasing runs
$3\mid 0,2,4\mid 1$.  The run containing $0$ is the second one, hence by
\cref{def:rooted} $h_\tau(3)=1$, $h_\tau(2)=h_\tau(4)=0$ and $h_\tau(1)=-1$.  Thus
$B_1=\{3\}$, $B_0=\{2,4\}$, $B_{-1}=\{1\}$, and $a_-(\tau)=3+1+1=5$,
$a_+(\tau)=1$.  The bounds \eqref{eq:schedule} are
\begin{align*}
 s_\tau(3)&=1+\#\{d\in B_1:d>3\}=1,\\
 s_\tau(2)&=1+\#\{d\in B_0:d<2\}+\#\{d\in B_1:d>2\}=2,\\
 s_\tau(4)&=1+\#\{d\in B_0:d<4\}+\#\{d\in B_1:d>4\}=2,\\
 s_\tau(1)&=\phantom{1+{}}\#\{d\in B_{-1}:d<1\}+\#\{d\in B_0:d>1\}=2,
\end{align*}
hence $\mathcal B_\tau$ has $1\cdot2\cdot2\cdot2=8$ elements.  We take
$\mathbf k=(k_1,k_2,k_3,k_4)=(0,1,0,0)$ and run the insertion of
\cref{thm:rooted-insertion}.  The labels enter in the order $3,2,4,1$, that is first
height $1$, then height $0$ in increasing order of the label, then height $-1$.

The insertion is easiest to follow with the word drawn as a skyline.  The $j$-th
entry is a box at horizontal position $j$, at height equal to its own height, the
two sentinels are shaded, and the dashed line separates the entries of height
$\ge1$ from those of height $\le0$, that is the part of the path strictly above the
line $y=x$ from the part on or below it.

Below each picture the carets mark the
anchors $\mathcal A(c)$ of the label about to be inserted, numbered in the direction
in which \cref{thm:rooted-insertion} counts them, from the right for a label of
height $\ge1$ and from the left for a label of height $\le0$, and the arrow marks the
slot chosen by $k_c$.  We start from the two sentinels,
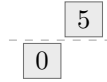
\[
\begin{tikzpicture}[scale=0.55,baseline=(current bounding box.center)]
  \draw[black!35,dashed] (-0.3,1.00) -- (2.30,1.00);
  \draw[fill=black!8,draw=black!55] (0.05,0.08) rectangle (0.95,0.92);
  \node[font=\small] at (0.50,0.50) {$0$};
  \draw[fill=black!8,draw=black!55] (1.05,1.08) rectangle (1.95,1.92);
  \node[font=\small] at (1.50,1.50) {$5$};
\end{tikzpicture}
\]
The label $3$, of height $1$, has the single anchor $0$, hence $k_3=0$ and it goes
immediately to its right,
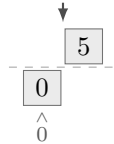
\[
\begin{tikzpicture}[scale=0.55,baseline=(current bounding box.center)]
  \draw[black!35,dashed] (-0.3,1.00) -- (2.30,1.00);
  \draw[fill=black!8,draw=black!55] (0.05,0.08) rectangle (0.95,0.92);
  \node[font=\small] at (0.50,0.50) {$0$};
  \draw[fill=black!8,draw=black!55] (1.05,1.08) rectangle (1.95,1.92);
  \node[font=\small] at (1.50,1.50) {$5$};
  \node[font=\scriptsize,black!65] at (0.50,-0.22) {$\scriptstyle\wedge$};
  \node[font=\scriptsize,black!65] at (0.50,-0.62) {$0$};
  \draw[-latex,black!70,thick] (1.00,2.55) -- (1.00,2.08);
\end{tikzpicture}
\]
The label $2$, of height $0$, has the two anchors $3$ and $5$ counted from the
left, and $k_2=1$ puts it immediately left of $5$,
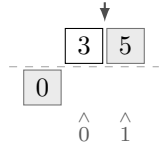
\[
\begin{tikzpicture}[scale=0.55,baseline=(current bounding box.center)]
  \draw[black!35,dashed] (-0.3,1.00) -- (3.30,1.00);
  \draw[fill=black!8,draw=black!55] (0.05,0.08) rectangle (0.95,0.92);
  \node[font=\small] at (0.50,0.50) {$0$};
  \draw[draw] (1.05,1.08) rectangle (1.95,1.92);
  \node[font=\small] at (1.50,1.50) {$3$};
  \draw[fill=black!8,draw=black!55] (2.05,1.08) rectangle (2.95,1.92);
  \node[font=\small] at (2.50,1.50) {$5$};
  \node[font=\scriptsize,black!65] at (1.50,-0.22) {$\scriptstyle\wedge$};
  \node[font=\scriptsize,black!65] at (1.50,-0.62) {$0$};
  \node[font=\scriptsize,black!65] at (2.50,-0.22) {$\scriptstyle\wedge$};
  \node[font=\scriptsize,black!65] at (2.50,-0.62) {$1$};
  \draw[-latex,black!70,thick] (2.00,2.55) -- (2.00,2.08);
\end{tikzpicture}
\]
The label $4$, also of height $0$, has the anchors $2$ and $5$, and $k_4=0$ puts it
immediately left of $2$,
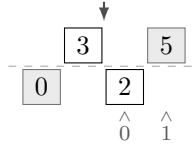
\[
\begin{tikzpicture}[scale=0.55,baseline=(current bounding box.center)]
  \draw[black!35,dashed] (-0.3,1.00) -- (4.30,1.00);
  \draw[fill=black!8,draw=black!55] (0.05,0.08) rectangle (0.95,0.92);
  \node[font=\small] at (0.50,0.50) {$0$};
  \draw[draw] (1.05,1.08) rectangle (1.95,1.92);
  \node[font=\small] at (1.50,1.50) {$3$};
  \draw[draw] (2.05,0.08) rectangle (2.95,0.92);
  \node[font=\small] at (2.50,0.50) {$2$};
  \draw[fill=black!8,draw=black!55] (3.05,1.08) rectangle (3.95,1.92);
  \node[font=\small] at (3.50,1.50) {$5$};
  \node[font=\scriptsize,black!65] at (2.50,-0.22) {$\scriptstyle\wedge$};
  \node[font=\scriptsize,black!65] at (2.50,-0.62) {$0$};
  \node[font=\scriptsize,black!65] at (3.50,-0.22) {$\scriptstyle\wedge$};
  \node[font=\scriptsize,black!65] at (3.50,-0.62) {$1$};
  \draw[-latex,black!70,thick] (2.00,2.55) -- (2.00,2.08);
\end{tikzpicture}
\]
Finally the label $1$, of height $-1$, has the anchors $4$ and $2$, and $k_1=0$ puts
it immediately left of $4$,
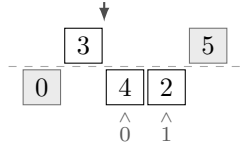
\[
\begin{tikzpicture}[scale=0.55,baseline=(current bounding box.center)]
  \draw[black!35,dashed] (-0.3,1.00) -- (5.30,1.00);
  \draw[fill=black!8,draw=black!55] (0.05,0.08) rectangle (0.95,0.92);
  \node[font=\small] at (0.50,0.50) {$0$};
  \draw[draw] (1.05,1.08) rectangle (1.95,1.92);
  \node[font=\small] at (1.50,1.50) {$3$};
  \draw[draw] (2.05,0.08) rectangle (2.95,0.92);
  \node[font=\small] at (2.50,0.50) {$4$};
  \draw[draw] (3.05,0.08) rectangle (3.95,0.92);
  \node[font=\small] at (3.50,0.50) {$2$};
  \draw[fill=black!8,draw=black!55] (4.05,1.08) rectangle (4.95,1.92);
  \node[font=\small] at (4.50,1.50) {$5$};
  \node[font=\scriptsize,black!65] at (2.50,-0.22) {$\scriptstyle\wedge$};
  \node[font=\scriptsize,black!65] at (2.50,-0.62) {$0$};
  \node[font=\scriptsize,black!65] at (3.50,-0.22) {$\scriptstyle\wedge$};
  \node[font=\scriptsize,black!65] at (3.50,-0.62) {$1$};
  \draw[-latex,black!70,thick] (2.00,2.55) -- (2.00,2.08);
\end{tikzpicture}
\]
Deleting the sentinels from the completed word leaves $\sigma=(3,1,4,2)$ with
heights $b=(1,-1,0,0)$, and the skyline is the sequence of heights of the labelled
path of \cref{ex:pathA} read from bottom to top,
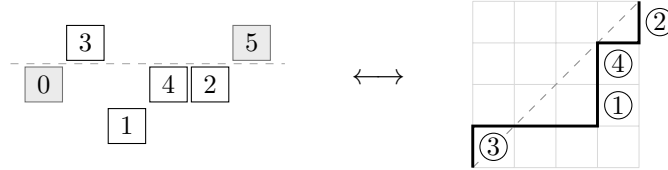
\[
\begin{tikzpicture}[scale=0.55,baseline=(current bounding box.center)]
  \draw[black!35,dashed] (-0.3,1.00) -- (6.30,1.00);
  \draw[fill=black!8,draw=black!55] (0.05,0.08) rectangle (0.95,0.92);
  \node[font=\small] at (0.50,0.50) {$0$};
  \draw[draw] (1.05,1.08) rectangle (1.95,1.92);
  \node[font=\small] at (1.50,1.50) {$3$};
  \draw[draw] (2.05,-0.92) rectangle (2.95,-0.08);
  \node[font=\small] at (2.50,-0.50) {$1$};
  \draw[draw] (3.05,0.08) rectangle (3.95,0.92);
  \node[font=\small] at (3.50,0.50) {$4$};
  \draw[draw] (4.05,0.08) rectangle (4.95,0.92);
  \node[font=\small] at (4.50,0.50) {$2$};
  \draw[fill=black!8,draw=black!55] (5.05,1.08) rectangle (5.95,1.92);
  \node[font=\small] at (5.50,1.50) {$5$};
\end{tikzpicture}
\qquad\longleftrightarrow\qquad
\begin{tikzpicture}[scale=0.55,baseline=(current bounding box.center)]
  \draw[step=1,black!15,very thin] (0,0) grid (4,4);
  \draw[black!45,dashed] (0,0) -- (4,4);
  \draw[line width=1.1pt] (0,0) -- (0,1) -- (3,1) -- (3,2) -- (3,3) -- (4,3) -- (4,4);
  \node[draw,circle,inner sep=0.8pt,fill=white,font=\small] at (0.5,0.5) {$3$};
  \node[draw,circle,inner sep=0.8pt,fill=white,font=\small] at (3.5,1.5) {$1$};
  \node[draw,circle,inner sep=0.8pt,fill=white,font=\small] at (3.5,2.5) {$4$};
  \node[draw,circle,inner sep=0.8pt,fill=white,font=\small] at (4.5,3.5) {$2$};
\end{tikzpicture}
\]

The four assertions of \cref{thm:rooted-insertion} can be read off against that
example.  Namely $a_-(\tau)=5=A^{(1)}_-$, $a_+(\tau)=1=A^{(1)}_+$,
$|\mathbf k|=1=D(P)$, and $\mathcal D_\tau(\mathbf k)=\{3\}=\operatorname{ides}(w(P))$,
the last because $h_\tau(4)=0<1=h_\tau(3)$ while the pairs $(1,2)$ and $(2,3)$ have
$h_\tau(2)>h_\tau(1)$ and $h_\tau(3)>h_\tau(2)$.
\end{example}

The two extreme sectors of \cref{thm:rooted-insertion} are the type-$A$ formula of
\cref{rem:typeA-schedules}, and the dictionary is exact.

\begin{proposition}\label{prop:typeA-sectors}
Let $n\ge1$.
\begin{enumerate}
\item Let $\tau\in S_{\{0,\dots,n\}}$ have $\tau_{n+1}=0$, and let $\tau'\in S_n$ be
 $\tau$ with the $0$ deleted.  Then $h_\tau(c)\ge1$ for every $c$, and
 \[
  s_\tau(\tau'_i)=w_i(\tau')\quad(1\le i\le n),\qquad a_-(\tau)=0,\qquad
  a_+(\tau)=2\operatorname{maj}(\tau')+n,
 \]
 with $w_i$ as in \eqref{eq:haglund-w}.  Conversely $a_-(\tau)=0$ forces
 $\tau_{n+1}=0$.
\item The bijection of \cref{thm:rooted-insertion} carries
 $\bigsqcup_{\tau_{n+1}=0}\mathcal B_\tau$ onto
 $\mathcal P_n:=\{P\in\mathsf{PF}^+_{n,1}:b_i(P)\ge1\text{ for all }i\}$, which is the
 set of type-$A$ parking functions.  On $\mathcal P_n$ one has $D=\dinv$,
 $A^{(1)}_-=0$ and $A^{(1)}_+=2\operatorname{area}+n$, and $\operatorname{ides}(w(P))$
 is the complement in $\{1,\dots,n-1\}$ of the inverse descent set of Haglund's
 reading word $\operatorname{read}(P)$.
\item Hence
 \[
  \sum_{\substack{\tau\in S_{\{0,\dots,n\}}\\ \tau_{n+1}=0}}
  q^{a_-(\tau)}t^{a_+(\tau)}\prod_{c=1}^n[s_\tau(c)]_{qt}
  =t^n\Bigl(\sum_{P\in\mathcal P_n}
  q^{\dinv(P)}t^{\operatorname{area}(P)}\Bigr)\Bigl|_{q\mapsto qt,\;t\mapsto t^2},
 \]
 and the sector $\tau_1=0$ contributes the image of this under $q\leftrightarrow t$.
\end{enumerate}
\end{proposition}

\begin{proof}
(1)  If $\tau_{n+1}=0$ then $\tau_n>0$, hence $0$ is a maximal run by itself, $s=r$
and $R_s=\{0\}$.  Deleting $0$ leaves the runs $R_1|\dots|R_{s-1}$ of $\tau'$, and
$h_\tau(c)=s-j\ge1$ for $c\in R_j$, $j\le s-1$.  For $h\ge2$ we have $B_h=R_{s-h}$
and $B_{h-1}=R_{s-h+1}$, the run of $\tau'$ immediately to the right of $B_h$, hence
the first line of \eqref{eq:schedule} is \eqref{eq:haglund-w}.

For $h=1$ we have
$B_1=R_{s-1}$, the last run of $\tau'$, whose successor in Haglund's convention is
the $(k+2)$nd run $\{0\}$.  Since $0<c$ for every $c$, his second count equals $1$,
which is the constant in the second line of \eqref{eq:schedule}.  There is no $c$
with $h_\tau(c)\le0$, hence $a_-(\tau)=0$ and
$a_+(\tau)=\sum_c(2h_\tau(c)-1)=2\sum_ch_\tau(c)-n$.  For $c\in R_j$ the number
$h_\tau(c)-1=(s-1)-j$ is exactly the length of the row of $c$ in
$\operatorname{cars}(\tau')$, whose total is $\operatorname{maj}(\tau')$
\cite[p.~81]{HaglundBook}.  Hence $a_+(\tau)=2(\operatorname{maj}(\tau')+n)-n$.

Conversely $a_-(\tau)=0$ means $B_h=\emptyset$ for every $h\le0$, that is
$R_j=\emptyset$ for $j>r$ and $R_r=\{0\}$.  Runs are nonempty, hence $r=s$ and
$\tau_{n+1}=0$.

(2)  The statistic $A^{(1)}_-(P)=\sum_i([1-b_i]_++[-b_i]_+)$ vanishes iff $b_i\ge1$
for all $i$, and $A^{(1)}_-(P)=a_-(\tau)$, hence the sector is carried onto the set
of $P$ with all $b_i\ge1$.  Such a $P$ is a path from $(0,0)$ to $(n,n)$ weakly above
$y=x$ together with a labelling increasing up the columns, that is a type-$A$ parking
function, and its $i$-th row has length $a_i=b_i-1$.  Haglund defines
\[
 \dinv(P)=\#\{i<j:a_i=a_j,\ \operatorname{occ}(i)<\operatorname{occ}(j)\}
 +\#\{i<j:a_i=a_j+1,\ \operatorname{occ}(i)>\operatorname{occ}(j)\}
\]
\cite[pp.~78--79]{HaglundBook}, the rows being indexed so that a larger index means
a higher row, and his gloss is ``pairs of rows of $P$ of the same length, with the
row above containing the larger car, or which differ by one in length, with the
longer row below the shorter, and the longer row containing the larger car''.

Substituting $a_i=b_i-1$ and $\operatorname{occ}(i)=\sigma_i$ turns this into
\eqref{eq:D-one}, term for term.  Next,
$A^{(1)}_+(P)=\sum_i(b_i+[b_i-1]_+)=\sum_i(2b_i-1)=2\sum_i(b_i-1)+n$ and
$\operatorname{area}(P)=\sum_i(b_i-1)$.

Finally, $\operatorname{read}(P)$ is
obtained by reading the cars ``along diagonals in a SE direction, starting with the
diagonal farthest from the line $y=x$, then working inwards''
\cite[p.~79]{HaglundBook}, that is by decreasing $(b_i,i)$, whereas $w(P)$ reads by
increasing $(b_i,i)$.  The two words are reverses of each other.  For any word $u$
and any $c$, exactly one of ``$c+1$ lies left of $c$ in $u$'' and ``$c+1$ lies left
of $c$ in the reverse of $u$'' holds, hence the two inverse descent sets are
complementary in $\{1,\dots,n-1\}$.

(3)  We substitute (1) into the second identity of \eqref{eq:rooted-forms} and use
(2) with $\langle F_{n,S},h_1^n\rangle=1$.  For the last sentence, the map which
reverses $\tau$ and then replaces each nonzero $c$ by $n+1-c$ is an involution of
$S_{\{0,\dots,n\}}$ exchanging the two sectors, and by \cref{lem:complementation} the
corresponding involution $P\mapsto P^c$ of $\mathsf{PF}^+_{n,1}$ fixes $D$ and
exchanges $A^{(1)}_-$ with $A^{(1)}_+$.
\end{proof}

Because \cref{thm:rooted-insertion} tracks $\operatorname{ides}$ and not only the
bidegree, the same computation refines \eqref{eq:haglund-53} to the quasisymmetric
level.  For $\tau\in S_n$ with runs $R_1|\dots|R_{k+1}$ write
$a_\tau(c)=k+1-j$ for $c\in R_j$, the length of the row of $c$ in
$\operatorname{cars}(\tau)$, and $w_\tau(c)=w_i(\tau)$ for $\tau_i=c$.

\begin{corollary}\label{cor:typeA-schedule}
For every $n\ge1$,
\begin{equation}\label{eq:typeA-quasisym}
 \sum_{P\in\mathcal P_n}q^{\dinv(P)}t^{\operatorname{area}(P)}
 F_{n,\operatorname{Des}(\operatorname{read}(P)^{-1})}(X)
 =\sum_{\tau\in S_n}t^{\operatorname{maj}(\tau)}
 \sum_{\mathbf k\in\mathcal B^A_\tau}q^{|\mathbf k|}F_{n,E_\tau(\mathbf k)}(X),
\end{equation}
where $\mathcal B^A_\tau=\prod_{c=1}^n\{0,1,\dots,w_\tau(c)-1\}$ and
$c\in E_\tau(\mathbf k)$ iff $a_\tau(c+1)>a_\tau(c)$, or $a_\tau(c+1)=a_\tau(c)$ and
$k_c>k_{c+1}$.  By the shuffle theorem \cite{CarlssonMellit} both sides equal
$\nabla e_n$.
\end{corollary}

\begin{proof}
We apply \cref{prop:typeA-sectors} to the first identity of \eqref{eq:rooted-forms}
restricted to the sector $\tau_{n+1}=0$, and we complement both descent sets.  The
substitution $h_\tau(c)=a_{\tau'}(c)+1$ turns the three clauses defining
$\mathcal D_\tau(\mathbf k)$ in \cref{thm:rooted-insertion} into two, the third being
vacuous on this sector, and complementing gives $E_\tau$.
\end{proof}

We did not find \eqref{eq:typeA-quasisym} in the literature.  Every schedule
identity in \cite{HaglundBook,HL,HagSergel} is a statement about
$q^{\dinv}t^{\operatorname{area}}$ alone, and a search for a quasisymmetric or
Gessel-$F$ refinement of \cite[Thm.~5.3]{HaglundBook} returned only the
Hilbert-series form.  The descent rule $E_\tau(\mathbf k)$ is what the monomial
$M_{\tau,\mathbf k}$ contributes, hence \eqref{eq:typeA-quasisym} is the
quasisymmetric shadow of the statement that the Carlsson--Oblomkov basis \cite{CO}
carries the $S_n$-character and not only the bigraded dimension.  The type-$C$
statement is \cref{conj:rooted-basis}.

The intermediate sectors are governed by the third line of \eqref{eq:schedule}, which
couples the labels on the two sides of the root.  They all carry the same
$\bigl(D,\operatorname{ides}\bigr)$ distribution, by the circular-street argument
which Haglund uses to count parking functions \cite[p.~77]{HaglundBook}.

\begin{proposition}\label{prop:chung-feller}
Identify $\mathsf{PF}^+_{n,1}$ with the set of functions
$f\colon\{1,\dots,n\}\to\{0,1,\dots,n\}$ by sorting the pairs $(f(c),c)$
lexicographically into $(e_i,\sigma_i)$, and let
$\varsigma(f)(c)=f(c)+1 \bmod (n+1)$.  Then $\varsigma$ fixes $w(P)$ and $D(P)$, and
each $\varsigma$-orbit has $n+1$ elements, one in each sector
$\mathcal S_p=\{P:\#\{i:b_i\ge1\}=p\}$, $0\le p\le n$.  Consequently
\[
 \sum_{P\in\mathcal S_p}(qt)^{D(P)}F_{n,\operatorname{ides}(w(P))}
 =\sum_{\tau\in S_n}\ \sum_{\mathbf k\in\mathcal B^A_\tau}(qt)^{|\mathbf k|}
 F_{n,\{1,\dots,n-1\}\setminus E_\tau(\mathbf k)}(X)
\]
for every $p$, independently of $p$, and by \cref{cor:typeA-schedule} this is
$\bigl(\omega\nabla e_n\bigr)\bigl|_{t=1,\ q\mapsto qt}$.
\end{proposition}

\begin{proof}
Write $a=\#f^{-1}(n)$, hence $e_{n-a+1}=\dots=e_n=n$ and
$\sigma_{n-a+1}<\dots<\sigma_n$.  Applying $\varsigma$ moves these $a$ cars to the
value $0$ and adds $1$ to the remaining $n-a$ values, hence
\[
 e(\varsigma f)=(\underbrace{0,\dots,0}_{a},e_1+1,\dots,e_{n-a}+1),\qquad
 \sigma(\varsigma f)=(\sigma_{n-a+1},\dots,\sigma_n,\sigma_1,\dots,\sigma_{n-a}),
\]
which is again a labelled path.  Its heights are $b'_i=i=b_{n-a+i}+a$ for $i\le a$
and $b'_i=b_{i-a}+a-1$ for $i>a$.  Thus every height rises by $a-1$, and the $a$
cars that wrapped around rise by one more.

A pair of cars both inside the wrapped
block, or both outside it, therefore keeps its height difference and its relative
order, and it contributes the same to \eqref{eq:D-one}.  For a pair consisting of a
wrapped car, with old index $i'>n-a$, and an unwrapped one, with old index
$j'\le n-a$, the order reverses and the difference increases by one,
$b'_{i'}-b'_{j'}=b_{i'}-b_{j'}+1$ with $i'$ now preceding $j'$.

The old pair
$(j',i')$ is counted by the first clause of \eqref{eq:D-one} iff $b_{j'}=b_{i'}$ and
$\sigma_{j'}<\sigma_{i'}$, which is exactly the condition for the new pair
$(i',j')$ to be counted by the second clause, namely $b'_{i'}=b'_{j'}+1$ and
$\sigma_{i'}>\sigma_{j'}$, and symmetrically the second clause for $(j',i')$ becomes
the first for $(i',j')$.  Thus $D$ is preserved pair by pair.

The reading word is
unchanged.  Two wrapped cars, or two unwrapped ones, keep their relative position in
the order by increasing $(b_i,i)$.

For a wrapped car with new index $i\le a$ and an
unwrapped one with new index $l+a$, the comparison is between $(i-a,n-a+i)$ and
$(b_l,l)$ before the shift and between $(i,i)$ and $(b_l+a-1,l+a)$ after it.  Before
the shift the wrapped car precedes the other iff $i-a<b_l$, the tie $i-a=b_l$ going
the other way because $l\le n-a<n-a+i$.  After the shift it precedes iff
$i<b_l+a-1$, or $i=b_l+a-1$, the tie now going its way because $i\le a<l+a$.  That
is, it precedes iff $i\le b_l+a-1$, which is again $i-a<b_l$.

For the sectors, let $m_j=\#f^{-1}(j)$ and $W(k)=\sum_{l<k}(m_l-1)$ for
$0\le k\le n+1$, hence $W(0)=0$ and $W(n+1)=-1$, and extend $W$ to $\Z$ by
$W(k+n+1)=W(k)-1$.  Since $e_i\le i-1$ iff at least $i$ of the values of $f$ are
$\le i-1$, that is iff $m_0+\dots+m_{i-1}\ge i$, that is iff $W(i)\ge0$, we have
\[
 p(f)=\#\{i:b_i\ge1\}=\#\{i\in\{1,\dots,n\}:W(i)\ge0\}=\#\{k\in\{1,\dots,n+1\}:W(k)\ge0\},
\]
the last step because $W(n+1)=-1$.

Now $\varsigma$ rotates $(m_0,\dots,m_n)$, and
the translates of $f$ are the functions whose $W$-walks are $i\mapsto W(k+i)-W(k)$
for $0\le k\le n$, the translate $\varsigma^{-k}f$ corresponding to $k$.  Hence
\begin{align*}
 p(\varsigma^{-k}f)&=\#\{i\in\{k+1,\dots,k+n+1\}:W(i)\ge W(k)\}\\
 &=\#\{i>k:W(i)\ge W(k)\}+\#\{i<k:W(i)>W(k)\},
\end{align*}
where in the second expression $i$ runs over $\{0,\dots,n\}$.  The terms with
$i\le k$ have been reduced modulo $n+1$, which lowers $W$ by one, and the term $i=k$
contributes $W(k+n+1)=W(k)-1<W(k)$, that is nothing.

The right-hand side is the
number of indices strictly greater than $k$ in the total order on $\{0,\dots,n\}$ in
which $i\prec k$ iff $W(i)<W(k)$, or $W(i)=W(k)$ and $i<k$.  As $k$ runs over
$\{0,\dots,n\}$ that number runs over $\{0,\dots,n\}$ bijectively, which is the
assertion about the orbits.  In particular all orbits are free.

For the displayed identity, we apply what has just been proved to the sector $p=n$,
which is $\mathcal P_n$, and we use \cref{prop:typeA-sectors}(2) and
\cref{cor:typeA-schedule} at $t=1$, the reading word $w(P)$ being the reverse of
$\operatorname{read}(P)$.  The identification with $\omega\nabla e_n$ uses
$\omega F_{n,S}=F_{n,\{1,\dots,n-1\}\setminus S}$ for the $F$-expansion of a
symmetric function, which follows from $s_\lambda=\sum_QF_{n,\operatorname{Des}(Q)}$
over the standard Young tableaux of shape $\lambda$, already used in the proof of
\cref{lem:labelled-basic}, together with $\omega s_\lambda=s_{\lambda'}$ and
$\operatorname{Des}(Q^{\mathsf t})=\{1,\dots,n-1\}\setminus\operatorname{Des}(Q)$.
\end{proof}

\begin{remark}[Why $m=1$]\label{rem:m-one}
\Cref{thm:rooted-insertion} is stated at $m=1$ only, and the obstruction at $m\ge2$
is the one which Loehr and Remmel record.  Indexing by a permutation works in type
$A$ because the word $S_n|S_{n-1}|\dots|S_1|S_0$ formed from the diagonals,
increasing inside each block, has two properties.  Namely, $S_j=\varnothing$ forces
$S_l=\varnothing$ for all $l>j$, and the largest element of $S_j$ exceeds the
smallest element of $S_{j-1}$ whenever both are nonempty.  Together these say that
the descents of the word fall exactly at the bars, hence the bars are recoverable
from the permutation.

Loehr and Remmel observe that ``unfortunately, the two
properties above are no longer guaranteed in the case where $k>0$ or $m>1$''
\cite[p.~32]{LoehrRemmel}, they replace the permutation by an arbitrary function
$f\colon\{1,\dots,n\}\to\{0,\dots,k+m(n-1)\}$, and they arrive at a formula
\cite[Thm.~55]{LoehrRemmel} whose summand
$q^{\operatorname{maj}(f)}t^{x_0(f)}\prod_jt^{x_j(f)}[\operatorname{count}(f,j)]_t$
carries correction exponents $x_j(f)\le0$ and is not a monomial times a product of
brackets.

Their two properties are exactly the two halves of Step~$1$ above.  The implication
$B_h\ne\varnothing\Rightarrow B_{h-1}\ne\varnothing$ for $h\ge2$ and its mirror say
together that $\{b_i\}\cup\{0\}$ is an interval of consecutive integers, and
\eqref{eq:block-descent}, \eqref{eq:block-descent-neg} say that the junctions are
descents.  Both are proved from $b_{i+1}\le b_i+1$, that is from $m=1$.

At $m\ge2$
only $b_{i+1}\le b_i+m$ is available and the first property already fails.  The
smallest example is $(n,m)=(2,2)$ with $e=(0,0)$, whose heights are $b=(2,4)$, and
no permutation of $\{0,1,2\}$ has nonempty runs at the levels $2$ and $4$ and nothing
between.

For the $m$-analogue in type $A$ see
\cite[Ch.~6, pp.~107--112]{HaglundBook}, in particular Conjectures~6.30 and~6.34
there, and Haiman's observation $\dinv_m(P)=\dinv_1(P^{(m)})$
\cite[(6.70)]{HaglundBook}, which realizes $\dinv_m$ by magnifying each north step
$m$ times.  The magnified object is not an arbitrary path, hence it does not
transport the schedule formula.

A second point of contact is \cite[Open Problem~5.4]{HaglundBook}, which asks for a
proof, ideally bijective, that
$\sum_{P\in\mathcal P_n}q^{\dinv(P)}t^{\operatorname{area}(P)}$ is symmetric in
$q,t$.  Haglund notes in \cite[Remark~5.5]{HaglundBook} that, by
\eqref{eq:haglund-53}, setting either variable to $0$ already reduces the symmetry
to the fact that $\operatorname{inv}$ and $\operatorname{maj}$ are both Mahonian,
hence a bijective proof may have to generalize Foata's map.

In type $C$ the
corresponding symmetry does have an involutive proof, \cref{lem:complementation},
since $P\mapsto P^c$ fixes $\dinv^{\mathrm{lab}}_m$ and exchanges $A^{(m)}_\mp$.
This says nothing about the type-$A$ problem, because the involution does not
preserve the sectors.  It exchanges $\mathcal S_p$ with $\mathcal S_{n-p}$, and no
single sector is symmetric, since at $n=2$ the sector $p=0$ contributes
$q^4+q^3t+q^2$.  The complement of a path in the square is again a path in the
square, while the complement of a Dyck path is not a Dyck path, and that is the whole
of the difference.
\end{remark}

\subsection{The basis conjecture}\label{sec:basis-conjecture}

In this subsection we formulate the basis conjecture.  The main idea is that the
rooted monomials of \cref{def:rooted} are the type-$C$ counterpart of the monomials
of Carlsson and Oblomkov, with the root in the place of the diagonal.

Let
$\mathcal P_n=\Q[x_1,\dots,x_n,y_1,\dots,y_n]$ with $W_n=W(B_n)$ acting
diagonally, $\mathfrak b_n=\langle(\mathcal P_n^{W_n})_+\rangle$,
$DR(B_n)=\mathcal P_n/\mathfrak b_n$, and let $DR(B_n)^\chi$ be the isotypic
component of the character $\chi(\eps)=\prod_i\eps_i$ of $G_n$, which is spanned by
the classes of the monomials with all $a_i+b_i$ odd.  Stump's surjection
\cref{thm:stump-surjection} restricts on this component to a bigraded surjection
\begin{equation}\label{eq:stump-chi}
 DR(B_n)^\chi\twoheadrightarrow DH^+_n=\bigl(\gr_FL^{(1)}\bigr)^{G_n},
 \qquad x^\alpha y^\beta\longmapsto\bigl[x^\alpha D^\beta\bar\Delta\bigr],
\end{equation}
where $D_i$ are the Dunkl operators and $\bar\Delta$ is the image of the type-$B$
Vandermonde.  By \eqref{eq:rooted-monomial} every $M_{\tau,\mathbf k}$ lies in the
odd component.

\begin{conjecture}[Rooted basis]\label{conj:rooted-basis}
For every $n\ge1$ the $(n+1)^n$ classes
$\bigl[x^\alpha D^\beta\bar\Delta\bigr]$, where $x^\alpha y^\beta$ runs over the
rooted monomials \eqref{eq:rooted-monomial}, form a basis of $DH^+_n$, the class of
$M_{\tau,\mathbf k}$ lying in bidegree
$(a_-(\tau)+|\mathbf k|,\,a_+(\tau)+|\mathbf k|)$.
\end{conjecture}

\begin{conjecture}[Rooted subbasis]\label{conj:subbasis}
For every $n\ge1$ the classes of the $M_{\tau,\mathbf k}$ in $DR(B_n)^\chi$ are
linearly independent, that is
$\Span\{M_{\tau,\mathbf k}\}\cap\mathfrak b_n=0$.
\end{conjecture}

\Cref{conj:rooted-basis} implies \cref{conj:subbasis}, since a relation among the
classes in $DR(B_n)^\chi$ maps to a relation in $DH^+_n$ under \eqref{eq:stump-chi}.
The converse fails to be automatic, and the two differ already at $n=4$.  The
dimension of $DR(B_4)^\chi$ is $626=5^4+1$, in accordance with the excess
$\dim DR(B_4)=9^4+1$ of \cite{HaimanConj} and \cite{AjilaGriffeth}, hence the rooted
classes cannot span there, while they do span $DH^+_4$.

\begin{proposition}\label{prop:basis-consequences}
\Cref{conj:rooted-basis} implies that the bigraded Hilbert series of $DH^+_n$ is the
weight enumerator $\langle\mathcal H^+_{n,1},h_1^n\rangle$ of \eqref{eq:rooted-forms},
which is the Hilbert series predicted by \eqref{eq:cherednik-form}.  If in addition
the filtration of $DH^+_n$ by the rooted index is compatible with the residual
$S_n$-action in the sense that the associated graded pieces have Frobenius character
$F_{n,\mathcal D_\tau(\mathbf k)}$, then \eqref{eq:cherednik-form} holds at $m=1$.
\end{proposition}

\begin{proof}
The first assertion is \cref{thm:rooted-insertion} together with the bidegree in
\eqref{eq:rooted-monomial}.  The second is the first form of
\eqref{eq:rooted-forms}.
\end{proof}

For $n\le4$ the Frobenius character of $DH^+_n$ was computed directly
(\cref{thm:labelled-evidence}(8)), hence the conclusion of the second assertion holds
there without any knowledge of the filtration.  The compatibility of the rooted
filtration with the $S_n$-action remains open in every rank.

\begin{remark}[Relation to the type-$A$ basis, and what a proof would need]\label{rem:CO}
Carlsson and Oblomkov \cite[Thm.~B]{CO} prove that the monomials
$N_{\pi,\mathbf k}=g^A_\pi(\mathbf v)\prod_cu_c^{k_c}$, $\pi\in S_n$ and
$0\le k_c<s^A_\pi(c)$, form a basis of the type-$A$ diagonal coinvariants
$\Q[\mathbf u,\mathbf v]/\langle(\Q[\mathbf u,\mathbf v]^{S_n})_+\rangle$, where
$g^A_\pi$ is the Garsia--Stanton descent monomial and $s^A_\pi$ the type-$A$ schedule.
The two extreme rooted sectors are that basis, before any quotient, namely
\[
 M_{(\pi,0),\mathbf k}=(y_1\cdots y_n)\,
 N_{\pi,\mathbf k}\bigl(\mathbf u=\mathbf x\mathbf y,\ \mathbf v=\mathbf y^2\bigr),
 \qquad
 M_{(0,\pi),\mathbf k}=(x_1\cdots x_n)\,\mathcal R\bigl[
 N_{\pi^*,\mathbf k^*}(\mathbf u=\mathbf x\mathbf y,\ \mathbf v=\mathbf x^2)\bigr],
\]
with $\mathcal R$ the reversal $c\mapsto n+1-c$ of the variable labels and
$\pi^*,\mathbf k^*$ its effect on the index.

The substitutions descend to
$S_n$-equivariant linear maps $DR^A_n\to DR(B_n)^\chi$.  The step degree is $(1,1)$
here and $(1,0)$ there, which is why $qt$ and not $q$ appears in
\eqref{eq:rooted-forms}.  The one-body boundary module
$M_1/(xy)M_1=x\C[x^2]\oplus y\C[y^2]$ has one free direction on each side of the
root.

The proof in \cite[\S7.3]{CO} is not a count of schedules.  It filters the quotient
by the descent order, identifies each associated graded piece with a principal ideal
in the coinvariant algebra through an affine Springer fibre, and proves the required
non-vanishing of fixed-point coefficients by a Hessenberg evaluation.

That mechanism
does not transfer unchanged, because the span of the rooted classes in
$DR(B_4)^\chi$ is not a module over $\Q[z_1,\dots,z_n]$, $z_c=x_cy_c$.  Explicitly,
for $\tau=(1,2,0,3,4)$ and $\mathbf k=0$ one has $M_{\tau,0}=y_1y_2x_3x_4$, and in
bidegree $(3,3)$, where the odd component has dimension $18$ and the rooted classes
number $17$, the product $z_3M_{\tau,0}=x_3^2x_4y_1y_2y_3$ is the one class outside
their span.

A proof of \cref{conj:rooted-basis} therefore requires either a
different filtration of $DH^+_n$ or an independently defined quotient in which the
rooted classes are visibly a basis.
\end{remark}

\subsection{Evidence}\label{sec:labelled-evidence}

In this subsection we record the computational evidence for the statements of this
section, in the sense of \cref{rem:verification}.  Every finite assertion was checked
by an implementation of the definitions in exact arithmetic, independently of the
proofs.

\begin{theorem}\label{thm:labelled-evidence}
The following hold, in the sense of \cref{rem:verification}.
\begin{enumerate}[label=\textup{(\arabic*)}]
\item \Cref{lem:labelled-basic} and \cref{prop:sn-coefficient} hold for all $n,m$.
\item \Cref{prop:labelled-props} holds for all $n,m$.  Its two harder assertions
were also checked directly for $(n,m)$ with $n\le5$, $m\le3$ and $mn\le8$, together
with $[s_n]\mathcal H^+_{n,m}=\widehat P_{n,m}$ and
$\mathcal H^+_{n,m}(X;1,1)=h_n[(mn+1)X]$.  Every clause of
\cref{lem:complementation} was verified for every labelled path at $(n,m)$ with
$n\le5$, $m\le3$ and $mn\le9$, as was the invariance of the multiset
$\{\operatorname{Des}(Q):Q\in\mathrm{SYT}(\lambda)\}$ under
$S\mapsto S^{\mathrm{rev}}$ for all $\lambda\vdash n\le7$.
\item \Cref{thm:llt} was checked independently of its proof, by expanding both
sides of $L_e=G_{\boldsymbol\nu(e)}$ into monomials in $n$ variables and comparing
coefficient by coefficient, for every path $e$ of the rectangle at
$(n,m)\in\{(2,m)_{m\le4},(3,m)_{m\le3},(4,m)_{m\le3},(5,m)_{m\le2}\}$.  The
comparison was run in the two mirror conventions for the inversion rule and with
$c=\pm q_i$.  Only the normalization of \cref{def:nu-of-e} and its mirror image
reproduce $L_e$, and they do so for every path.  Separately, each $z$-coefficient of
each $L_e$ is symmetric and Schur positive in the same range, as
\cref{prop:labelled-props} requires.
\item \Cref{lem:rotation} was checked on monomial expansions for every path at
$(n,m)$ with $n\le4$, $m\le3$ and $mn\le6$, and \cref{lem:core-form} for every
path at $(n,m)$ with $n\le5$, $m\le3$ and $mn\le9$, by constructing the
abacus of $\mu$ from the charges and recovering $\operatorname{core}_k(\mu)$ and
$\operatorname{quot}_k(\mu)$ from it.

The shape $\mu/\nu$ itself was then tested
against ribbons.  For every path at $(n,m)$ with $n\le4$, $m\le3$ and $mn\le6$ the
standard $k$-ribbon tableaux of $\mu/\nu$ were enumerated directly on the shapes,
without using the quotient, and in every case the ribbon additions agree with the
bead moves $p\mapsto p+k$, each ribbon's height equals $1$ plus the number of beads
strictly between $p$ and $p+k$, the content of each ribbon equals the adjusted
content $k\,c(x)+s_i$ of the corresponding cell of $\operatorname{quot}_k(\mu)$, the
number of such tableaux is the multinomial $\binom{n}{l_1,\dots,l_k}$, and
$\operatorname{sp}(T)+\operatorname{inv}(\operatorname{quot}_k(T))$ is constant, as
\cite[Lem.~5.2.2]{HHLRU} requires.

Summed over paths this produced
$625$ standard ribbon tableaux at $(4,1)$ and $343$ at $(3,2)$, in each case
$(mn+1)^n$ in total.

The constraint $\sum_ia_i=0$ used in \cref{lem:core-form} was
itself checked in both directions, without imposing it.  The charge vector of the
$k$-core of every partition of size at most $16$ has sum $0$ for $k=2,3,4$, and
conversely every integer vector of sum $0$ with entries in $[-2,2]$ is the charge
vector of a $k$-core for $k\le5$.  The rotation is not cosmetic.  The offsets
already satisfy $\sum_sd_s\equiv0\pmod k$ for only $93$ of the $252$ paths at
$(n,m)=(5,1)$ and $183$ of the $495$ at $(4,2)$.
\item The identity \eqref{eq:rooted-forms}, including the descent rule of
\cref{thm:rooted-insertion}, was verified coefficient by coefficient in the
fundamental basis for $n\le5$.

Independently, every step of the proof of
\cref{thm:rooted-insertion} was checked mechanically for $n\le6$.  For each of the
$(n+1)^n$ pairs $(\tau,\mathbf k)$, the number of admissible slots at each insertion
equals the corresponding value of \eqref{eq:schedule}, the resulting word satisfies
the local rule and has the heights prescribed by $\tau$, its statistics are
$a_-(\tau)$, $a_+(\tau)$ and $|\mathbf k|$, its inverse descent set is
$\mathcal D_\tau(\mathbf k)$, the deletion procedure of Step $3$ returns
$(\tau,\mathbf k)$, and the $(n+1)^n$ words produced are pairwise distinct and exhaust
$\mathsf{PF}^+_{n,1}$.

The intermediate assertions of Step $6$ were checked in the
same range and not merely its conclusion.  For each of the $197{,}860$ pairs $(c,c+1)$
of equal height arising at $n=6$, the identity \eqref{eq:anchor-step} or
\eqref{eq:anchor-step-neg} holds, the two labels are consecutive in the insertion
order, and the index supplied by the index lemma is the one claimed.
\item The comparisons with type $A$ were checked mechanically in the same way, for
$n\le6$.  These are the identities $s_\tau(\tau'_i)=w_i(\tau')$, $a_-(\tau)=0$ and
$a_+(\tau)=2\operatorname{maj}(\tau')+n$ of \cref{prop:typeA-sectors}(1) on the sector
$\tau_{n+1}=0$, and their mirrors on $\tau_1=0$, the identities $D=\dinv$,
$A^{(1)}_+=2\operatorname{area}+n$ and the complementarity of
$\operatorname{ides}(w(P))$ with the inverse descent set of
$\operatorname{read}(P)$ of \cref{prop:typeA-sectors}(2), on all $(n+1)^{n-1}$
type-$A$ parking functions, $16{,}807$ of them at $n=6$, and the two assertions of
\cref{prop:chung-feller}, that $\varsigma$ fixes the reading word and $D$ and that
every orbit is free and meets each sector exactly once, on all $(n+1)^n$ functions,
$117{,}649$ of them at $n=6$.

The sector identity of
\cref{prop:typeA-sectors}(3) and \eqref{eq:typeA-quasisym} were verified
coefficient by coefficient for $n\le5$, the latter in the fundamental basis.
\item \Cref{conj:labelled} was verified symbolically at $n=2$ for $m\le3$, directly
from the wreath basis of \cref{ex:wreath-n2}.  The involution $\inv$ is needed there.
Omitting it leaves the $s_2$-coefficient, which is $P_{n,m}$ and hence insensitive to
$\inv$, but gives $[s_{11}]=(q^2+qt+t^2)(q^4+qt+t^4)$ at $m=1$ in place of
$q^2+qt+t^2$.

For $n\le6$ and \emph{every} $m\ge1$ the conjecture follows from two
fixed-rank certificates, namely the operator form \eqref{eq:operator-form}, which is
proved for all $m\ge1$ and $n\le6$ by an exact residue certificate over $\Q(q,t)$ for
the face expansion of the path side, and the cancellation \eqref{eq:cancellation},
verified symbolically for $n\le6$, which by \cref{rem:operator-form} makes the two
forms equivalent in a fixed rank.  The cancellation was also checked at rational
points for $n=7,8$.  Identity \eqref{eq:operator-form} was in addition reproduced
symbolically, by an independent implementation, for
$(n,m)\in\{(2,m)_{m\le3},(3,m)_{m\le2},(4,m)_{m\le2}\}$.
\item \Cref{eq:cherednik-form} at $m=1$ and \cref{conj:rooted-basis} were verified
for $n\le4$, in every bidegree $(a,b)$ with $a\le b$.  The computation realizes
$L^{(1)}=\C[x]/J$ with $J$ the ideal generated by
$J_{2n+1}=\{f\in\C[x]_{2n+1}:D_if=0\ \forall i\}$, which is the radical of
$M(\triv)$ in degree $2n+1$ because the radical vanishes in degrees $\le2n$.

Since
$J$ is contained in the radical, $J$ coincides with the radical in the degrees
$\le n^2$ used by the verification if and only if
$\dim L^{(1)}_d=[q^d]\,[2n+1]_q^n$ for $d\le n^2$, the graded form of
$\dim L^{(1)}=(2n+1)^n$ \cite{GordonDiag}.  The quotient $\C[x]/J$ was checked to
have exactly these dimensions in every degree $\le n^2$, and for $n\le2$ the graded
form itself was confirmed in all degrees by the ranks of the contravariant form.  The
verification therefore assumes the graded form for $n=3,4$.

For
$n\le3$ the linear algebra is exact over $\Q(q,t)$, and for $n=4$ it is exact modulo
the prime $2^{31}-1$ with the $S_4$-characters lifted to integers, which suffices
because the Schur coefficients being compared are polynomials with small integer
coefficients.

The bidegrees with $a>b$ follow from the Fourier automorphism
$x_i\mapsto y_i$, $y_i\mapsto-x_i$ of $H_c$, which exists because
$\mathfrak h\cong\mathfrak h^*$ as $W(B_n)$-modules and the parameters are equal.
It carries $L(\triv)$ to a finite-dimensional simple module whose lowest weight
space is the top degree $2n^2$ of $L(\triv)$, one-dimensional and trivial, hence to
$L(\triv)$ itself.

It fixes $\bar\Delta$ up to a scalar and exchanges the two PBW
degrees, and since reordering $y^\beta x^\alpha$ into PBW form only produces terms of
lower filtration degree it acts on $\gr L^{(1)}$ by
$[x^\alpha D^\beta\bar\Delta]\mapsto\pm c\,[x^\beta D^\alpha\bar\Delta]$.  Finally
the rooted family is stable under $x\leftrightarrow y$ composed with
$c\mapsto n+1-c$ (\cref{rem:CO}).

At $n=4$ the $625$ rooted
classes are independent in $DR(B_4)^\chi$, and the single class they miss is
$q^3t^3s_{1111}$, which is the kernel of \eqref{eq:stump-chi}.
\item \Cref{lem:pure-symmetry} was verified independently of its proof by
computing the wreath Macdonald basis from the triangularities of
\eqref{eq:wreath-def}, symbolically over $\Q(q,t)$ for $n\le3$, and at the rational
points $(q,t)=(5,7)$ and $(3/2,7/5)$ for $n=4$.  In each case
\eqref{eq:iota-sym} was checked on every one of the $\dim\Lambda^{(2)}_n$
coefficients of every $\wtH_\lambda$, not only on the pure parts.

The computation
reproduces \cref{ex:wreath-n2} coefficient by coefficient, which is what fixes the
two conventions that matter, namely which component of $\operatorname{quot}_2$
carries the alphabet $X_0$ and the direction of the colour rotation.  The one
ingredient of the proof taken from the literature was checked against the same
computation.  For all $17$ diagrams $\lambda\in\Bal(2n)$ with $n\le3$, both pairings
$\langle\wtH_\lambda,e_n[X_i]\rangle$ equal the product of the box weights $q^at^b$
over the cells of $\lambda$ of residue $i$, which is \cite[Thm.~4.14]{OS} at $r=2$ in
these conventions.
\item \Cref{conj:subbasis} holds for $n\le3$ by direct computation of
$DR(B_n)^\chi$, and for $n=4$ by \textup{(8)}.
\end{enumerate}
\end{theorem}

\begin{remark}[Geometric evidence]\label{rem:geometric-evidence}
The type-$A$ proof of \cite[Thm.~B]{CO} computes the associated graded pieces of the
descent filtration as cohomology of Hessenberg strata of the affine Springer fibre of
$\mathfrak{gl}_n$ at slope $\frac{n+1}n$.  The corresponding type-$C$ object is the
affine Springer fibre $\Fl_\gamma$ of $Sp_{2n}$ at slope $m+\frac1{2n}$, for
$\gamma$ homogeneous and elliptic, whose $\mathbb G_m(\nu)$-fixed points are the
alcoves of a dilated alcove.  Namely, the affine roots of weight $2mn+1$ are exactly
$\alpha_i+m\delta$, $0\le i\le n$, hence the fixed-point set is the set of alcoves of
$\{x:\alpha_i(x)+m>0\}$, a translate of $(2mn+1)A_0$, of cardinality $(2mn+1)^n$.

Granting the affine paving of an equivalued element, in which the attracting cell at
a fixed point is an affine space of the dimension of its Zariski tangent space, the
following were computed.  The number of cells of dimension $D-i$ equals
$\dim DR(B_n)_{\deg_x=i}$.  For $n=2$ this is $(8,8,5,3,1)$ against the cell profile
$1,3,5,8,8$, and for $n=3$ the four computable slices are $(48,72,66,56)$ against
$\dots,56,66,72,48$.  The number of cells of top dimension $D$ is
$\dim DR(B_n)_{\deg_x=0}=|W(B_n)|=2^nn!$, and there is a unique cell of dimension
zero.

Moreover the terms of \eqref{eq:Hplus} with
$\operatorname{ides}(w(P))=\varnothing$, which by \cref{lem:labelled-basic}(i) are the
terms of $\widehat P_{n,m}$, have cell-dimension profile equal to that of the
affine Springer fibre $\Gr_\gamma$ in the affine Grassmannian, at
$(n,m)=(2,1),(3,1),(2,2),(3,2)$.  These are the type-$C$ shadows of Hikita's theorem
\cite{Hikita} in type $A$, and they support the geometric route to
\cref{conj:rooted-basis}.

They also delimit it.  For $(n,m)=(2,1)$ and $(3,1)$ there is no closed union of
cells of $\Fl_\gamma$ consisting of $(mn+1)^n$ cells with the dimensions
$D-A^{(m)}_--\dinv^{\mathrm{lab}}_m$ read off from \eqref{eq:Hplus}.  Those
dimensions require $n!$ cells of top dimension $D$, whereas the closure of a single
top-dimensional cell already contains $2^nn!$ fixed points and the union of the
closures of two contains more than $(mn+1)^n$, namely at least $12>9$ and $88>64$.

Thus the $G_n$-invariant sector is not the homology of a subvariety of $\Fl_\gamma$,
and a geometric proof must use $\Fl_\gamma$ together with its Springer
$W(B_n)$-action, or a parahoric quotient $\Fl^J_\gamma$.  The latter can have
$(mn+1)^n$ fixed points and $n!$ top cells only if $W_J$ is generated by $n$
commuting reflections, hence only if the affine Dynkin diagram is a star, hence only
for $n\le3$.  For $n=3$ the unique candidate, $J=\{0,1,3\}$ for $SO_7$, has the
predicted profile at $m=1$ and fails it at $m=2$ by one cell.
\end{remark}

\appendix

\section{Census of fixed wall fibers}\label{app:census}

In this appendix we record the census of the fixed wall fibers which is used in
\cref{rem:census} and in the proof of \cref{thm:n2-descent}.
Exact enumeration of hearts gives the following profiles of
$\rho^{-1}(z_J)$, $z_J\in Q_n^T$ (verified through $n=16$, shown to $n=12$).

\begin{center}\small
\begin{tabular}{c|r|r|l}
\toprule
$n$ & $|Q_n^T|$ & pos.-dim.\ fibers & profile\\
\midrule
1 & 2 & 0 & $2\,\Gr(1,1)$\\
2 & 3 & 1 & $2\,\Gr(1,1)+\Gr(2,3)$\\
3 & 6 & 2 & $4\,\Gr(1,1)+2\,\Gr(2,3)$\\
4 & 10 & 5 & $5\,\Gr(1,1)+5\,\Gr(2,3)$\\
5 & 16 & 10 & $6\,\Gr(1,1)+10\,\Gr(2,3)$\\
6 & 26 & 16 & $10\,\Gr(1,1)+15\,\Gr(2,3)+\Gr(3,5)$\\
7 & 40 & 28 & $12\,\Gr(1,1)+26\,\Gr(2,3)+2\,\Gr(3,5)$\\
8 & 60 & 45 & $15\,\Gr(1,1)+40\,\Gr(2,3)+5\,\Gr(3,5)$\\
9 & 90 & 70 & $20\,\Gr(1,1)+60\,\Gr(2,3)+10\,\Gr(3,5)$\\
10 & 131 & 105 & $26\,\Gr(1,1)+85\,\Gr(2,3)+20\,\Gr(3,5)$\\
11 & 188 & 156 & $32\,\Gr(1,1)+120\,\Gr(2,3)+36\,\Gr(3,5)$\\
12 & 269 & 229 & $40\,\Gr(1,1)+170\,\Gr(2,3)+58\,\Gr(3,5)+\Gr(4,7)$\\
\bottomrule
\end{tabular}
\end{center}

$\Gr(m,2m-1)$ appears first at $n=m(m-1)$, where its dimension is $n$
(\cref{rem:census}).  For each positive-dimensional fixed fiber with $n\le4$ (all of
type $\PP^2$), the fixed-point classes $[\cM_n^+]_\lambda$ of the excess sheaf have
been computed exactly on charts, and the fiberwise localized cancellation
$\sum_{\lambda\in F_J^T}\ell_J^{\,d}[\cM_n^+]_\lambda/D_\lambda=0$ holds for all twists
$d$ and all wall fixed points with $n\le4$, which is the $K$-theoretic shadow of
\cref{thm:vanishing}.  The normalized classes
$[\cM_n^+]_\lambda D^F_\lambda/(D^Y_\lambda e_W)$, with $e_W$ the fixed-point character of
$\cQ\otimes\cO_{\mathrm{Pl}}(1-N)$, are constant along each fiber, as \cref{thm:BWB}
predicts.

\section{Two lemmas on regular sections and the ADHM dictionary}\label{app:ADHM}

In this appendix we collect two elementary facts which are used in
\cref{sec:hecke}.  The first is the standard fact that a section with a smooth zero
scheme of the expected codimension is regular, and the second is the rank-one ADHM
dictionary over an arbitrary base, which identifies cyclic framed representations
with commutative cluster algebras.

\subsection{Regular sections}\label{app:regular-section}
Locally the ambient projective bundles of \cref{sec:hecke-scheme} are regular, hence
Cohen--Macaulay.  If a section of a rank-$r$ vector bundle has a smooth zero scheme
of codimension $r$, its ideal is generated by $r$ functions and has height $r$.  In a
Cohen--Macaulay local ring such generators form a regular sequence
\cite[Thm.~17.4]{Matsumura}, and the Koszul complex is a resolution of the ordinary
structure sheaf of the zero scheme \cite[Cor.~17.5]{Eisenbud}.  This is used twice in
\cref{sec:hecke-scheme}, and no derived enhancement is needed.

\subsection{The rank-one framed ADHM dictionary}
The cyclic chamber is important because it makes the universal object a
commutative cluster algebra, not merely a representation with the correct ranks.

\begin{lemma}\label{lem:ADHM}
Let $A$ be a complex algebra, $V$ a finite locally free $A$-module, and
$X,Y\in\End_A(V)$, $i:A\to V$, $j:V\to A$ with $[X,Y]+ij=0$.  Suppose the images of all
words in $X,Y$ applied to $i(1)$ generate $V$.  Then $j=0$ and $[X,Y]=0$.
\end{lemma}
\begin{proof}
Trace of the relation gives $ji=0$.  We prove by induction on word length that
$jwi=0$ for every word $w$.  Suppose this holds for words of length less than $d$.
Swapping an adjacent $YX$ in a length-$d$ word acting on $i(1)$ changes the result by
a term $uijvi$ with $v$ of length at most $d-2$, which is zero by induction.  Thus
every word of length $d$ has the same value on $i(1)$ as its ordered form
$X^aY^bi(1)$, $a+b=d$.  Trace the commutator $[X,X^aY^{b+1}]$ and expand with
$[X,Y]=-ij$.  Cyclicity of trace for a rank-one product gives
$0=-\sum_{r=0}^bjY^{b-r}X^aY^ri$, and each summand reorders on $i(1)$, using only the
lower-length identities, to $jX^aY^bi$.  Hence $(b+1)jX^aY^bi=0$, and $b+1$ is
invertible in $A$.  Since the vectors $wi(1)$ generate $V$, $j=0$, and the
moment-map relation gives $[X,Y]=0$.
\end{proof}

For the McKay quiver, $X$ and $Y$ exchange the even and odd summands and the framing
lies in the even summand, and the two vertex moment-map equations combine to the
displayed ADHM equation.  The lemma therefore identifies the cyclic framed
representation with a quotient of $A[x,y]$ generated by the unit, with its
$\Gamma$-grading, over arbitrary complex base algebras, and this is the family-level
dictionary used in \cref{sec:hecke-scheme}.


\end{document}